\documentclass[12pt]{amsart}
\pdfoutput=1
\usepackage{amsfonts,amssymb, amscd, stmaryrd, mathrsfs}
\usepackage{latexsym, graphicx, psfrag, color, xcolor, float}
\usepackage[normalem]{ulem}

\usepackage[all, knot, poly]{xy}
\xyoption{knot}

\usepackage{times}
\usepackage{bbm}

\usepackage[bookmarks=true, bookmarksopen=true,%
bookmarksdepth=3,bookmarksopenlevel=2,%
colorlinks=true,%
linkcolor=blue,%
citecolor=blue,%
filecolor=blue,%
menucolor=blue,%
urlcolor=blue]{hyperref}

\usepackage{pinlabel}

\definecolor{darkgreen}{rgb}{0.0, 0.7, 0.0}
\definecolor{purple}{rgb}{0.5, 0.0, 0.5}
\definecolor{brown}{rgb}{0.7, 0.4, 0.0}
\definecolor{greenblue}{rgb}{0.0, 0.7, 0.7}
\definecolor{figuregreen}{rgb}{0, 0.651, 0.318}

\usepackage[draft]{say}

\numberwithin{equation}{section}
\numberwithin{figure}{section}

\makeatletter
\def\@tocline#1#2#3#4#5#6#7{\relax
  \ifnum #1>\c@tocdepth 
  \else
    \par \addpenalty\@secpenalty\addvspace{#2}%
    \begingroup \hyphenpenalty\@M
    \@ifempty{#4}{%
      \@tempdima\csname r@tocindent\number#1\endcsname\relax
    }{%
      \@tempdima#4\relax
    }%
    \parindent\z@ \leftskip#3\relax \advance\leftskip\@tempdima\relax
    \rightskip\@pnumwidth plus4em \parfillskip-\@pnumwidth
    #5\leavevmode\hskip-\@tempdima
      \ifcase #1
       \or\or \hskip 1em \or \hskip 2em \else \hskip 3em \fi%
      #6\nobreak\relax
    \dotfill\hbox to\@pnumwidth{\@tocpagenum{#7}}\par
    \nobreak
    \endgroup
  \fi}
\makeatother

\makeatletter
\let\origsubsection\subsection
\renewcommand\subsection{\@ifstar{\starsubsection}{\nostarsubsection}}

\newcommand\nostarsubsection[1]
{\subsectionprelude\origsubsection{#1}\subsectionpostlude}

\newcommand\starsubsection[1]
{\subsectionprelude\origsubsection*{#1}\subsectionpostlude}

\newcommand\subsectionprelude{%
  \vspace{1em}
}

\newcommand\subsectionpostlude{%
 $ $ \vspace{1em}

}

\let\origsubsubsection\subsubsection
\renewcommand\subsubsection{\@ifstar{\starsubsubsection}{\nostarsubsubsection}}

\newcommand\nostarsubsubsection[1]
{\subsubsectionprelude\origsubsubsection{#1}\subsubsectionpostlude}

\newcommand\starsubsubsection[1]
{\subsubsectionprelude\origsubsubsection*{#1}\subsubsectionpostlude}

\newcommand\subsubsectionprelude{%
  \vspace{1em}
}

\newcommand\subsubsectionpostlude{%
 $ $ \vspace{1em}

}
\makeatother

\newtheorem{dummy}{dummy}[section]
\newtheorem{lemma}[dummy]{Lemma}
\newtheorem{theorem}[dummy]{Theorem}

\newtheorem{corollary}[dummy]{Corollary}
\newtheorem{proposition}[dummy]{Proposition}

\theoremstyle{definition}
\newtheorem{definition}[dummy]{Definition}
\newtheorem{convention}[dummy]{Convention}
\newtheorem{notation}[dummy]{Notation}
\newtheorem{example}[dummy]{Example}
\newtheorem{remark}[dummy]{Remark}

\newcommand{\AugBC}{\mathcal{A}ug_-}

\newcommand{\bF}{\mathbb{F}}

\newcommand{\bR}{\mathbb{R}}
\newcommand{\bZ}{\mathbb{Z}}

\newcommand{\cA}{\mathcal{A}}

\newcommand{\cC}{\mathcal{C}}

\newcommand{\cF}{\mathcal{F}}
\newcommand{\cG}{\mathcal{G}}

\newcommand{\cL}{\mathcal{L}}

\newcommand{\cR}{\mathcal{R}}
\newcommand{\cS}{\mathcal{S}}
\newcommand{\cT}{\mathcal{T}}

\newcommand{\hh}{\mathfrak{h}}

\newcommand{\sS}{\mathcal{S}}
\newcommand{\rR}{\mathcal{R}}
\newcommand{\tT}{\mathcal{T}}

\newcommand{\Ext}{\mathrm{Ext}}

\newcommand{\Hom}{\mathrm{Hom}}

\renewcommand{\SS}{\mathit{SS}}

\newcommand{\redact}[1]{}

\newcommand{\cross}[9]{\xymatrix{
{}\ar@{.}[ddrr]&#4&\ar@{.}[ddll]\\
#2\ar[ur]^{#7}&{}&#3\ar[ul]_{#8}\\
{}&#1\ar[ul]^{#5}\ar[ur]_{#6}\ar@{~>}[uu]_{#9}&{}\\
}}

\newcommand{\e}{\epsilon}

\newcommand{\coeffs}{\mathbbm{k}}
\newcommand{\field}{\mathbbm{k}}
\newcommand{\alg}{\mathcal{A}}
\newcommand{\dd}{\partial}
\newcommand{\lin}{\operatorname{lin}}
\newcommand{\Aug}{\mathcal{A}ug_+}
\newcommand{\Augpm}{\mathcal{A}ug_\pm}
\renewcommand{\hom}{\Hom_+}
\newcommand{\homBC}{\Hom_-}
\newcommand{\hompm}{\Hom_\pm}
\newcommand{\LCH}[1]{H^{#1}\Hom_+}
\newcommand{\LCHBC}[1]{H^{#1}\Hom_-}
\newcommand{\LCC}[1]{\Hom_+^{#1}}
\newcommand{\LCCBC}[1]{\Hom_-^{#1}}
\newcommand{\LCHpm}[1]{H^{#1}\Hom_\pm}
\newcommand{\LCCpm}[1]{\Hom_\pm^{#1}}
\newcommand{\dga}{DGA}
\newcommand{\dgas}{DGAs}

\newcommand{\hfp}{\Hom_{Fuk_\epsilon}}
\newcommand{\hfm}{\Hom_{Fuk_{-\epsilon}}}

\newcommand{\R}{\mathbb{R}}

\newcommand{\zz}{\mathbb{Z}}

\newcommand{\strip}{\bigsqcup\!\!\!\!\!\!\bigsqcap}

\newcommand{\Bmu}{\mbox{$\raisebox{-0.59ex}
  {$l$}\hspace{-0.18em}\mu\hspace{-0.88em}\raisebox{-0.98ex}{\scalebox{2}
  {$\color{white}.$}}\hspace{-0.416em}\raisebox{+0.88ex}
  {$\color{white}.$}\hspace{0.46em}$}{}}

\newcommand{\crossk}{  {}^k \!\! \mathrel{\vcenter{\offinterlineskip%
  \hbox{$-$}\vskip-.5ex \hbox{$\times$}\vskip-.5ex\hbox{$-$}}} }

\begin{document}

\title[Augmentations are sheaves]{Augmentations Are Sheaves}

\author{Lenhard Ng}
\address{Lenhard Ng \\ Department of Mathematics \\ Duke University}
\email{ng@math.duke.edu}

\author{Dan Rutherford}
\address{Dan Rutherford \\ Department of Mathematical Sciences \\ Ball State University}
\email{rutherford@bsu.edu}

\author{Vivek Shende}
\address{Vivek Shende \\ Department of Mathematics \\ UC Berkeley}
\email{vivek@math.berkeley.edu}

\author{Steven Sivek}
\address{Steven Sivek \\ Department of Mathematics \\ Imperial College London}
\email{s.sivek@imperial.ac.uk}

\author{Eric Zaslow}
\address{Eric Zaslow \\ Department of Mathematics \\ Northwestern University}
\email{zaslow@math.northwestern.edu}

\begin{abstract}
We show that the set of augmentations of the Chekanov--Eliashberg algebra of a Legendrian link
underlies the structure of a unital A-infinity category.  This differs from the non-unital category constructed
in \cite{BC}, but is related to it in the same way that cohomology is related to compactly supported cohomology.
The existence of such a category  was predicted by \cite{STZ}, who moreover conjectured its
equivalence to a category of sheaves on the front plane with singular support meeting infinity in the
knot.  After showing that the augmentation category forms a sheaf over
the $x$-line, we are able to prove this conjecture by calculating both categories on thin slices of the front plane.  In
particular, we conclude that every augmentation comes from geometry.
\end{abstract}

\maketitle

\thispagestyle{empty}

{\small \tableofcontents}


\section{Introduction}

A powerful modern approach to studying a Legendrian submanifold $\Lambda$ in a contact manifold $V$
is to encode Floer-theoretic data into a
differential graded algebra
$\alg(V, \Lambda)$, the Chekanov--Eliashberg \dga{}.
The generators of this algebra are indexed by
Reeb chords; its differential counts holomorphic disks
in the symplectization $\R\times V$ with boundary lying along the 
Lagrangian $\R\times \Lambda$ and
meeting the Reeb chords at infinity \cite{Eli,EGH}.
Isotopies of Legendrians induce homotopy equivalences of algebras, and
the homology of this algebra is called Legendrian contact homology.

A fundamental insight of Chekanov \cite{C} is that, in practice, these homotopy equivalence classes of infinite
dimensional algebras can often be distinguished by the techniques of algebraic geometry.
For instance,  the {\em functor of points}
\begin{eqnarray*}
\mbox{fields} & \to & \mbox{sets} \\
\coeffs & \mapsto & \{\mbox{\dga{} morphisms $\alg(V, \Lambda) \to \coeffs$}\}/\mbox{\dga{} homotopy}
\end{eqnarray*}
is preserved by homotopy equivalences of algebras $\alg(V, \Lambda)$ 
\cite[Lem. 26.3]{FHT}, and thus furnishes an invariant.
Collecting together the linearizations (``cotangent spaces'')
$\ker \e / (\ker \e)^2$ 
of the augmentations (``points'') $\e : \alg(V, \Lambda) \to \coeffs$ gives a stronger invariant:
comparison of these linearizations as differential graded vector spaces is one way that Legendrian knots have been distinguished in practice since the work of Chekanov.

As the structure coefficients of the \dga{} $\alg(V, \Lambda)$ come from the contact geometry of $(V, \Lambda)$,
it is natural to ask for direct contact-geometric interpretations of the algebro-geometric constructions above, and in particular
to seek the contact-geometric meaning of the -- a priori, purely algebraic -- augmentations.
In some cases, this meaning is known.
As in topological field theory, exact Lagrangian cobordisms
between Legendrians give rise (contravariantly) to morphisms of the
corresponding
\dgas{} \cite{EGH, Ekh09, EHK}.
In particular, exact Lagrangian fillings are cobordisms from the empty
set, and so give augmentations.

However, not all augmentations arise in this manner.  Indeed, consider
pushing an exact filling surface $L$ of a Legendrian knot $\Lambda$ in the Reeb direction: 
on the one hand, this is a deformation of $L$ inside $T^*L$,
and so intersects $L$ -- an exact Lagrangian --
in a number of points which, counted with signs,
is $-\chi(L)$.  On the other hand, this intersection can be computed
as the linking number at infinity, or in other words, the
Thurston--Bennequin number of $\Lambda$:
$tb(\Lambda) = -\chi(L)$.
Now there is a Legendrian figure eight knot 
with $tb=-3$ (see e.g. \cite{atlas} for this
and other examples); its \dga{} has augmentations, and yet any
filling surface would necessarily have genus $-2$. 

This obstruction has a categorification, originally due to Seidel and made
precise in this context by Ekholm \cite{Ekh09}.
Given an exact filling $(W, L)$ of $(V, \Lambda)$ (where we will
primarily focus on the case $V=\bR^3$ and $W=\bR^4$), consider
the Floer homology $HF_t(L, L)$, where the differential accounts only for disks
bounded by $L$ and a controlled Hamiltonian perturbation
of $L$ for time $<t$, i.e. loosely those disks with action bounded by $t$.
There is an inclusion $HF_{-\varepsilon}(L, L) \to HF_{\infty}(L, L)$.
The former has generators given by self-intersections of $L$ with a small
perturbation of itself, and the latter has generators given by these
together with Reeb chords of $\Lambda$.  The quotient
of these chain complexes 
leads to what is called ``linearized contact cohomology'' in the literature;
for reasons to be made clear shortly,
we write it as $\Hom_-(\e, \e)[1]$.  That is, we have:
\begin{equation} \label{eq:ekholm-seidel} HF_{-\varepsilon}(L, L) \to HF_{\infty}(L, L) \to \Hom_-(\e, \e)[1]  \xrightarrow{1}. \end{equation}
Finally, since the wrapped Fukaya category of $\bR^4$ is trivial, we get an isomorphism
$\Hom_-(\e, \e) \cong HF_{-\varepsilon}(L, L) $.  On the other hand, $HF_{-\varepsilon}(L, L) \cong H^*_c(L;\coeffs)$.
In particular, taking Euler characteristics recovers:
$$-tb(\Lambda) =  \chi(\Hom_-(\e, \e)) =  \chi( H^*_c(L;\coeffs)) = \chi(L).$$

One could try to construct the missing augmentations from more general objects in the derived Fukaya category.
To the extent that this is possible, the  above sequence implies that the {\em categorical} structures present in the
symplectic setting should be visible on  
the space of augmentations.
An important step in this direction
was taken by Bourgeois and Chantraine \cite{BC}, who define a \emph{non-unital}
$A_\infty$ category which we denote $\AugBC$.  Its objects are augmentations of the Chekanov--Eliashberg \dga{}, and its hom spaces $\homBC(\e, \e')$
have the property that the
self Homs are the linearized contact cohomologies.
The existence of this category
was strong evidence that augmentations could indeed be built from geometry.

On the other hand, when $V = T^{\infty} M$ is the cosphere bundle over a manifold, $\Lambda \subset V$ is a 
Legendrian, and $\coeffs$ is a field,
a new source of Legendrian invariants is provided by the category $Sh(M,\Lambda; \coeffs)$ of constructible sheaves of $\coeffs$-modules
on $M$ whose singular support meets $T^\infty M$ in $\Lambda$ \cite{STZ}. The introduction of this category is motivated by the
microlocalization equivalence of the category of sheaves on a manifold with the infinitesimally
wrapped Fukaya category of the cotangent bundle \cite{NZ, N}:
$$\mu: Sh(M; \coeffs) \xrightarrow{\sim} Fuk_\varepsilon(T^*M; \coeffs).$$
In particular, to a Lagrangian brane $L \subset T^* M$ ending on $\Lambda$, there corresponds a sheaf $\mu^{-1}(L)$ with the property that
$$\Hom_{Sh(M)}( \mu^{-1}(L), \mu^{-1}(L) ) = \Hom_{Fuk_\e(T^* M)}(L, L) = HF_{+\varepsilon}(L, L) = H^*(L; \coeffs),$$
and we write $Sh(M, \Lambda; \coeffs) := \mu^{-1}(L)$.

Like ordinary cohomology, the category $Sh(M, \Lambda; \coeffs)$ is unital; like compactly supported cohomology,
the Bourgeois--Chantraine augmentation category $\AugBC(\Lambda; \coeffs)$ is not.
In an augmentation category matching the sheaf category, the Hom spaces would fit
naturally into an exact sequence
\begin{equation} \label{eq:plus} HF_{+\varepsilon}(L, L) \to HF(L, L) \to \hom(\e, \e)[1]  \xrightarrow{1}. \end{equation}

Together these observations suggest the following modification to the Bourgeois--Chantraine construction.
As noted in \cite{BC}, $\AugBC$ can be defined from the $n$-copy of the Legendrian, ordered
with respect to the displacement in the Reeb direction.  To change the
sign of the perturbations,
in the front diagram of the Legendrian we
re-order the $n$-copy from top to bottom, instead of from bottom to top.  The first main result of this article, established
in Sections \ref{sec:augcatalg} and \ref{sec:augcat},
is that doing so yields a {\em unital} $A_\infty$ category.

\begin{theorem}[see Definition~\ref{def:augplus-xy-xz} and Theorem~\ref{prop:Invariance}]
Let $\Lambda$ be a Legendrian knot or link in $\R^3$.
We define a unital $A_\infty$ category $\Aug(\Lambda; \coeffs)$ whose objects are \dga{} maps
$\e : \alg(\R^3, \Lambda) \to \coeffs$, i.e., augmentations.
This category is invariant up to $A_\infty$ equivalence under Legendrian isotopies of $\Lambda$.
\end{theorem}

It turns out that the cohomology $H^*\hom(\epsilon,\epsilon)$ of the self-hom
spaces in $\Aug(\Lambda;\coeffs)$ is exactly (up to a grading shift)
what is called linearized Legendrian contact homology in the
literature; see Corollary~\ref{cor:hom-is-LCH}.  Moreover, if
$\Lambda$ is a knot with a single base point, then two objects of
$\Aug(\Lambda;\coeffs)$ are isomorphic in the cohomology category
$H^*\Aug$ if and only if they are homotopic as \dga{} maps
$\alg(\R^3,\Lambda) \to \coeffs$; see
Proposition~\ref{prop:Homotopy}. In particular, it follows from work
of Ekholm, Honda, and K\'alm\'an \cite{EHK} that augmentations
corresponding to isotopic exact fillings of $\Lambda$ are isomorphic.

There is a close relation  between $\AugBC(\Lambda)$ and $\Aug(\Lambda)$.  Indeed,
our construction gives both, and a morphism from one to the other.  We investigate these in Section \ref{sec:augprops}, and find:

\begin{theorem}[see Propositions~\ref{prop:nu-morphism},
  \ref{prop:pmexactseq}, and \ref{prop:duality}]
There is an $A_\infty$ functor $\AugBC \to \Aug$ carrying every augmentation to itself.
\label{thm:properties}
On morphisms, this functor extends to an exact triangle
$$\homBC(\e, \e') \to \hom(\e, \e') \to H^*(\Lambda; \coeffs) \xrightarrow{[1]}. $$
Moreover, there is a duality
$$\hom(\e, \e') \cong \homBC(\e', \e)^\dag[-2].$$
Here, the $^\dag$ denotes the cochain complex dual of a cochain
complex,
i.e., the underlying vector space is
dualized, the differential is transposed, and the degrees are negated.
\end{theorem}

When $\e=\e'$, Sabloff \cite{Sabloff} first constructed this duality, and
the exact sequence in this case is given in \cite{EESab}.
When the augmentation comes
from a filling $L$, the duality is Poincar\'e duality, and
the triangle is identified with the long exact sequence in cohomology
$$H^*_c(L; \coeffs) \to H^*(L; \coeffs) \to H^*(\Lambda; \coeffs) \xrightarrow{[1]}.$$
That is, there is a map of triangles \eqref{eq:ekholm-seidel} $\to$ \eqref{eq:plus}, so that the connecting homomorphism
identifies the inclusion $\homBC(\e, \e) \to \hom(\e, \e)$ with the inclusion
$HF_{-\varepsilon}(L, L) \to HF_{+\varepsilon}(L, L)$.


The category $\Aug$ in hand, we provide the hitherto elusive connection between augmentations
and the Fukaya category.  We write $\cC_1(\Lambda;\coeffs)\subset
Sh(\bR^2,\Lambda; \coeffs)$ for the sheaves with ``microlocal rank one'' along $\Lambda$,
and with acyclic stalk when $z \ll 0$.

\begin{theorem}[see Theorem~\ref{thm:main}] \label{thm:intro-main}
Let $\Lambda \subset \R^3$ be a Legendrian knot, and let $\mathbbm{k}$ be a field.  Then there is an $A_\infty$ equivalence of categories
$$\Aug(\Lambda; \mathbbm{k}) \xrightarrow{\sim} \cC_1(\Lambda; \mathbbm{k}).$$
\end{theorem}

\noindent
Via the equivalence between constructible sheaves and the Fukaya category, we view
this theorem as asserting that all augmentations come from geometry.  In total, we
have a host of relations among categories of sheaves, Lagrangians and augmentations.
These are summarized in Section \ref{sec:exseq}.

We remark that the first four authors \cite{NRSS} have shown that the groupoid of isomorphisms in the truncation $\pi_{\geq 0}\Aug(\Lambda; \bF_q)$ has homotopy cardinality $q^{tb(\Lambda)/2} R_\Lambda(q^{1/2}-q^{-1/2})$, where $R_\Lambda(z)$ is the ruling polynomial of $\Lambda$; thus the same is true of $\cC_1(\Lambda; \bF_q)$, resolving Conjecture~7.5 of \cite{STZ}.

The Bourgeois--Chantraine category ${\mathcal Aug}_-(\Lambda;  \mathbbm{k})$
can also be identified with a category of sheaves.  If we define
$\Hom_{Sh_-}(\cF, \cG) := \Hom_{Sh}(\cF, r_{-\varepsilon}^* \cG)$
where $r_t$ is the front projection of Reeb flow, then there is a non-unital dg category $Sh_{-}(\mathbb{R}^2, \Lambda; \mathbbm{k})$ whose
morphism spaces are $\Hom_{Sh_-}$.  We write $\cC_1^{(-)}$ for the sheaves with ``microlocal rank one'' along $\Lambda$
and with acyclic stalk when $z \ll 0$.  Similar arguments (which we do
not give explicitly in this paper) yield an equivalence
${\mathcal Aug}_{-}(\Lambda;  \mathbbm{k}) \xrightarrow{\sim} \cC_1^{(-)}(\Lambda;  \mathbbm{k})$.
Further properties and relations to existing constructions are discussed in Section \ref{sec:augprops}.

\vskip 0.2in
\noindent{\bf Summary of the Paper}
\vskip 0.2in
The preceding gives an account of the main results of this paper and their relevance
to the study of Legendrian knots.  Since much of the remainder of the paper is technical,
a straightforward summary in plain English may be helpful for the casual reader, or as a reference 
for those who get lost in the weeds.  We address only topics not already discussed above.

To create a category whose objects are augmentations, we must define the
morphisms and compositions.
At first glance, there seems to be little to do beyond 
adapting the definitions
that already appear in the work of Bourgeois and Chantraine \cite{BC}
to account for the reversal in ordering link components.
Yet there is an important distinction.  In ordering the perturbations
as we do,
we are forced to consider the presence of ``short'' Reeb chords,
traveling from the original Legendrian
to its perturbation.
These short chords were also considered in \cite{BC} and indeed
have appeared in a number of papers in contact topology; however,
Bourgeois and Chantraine ultimately do not need them to formulate their
augmentation category, whereas they are crucial to our formulation.

The
higher products in the augmentation category involve multiple perturbations and counts of disks bounding chords -- including
short chords --
traveling between the different perturbed copies.
The way to treat this scenario is to consider the Legendrian and its perturbed copies as a single link,
then to encode the data of which copies the chords connect with the notion of a
``link grading'' \cite{Mishachev}.
So we must consider the \dga{} of a link constructed
from a number of
copies of an original Legendrian, each with different perturbations --- and we must repeat
this construction for each natural number to define all the different higher products in the
$A_\infty$ category.  As the different products must interact with one another according
to the $A_\infty$ relations, we must organize all these copies and perturbations and \dgas{}
coherently, leading to the notion of a consistent sequence of \dgas{}.
We provide this definition and show that a consistent sequence of \dgas{} with a link
grading produces an $A_\infty$ category, which in the case described above
will be the augmentation category $\Aug(\Lambda)$.
To keep these general algebraic aspects distinct from the specific application, we
have collected them all in Section \ref{sec:augcatalg}.

In Section~\ref{sec:augcat}, we construct consistent sequences of
\dgas{} for Legendrian knots $\Lambda$ in $\bR^3$, resulting in the
category $\Aug(\Lambda)$.
It is important to note that the consistent sequence of \dgas{} that
we construct for a Legendrian knot
does
not apply to Legendrians in higher than one dimension: see Remark~\ref{rem:why-1d} for some brief discussion of this.
Accordingly, as distinct from the category of Bourgeois and Chantraine,
a general version of our category in higher dimensions would not be algebraically determined by the \dga{} of
the Legendrian in general, although we show that it is for
one-dimensional knots
(see Proposition \ref{prop:complete-aug-description}).
Another complication in the definition of the category
is that it includes ``base points,'' additional
generators of the \dga{} which are needed both for the comparison to sheaves,
i.e. to reduce \dga{} computations to purely local considerations,
and in order to prove independence of perturbation.
We have so far been vague about what ``perturbation'' means.
We can perturb a Legendrian in a $1$-jet bundle with a Morse function, and we do this,
but we might also take a copy of the front projection translated in the Reeb direction, and then
use the resolution procedure \cite{NgCLI}.  (If one simply translates
a Lagrangian projection in the Reeb direction, every point in the projection would correspond to a chord!)  Of course, one wants
to show independence of choices as well as invariance of the category
under Legendrian isotopy, all up to
$A_\infty$ equivalence.  This is done in Theorem \ref{prop:Invariance}.
The reader who wants to see how the definition plays out in explicit examples is
referred to Section \ref{ssec:exs}. We then establish a number of
properties of $\Aug$ in Section \ref{sec:augprops}, including the
exact triangle and duality stated in Theorem~\ref{thm:properties}.

With the category in hand, we are in a position to compare with sheaves.
Of course, Fukaya--Floer type categories are non-local, depending as they do on
holomorphic disks which may traverse large distances.  Sheaves, on the other hand,
are local.  Comparison is made possible by the bordered construction of the \dga{}
\cite{sivek-bordered}, where locality is proven:  the \dga{} of the union of two sets is determined
by the two pieces and their overlap.  These results are reviewed and extended
for the present application in Section \ref{sec:bordered}.  The idea of the
bordered construction is simple:  holomorphic disks exiting a
vertical strip would do so along a chord connecting two strands.  By including
such chords in the definition of the bordered algebra one shows that the \dga{} of
a diagram glued from a left half and a right half is the pushout of the \dga{} of
the two halves over the algebra of purely horizontal strands.

Now once we put the front diagram in plat position and slice it into horizontal strips,
we can apply the bordered construction and achieve locality as discussed above.
Since sheaves are by definition 
local --- this is the sheaf axiom --- we are in a position
to compare the two categories, and can do so strip by strip.  We can further prepare the
strips so that each is 
one
of the following four possibilities:  no crossings or cusps, one crossing,
all the left cusps, all the right cusps.
Note that
to ensure that the gluings are themselves compatible,
we also must compare the restriction functors from these cases to the left and right
horizontal strip categories.  Interestingly, while all these cases are dg categories, the
restriction functors are only equivalent as $A_\infty$ functors, and this accounts for
the difference in the glued-together categories at the end:  the augmentation category
is $A_\infty$ and the sheaf category is dg.
All these equivalences and compatibilities are shown in Section \ref{sec:aug=sh}.
The case-by-case nature means the proof is somewhat lengthy, but it is straightforward.
And that's it.

\begin{remark} \label{rem:why-1d}
It should be possible to construct the augmentation category
$\Aug(\Lambda)$ for general Legendrians in arbitrary (and in
particular higher-dimensional) $1$-jet spaces. 
Here we explain why we restrict ourselves in this paper to the setting
of $J^1(\R)$, in contrast to Bourgeois and Chantraine's more general
treatment of $\AugBC(\Lambda)$ in \cite{BC}.
The consistent
sequences of \dgas{} mentioned in the above summary are constructed
from the $n$-copies of $\Lambda$, each of which is built by using a
Morse function $f$ to perturb $n$ copies of $\Lambda$ and then further
perturbing at the critical points of $f$ to make the $xy$-projection
generic.  For Legendrians in $\R^3$, the latter perturbation can be
done explicitly so that the resulting \dgas{} are described in terms
of $\alg(\R^3,\Lambda)$ by the algebraic construction of
Subsection~\ref{ssec:unital-construction}, and this produces the
``short'' Reeb chords mentioned above.  Even with this algebraic
description, the proof of invariance of $\Aug(\Lambda)$ requires
substantial effort. In higher dimensions, one can define
$\Aug(\Lambda)$ by making choices for the necessary perturbations, but
these choices are non-canonical and proving invariance under all
choices of perturbations is much harder.
By contrast, the invariance of $\AugBC(\Lambda)$ is simpler since it omits the short Reeb chords, but as a consequence it fails to be a unital category.

In particular, most of the technical material in Sections~\ref{sec:augcatalg} and \ref{sec:augcat} is developed from scratch specifically to deal with the incorporation of the short Reeb chords into the augmentation category.  The extra trouble required to do so (in comparison with \cite{BC}) is worthwhile in the end, because many interesting properties of $\Aug(\Lambda)$ are either false or unknown for $\AugBC(\Lambda)$.  Most importantly, $\Aug(\Lambda)$ satisfies Theorem~\ref{thm:intro-main} while $\AugBC(\Lambda)$ does not; in addition, as mentioned above, we give a precise characterization of isomorphism in $\Aug(\Lambda)$ in Subsection~\ref{ssec:augisom}, and we show in \cite{NRSS} that the homotopy cardinality of $\Aug(\Lambda; \bF_q)$ recovers the ruling polynomial of $\Lambda$.
\end{remark}

\vskip 0.2in
\noindent {\bf Acknowledgments}
\vskip 0.2in
The work reported here stems from discussions at a SQuaRE sponsored by
the American Institute of Mathematics, and we are very grateful to AIM
for its support.
This meeting also included
David Treumann, whom we especially thank for his initial involvement with this
project.  We would also like to thank Mohammad Abouzaid, Denis Auroux,
Ben Davison, Tobias Ekholm, Sheel Ganatra, Paolo Ghiggini,
 and Emmy Murphy for helpful discussions.
The work of LN was supported by NSF grants DMS-0846346 and DMS-1406371.
The work of VS was supported by NSF grant DMS-1406871.
The work of SS was supported by NSF postdoctoral fellowship DMS-1204387.
The work of EZ was supported by NSF grants DMS-1104779 and DMS-1406024.

\newpage

\section{Background}
\label{sec:nrd}

\subsection{Contact geometry}

To denote a choice of coordinates,
we write $\bR_x$ to mean the space $\bR$ coordinatized by $x$, and similarly for $\bR^2_{xz},$ etc.
We consider Legendrian knots and links $\Lambda$ in $J^1(\bR_x) \cong T^*\bR_x \times \bR_z \cong \bR^3_{xyz}$
and their front projections $\Phi_\Lambda = \pi_{xz} (\Lambda)$ where $\pi_{xz} : \bR^3_{xyz} \to \bR^2_{xz}.$
We take the contact form for the standard contact structure on $J^1(\bR)$ to be $\alpha = dz - y\,dx$ with
Reeb vector field $R_\alpha = \partial_z$.
In higher dimensions one could take $\Lambda\subset J^1(\bR^n)\cong T^*\bR^n\times \bR_z$,
in which case $\alpha = dz - \sum_i y_i dx^i$ and $R_\alpha = \partial_z,$
but we focus on $1$-dimensional knots and links in this paper.

Consider $T^*\bR^2_{xz}$ with coordinates $(x,z,p_x,p_z)$ and exact symplectic structure
$\omega = d\theta$ defined by the primitive $\theta = -p_x dx -p_z dz.$
For any $\rho > 0$ the cosphere bundle $S^*_\rho\bR^2_{xz} := \{p_x^2+p_z^2=\rho^2\}\subset T^*\bR^2_{xz}$
with induced contact form $\alpha = -p_x dx - p_z dz$ defined by restricting $\theta$
is contactomorphic to the unit cosphere bundle $S^*_1\bR^2_{xz}$ via
dilation by $1/\rho$ in the fibers.  We define
$T^\infty\bR^2_{xz} := S^*_1\bR^2_{xz}$, thinking of $\rho$ large describing the ``cosphere at infinity.''
There is a contact embedding of $\bR^3_{xyz}$ as a hypersurface of $T^*\bR^2_{xz}$ by the map
$(x,y,z)\mapsto (x=x,z=z,p_x=y,p_z=-1)$.
By scaling $(x,z,p_x,p_z)\mapsto (x,z,\frac{p_x}{\sqrt{p_x^2+p_z^2}},\frac{p_z}{\sqrt{p_x^2+p_z^2}})$ this
hypersurface is itself contactomorphic to an open subset of $T^\infty\bR^2_{xz}$ which we call
$T^{\infty,-}\bR^2_{xz}$ or just $T^{\infty,-}\bR^2,$ the minus sign indicating the downward direction of
the conormal vectors.  In this way, we equate, sometimes without further mention, the standard
contact three-space with the open subset $T^{\infty,-}\bR^2$ of the cosphere bundle of the plane.
Our knots and links live in this open set.

Given a front diagram $\Phi_\Lambda$, we sometimes use planar
isotopies and Reidemeister II
moves to put the diagram in ``preferred plat'' position: 
with crossings at different values of $x$, all left cusps horizontal and at the
same value of $x$, and likewise for right cusps.
The maximal smoothly immersed
submanifolds of $\Phi_\Lambda$ are called \emph{strands}, maximal embedded submanifolds are called \emph{arcs},
and maximal connected components of the complement of $\Phi_\Lambda$ are called \emph{regions}.
A Maslov potential $\mu$ is a map from the set of strands to $\bZ/2k$ such that at a cusp, the
upper strand has value one greater than the lower strand.  Here $k$ is any integer dividing the gcd of the rotation numbers
of the components of $\Lambda$.

\subsection{The LCH differential graded algebra}
\label{ssec:dga-background}

In this subsection, we review the Legendrian contact homology \dga{} for Legendrian knots and links
in $\bR^3$. For a more detailed introduction we refer the reader, for example, to \cite{C,NgCLI, ENS}.  Here, we discuss a version of the \dga{} that allows for an arbitrary number of base points to appear, as in \cite{NgR}, and our sign convention follows \cite{ekholm-ng} (which essentially agrees with the one used in \cite{ENS}).

\subsubsection{The \dga{}}

Let $\Lambda$ be a Legendrian knot or link in the contact manifold $\bR^3 = J^1(\bR) = T^{\infty, -} \bR^2$.  The \dga{} of $\Lambda$ is most naturally defined via the {\it Lagrangian
 projection} (also called the $xy$-projection) of $\Lambda$, which is the image of $\Lambda$ via the projection $\pi_{xy}: J^1(\bR) \rightarrow \bR_{xy}$.
The image  $\pi_{xy}(\Lambda) \subset \R_{xy}$ is a union of immersed curves.  After possibly modifying $\Lambda$ by a small Legendrian isotopy, we  may assume that $\pi_{xy}|_\Lambda$ is one-to-one except for some finite number of transverse double points which we denote $\{a_1, \ldots, a_r\}$.  We note that the $\{a_i\}$ are in bijection with {\it Reeb
chords} of $\Lambda$, which are trajectories of the Reeb vector field $R_\alpha = \partial_z$ that begin and end on $\Lambda$.

To associate a \dga{} to $\Lambda$, we fix a Maslov potential $\mu$
for the front projection $\pi_{xz}(\Lambda)$, taking values in $\bZ/2r$ where $r$ is the gcd of the rotation numbers of the components of $\Lambda$.
In addition, we choose sufficiently many base points $*_1, \ldots, *_M \in \Lambda$ so that every component
of $\Lambda \setminus \{ *_i \}$ is contractible, i.e., at least one on each component of the link.

The {\it Chekanov--Eliashberg \dga{}} (C--E \dga{}), also called the {\it Legendrian contact homology \dga{}}, is denoted  simply $(\alg, \partial)$, although we may write $\alg(\Lambda, *_1, \ldots, *_M)$ when the choice of base points needs to be emphasized.
The underlying graded algebra, $\alg$, is the noncommutative unital (associative) algebra
generated over $\bZ$ by the symbols $a_1, \ldots, a_r, t_1, t_1^{-1}, \ldots, t_M, t_M^{-1}$ subject only to the relations $t_i t_i^{-1} = t_i^{-1} t_i = 1$.  (In particular, $t_i$ does not commute with $t_j^{\pm 1}$ for $j\neq i$ or with any of the $a_k$.)

A $\bZ/2r$-valued grading is given by assigning degrees to generators and requiring that for homogeneous elements $x$ and $y$, $x\cdot y$ is also homogeneous with $|x\cdot y| = |x|+ |y|$.   To this end, we set  $|t_i| = |t_i^{-1}| = 0$.  A Reeb chord $a_i$ has its endpoints on distinct strands of the front projection, $\pi_{xz}(L)$, and moreover the tangent lines to $\pi_{xz}(\Lambda)$ at the endpoints of $a_i$ are parallel.  Therefore, near the upper (resp. lower) endpoint of $a_i$, the front projection is a graph  $z = f_u(x)$ (resp. $z = f_l(x)$) where the functions $f_u$ and $f_l$ satisfy
\[
(f_u-f_l)'(x(a_i)) = 0,
\]
and the critical point at $x(a_i)$ is a nondegenerate local maximum or minimum (by the assumption that $a_i$ is a transverse double point of $\pi_{xy}(\Lambda)$).  The degree of $a_i$ is
\[
|a_i| =  \mu(a_i^u) - \mu(a_i^l) + \left\{ \begin{array}{cr} 0, & \mbox{if $f_u-f_l$ has a local maximum at $x(a_i)$,} \\ -1, & \mbox{if $f_u-f_l$ has a local minimum at $x(a_i)$,} \end{array} \right.
\]
where $\mu(a_i^u)$ and $\mu(a_i^l)$ denote the value of the Maslov
potential at the upper and lower endpoint of $a_i$. (For this index
formula in a more general setting, see
\cite[Lemma~3.4]{EES-nonisotopic}.)

\begin{remark} Note that adding an overall constant to $\mu$ does not change the grading of $\alg$.  In particular, when $\Lambda$ is connected, $|a|$ is independent of the Maslov potential and corresponds to
the Conley--Zehnder index associated to the Reeb
chord $a$.  This can be computed from  
the rotation
number in $\bR^2$ of the projection to the $xy$-plane of a path along
$\Lambda$ joining the endpoints of $a$; see \cite{C}.
\end{remark}

The differential $\partial :\thinspace \alg\to\alg$ counts holomorphic disks in the symplectization $\bR \times J^1(\bR)$ with boundary on the Lagrangian cylinder $\bR \times \Lambda$, with one boundary puncture limiting to a Reeb chord of $\Lambda$ at $+\infty$ and some nonnegative number of boundary punctures limiting to Reeb chords at $-\infty$.  For Legendrians in $J^1(\R)$, we have the following equivalent (see \cite{ENS}) combinatorial description.

At each crossing $a_i$ of $\pi_{xy}(\Lambda)$, we assign {\it Reeb signs} to the four quadrants at the crossing according to the condition that the two quadrants that appear counterclockwise (resp. clockwise) to the over-strand have positive (resp. negative) Reeb sign.  
In addition, to define $(\alg, \partial)$ with $\bZ$ coefficients, we have to make a choice of orientation signs as follows:  At each crossing, $a_i$, such that $|a_i|$ is even,
we assign negative {\it orientation signs} to the two quadrants that lie on a chosen side of the understrand at $a_i$.  All other quadrants have positive orientation signs.  See Figure \ref{fig:ReebSigns}.

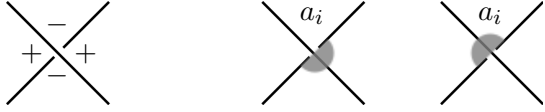
\begin{figure}
\labellist
\small\hair 2pt
\pinlabel $-$  at 32 50
\pinlabel $-$  at 32 18
\pinlabel $+$  at 50 32
\pinlabel $+$  at 16 32
\pinlabel $a_i$  at 192 56
\pinlabel $a_i$  at 304 56
\endlabellist
\centering
\includegraphics[scale=0.6]{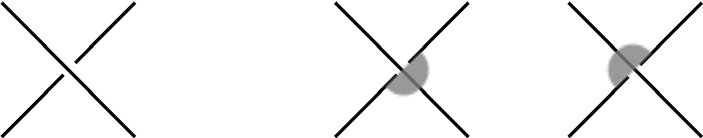}
\caption{
Left: the Reeb signs of the quadrants of a crossing of
$\pi_{xy}(\Lambda)$.  Right: he two possible choices of orientation signs at a crossing $a_i$ with $|a_i|$ even.  The shaded quadrants have negative orientation signs while the unshaded quadrants have positive orientation signs. At a crossing of odd degree, all quadrants have positive orientation signs. }
\label{fig:ReebSigns}
\end{figure}

For $l \geq 0$, let $D^2_l = D^2 \setminus \{p,q_1,\dots,q_l\}$ denote a disk with $l+1$ boundary punctures labeled $p,q_1,\dots,q_l$ in counterclockwise order.  Given generators $a, b_1, \ldots, b_l \in \alg$, we define $\Delta(a;b_1,\dots,b_l)$ to be the space of smooth, orientation-preserving immersions $u: (D^2_l, \partial D^2_l) \to (\R^2_{xy}, \pi_{xy}(\Lambda))$ up to reparametrization, such that
\begin{itemize}
\item $u$ extends continuously to $D^2$; and
\item $u(p) = a$ and $u(q_i)= b_i$  for each $1 \leq i \leq l$, and the image of a neighborhood of $p$ (resp. $q_i$) under $u$ is
a single quadrant at $a$ (resp. $b_i$)  with positive (resp. negative) Reeb sign.
\end{itemize}
We refer to the $u(p)$ and $u(q_i)$ as the corners of this disk.  Traveling counterclockwise around $\overline{u(\partial D_l)}$ from $a$, we encounter a sequence $s_1,\dots,s_m$ ($m \geq l$) of corners and base points, and we define a monomial
\[
w(u) = \delta \cdot w(s_1)w(s_2)\dots w(s_m),
\]
 where $w(s_i)$ is defined by:
\begin{itemize}
\item If $s_i$ is a corner $b_j$, then $w(s_i)=b_j$.
\item If $s_i$ is a base point $*_j$, then $w(s_i)$ equals $t_j$ or $t_j^{-1}$ depending on whether the boundary orientation of $u(\partial D^2_l)$ agrees or disagrees with the orientation of $\Lambda$ near $*_j$.
\item The coefficient $\delta = \pm1$ is the product of orientation signs assigned to the  quadrants that are occupied by $u$ near the corners at $a$, $b_1, \ldots, b_l$. (See also Remark~\ref{rmk:dgasigns} concerning sign choices.)
\end{itemize}

We then define the differential of a Reeb chord generator $a$ by
\[ \dd a = \sum_{u\in \Delta(a;b_1,\dots,b_l)} w(u) \]
where we sum over all tuples $(b_1,\dots,b_l)$, including possibly the empty tuple.  Finally, we let $\dd t_i = \dd t_i^{-1} = 0$ and extend $\dd$ over the whole \dga{} by the Leibniz rule $\dd(xy) = (\dd x)y + (-1)^{|x|}x(\dd y)$.

\begin{remark}
An equivalent definition with more of the flavor of Floer homology can
be made by taking $\Delta(a;b_1,\dots,b_l)$ to consist of holomorphic
disks in $\bR \times J^1(\bR)$, modulo conformal reparametrization and
vertical translation.  If this approach is taken, then the location of the boundary punctures $p, q_1, \ldots, q_l$ needs to be allowed to vary along $\partial D^2$ in a manner that preserves their cyclic ordering.  See \cite{ENS}.
\end{remark}

\begin{theorem}[\cite{C, ENS}] For any Legendrian $\Lambda \subset J^1(\bR)$ with base points $*_1, \ldots, *_M$, the differential $\partial : \alg(\Lambda, *_1, \ldots, *_M) \rightarrow \alg(\Lambda, *_1, \ldots, *_M)$ is well-defined, has degree $-1$, and satisfies $\partial^2=0$.
\end{theorem}

An {\it algebraic stabilization} of a \dga{} $(\alg, \partial)$  is a
\dga{} $(S(\alg), \partial')$ obtained as follows:  The algebra
$S(\alg)$ is obtained from $\alg$ by adding two new generators $x$ and
$y$ with $|x| = |y|+1$ (without additional relations), and the
differential $\partial'$ satisfies $\partial' x = y$, $\partial' y =
0$, and $\partial'|_\alg = \partial$.

\begin{theorem} \label{thm:CEDGAInvariance} Let $\Lambda_1, \Lambda_2
  \subset J^1(\R)$ be Legendrian links with base points chosen so that each component of $\Lambda_1$ and $\Lambda_2$ contains exactly one base point.
  If $\Lambda_1$ and $\Lambda_2$ are Legendrian isotopic, then for any choice of Maslov potential on $\Lambda_1$, there is a corresponding Maslov potential on $\Lambda_2$ such that the Legendrian contact homology \dgas{} $(\alg_1, \partial_1)$ and $(\alg_2, \partial_2)$ are stable tame isomorphic.
\end{theorem}

The meaning of the final statement is that after stabilizing both the \dgas{} $(\alg_1, \partial_1)$ and $(\alg_2, \partial_2)$ some possibly different number of times they become isomorphic.  Moreover, the \dga{} isomorphism may be assumed to be tame, which means that the underlying algebra map is a composition of certain elementary isomorphisms with have a particular simple form on the generators.  (We will not need to use the tame condition in this article.)

Allowing more than one base point on some components of $\Lambda$ provides essentially no new information, yet is convenient in certain situations.  The precise relationship between \dgas{} arising from the same link equipped with different numbers of base points is given in Theorems 2.21 and 2.22 of \cite{NgR}.  See also the proof of Proposition \ref{prop:independentoff} of this article where relevant details are discussed.

\subsubsection{The resolution construction}
Often, a Legendrian link $\Lambda \subset J^1(\bR)$ is most conveniently presented via its front projection.  For computing Legendrian contact homology, we can obtain the Lagrangian projection of a link $\Lambda'$ that is Legendrian isotopic to $\Lambda$ by resolving crossings so that the strand with lesser slope in the front projection becomes the overstrand, smoothing cusps, and adding a right-handed half twist before each right cusp; the half twists result in a crossing of degree $1$ appearing before each right cusp.  See Figure \ref{fig:trefoil-1}
 below for an example.  We say that $\Lambda'$ is obtained from $\Lambda$ by the {\it resolution construction}.  (See \cite{NgCLI} for more details.)

Thus, by applying the resolution procedure to a Legendrian $\Lambda$
with a given front diagram and Maslov potential $\mu$, we obtain a
\dga{} $(\alg, \partial)$ (for $\Lambda'$) with Reeb chord generators in bijection with the crossings and right cusps of $\pi_{xz}(\Lambda)$.  The grading of a crossing of $\pi_{xz}(\Lambda)$ is the difference in Maslov potential between the overstrand and understrand of the crossing (more precisely, overstrand minus understrand), and the grading of all right cusps is $1$.  Moreover, supposing that $\Lambda$ is in preferred plat position, the disks involved in computing $\partial$ have almost the same appearance on $\pi_{xz}(\Lambda)$ as they do on the Lagrangian projection of $\Lambda'$.  The exception here is that when computing the differential of a right cusp $c$, we count disks that have their initial corner at the cusp itself, and there is an ``invisible disk'' whose boundary appears in the Lagrangian projection as the loop to the right of the crossing before $c$ that was added as part of the resolution construction.  Invisible disks contribute to $\partial c$ a term that is either $1$ or the product of $t_i^{\pm1}$ corresponding to base points located on the loop at the right cusp.

\subsubsection{The link grading}   \label{sssec:linkG}

Assume now that $\Lambda$ is a Legendrian link with
\[
\Lambda = \Lambda_1 \sqcup \cdots \sqcup \Lambda_m,
\]
where each $\Lambda_i$ is either a connected component or a union of
connected components.
In this setting, there is an additional structure on the \dga{}
$\alg(\Lambda)$, the ``link grading'' of Mishachev
\cite{Mishachev}.

\begin{definition}
Write $\rR^{ij}$ for the collection of Reeb chords of $\Lambda$ that
\textit{end} on $\Lambda_i$ and
\textit{begin} on $\Lambda_j$, so that $\rR = \sqcup_{i,j=1}^m
\rR^{ij}$. The Reeb chords in
$\rR^{ij}$ are called \textit{pure chords} if $i=j$ and \textit{mixed
  chords} if $i\neq j$.
\end{definition}

In addition, write $\cT^{ii}$ for the collection of generators $t_j,t_j^{-1}$ corresponding to base points belonging to $\Lambda_i$, and set $\cT^{ij} = \emptyset$ for $i \neq j$.  Finally, put $\sS^{ij} = \rR^{ij} \sqcup \cT^{ij}$.

For $1 \leq i, j \leq m$, we say that a word $a_{\ell_1}\cdots
a_{\ell_k}$ formed from generators in $\sS = \sqcup \sS^{ij}$ is {\it
  composable} from $i$ to $j$ if there is some sequence of indices
$i_0,\ldots,i_k$ with $i_0=i$ and $i_k=j$, such that $a_{\ell_p}
\in \sS^{i_{p-1}i_p}$ for $p=1,\ldots,k$.
Observe that the LCH differential $\dd(a)$ of a Reeb chord $a\in\rR^{ij}$
is a $\bZ$-linear combination of composable words from $i$ to $j$.
One sees this by following the
boundary of the holomorphic disk: this is in $\Lambda_i$ between $a$
and $a_{\ell_1}$, in some $\Lambda_{i_1}$ between $a_{\ell_1}$ and
$a_{\ell_2}$, and so forth. Note in particular that a mixed chord cannot
contain a constant term (i.e., an integer multiple of $1$) in its
differential.  That the differentials of generators, $\dd(a)$, are sums of composable words allows various algebraic invariants derived from $(\alg, \partial)$ to be split into direct summands.  A more detailed discussion appears in a purely algebraic setting in Section \ref{sec:augcatalg}, and the framework developed there is a crucial ingredient for the construction of the augmentation category in Section  \ref{sec:augcat}.

The invariance result from Theorem \ref{thm:CEDGAInvariance} can be
strengthened to take link gradings into account.  Specifically, if
$(\alg, \partial)$ is the \dga{} of a link $\Lambda = \Lambda_1 \sqcup
\cdots \sqcup \Lambda_m$ with generating set $\sS = \amalg^m_{i,j=1}
\sS^{ij}$, then we preserve the decomposition of the generating set
when considering algebraic stabilizations by requiring that new
generators $x,y$ are placed in the same subset $\sS^{ij}$ for some $1
\leq i,j \leq m$.  We then have:

\begin{proposition}[\cite{Mishachev}]  \label{prop:InvarianceOfLinkGrading} If $\Lambda =  \Lambda_1 \sqcup \cdots \sqcup \Lambda_m$ and $\Lambda' = \Lambda'_1 \sqcup \cdots \sqcup \Lambda'_m$ are Legendrian isotopic via an isotopy that takes $\Lambda_i$ to $\Lambda'_i$ for $1\leq i \leq m$, then there exist (iterated) stabilizations of the \dgas{} of $\Lambda$ and $\Lambda'$, denoted $(S \alg, \partial)$ and $(S\alg', \partial')$, that are isomorphic via a \dga{} isomorphism $f: S\alg \rightarrow S\alg'$, with the property that for a generator $a \in \mathcal{S}^{ij}$ of $S\alg$, $f(a)$ is a $\bZ$-linear combination of composable words from $i$ to $j$ in $S\alg'$.  (Multiples of $1$ may appear if $i =j$.)  Moreover, if each $\Lambda_i$ and $\Lambda'_i$ contains a unique basepoint $t_i$ and the isotopy takes the orientation of $\Lambda_i$ to the orientation of $\Lambda'_i$, then we have $f(t_i) = t_i$.
\end{proposition}

\subsection{$A_\infty$ categories}
\label{sec:a-infinity}

We follow the conventions of Keller \cite{Keller}, which are the same as the conventions of Getzler--Jones \cite{GJ} except
that in Keller the degree of $m_n$ is $2-n$ whereas in Getzler-Jones
it is $n-2$.  In particular we will use the Koszul sign rule: for graded vector spaces, we choose the identification $V \otimes W \to W \otimes V$ to
come with a sign $v\otimes w \mapsto (-1)^{|v||w|} w \otimes v$, or equivalently,
we demand $(f \otimes g)(v \otimes w) = (-1)^{|g||v|} f(v) \otimes g(w)$.
Note that the sign conventions that we use differ from, say, the conventions of Seidel \cite{Seidel}; so for instance, reading off the multiplication operations from the differential in Legendrian contact homology requires the introduction of a sign, see \eqref{eq:ms}.

An \textit{$A_\infty$ algebra} $A$ is a graded module equipped with operations $m_n: A^{\otimes n} \to A$ for $n \ge 1$.
These operations have degree $2-n$ and obey a complicated tower of relations.  The first is that $(m_1)^2 = 0$, and the
second ensures that $m_2$ is associative after passing to cohomology with respect to $m_1$.

The relations are nicely expressed in terms of the bar construction.  This goes as
follows.  Let $\overline{T}(A[1]) := \bigoplus_{k \ge 1} A[1]^{\otimes k}$ be the positive part of the tensor co-algebra.  Let
$b: \overline{T}(A[1]) \to \overline{T}(A[1])$ be a co-derivation --
i.e., a map satisfying the co-Leibniz rule -- of degree 1.  Then, by
the co-Leibniz rule,
$b$ is determined by the components $b_k: A[1]^{\otimes k} \to A[1]$.

Let $s: A \to A[1]$ be the canonical degree $-1$
identification $a \mapsto a$.
 Taking $m_k, b_k$ to be related by $s \circ m_k = b_k \circ s^{\otimes k}$, the $A_\infty$ relations
are equivalent to the statement that $b$ is a co-differential, i.e., $b^2 = 0$.  It is even more complicated to write, in terms of the
$m_k$, the definition of a morphism $A \to B$ of $A_\infty$ algebras; suffice it here to say that the definition is equivalent to asking for
a co-\dga{} morphism $\overline{T}(A[1]) \to \overline{T}(B[1])$.  That is:

\begin{proposition}[\cite{Stasheff-II, Kadeishvili}] \label{prop:bar}
Let $A$ be a graded free module, and let $\overline{T}A = \bigoplus_{k \geq 1} A^{\otimes k}$.
Then there is a natural bijection between $A_\infty$ algebra structures on $A$ and square zero degree 1
coderivations on the coalgebra $\overline{T}(A[1])$.  This equivalence extends to a bijection between $A_\infty$ morphisms
$A \to B$ and dg-coalgebra morphisms $\overline{T}(A[1]) \to \overline{T}(B[1])$, which preserves the underlying map $A \to B$.
\end{proposition}

Because in practice our $A_\infty$ algebras will be given in terms of $b$ but we will want to make explicit calculations of the $m_k$, especially
$m_1$ and $m_2$, we record here the explicit formula relating their behavior on elements.    For elements $a_i \in A$, the Koszul sign rule
asserts
\begin{eqnarray*}
s^{\otimes k}(a_1 \otimes \cdots \otimes a_k) & = &  (-1)^{|a_{k-1}| + |a_{k-2}| + \cdots + |a_{1}|}  s^{\otimes k-1}(a_1 \otimes \cdots \otimes a_{k-1}) \otimes s(a_k) \\
& = &  (-1)^{|a_{k-1}| + |a_{k-3}| + |a_{k-5}| + \cdots }  s(a_1) \otimes s(a_2) \otimes \cdots \otimes s(a_k)
\end{eqnarray*}
so:
\begin{eqnarray*} m_k(a_1, a_2, \ldots, a_k) &  = &  s^{-1} \circ b_k \circ s^{\otimes k} (a_1 \otimes a_2 \otimes \cdots \otimes a_k) \\
& = & (-1)^{|a_{k-1}| + |a_{k-3}| + |a_{k-5}| + \cdots} s^{-1} b_k (s(a_1) \otimes s(a_2) \otimes \cdots \otimes s(a_k)).
\end{eqnarray*}

\noindent
In terms of the $m_k$, the first three $A_\infty$ relations are:
\begin{align*}
m_1(m_1(a_1)) &= 0 \\
m_1(m_2(a_1,a_2)) &= m_2(m_1(a_1),a_2) + (-1)^{|a_1|}
m_2(a_1,m_1(a_2)) \\
m_2(a_1,m_2(a_2,a_3)) - m_2(m_2(a_1,a_2),a_3) &=
m_1(m_3(a_1,a_2,a_3)) + m_3(m_1(a_1),a_2,a_3) \\
&\qquad + (-1)^{|a_1|} m_3(a_1,m_1(a_2),a_3) \\
&\qquad + (-1)^{|a_1|+|a_2|} m_3(a_1,a_2,m_1(a_3)).
\end{align*}
These are the standard statements that $m_1$ is a differential on $A$,
$m_1$ is a derivation with respect to $m_2$, and $m_2$ is associative
up to homotopy.  In general, the $A_\infty$ relations  are
\begin{equation} \label{eq:Ainftyrelations}
\sum (-1)^{r+s t}m_u (1^{\otimes r} \otimes m_s \otimes 1^{\otimes t}) = 0
\end{equation}
for $n \geq 1$, where we sum over all $r,s,t \geq 0$ with $r+s+t =n$ and put $u = r+ 1 + t$.  Note that when the left hand side is applied to elements, more signs appear from the Koszul convention.

The notion of an $A_\infty$ morphism of $A_\infty$ algebras
$f :\thinspace A\to B$ can also be described explicitly,
as a collection of maps
$f_n :\thinspace A^{\otimes n} \to B$ of degree $1-n$ satisfying
certain relations; see \cite{Keller}. We record the explicit expressions for the first two here:
\begin{align*}
f_1(m_1(a_1)) &= m_1(f_1(a_1)) \\
f_1(m_2(a_1,a_2)) &= m_2(f_1(a_1),f_1(a_2)) + m_1(f_2(a_1,a_2))
+ f_2(m_1(a_1),a_2) + (-1)^{|a_1|} f_2(a_1,m_1(a_2)).
\end{align*}
These assert that $f_1$ commutes with the differential, and respects the product
up to a homotopy given by $f_2$.

The notions of {\it $A_\infty$ categories} and {\it $A_\infty$ functors} are generalizations of $A_\infty$ algebras and their morphisms.  An $A_\infty$ category has,
for any two objects $\e_1, \e_2$, a graded module $\Hom(\e_1, \e_2)$.
For $n\geq 1$ and objects $\e_1, \ldots, \e_{n+1}$, there is a degree $2-n$ composition
\[
m_n: \Hom(\epsilon_n,\epsilon_{n+1}) \otimes \cdots \otimes \Hom(\epsilon_1,\epsilon_{2})
\to \Hom(\epsilon_1,\epsilon_{n+1})
\]
satisfying (\ref{eq:Ainftyrelations}) where the operations appearing on the left are understood to have appropriate sources and targets as determined by
$\epsilon_1, \ldots, \epsilon_{n+1}$.

\begin{remark} \label{rem:ainftycatrelations}
An equivalent way to formulate the $A_\infty$ condition on a category is as follows.  For a finite collection of objects $\epsilon_1, \ldots, \epsilon_n$, let
$A(\epsilon_1, \ldots, \epsilon_n) := \bigoplus \Hom(\epsilon_i, \epsilon_j)$ carry compositions $M_k$ defined
by first multiplying matrices and then applying the $m_k$.  (I.e., form $\mathrm{End}(\bigoplus \epsilon_i)$ without assuming
$\bigoplus \epsilon_i$ exists.)  The condition that the category is $A_\infty$ is just the requirement that all
$A(\epsilon_1, \ldots, \epsilon_n)$ are $A_\infty$ algebras.
\end{remark}

The definition of an $A_\infty$-functor $F$ is a similar generalization of morphism of $A_\infty$ algebras;
along with a correspondence of objects  $\e \mapsto F(\e)$ we have for any objects $\e_1, \ldots, \e_{n+1}$ a map
\[
F_n:  \Hom(\epsilon_n,\epsilon_{n+1}) \otimes \cdots \otimes \Hom(\epsilon_1,\epsilon_{2})
\to \Hom(F(\epsilon_1),F(\epsilon_{n+1}))
\]
satisfying appropriate relations.

Often, $A_\infty$ categories are not categories in the usual sense due to the absence of identity morphisms and the failure of associativity of composition (which only holds up to homotopy).
However, associativity does hold at the level of the {\it cohomology category} which is defined as follows.
The first $A_\infty$ relation shows that
\[
m_1 :\thinspace \Hom(\epsilon_1,\epsilon_2) \to
\Hom(\epsilon_1,\epsilon_2)
\]
is a differential: $m_1^2=0$. The cohomology category is defined to have the same objects as the underlying $A_\infty$ category, but with morphism spaces given by the cohomology
$H^*(\Hom(\epsilon_1,\epsilon_2))$. Composition is induced by $m_2$, which descends to
an associative multiplication map
\[
m_2 :
H^* \Hom(\epsilon_2,\epsilon_3) \otimes H^* \Hom(\epsilon_1,\epsilon_2) \to H^*\Hom(\epsilon_1,\epsilon_3).
\]

An $A_\infty$ category
is {\it strictly unital} if for any object $\epsilon$, there is a morphism
$e_\epsilon \in \Hom(\epsilon,\epsilon)$ of degree $0$ such that:
\begin{itemize}
\item
$m_1(e_\epsilon) = 0$;
\item
for any objects $\epsilon_1,\epsilon_2$, and any
$a\in\Hom(\epsilon_1,\epsilon_2)$,
$m_2(a,e_{\epsilon_1}) = m_2(e_{\epsilon_2},a) = a$;
\item
all higher compositions involving $e_\epsilon$ are $0$.
\end{itemize}

\begin{proposition}  \label{prop:ainftyunit}
For any $A_\infty$ category, the corresponding cohomology category is a
(usual, possibly non-unital) category, and it is unital if the $A_\infty$ category is strictly unital.
\end{proposition}

An $A_\infty$ functor $F$ induces an ordinary (possibly, non-unital) functor between the corresponding cohomology categories.  In the case that the two $A_\infty$ categories have unital cohomology categories, $F$ is called an {\it $A_\infty$ equivalence}
(or {\it quasi-equivalence})  if the induced functor on cohomology categories is an equivalence of categories in the usual sense, in particular preserving units.
The notion of $A_\infty$ equivalence satisfies the properties of an equivalence relation, cf. Theorem 2.9 of \cite{Seidel}.

To verify that a given $A_\infty$ functor $F$ is an equivalence, it suffices to check that, on cohomology categories, $F$ is essentially surjective (i.e. every object is isomorphic to one that appears in the image of $F$) and  fully faithful (i.e. induces isomorphisms on hom spaces).  The property of preserving units in cohomology follows as a consequence.

\subsection{Legendrian invariants from sheaves}

In this section we review some notions of sheaf theory, and how they are applied in 
 \cite{STZ} to the study of Legendrian knots.

First we recall the definition; explanations follow.
Put $M = \bR^2_{xz}$ and let $\Lambda \subset \bR^3 \cong T^{\infty,-}M$ be a Legendrian knot.
Then $Sh_\Lambda(M; \coeffs)$ is the dg category of sheaves 
with coefficients in $\coeffs$,
singular support
at infinity contained in $\Lambda$, and with compact support in $M$.
In fact, we use a slight variant:  when we take $M=I_x\times \bR_z$
with $I_x\subset \bR_x$, we
will require only that sheaves have zero support for $z\ll 0$.
By \cite{GKS,STZ}, a Legendrian isotopy $\Lambda \rightsquigarrow \Lambda'$ induces
an equivalence of categories $Sh_\Lambda(M; \coeffs) \cong Sh_{\Lambda'}(M; \coeffs)$.  

\subsubsection{Sheaves}
For a topological space\footnote{We always assume our topological spaces are locally compact Hausdorff; in fact in this article we will only be concerned with sheaves on manifolds.} $T$, we write $Op(T)$ for the category whose objects are open sets
of $T$ and whose morphisms are inclusions of open sets.  A presheaf on $T$ valued in some 
category $\mathcal{C}$ is by definition a functor $\cF: Op(T)^{op} \to \mathcal{C}$.  In particular,
when $U \subset V$ there is a restriction map $\cF(V) \to \cF(U)$. 

A presheaf is said to be a sheaf if the corresponding functor takes covers to limits.  More precisely, 
whenever given a collection of opens $U_i$ indexed by $i \in I$, the restriction maps induce 
a morphism
$$ \cF \left( \bigcup_{i \in I} U_i \right) \to \lim_{\emptyset \ne J \subset I} \cF \left( \bigcap_{j \in J} U_j \right)$$
One says $\cF$ is a sheaf assuming these morphisms are all isomorphisms.  When $\mathcal{C}$
is the category of sets or abelian groups, 
the limit on the right is already determined as the 
equalizer of the diagram $\prod \cF(U_i) \rightrightarrows \prod \cF(U_i \cap U_j)$.  However, the definition 
above makes sense in more general settings, in particular for various sorts of homotopical
categories, e.g. the $(\infty, 1)$ categories of \cite{Lurie-HTT}.  In particular, this definition
of sheaf is appropriate to define sheaf of categories, or sheaf of dg categories, or sheaf of $A_\infty$ categories. 

In classical references such as \cite{KS},
the derived category of sheaves is defined by beginning
with the category of sheaves of $\coeffs$-modules,  
taking complexes of such objects, and then taking the Verdier localization along quasi-isomorphisms.  
The well-behavedness of this localization is underwritten by the existence of injective
resolutions.  One then showed a posteriori that the resulting category was `triangulated'.  

From a more modern point of view, the category of complexes of sheaves is a dg category, 
and thus so is its localization along quasi-equivalences \cite{D}.
In this dg version, the 
natural hom space between objects is itself a complex, whose $H^0$ is the old hom.  
Some discussion of how to set up various sheaf-theoretic functors in the dg context 
can be found in \cite[Sec. 2.2]{N}. 

(From an even more modern point of view, one could just directly consider sheaves valued
in an appropriate $(\infty, 1)$ category of complexes.)

In any event, we write the resulting dg category of sheaves as $Sh(M; \coeffs)$. 
The extra information in the dg version is crucial in gluing arguments. 
In addition, we want to prove an 
equivalence between a category of sheaves and the $A_\infty$ category of augmentations.
As the latter is an $A_\infty$ category, we certainly need the $dg$ structure on the former.

\subsubsection{Microsupport}

To each complex $F$ of (not necessarily constructible) sheaves of $\coeffs$-modules is attached a closed conic subset $\SS(F)$, called the
``singular support'' of $F$.  This captures the failure of $F$ to be locally constant for a sheaf $F$, or cohomologically locally constant for a complex. This notion was introduced by Kashiwara and Schapira, and extensively developed in \cite{KS}.  We recall one of several equivalent definitions provided in \cite[Chap. 5]{KS}. 

Consider a covector $\xi \in T^*M$. 
If there is some $C^1$ function $f$, locally defined near $x$, with $f(x) = 0$ and $df_x = \xi$, such that 
$$\mathrm{colim}_{U \ni x} H^*(U;F) \to \mathrm{colim}_{V \ni x} H^*(f^{-1}(-\infty,0)\cap V;F)$$
is {\em not} an isomorphism, then we say $\xi$ is singular for $F$. 
(The 
map on colimits is induced by the evident restriction map for $V \subset U$.)  
The singular support is the closure of the locus of singular covectors.  


We define $Sh_\Lambda(M; \coeffs) \subset Sh(M; \coeffs)$ to be the full subcategory defined by
such $F$ with $SS(F)\subset \Lambda$ for a Legendrian subspace
$\Lambda$ of $T^\infty M$, and similarly
for $Sh_\Lambda(M; \coeffs)$.
%
%
%

\subsubsection{Constructible sheaves and combinatorial models}

When $\Lambda$ is the union of conormals to a subanalytic stratification of $M$, then 
$Sh_\Lambda(M; \coeffs)$ consists of sheaves {\em constructible} with respect 
to the stratification -- i.e., locally constant when restricted to each stratum. 
The theory of constructible sheaves in this sense is developed in detail in \cite[Chap. 8]{KS}. 

For sufficiently fine stratifications, the category of constructible sheaves admits
a well known combinatorial description.

\begin{definition}
Given a stratification $\cS$, the star of a stratum $s \in \cS$ is the union of strata that contain $s$ in their closure.
We view $\cS$ as a poset category in which every stratum has a unique map (generization) to every stratum in its star.
We say that $\cS$ is a regular cell complex if every stratum is contractible and moreover
the star of each stratum is contractible.
\end{definition}

Now if $C$ is any category and $A$ is an abelian category,
we write $Fun_{\mathit{naive}}(C,A)$ for the dg category of functors from $C$ to the category
whose objects are cochain complexes in $A$,
and whose maps are the cochain maps.
We write $Fun(C,A)$ for the dg quotient \cite{D}
of  $Fun_{\mathit{naive}}(C, A)$ by the thick subcategory of functors taking values in acyclic complexes.
For a ring $\coeffs$, we abbreviate the case where $A$ is the abelian category of $\coeffs$-modules to $Fun(C,\coeffs)$.

\begin{proposition}[{\cite[Thm. 1.10]{kashiwara1984riemann}, \cite{Shepard},\cite[Lemma 2.3.3]{N}}]
 \label{prop:star}
Let $\cS$ be a Whitney stratification of the space $M$.  Consider the functor
\begin{equation}
\label{eq:luciustarquiniuspriscus}
 \Gamma_{\cS}: Sh_\cS(M;\coeffs)  \to  Fun(\cS,\coeffs) \qquad \qquad
F  \mapsto  [s \mapsto \Gamma(\text{\rm star of $s$};F) ].
\end{equation}
If $\cS$ is a regular cell complex, then $\Gamma_{\cS}$ is a quasi-equivalence.
\end{proposition}

\begin{remark} Note in case $\cS$ is a regular cell complex, the restriction map from
$ \Gamma(\text{star of $s$};F)$ to the stalk of $F$ at any point of $s$ is a
quasi-isomorphism. \end{remark}

We use these constructions as follows. 
Our $\Lambda$ is not a union of conormals; but it will be contained in such a union
(possibly after a small contact isotopy to make $\Lambda$ subanalytic), 
so $Sh_\Lambda(M; \coeffs)$ can be described as constructible sheaves satisfying 
certain extra conditions.  

Specifically, for us $\Lambda \subset \bR^3 \cong T^{\infty,-}\bR^2 \subset T^\infty \bR^2.$
If we take $\cS$ the stratification of $\bR^2$ in which the zero-dimensional strata are the cusps and crossings, the one-dimensional strata are the arcs, and the two-dimensional strata are the regions,
and $\Lambda_{\cS}$ the union of conormals to these strata, then 
\[
Sh_\Lambda(\bR^2; \coeffs) \subset Sh_{\Lambda_{\cS}}(\bR^2; \coeffs) = Sh_{\cS}(\bR^2; \coeffs).
\]

Because $\Lambda \subset \bR^3 \cong T^{\infty,-}\bR^2$, every covector with $p_z > 0$ is nonsingular,
which means that every local restriction map which is \emph{downward} is required
to be a quasi-isomorphism.  
The easiest objects to describe are those in which all downward morphisms are in fact required
to be {\em identities}.  In Section 3.4 of \cite{STZ}, we term such objects {\em legible}.  
Such objects can be described in terms of a diagram of maps between the stalks at top
dimensional strata, as depicted in Figure \ref{fig:legcats} near an arc, a cusp, or a crossing. 

\begin{figure}[H]
\begin{center}
\includegraphics[scale = .3]{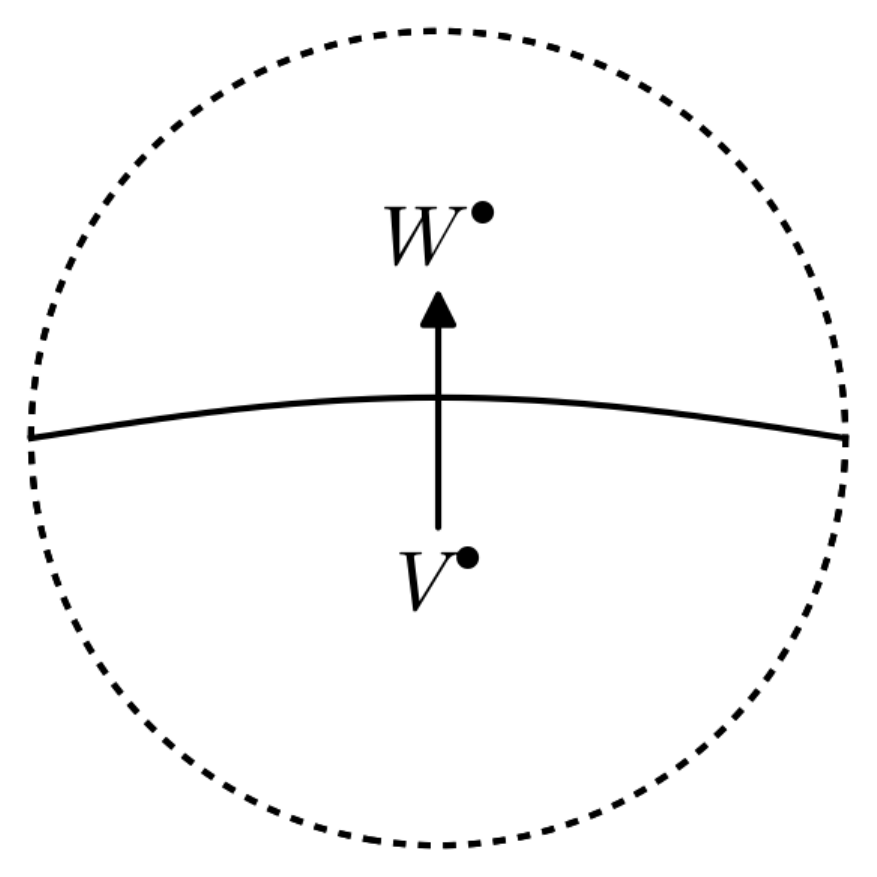}
\includegraphics[scale=.3]{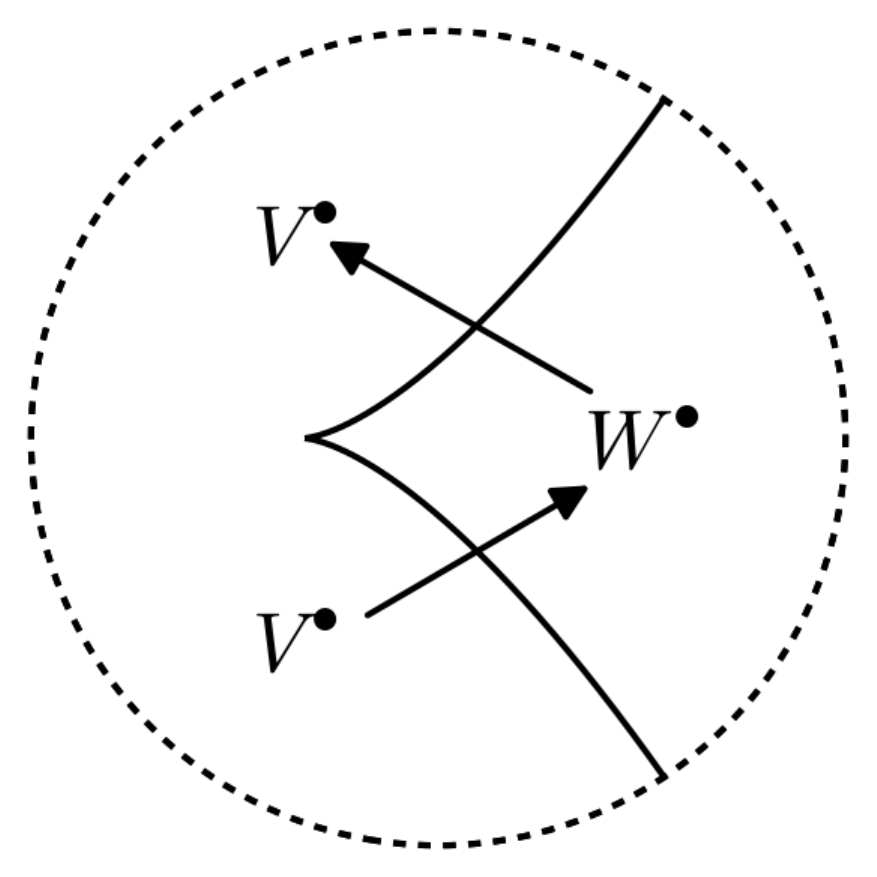}
\includegraphics[scale = .3]{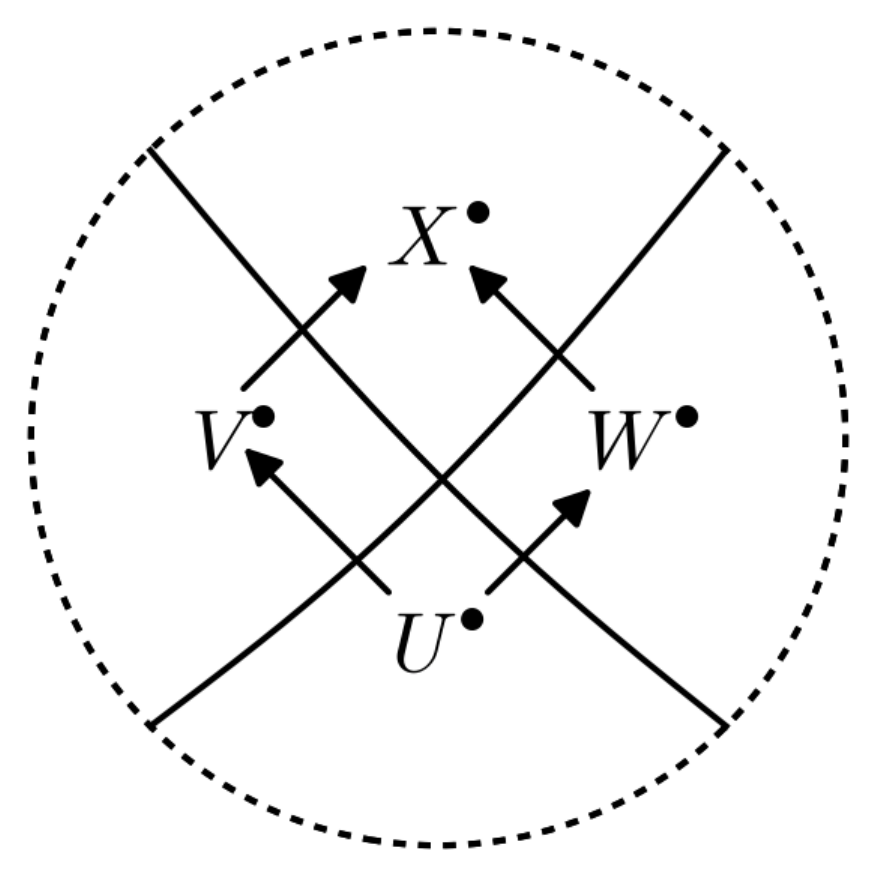}
\end{center}
\caption{Legible objects in various neighborhoods of a front diagram.}
\label{fig:legcats}
\end{figure}

To recover the sort of diagram described in the equivalence of Proposition \ref{prop:star} from
such a description of a legible object, one assigns to each stratum 
the chain complex placed in the region below, and takes the corresponding downward generization map 
to be the identity.  Then the upward generization maps are defined as the composition with this
downward equality and the map depicted in the legible diagram. 

Using the microsupport conditions, it is calculated in \cite[Section 3.4]{STZ} what additional
conditions must be satisfied by the maps indicated in Figure \ref{fig:legcats} in order 
that the corresponding sheaf have microsupport in the Legendrian lift of the depicted front.  
There is no condition along a line, as in the leftmost diagram.  At a cusp, 
the composition of the maps on the cusps is required to be the identity map of $V^\bullet$.  
At a crossing, the square around the crossing must commute and have acyclic total complex. 

For front diagrams of Legendrian tangles {\em with no cusps}, it is shown in \cite[Proposition 3.22]{STZ} that all sheaves with the corresponding microsupport are in fact quasi-isomorphic to sheaves
associated to legible objects.  The same is not true for arbitrary front diagrams. 

\subsubsection{Microlocal monodromy}
\label{sec:mumon}

Given an object $F\in Sh(\bR^2,\Lambda; \coeffs)\subset Sh(\bR^2,\Lambda_{\cS}; \coeffs)$,
there corresponds under $\Gamma_{\cS}$ of Proposition \ref{prop:star} a functor $\Gamma_{\cS}(F)$ from the poset
category of $\cS$ to chain complexes of $\coeffs$-modules.  Then to
a pair of an arc $a$ on a strand and a region $r$ above it (so $r = \text{star of } r$ is an
open subset of $\text{star of } a$), we have a morphism $a\to r$ and
there is an associated upward generization map $\rho = \Gamma_{\cS}(F)(a\to r)$
given by $\rho: \Gamma(\text{star of } a; F) \to \Gamma(r;F).$
If we take a legible representative for $\Gamma_{\cS}(F)$ then $\rho$ can also be
associated to a
map from the region $s$ below $a$ to the region $r$ above, as in Figure \ref{fig:legcats}.
The microlocal monodromy will be constructed from the map $\rho.$

Recall that a Maslov potential $\mu$ on the front diagram of a Legendrian knot
$\Lambda$ (with rotation number $0$) is a map from strands to $\bZ$ such that
the value assigned to the upper strand at a cusp is one more than the value
of the lower strand.
Now let $\Delta$ be the unique lift of $\cS\vert_{\pi_{xz}\Lambda},$ i.e. the induced
stratification of the knot $\Lambda$ itself.  Note there is one arc in $\Delta$ for each arc
of $\cS$, but two points for each crossing.
The microlocal monodromy of an object $F\in Sh(\bR^2,\Lambda)$,
denoted $\mu mon(F)$,
begins life as a functor from strata of $\Delta$ to chain complexes:
$\mu mon(F)(a) = {\sf Cone}(\rho)[-\mu(a)]$.  Note the Maslov potential
is used to determine the shift.
In \cite[Section 5.1]{STZ} it is shown how to treat the zero-dimensional
strata of $\Delta$ and that $\mu mon$ maps arrows of the $\Delta$ category
to quasi-isomorphisms --- see \cite[Proposition 5.5]{STZ}.
As a result, $\mu mon$ defines a functor from $Sh_c(\bR^2,\Lambda; \coeffs)$ to
local systems (of chain complexes) on $\Lambda:$
$$\mu mon:  Sh_c(\bR^2,\Lambda; \coeffs) \to Loc(\Lambda; \coeffs).$$

\begin{definition}
\label{def:rankone}
We define $\cC_1(\Lambda,\mu; \coeffs) \subset Sh_c(\bR^2,\Lambda)$ to be the full
subcategory consisting of objects $F$ such that $\mu mon(F)$ is a local system
of rank-one $\coeffs$-modules in cohomological degree zero.
\end{definition}

\begin{example}
Let $\equiv_n$ be the front diagram with $n$ infinite horizontal lines
labeled $1, 2, \ldots, n$ from top to bottom,
and let $\Lambda$ be the corresponding Legendrian.
Let $\mu$ be the Maslov potential $\mu(i) = 0$ for all $i$.
The associated stratification $\cS$ is a regular cell complex, and therefore
every object of $\cC_1(\Lambda,\mu; \coeffs)\subset Sh_c(\bR^2,\Lambda; \coeffs)$ has a legible representative.
To the bottom region we must assign $0$ due to the subscript ``$c$.''  If $V^\bullet$ is
assigned to the region above
the $n$-th strand, then the microlocal monodromy on the $n$th
strand is the cone of the unique map from $0$ to $V^\bullet$, i.e. $V^\bullet$ itself.
Microlocal rank one means then that $V^\bullet$ is a rank-one $\coeffs$-module in degree
zero.  Moving up from the bottom we get a complete flag in the rank-$n$ $\coeffs$-module
assigned to the top region.
For details and further considerations, see Section \ref{sec:locSh}.
\end{example}

In Theorem \ref{thm:main} we show that the category $\cC_1(\Lambda,\mu; \coeffs)$ is
equivalent to the category of augmentations to be defined in Section \ref{sec:augcat}.

\subsubsection{Sheaves of categories}
As mentioned above, it makes sense to consider sheaves valued in any category $\mathcal{C}$
in which the notion of limit makes sense; in particular, in a category of categories.  Here we want
to work with categories of dg or $A_\infty$ categories.  To do this we need some appropriate homotopical
framework for category theory, for instance the 
$(\infty,1)$-categories as developed in \cite{Lurie-HTT}.  It is also possible, and equivalent, to 
work in the older `model category' framework --- the model structures on the category of 
dg or $A_\infty$ categories present the corresponding $(\infty,1)$-category.  The relevant notion
of limit is what is called a homotopy limit in the model category setting, and just the limit in the 
setting of $(\infty, 1)$-categories.  

If $X$ is a 
locally compact Hausdorff topological space, and $A$ is a 1-category such as sets or $\mathbb{Z}$-modules, 
it is a standard result 
that the assignment $U \mapsto Sh(U)$ extends to a sheaf on $X$ valued in the $(2,1)$-category of categories.
(That is, restriction maps do not compose strictly, but only up to a homotopy.)  Such sheaves of categories
are sometimes called `stacks' in the old literature. 

In the $(\infty,1)$-categorical framework, there is a similar result for categories of sheaves which themselves
take values in a (presentable) $(\infty,1)$-category $C$, for instance in an appropriate 
$(\infty,1)$-category of $A_\infty$ categories.   The fact that a category of sheaves assembles itself into
a presheaf of categories is a tautology.
One must check that covers are carried to limits; we do not know of a reference for this result in the literature, so provide an argument here in a  
footnote.\footnote{One way to extract this result from what is written (as explained to us by Nick Rozenblyum, errors in translation due to us):  
it suffices to check the `universal' example where $C$ is the category of spaces, since in general $Sh(X, C) \cong Sh(X, Spaces) \otimes C$ by \cite[Prop. 4.8.1.17]{Lurie-HA} and 
\cite[Prop. 1.1.12]{Lurie-DAGV}.     For an open inclusion $u: U \subset X$, the restriction 
$u^*: Sh(X, Spaces) \to Sh(U, Spaces)$ has a fully faithful 
left adjoint given by the extension by the empty set, denote it $u_!$.  Note that $u_!(Sh(U, Spaces))$ is
identified with the overcategory $Sh(X, Spaces)_{/ u_! pt}$; this is because the point is initial and
there are no maps to the empty set.  By adjunction, showing that the limit over the restriction maps of the 
$Sh(U, Spaces)$ is $Sh(X, Spaces)$ is equivalent to showing that the same for the 
colimit over the extension by zero maps.  A homotopy cover (i.e. including overlaps) $u_\alpha: U_\alpha \to X$ means literally that $X$ is the gluing of the $U_\alpha$ in the category of topological spaces; it
follows easily that $\mathrm{colim}\,  u_{\alpha!} pt = pt_X$.  The fact that the overcategories over these objects obey
the same colimit is \cite[Thm. 6.1.3.9(3); Prop 6.1.3.10(2)]{Lurie-HTT}.} 

This fact as stated, in terms of presentable categories, applies directly to the categories of all sheaves of 
unbounded complexes (localized along quasi-isomorphisms).  However the full subcategories with perfect
stalks, or constructible with respect to some prescribed stratification, or with some prescribed microsupport,
are all characterized locally, hence form subsheaves of full subcategories.  

In fact, we make little essential use of the above generalities, since for the most part we only work with
constructible sheaves of categories on $\mathbb{R}$. 
One source of these is the following: 

\begin{lemma}
Let $\pi: \mathbb{R}^2 \to \mathbb{R}$ be the projection to the second factor. 
Let $\cS$ be a reasonable (e.g. Whitney B) 
stratification of $\mathbb{R}^2$, such that $\pi$ on each stratum has maximally
nondegenerate derivative.  Let $\cS'$ be a stratification of $\mathbb{R}$ such that 
the image under $\pi$ of any stratum in $\cS$ is $\cS'$-constructible.  
Then the assignment $U \mapsto Sh_{\cS}(\pi^{-1}(U))$ extends to a 
sheaf of dg categories
on $\mathbb{R}$, constructible with respect to $\cS'$. 
\end{lemma}

To compute with constructible sheaves of categories on $\mathbb{R}$, we note that 
Proposition \ref{prop:star} holds for sheaves of categories, so long as one understands
$Fun(\mathcal{S}, \mathcal{C})$ in the appropriate sense, i.e. as functors of quasicategories 
from (the nerve of) $\mathcal{S}$ to $\mathcal{C}$.  That is, 2-cells go to homotopies, etc.  

However, when $\mathcal{S}$ is a stratification of $\mathbb{R}$, the corresponding poset looks like 
$\bullet \leftarrow \bullet \rightarrow \bullet \leftarrow \cdots \leftarrow \bullet \rightarrow \bullet$, hence 
there are no nondegenerate 
2-cells in the nerve, so this complication can be essentially ignored for the purpose of 
describing objects of $Sh_{\mathcal{S}}(\mathbb{R})$.  

In describing morphisms between objects of $Fun(\mathcal{S}, \mathcal{C})$, it is important
to remember that such a morphism is a diagram which commutes up to specified homotopies 
(the possible homotopies being encoded by the 2-cells in $\mathcal{C}$).  When taking a 
stratification of $\mathbb{R}$, one has homotopy-commutative squares in the maps between diagrams, 
but never nontrivial compositions of such squares, so there is no need to consider higher homotopies. 
These considerations will arise later when describing maps of sheaves of $A_\infty$-categories. 

Finally, we will want to compute the global sections of a constructible sheaf of categories on 
$\mathbb{R}$, which is given as a functor out of a diagram 
$\cdots  \leftarrow \bullet \rightarrow \bullet \leftarrow \bullet \rightarrow \cdots$.  By definition
of sheaf (and recalling how this diagram is obtained from the cover by stars in Proposition \ref{prop:star}), 
the global sections is the limit of this diagram.  Evidently this can be computed by iterated pullbacks; 
we recall in the next subsection how to actually compute such pullbacks.

\subsubsection{Pullbacks of categories} 

To actually calculate limits in the $(\infty,1)$-category of dg or $A_\infty$ categories, the model
structure provided by Tabuada and To\"en \cite{Tabuada,Toen} is useful.  In fact
we will be only interested in calculating pullbacks, which are given by the following formula.  

Let $p:A\to C$ and $q:B\to C$ be two functors between dg categories $A, B, C.$
Objects of the fiber product dg category $A\times_C B$ are triples $x=(a,b,f)$, with $f \in {{hom}_C}^0(p(a),q(b))$ a closed isomorphism. 
Morphisms are
$$hom^k(x,x') = hom^k_A(a,a')\oplus hom^k_B(b,b')\oplus hom_C^{k-1}(p(a),q(b'))$$
with differential
$D = d + d'$ where $d = d_A \oplus d_B \oplus d_C$
and $d':hom_A^k(a,a')\oplus hom_B^k(b,b')\to hom_C^k(p(a),q(b'))$
defined by
$$d'(u\oplus w) =f' \circ p(u) - q(w)\circ f$$
The composition between $(u,w,v) \in hom(x,x')$ and $(u',w',v')\in hom(x',x'')$
is
$$u'\circ u \oplus w'\circ w \oplus  q(w')\circ v - v' \circ p(u)$$
which lies in $hom^k_A(a,a'')\oplus hom^k_B(b,b'')\oplus hom^{k-1}_C(p(a),q(a''))$ as required.
It is associative on the nose.


In our application, we will prefer to require $f$ to be {\em the identity} rather than an isomorphism.  
We write $(A \times_C B)_{strict}$ for the full subcategory of the product whose objects can be obtained
in this manner.    In general, the inclusion $(A \times_C B)_{strict} \to A \times_C B$ is not essentially
surjective, and in fact we will see examples where this fails.  However we have the following: 

\begin{definition}
A morphism of categories $q: B \to C$ has the isomorphism lifting property if, whenever there is some 
isomorphism 
$\phi:q(b) \cong c$, then in fact there is some $b' \in B$ with $c = q(b')$ and some isomorphism
$\widetilde{\phi}: b \cong b'$ with $\phi = q(\widetilde{\phi})$. 

A morphism of dg or $A_\infty$ categories $q: B \to C$ has the weak isomorphism lifting property
if $ho(q): ho(B) \to ho(C)$ has the isomorphism lifting property.  It has the strict isomorphism lifting
property if for any closed degree zero map $\phi:q(b) \cong c$ which becomes an isomorphism
in $ho(C)$, there is some $b'$ with $q(b')=c$ and some closed degree zero map 
$\widetilde{\phi}: b \cong b'$ which
becomes an isomorphism in $ho(B)$, such that $q(\widetilde{\phi}) = \phi$. 
\end{definition}

\begin{remark}
This is evidently some sort of fibrancy condition, but we do not know exactly how it relates to the 
model structures on dg categories. 
\end{remark}

\begin{lemma} \label{lem:strictify}
Given $p:A\to C$ and $q:B\to C$ morphisms of dg categories, suppose that $q$ has the strict 
isomorphism lifting property. Then the inclusion $A \times_C^{str} B \to A \times_C B$ 
is an equivalence. 
\end{lemma}
\begin{proof}
We need only check essential surjectivity.  Consider some object $(a, b, \phi:p(a) \cong q(b)) \in A \times_C B$. 
By the lifting property, there must be some $b' \in B$ with $q(b') = p(a)$, and a (closed degree zero quasi-) isomorphism $b \cong b'$ in $B$.  Consider the object
$(a, b', id: p(a) = q(b')) \in A \times_C^{str} B$. The map $b \cong b'$ induces an isomorphism 
$(a, b', id: p(a) = q(b'))  \cong (a, b, \phi:p(a) \cong q(b))$. 
\end{proof}


We end this section by summarizing some properties of constructible sheaves of $A_\infty$ categories over a line.

\begin{proposition} \label{prop:linesummary}
Let $\cC$ be a constructible sheaf of categories on a line, with respect to a stratification $\mathcal{Z}$ with zero dimensional strata $z_i$
and one-dimensional strata $u_{i, i+1} = (z_i, z_{i+1})$.  The associated diagram (as in Prop. \ref{prop:star}) has maps 
$$\cC(u_{i-1, i}) \xleftarrow{\rho_{L}} \cC_{z_{i}} \xrightarrow{\rho_{R}} \cC(u_{i, i+1}).$$

If $z_i <  z_{i+1} < \ldots < z_j$ are the zero dimensional strata in the interval $(a, b)$, then sections are 
calculated by 
$$\cC( (a, b) ) \equiv \cC_{z_{i}} \times_{\cC(u_{i, i+1})}  \cC_{z_{i+1}} \times \ldots \times \cC_{z_{j}}.$$

Objects of this fibre product are tuples $(\xi_i, \xi_{i+1}, \ldots, \xi_j; f_{i, i+1}, \ldots, f_{j-1, j})$ where
$\xi_k \in \cC_{z_{k}}$ and $f_{k, k+1}: \rho_R(\xi_k) \to \rho_L(\xi_{k+1})$ is an isomorphism in $\cC$.

This fibre product contains a full subcategory
$$(\cC_{z_{i}} \times_{\cC(u_{i, i+1})}  \cC_{z_{i+1}} \times \ldots \times \cC_{z_{j}})_{strict}$$
in which the $f_{k, k+1}$ must all be the identity morphism, i.e., $\rho_{R}(\xi_k) = \rho_L(\xi_{k+1})$.

If all $\rho_L$ have the ``isomorphism lifting property,'' i.e., that
any isomorphism $\phi: \rho_L(\xi) \sim \eta'$ is in fact the image under $\rho_L$ of some  isomorphism
$\psi: \xi \sim \xi'$.  Then the inclusion of the strict fibre product in the actual fibre product is an equivalence.
\end{proposition}
\begin{proof}
The only new thing is we have allowed many strata in the strictification; the result follows from Lemma \ref{lem:strictify} by induction. 
\end{proof}

\newpage

\section{Augmentation category algebra}
\label{sec:augcatalg}

In this section, we describe how to obtain a unital $A_\infty$ category from what we call a ``consistent sequence of differential graded algebras.'' Our motivation is the fact that if we start with a Legendrian knot or link $\Lambda$ in $\bR^3$ and define its $m$-copy $\Lambda^m$ to be the link given by $m$ copies of $\Lambda$ perturbed in the Reeb direction, then the collection of Chekanov--Eliashberg \dgas{} for $\Lambda^m$ ($m\geq 1$) form such a consistent sequence, as we will see in Section~\ref{sec:augcat}. First, however, we present a purely algebraic treatment, defining a consistent sequence of \dgas{} $\alg^{(\bullet)}$ and using it to construct the augmentation category $\Aug(\alg^{(\bullet)},\coeffs)$ along with a variant, the negative augmentation category $\AugBC(\alg^{(\bullet)},\coeffs)$. We then show that $\Aug(\alg^{(\bullet)},\coeffs)$ is unital, though $\AugBC(\alg^{(\bullet)},\coeffs)$ may not be (see Section~\ref{sec:augcat} or \cite{BC}).

\subsection{Differential graded algebras and  augmentations}
\label{ssec:dga-augs}

For the following definition, by a \dga{}, we mean an associative
$\bZ$-algebra $\alg$ equipped with a $\bZ/m$ grading for some even
$m \geq 0$, 
and a degree $-1$ differential $\partial$ that is a derivation.
The condition that $m$ is even is necessary for the Leibniz rule
$\partial(xy) = (\partial x)y+(-1)^{|x|}x(\partial y)$ to make sense,
though many of our results continue to hold if $m$ is arbitrary and
$\alg$ is instead an $R$-algebra where $R$ is a commutative unital
ring with $-1=1$ (e.g., $R=\bZ/2$).

\begin{definition}
A \textit{semi-free \dga{}} is a \dga{} equipped with a set $\sS = \rR \sqcup \tT$ of homogeneous generators
\begin{eqnarray*}
\rR & = & \{a_1, \ldots, a_r\} \\
\tT & = & \{t_1, t_1^{-1}, \ldots, t_{M}, t_{M}^{-1}\}
\end{eqnarray*}
such that $\alg$ is the result of taking the free noncommutative unital algebra over $\bZ$ generated by the elements of
$\sS$ and quotienting by the relations $t_i \cdot t^{-1}_i = t^{-1}_i \cdot t_i = 1$.
We require in addition that $|t_i| = 0$ and $\partial t_i = 0$.
\end{definition}

\noindent
We note that our use of ``semi-free'' is nonstandard algebraically but
roughly follows \cite{C}.

\begin{definition}
Let $\coeffs$ be a field;
we view it as a \dga{} by giving it the zero grading and differential.
A \textit{$\coeffs$-augmentation} of a semi-free \dga{} $\alg$ is a
\dga{} map $\epsilon :\thinspace \alg \to \coeffs$.  That is, it is a map of the underlying
unital algebras,
annihilating all elements of nonzero degree, and satisfying $\epsilon \circ \dd = 0$.
\end{definition}

\begin{remark}
An augmentation $\epsilon$ is uniquely determined by $\epsilon(a_i)
\in \coeffs$ for each $a_i\in\rR$, along with invertible elements $\epsilon(t_i) \in \coeffs$.
\end{remark}

Given an augmentation $\epsilon : \alg \to \coeffs$, we 
define the $\coeffs$-algebra
$$\alg^\epsilon := (\alg \otimes \coeffs) / (t_i = \epsilon(t_i) ).$$
Since $\partial t_i = 0$,
the differential $\partial$ descends to $\alg^\epsilon$.

We write $C$ for the free $\coeffs$-module with basis $\rR$.
We have
\[
\alg^{\epsilon} = \bigoplus_{k \ge 0} C^{\otimes k};
\]
and we further define $\alg^\epsilon_+ \subset \alg^\epsilon$ by
\[
\alg^\epsilon_+ := \bigoplus_{k \ge 1} C^{\otimes k}.
\]

\noindent
Note that $\partial$ need not preserve $\alg^\epsilon_+$.
A key observation, used extensively in Legendrian knot theory starting with Chekanov \cite{C}, is that this
can be repaired.

Consider the $\coeffs$-algebra automorphism $\phi_\epsilon :\thinspace \alg^\epsilon \to \alg^\epsilon$, determined by
$\phi_\epsilon(a) = a+\epsilon(a)$ for $a\in\rR$.
Conjugating by this automorphism gives rise to a new differential
\[
\dd_\epsilon := \phi_\epsilon \circ \dd \circ \phi_\epsilon^{-1} :
\alg^\epsilon \to \alg^\epsilon.
\]

\begin{proposition}
The differential $\dd_\epsilon$ preserves $\alg^\epsilon_+$, and in particular,
descends to a differential on $\alg^\epsilon_+/(\alg^\epsilon_+)^2 \cong C$.
\end{proposition}
\begin{proof}
Write $\alg^\epsilon = \coeffs \oplus \alg^\epsilon_+$ and denote the projection map $\alg^\epsilon \to \coeffs$ by $\pi$;
then $\pi \dd_\epsilon(a_i) = \pi \phi_\epsilon \dd(a_i) = \epsilon\dd(a_i) = 0$, and it follows that $\pi\dd_\epsilon$ sends $\alg^\epsilon_+$ to $0$.
\end{proof}

Let $C^* := \Hom_\coeffs(C, \coeffs)$. 
The generating set $\rR = \{a_i\}$ for $C$ gives a dual generating
set $\{a_i^*\}$ for $C^*$ with $\langle a_i^*,a_j\rangle = \delta_{ij}$, and we grade $C^*$ by $|a_i^*| = |a_i|$.

Recall that for a $\coeffs$-module $V$,
we write $T(V):= \bigoplus_{n \ge 0} V^{\otimes n}$ for the tensor algebra, and $\overline{T}(V) := \bigoplus_{n \ge 1} V^{\otimes n}$.
The pairing
extends to a pairing between $T(C^*)$ and $T(C)$ determined by
\[
(a_{i_1}a_{i_2}\cdots a_{i_k})^* = (-1)^{\sum_{p<q}|a_{i_p}||a_{i_q}|} a_{i_k}^* \cdots a_{i_2}^* a_{i_1}^*:
\]
that is, $\langle a_{i_k}^*
\cdots a_{i_2}^* a_{i_1}^*,a_{i_1}a_{i_2}\cdots a_{i_k}\rangle = (-1)^{\sum_{p<q}|a_{i_p}||a_{i_q}|}$
and all other pairings are $0$.
(The sign comes from the fact that we are reversing the order of the
$a_i$'s, and is necessary for the dual of a derivation to be a
coderivation, which in turn we need for the correspondence between $A_\infty$ algebras and duals of \dgas{}.)
On the positive part $\overline{T}(C^*)$ of the tensor algebra $T(C^*)$, we define $\dd_\epsilon^*$ to be the co-differential dual to $\dd_\epsilon$:
\[
\langle \dd_\epsilon^* x, y \rangle = \langle x, \dd_\epsilon y\rangle.
\]

Shift gradings by defining
$C^\vee:= C^*[-1]$; then $\overline{T}(C^*) = \overline{T}(C^\vee[1])$.  By Proposition~\ref{prop:bar}, the co-differential
$\dd_\epsilon^*$ now determines an $A_\infty$ structure on $C^\vee$.  We write the corresponding
multiplications as
\[
m_k(\epsilon) :\thinspace (C^\vee)^{\otimes k} \to C^\vee.
\]

Concretely, $m_k(\epsilon)$ is given as follows. For $a\in\rR$, $a$ is a generator of $C$ with dual $a^*\in C^*$. Write the corresponding element of $C^\vee$ as $a^\vee = s^{-1}(a^*)$,
where $s:\thinspace C^\vee \to C^\vee[1]= C^*$ is the degree $-1$ suspension map, and note that
\[
|a^\vee| = |a^*| + 1 = |a| + 1.
\]
Now we have
\begin{align*}
m_k(\epsilon)(a_{i_1}^\vee, \cdots, a_{i_k}^\vee) &= (-1)^{|a_{i_{k-1}}^\vee| + |a_{i_{k-3}}^\vee| + \cdots} s^{-1} \dd_\epsilon^* (a_{i_1}^* \cdots a_{i_k}^*) \\
&= (-1)^{\sum_{p<q} |a_{i_p}||a_{i_q}| + |a_{i_{k-1}}^\vee| + |a_{i_{k-3}}^\vee| + \cdots} s^{-1}
\dd_\epsilon^* (a_{i_k} \cdots a_{i_1})^*,
\end{align*}
and also
$$ \langle \dd_\epsilon^* (a_{i_k} \cdots a_{i_1})^*, a \rangle = \langle (a_{i_k} \cdots a_{i_1})^*,
\dd_\epsilon a \rangle =
\mathrm{Coeff}_{a_{i_k} \cdots a_{i_1}}(\dd_\epsilon a).$$
Combining these, and using the fact that
\[
\sum_{1\leq p<q\leq k} |a_{i_p}||a_{i_q}| + |a_{i_{k-1}}^\vee| +
|a_{i_{k-3}}^\vee| + \cdots
\equiv
\sum_{1\leq p<q\leq k} |a_{i_p}^\vee| |a_{i_q}^\vee|  + \sum_j
(j-1)|a_{i_j}^\vee| + k(k-1)/2 \pmod{2},
\]
we obtain the following formula for $m_k$ in terms of the differential
$\dd_\epsilon$:
\begin{equation} \label{eq:ms}
m_k(\epsilon)(a_{i_1}^\vee, \cdots, a_{i_k}^\vee) =  (-1)^\sigma
\sum_{a \in \rR} a^\vee
\cdot \mathrm{Coeff}_{a_{i_k} \cdots a_{i_1}}(\dd_\epsilon a),\end{equation}
where
\[
\sigma =k(k-1)/2 + \left(\sum_{p<q}
  |a_{i_p}^\vee||a_{i_q}^\vee|\right) + |a_{i_2}^\vee| +
|a_{i_4}^\vee| + \cdots.
\]
For future reference, we note in particular that
\[
\sigma = \begin{cases} 0 & k=1 \\
|a_{i_1}^\vee||a_{i_2}^\vee| + |a_{i_2}^\vee|+1 & k=2.
\end{cases}
\]

We write $C^\vee_\epsilon := (C^\vee, m_1(\epsilon), m_2(\epsilon), \ldots)$ to mean $C^\vee$ viewed as an $A_\infty$ algebra, rather than just as a $\coeffs$-module.
In this context, and when there is no risk of confusion, we simply write $m_k$ for $m_k(\epsilon)$.

\subsection{Link grading}
\label{ssec:alglinkgrading}

Here we give several viewpoints on link grading, which is an additional structure on the \dga{} of a Legendrian link in the case where the link has multiple components; the notion and name are due to Mishachev \cite{Mishachev}. We then discuss how it interacts with the $A_\infty$ structure from Section~\ref{ssec:dga-augs}.

\begin{definition}
Let $(\alg,\dd)$ be a semi-free \dga{} with generating set $\sS = \rR \sqcup \tT$. An \textit{$m$-component weak link grading} on $(\alg,\dd)$ is a choice of a pair of maps
\[
r,c : \sS \to \{1,\ldots,m\}
\]
satisfying the following conditions: \label{def:linkgrading}
\begin{enumerate}
\item
for any $a \in \rR$ with $r(a) \neq c(a)$, each term in $\dd a$ is an integer multiple of a word of the form $x_1 \cdots x_k$ where $c(x_i)=r(x_{i+1})$ for $i=1,\ldots,k-1$ and $r(x_1) = r(a)$, $c(x_k) = c(a)$ (such a word is called ``composable'');
\item
for any $a \in \rR$ with $r(a) = c(a)$, each term in $\dd a$ is either composable or constant (an integer multiple of $1$);
\item for any $i$, we have $r(t_i) = c(t_i^{-1})$ and $c(t_i) = r(t_i^{-1})$.
\end{enumerate}
The maps $r,c$ form an \emph{$m$-component link grading} if they also satisfy
\begin{enumerate}
\setcounter{enumi}{3}
\item $r(t_i) = c(t_i) = r(t_i^{-1}) = c(t_i^{-1})$ for all $i$.
\end{enumerate}
We write $\sS^{ij} := (r \times c)^{-1}(i, j)$, and likewise $\rR^{ij}$ and $\tT^{ij}$.
We call elements of $\sS^{ii}$ {\em diagonal} and elements of $\sS^{ij}$ for $i \ne j$ {\em off-diagonal}.
Note that all elements of $\tT$ are required to be diagonal in a link grading.
\end{definition}

The motivation for this definition is that if $\Lambda = \Lambda_1 \sqcup \cdots \sqcup \Lambda_m$ is an $m$-component Legendrian link, then the \dga{} for $\Lambda$ has an $m$-component link grading: for each Reeb chord $a$, define $r(a)$ (respectively $c(a)$) to be the number of the component containing the endpoint (respectively beginning point) of $a$, and define $r(t_i)  = c(t_i) = r(t_i^{-1}) = c(t_i^{-1})$ to be the number of the component containing the corresponding base point. More generally, if $\Lambda$ is partitioned into a disjoint union of $m$ sublinks (where each may consist of more than one link component), then the \dga{} for $\Lambda$ similarly has a natural $m$-component link grading.

Given a \dga{} with an $m$-component weak link grading, a related \dga{} to consider is the ``composable \dga{}'' $(\alg',\dd')$, cf.\ \cite[\S 4.1]{BEE}. Here $\alg'$ is generated over $\bZ$ by
\[
\sS' = \rR \sqcup \tT \sqcup \{e_1,\ldots,e_m\}
\]
with $r,c$ extended to $\sS'$ by defining $r(e_i) = c(e_i) = i$, quotiented by the relations
\begin{itemize}
\item
$xy=0$ if $x,y \in \sS'$ with $c(x) \neq r(y)$
\item
$t_i \cdot t_i^{-1} = e_{r(t_i)}$ and $t_i^{-1} \cdot t_i = e_{c(t_i)}$
\item
for $x \in \sS'$, $x e_i = x$ if $c(x) = i$, and $e_i x = x$ if $r(x) = i$
\item $1 = \sum_{i=1}^m e_m$.
\end{itemize}
The differential $\dd'$ is defined identically to $\dd$, extended by $\dd'(e_i) = 0$, except that for each Reeb chord $a$ with $r(a) = c(a)$, each constant term $n\in\bZ$ in $\dd a$ is replaced by $n e_{r(a)}$: that is, the idempotent $e_i$ corresponds to the empty word on component $i$. We can now write
\[
\alg' = \bigoplus_{i,j=1}^m (\alg')^{ij}
\]
where $(\alg')^{ij}$ is generated by words $x_1\cdots x_k$ with $r(x_1) = i$ and $c(x_k) = j$, and $\dd'$ splits under this decomposition.

It will be useful for us to have a reformulation of the composability properties of $(\alg',\dd')$ in terms of matrices. To this end, consider the algebra morphism
\begin{align*}
\ell: \alg' & \to \alg \otimes \mathrm{End}(\bZ^m) && \\
x & \mapsto x \otimes |r(x) \rangle \langle c(x)| &&x \in \sS \\
e_i & \mapsto 1 \otimes |i \rangle \langle i| &&i=1,\ldots,m,
\end{align*}
where $|r \rangle \langle c|$ is the $m\times m$ matrix whose $(r,c)$ entry is $1$ and all other entries are $0$. Note that $\alg \otimes \mathrm{End}(\bZ^m)$, i.e., the $m\times m$ matrices with coefficients in $\alg$, is naturally a \dga{}: it is a tensor product of \dgas{}, where $\mathrm{End}(\bZ^m)$ carries the $0$ differential. (That is, the differential $\partial$ on
$\alg \otimes \mathrm{End}(\bZ^m)$ acts entry by entry.) The weak link grading property now just states that $\ell$ is a \dga{} map from $(\alg',\dd')$ to $(\alg \otimes \mathrm{End}(\bZ^m),\dd)$.

For a variant on this perspective, and the one that we will largely use going forward, suppose that $r,c$ is a weak link grading and that $\epsilon :\thinspace \alg \to \coeffs$ is an augmentation. We say that $\epsilon$ \textit{respects the link grading on $\alg$} if $\epsilon(a) = 0$ for all $a\in \rR$ with $r(a) \neq c(a)$ (``mixed Reeb chords''); note that $\epsilon(t_i) = \epsilon(t_i^{-1})^{-1} \neq 0$ for all $i$, so $r(t_i)=c(t_i)$ and thus $r,c$ must be an actual link grading.  In this case, the twisted differential $\dd_\epsilon = \phi_\epsilon \circ \dd \circ \phi_\epsilon^{-1}$ preserves the link grading, and we can drop the discussion of idempotents $e_i$ since $\dd_\epsilon$ contains no constant terms. More precisely, recall that $\alg^\epsilon$ is the $\coeffs$-algebra
$(\alg \otimes \coeffs) / (t_i = \epsilon(t_i) )$, and define the $\coeffs$-algebra map
\begin{align*}
\ell : \alg^\epsilon& \to \alg^\epsilon \otimes \mathrm{End}(\bZ^m) && \\
a & \mapsto a \otimes |r(a) \rangle \langle c(a)| &&a\in \rR.
\end{align*}
Then the structure of the $m$-component link grading implies that $\ell$ is a \dga{} map from $(\alg^\epsilon,\dd_\epsilon)$ to $(\alg^\epsilon \otimes \mathrm{End}(\bZ^m),\dd_\epsilon)$.\footnote{For yet another perspective, one can combine the twisted differential with the composable algebra. Consider the path algebra $\alg''$ over $\coeffs$ on the quiver whose vertices are $1,\ldots,m$ and whose edges are the Reeb chords $a$, where edge $a$ goes from vertex $i$ to vertex $j$ if $r(a) = i$ and $c(a)=j$. Then $\dd_\epsilon$ descends to a differential on $\alg''$ that respects the splitting $\alg'' = \oplus_{i,j} (\alg'')^{ij}$, where $(\alg'')^{ij}$ is generated as a $\coeffs$-module by paths beginning at $i$ and ending at $j$. In this context, the idempotent $e_i$ corresponds to the empty path at $i$.}

For the remainder of this subsection, we suppose that $(\alg, \partial)$ is a semi-free \dga{} equipped with a link grading.

\begin{proposition}
The two-sided ideal generated by the off-diagonal generators is preserved by $\dd$.  More generally,
if $\pi: \{1, \ldots, m \} = P_1 \sqcup \cdots \sqcup P_k$ is any partition, let $J_\pi$ be the two-sided ideal
generated by all elements $a$ with $r(a), c(a)$ in different parts.  Then $J_\pi$ is preserved by $\dd$.

\end{proposition}
\begin{proof} Let $g$ be an off-diagonal generator, and
$y_1 \cdots y_k$ be a word in $\dd g$.  Then $r(g) = r(y_1)$, $c(y_i) = r(y_{i+1})$, and $c(y_k) = c(g)$.
So if moreover $r(y_i) = c(y_i)$ for all $i$, we would have $r(g) = c(g)$, a contradiction.

The argument in
the more general case is similar.
\end{proof}

Note that $\alg / J_\pi$
remains a semi-free algebra with generators $\tT$ and some subset of $\rR$;
it moreover inherits the link grading.

\begin{definition}
For a partition $\pi$ of $\{1,\ldots,m\}$, we write $\alg_\pi := \alg/J_\pi$. In the special case where $\pi = I \sqcup I^c$ for some $I \subset \{1,\ldots,n\}$, we write $\alg_I$ for the subalgebra of $\alg_{I \sqcup I^c}$ generated
by the elements of $\coprod_{i,j \in I} \sS^{ij}$ (that is, further quotient by elements $a$ with $r(a),c(a) \in I^c$).  Finally, we will write $\alg_i := \alg_{\{i\}}$.
\end{definition}

\begin{proposition} For any $I \subset \{1, \ldots, m\}$, the algebra $\alg_I$ is preserved by the differential inherited
by $\alg_{I \sqcup I^c}$.
\end{proposition}

\begin{proposition} \label{prop:restrictedfactors}
For any partition $\pi: \{1, \ldots, m \} = P_1 \sqcup \cdots \sqcup P_k$, we have
$\alg_\pi = \alg_{P_1} \star \cdots \star \alg_{P_k}$.
\end{proposition}

In particular, an augmentation of $\alg$ which
annihilates generators $a$ with $r(a), c(a)$ in different parts is the same as a tuple of augmentations of the $\alg_{P_\alpha}$.

Let $\epsilon: \alg \to \coeffs$ be an augmentation.
We write $C^{ij}$ for the free $\coeffs$-submodule of $C$ generated by
$\rR^{ij}$, so that $C = \bigoplus_{i,j} C^{ij}$. Similarly we split $C^{\vee} = \bigoplus C^{\vee}_{ij}$.
The product then splits into terms
$$m_k(\epsilon): C_{i_1j_1}^\vee \otimes C_{i_2j_2}^\vee \otimes \cdots \otimes C_{i_kj_k}^\vee \to C_{ij}^\vee.$$
\begin{proposition}
Assume $\epsilon$ respects the link grading.
Then the product $m_k(\epsilon): C_{i_1j_1}^\vee \otimes C_{i_2j_2}^\vee \otimes \cdots \otimes C_{i_kj_k}^\vee \to C_{ij}^\vee$ vanishes
unless $i_k = i$, $j_1 = j$, and $i_r = j_{r+1}$.
\end{proposition}
\begin{proof}
Up to a sign, the coefficient of $a^\vee$ in the product $m_k(a_{n_1}^\vee,\ldots,a_{n_k}^\vee)$ is the coefficient of $a_{n_k} \cdots a_{n_1}$ in
$\dd_\epsilon a$.
Since $a \in \rR^{ij}$, this vanishes unless $i_k = i$, $j_1 = j$, and $i_r = j_{r+1}$.
\end{proof}

\noindent
That is, the nonvanishing products are:
\[
m_k :\thinspace C_{i_k i_{k+1}}^\vee \otimes \cdots \otimes C_{i_1 i_2}^\vee  \to
C_{i_1 i_{k+1}}^\vee.
\]

\begin{proposition} \label{prop:linkcat}
Let $\alg$ be a semi-free \dga{} with an $m$-component link grading.  Let $\epsilon$ be an augmentation
which respects the link grading.
There is a (possibly nonunital)
$A_\infty$ category on the objects $\{1, \ldots, m\}$ with morphisms
$\Hom(i, j) = C_{ij}^\vee$, with multiplications $m_k$ as above.
\end{proposition}
\begin{proof}
The $A_\infty$ relations on the category follow from the $A_\infty$ relations on the algebra $C^\epsilon$,
as per Remark \ref{rem:ainftycatrelations}.
\end{proof}

\begin{proposition} \label{prop:purity}
Let $\epsilon: \alg \to \coeffs$ be an augmentation respecting the link grading.
Let $\pi$ be a partition of $\{1,\ldots,m\}$.
Suppose $i_0, \ldots, i_k$ are in the same part $P$ of $\pi$.
Then computing $m_k$ in $\alg$, $\alg_{\pi}$, and $\alg_P$ gives the same result.
\end{proposition}
\begin{proof}
The element $m_k$ is computed using the length $k$ terms of the twisted differential in which
the terms above appear.  The assumption that the augmentation respects the link grading means
that off-diagonal terms will not contribute new things to the twisted differential.
\end{proof}

\subsection{$A_\infty$-categories from sequences of \dgas{}}

For bookkeeping, we introduce some terminology.
We write $\Delta_+$ for the category whose objects are the sets $[m] :=\{1, \ldots, m\}$ and whose morphisms are the order-preserving inclusions.
Such maps $[m] \to [n]$ are enumerated by $m$-element subsets of $[n]$; we denote the map corresponding to $I \subset [n]$ by $h_I: [m] \to [n]$.
We call a covariant functor $\Delta_+ \to \cC$ a co-$\Delta_+$ object of $\cC$.\footnote{Co-$\Delta_+$ is pronounced ``semi-cosimplicial.''
We only use $\Delta_+$  for bookkeeping -- while the following construction bears some family resemblance to taking a resolution of $\alg^{(1)}$, we have
been unable to express it in this manner.}
For a co-$\Delta_+$ object
$X: \Delta_+ \to \cC$, we write $X[m] := X(\{1, \ldots, m\})$.  We denote the structure map $X[m] \to X[n]$ corresponding to a subset $I \subset [n]$ also
by $h_I$.

For example, $\Delta_+$ itself, or more precisely the inclusion $\Delta_+ \to \mathrm{Set}$,
is a co-$\Delta_+$ set.    Another example of a
co-$\Delta_+$ set is the termwise square of this, $\Delta_+^2$, which has $\Delta_+^2[m] = \{1, \ldots, m\}^2$.

\begin{definition} \label{def:consistent}
A sequence  $\alg^{(\bullet)}$ of semi-free \dgas{} $(\alg^{(1)}, \partial)$, $(\alg^{(2)},\partial), \ldots$ with generating sets
$\sS_1, \sS_2, \ldots$ is \emph{consistent} if it comes equipped with the following additional structure:
\begin{itemize}
\item the structure of a co-$\Delta_+$ set $\sS$ with $\sS[m] = \sS_m$;
\item  link gradings $\sS_m \to \{1,\ldots, m\} \times \{1, \ldots, m\}$.
\end{itemize}
This structure must satisfy the following conditions.  First, the  link grading should give a morphism of co-$\Delta_+$ sets
$\sS \to \Delta_+^2$.  Second, for any $m$-element subset $I \subset [n]$, note that the map $h_I: \sS_m \to \sS_n$
induces a morphism of algebras $h_I: \alg^{(m)} \to \alg^{(n)}_{I}$.
We require this map be an isomorphism of \dgas{}.
\end{definition}

\begin{remark}
There is a co-$\Delta_+$ algebra $\alg$ with $\alg[m] = \alg^{(m)}$ and the structure maps induced from the structure maps
on the $\sS_m$.   This however is generally {\em not} a co-$\Delta_+$ \dga{} -- the morphisms do not respect the differential.
\end{remark}

\begin{lemma}
Let $\alg^{(\bullet)}$ be a consistent sequence of \dgas{}.  Then in particular:
\begin{itemize}
\item The map $h_i: \alg^{(1)} \to \alg^{(m)}_i$ is an isomorphism, and
$$\alg^{(m)}_{\{1\} \sqcup \{2\} \sqcup \cdots \sqcup \{m\}} = \alg^{(m)}_1 \star \alg^{(m)}_2 \star \cdots \star \alg^{(m)}_{m} =
h_1(\alg^{(1)}) \star \cdots \star h_m(\alg^{(1)})= \alg^{(1)} \star \cdots \star \alg^{(1)}.$$
In particular, an $m$-tuple of augmentations of $\alg^{(1)}$ induces a augmentation of $\alg^{(m)}$ which respects the link grading.
\item The map $h_{ij}: \sS_2 \to \sS_m$ induces a bijection $h_{ij}: \sS_2^{12} \to \sS_m^{ij}$ and hence an isomorphism
$h_{i,j}: C^{\vee}_{12} \to C^{\vee}_{ij}$.
\item Let $(\epsilon_1, \epsilon_2, \ldots, \epsilon_{m})$ be a tuple of augmentations of $\alg^{(1)}$, and let $\epsilon$ be the
corresponding diagonal augmentation of $\alg^{(m)}$.  Let $1 \le i_1 < i_2  < \cdots < i_{k+1} \le m$ be any increasing sequence.
Then the composite morphism
\begin{equation} \label{eq:composition}
C^\vee_{12} \otimes \cdots \otimes C^\vee_{12}
\xrightarrow{h_{i_{k}i_{k+1}} \otimes \cdots \otimes h_{i_1i_2}}
C^\vee_{i_ki_{k+1}} \otimes \cdots \otimes  C^\vee_{i_1i_2} \xrightarrow{m_k(\epsilon)} C^\vee_{i_1i_{k+1}} \xrightarrow{h_{i_1i_{k+1}}^{-1}} C^\vee_{12}
\end{equation}
does not depend on anything except the tuple $(\epsilon_{i_1}, \ldots, \epsilon_{i_{k+1}})$.   
\end{itemize}
\end{lemma}
\begin{proof}
The first part of the first statement holds by definition; the equation following is Proposition \ref{prop:restrictedfactors}.  The second
statement again holds by definition.  The third statement is Proposition \ref{prop:purity}.
\end{proof}

We will associate an $A_\infty$-category to a consistent sequence of \dgas{}.

\begin{definition} \label{def:aug-from-consistent-sequence}
Given a consistent sequence of \dgas{} $(\alg^{(m)},\partial)$ and a coefficient field
$\coeffs$, we define the \textit{augmentation category}
$\Aug(\alg^{(\bullet)},\coeffs)$
as follows:
\begin{itemize}
\item
The objects are augmentations $\epsilon: \alg^{(1)} \to \coeffs$.
\item
The morphisms are
\[
\hom(\epsilon_1,\epsilon_2) := C_{12}^\vee \subset \alg^{(2)},
\]
where $\epsilon$ is the diagonal augmentation $(\epsilon_1,\epsilon_2)$.
\item
For $k\geq 1$, the
composition map
\[
m_k :\thinspace \hom(\epsilon_{k} ,\epsilon_{k+1}) \otimes \cdots \otimes \hom(\epsilon_{2},\epsilon_3) \otimes \hom(\epsilon_{1},\epsilon_{2})
\to \hom(\epsilon_{1},\epsilon_{k+1})
\]
is defined to be
the map  
of \eqref{eq:composition}.
\end{itemize}
\end{definition}

\begin{proposition} \label{prop:consistentC}
$\Aug(\alg^{(\bullet)},\coeffs)$
is an $A_\infty$ category.
\end{proposition}
\begin{proof}
The $A_\infty$ relations can then be verified by observing that all compositions relevant to any finite relation can be computed
in some fixed $A_\infty$ category of the sort constructed in Proposition \ref{prop:linkcat}.
\end{proof}

\begin{remark}
We emphasize that the $A_\infty$ algebra $\Hom_+(\epsilon, \epsilon)$ {\em is not} the $A_\infty$ algebra obtained by
dualizing $(\alg^{(1)})^\epsilon$.  In particular, the former can be
unital when the latter is not.
\end{remark}

\begin{definition}  \label{def:consistentsequencemorphisms}
Given two consistent sequences $(\alg^{(\bullet)}, \partial)$ and $(\mathcal{B}^{(\bullet)},\partial)$,
we say a sequence of \dga{} morphisms
\[
f^{(m)} :(\alg^{(m)}, \partial) \rightarrow (\mathcal{B}^{(m)},\partial)
\]
is \emph{consistent} if:
\begin{enumerate}
\item Each $f^{(m)}$ preserves the subalgebra generated by the invertible generators.
\item The $f^{(m)}$ are compatible with the  link gradings in the following sense.  For any generator, $a_i \in \sS_m$, $f(a_i)$ is a $\bZ$-linear combination of  composable words in $\mathcal{B}^{(m)}$ from $r(a_i)$ to $c(a_i)$,  i.e. words of the form $x_1 \cdots x_k$ with $c(x_{i}) = r(x_{i+1})$ for $i =1, \ldots, k-1$, and $r(x_1) = r(a_i)$, $c(x_k) = c(a_i)$.  Note that constant terms are allowed if $r(a_i) = c(a_i)$.

As a consequence of this requirement, a well-defined \dga{} morphism
of composable algebras arises from taking $(f^{(m)})'(a_i)$ to be
$f^{(m)}(a_i)$ with all occurrences  $1$ replaced with the idempotent $e_{r(a_i)}$ for generators $a_i$ of $\alg^{(m)}$.
 Moreover, the following square commutes:
\begin{equation} \label{eq:composable}
\begin{gathered}
\xymatrix{
(\alg^{(m)})' \ar[r]^{(f^{(m)})'} \ar[d]^{\ell} & (\mathcal{B}^{(m)})' \ar[d]^{\ell} \\
\alg^{(m)} \otimes \mathrm{End}(\bZ^m)  \ar[r]^{f^{(m)} \otimes 1} &
\mathcal{B}^{(m)} \otimes \mathrm{End}(\bZ^m).
}
\end{gathered}
\end{equation}

\item
For any $I: [m] \hookrightarrow [n]$, 
note that, by the previous axiom, $f^{(n)}$ induces a well defined homomorphism $f^{(n)}_I: (\alg^{(n)}_I, \partial) \rightarrow (\mathcal{B}^{(n)}_I, \partial)$.  We require the following diagram to commute:
\begin{equation}
\begin{gathered}
\label{eq:consistency}
\xymatrix{
\alg^{(m)} \ar[r]^{f^{(m)}} \ar[d] & \mathcal{B}^{(m)} \ar[d] \\
\alg^{(n)}_I \ar[r]^{f^{(n)}_I} & \mathcal{B}^{(n)}_I,
}
\end{gathered}
\end{equation}
where the vertical arrows are the definitional isomorphisms $h_I$.
\end{enumerate}
\end{definition}

A consistent sequence of \dga{} morphisms
$f^{(m)}: (\alg^{(m)}, \partial) \rightarrow
({\mathcal{B}}^{(m)}, \partial)$
gives rise to an $A_\infty$-functor
\[
F:\Aug({\mathcal{B}}^{(\bullet)}, \coeffs) \rightarrow \Aug(\alg^{(\bullet)}, \coeffs)
\]
according to the following construction.  On objects, for an augmentation $\epsilon : ({\mathcal{B}}^{(1)}, \partial) \rightarrow (\coeffs, 0)$ we define
\[
F(\epsilon) = f^\ast \epsilon := \epsilon \circ f
\]
where $f := f^{(1)}: (\alg^{(1)}, \partial) \rightarrow ({\mathcal{B}}^{(1)}, \partial)$.  Next, we need to define maps
\[
F_k: \hom^{{\mathcal{B}}}(\e_k,\e_{k+1})  \otimes \cdots \otimes  \hom^{{\mathcal{B}}}(\e_1,\e_2)\rightarrow \hom^{\alg}(f^*\e_1, f^*\e_{k+1}).
\]
Consider the diagonal augmentation $\e= (\e_1, \ldots, \e_{k+1})$ of ${\mathcal{B}}^{(k+1)}$, and let
$f^{(k+1)}_{\e} := \Phi_\e \circ f^{(k+1)} \circ \Phi^{-1}_{(f^{(k+1)})^*\epsilon}$.  Here, we used that $f^{(k+1)}$ passes to a well defined map $(\alg^{(k+1)})^{(f^{(k+1)})^*\e}  \rightarrow ({\mathcal{B}}^{(k+1)})^\e$.
Observe that  $f^{(k+1)}_{\e} ( (\alg^{(k+1)})^{(f^{(k+1)})^*\e}_+ )\subset ({\mathcal{B}}^{(k+1)})^\e_+$,
i.e. no constant terms appear in the image of generators.  We then define $F_k$, up to the usual grading shift,
by dualizing the component of $f^{(k+1)}_{\e}$ that maps from
\[
C^{1,k+1} \rightarrow {C'}^{1,2} \otimes \cdots \otimes {C'}^{k,k+1}
\]
and making use of the consistency of the sequence to identify the grading-shifted duals ${C'}^\vee_{i,i+1}$ and  $C^\vee_{1,k+1}$ with
$\hom^{\mathcal{B}}(\e_{i},\e_{i+1})$ and $\hom^{\alg}(f^\ast \e_1, f^\ast \e_{k+1})$ respectively.

\begin{proposition}  \label{prop:consistentF}
If the sequence of \dga{} morphisms $f^{(m)}$ is consistent, then $F$ is an $A_\infty$-functor.  Moreover, this construction defines a functor from the category of consistent sequences of DGAs and DGA morphisms to $A_\infty$ categories.
\end{proposition}

\begin{proof}
Using the third stated property of a consistent sequence, we see that the required relation for the map
$F_k: \hom^{\mathcal{B}}(\e_k,\e_{k+1}) \otimes \cdots \otimes \hom^{\mathcal{B}}(\e_1,\e_2)  \rightarrow \hom^{\alg}(f^*\e_1, f^*\e_{k+1})$
follows from the identity $f^{(k+1)}_\epsilon \partial_{(f^{(k+1)})^\ast \epsilon} = \partial_{ \epsilon} f^{(k+1)}_\epsilon$.  That the construction preserves compositions and identity morphisms is clear from the definitions.
\end{proof}

\subsubsection{The negative augmentation category}

For a given consistent sequence of \dgas{}
$(\alg^{(\bullet)}, \partial)$ ,
there is a kind of dual consistent sequence obtained by reversing the order of components in the link grading.  That is, for each $m \geq 1$, we form a new link grading, $(r \times c)^*$, as the composition
\[
 \sS_m \stackrel{r\times c}{\rightarrow} \{1, \ldots, m\} \stackrel{\tau}{\rightarrow} \{m, \ldots, 1 \}
\]
where $\tau$ reverses the ordering:  $\tau(k) = m-k+1$.  The structure of a consistent sequence for this new link grading is then provided by altering the maps $h_I$ to $h_I^* = h_{\tau(I)}$.

\begin{definition} \label{def:augbc}
Given a consistent sequence of \dgas{} 
$(\alg^{(\bullet)},\partial)$ and a coefficient ring
$\coeffs$, we define the \textit{negative augmentation category}
$\AugBC(\alg^{(\bullet)}, \coeffs)$
to be the augmentation category associated, as in Definition \ref{def:aug-from-consistent-sequence},  to the sequence of \dgas{} $(\alg^{(m)}, \partial)$ with link grading $(r \times c)^*$ and co-$\Delta_+$ set structure on the $\sS_m$ given by the $h_I^*$.
\end{definition}

The category 
$\AugBC(\alg^{(\bullet)}, \coeffs)$
can also be described in a straightforward manner in terms of the original link grading and $h_I$ for $(\alg^{(m)}, \coeffs)$ as follows:
\begin{itemize}
\item
The objects are augmentations $\epsilon: \alg^{(1)} \to \coeffs$.
\item
The morphisms are
\[
\homBC(\epsilon_2,\epsilon_1) := C_{21}^\vee \subset \alg^{(2)}
\]
where $\epsilon$ is the diagonal augmentation
$(\epsilon_1,\epsilon_2)$
(note the reversal of the order of inputs).
\item For $k \geq 1$, let $(\epsilon_1, \epsilon_2, \ldots, \epsilon_{k+1})$ be a tuple of augmentations of $\alg^{(1)}$, and let $\epsilon$ be the
corresponding diagonal augmentation of $\alg^{(k+1)}$.
Then
\[
m_k :\thinspace \homBC(\epsilon_{2},\epsilon_{1}) \otimes \homBC(\epsilon_{3},\epsilon_2) \otimes \cdots \otimes
\homBC(\epsilon_{k+1} ,\epsilon_{k})  \to \homBC(\epsilon_{k+1},\epsilon_{1})
\]
is the composite morphism
\begin{equation} 
C^\vee_{21} \otimes \cdots \otimes C^\vee_{21}
\xrightarrow{h_{1 2} \otimes \cdots \otimes  h_{k,k+1} }
  C^\vee_{21} \otimes \cdots \otimes  C^\vee_{k+1,k}  \xrightarrow{m_k(\epsilon)} C^\vee_{k+1, 1} \xrightarrow{h_{1, k+1}^{-1}} C^\vee_{21}.
\label{eq:compositionminus}
\end{equation}
\end{itemize}

\begin{remark}
\begin{itemize}
\item[(i)]  In the preceding formulas, objects were indexed in a
  manner that is reverse to our earlier notations.  This is to allow
  for easy comparison of the operations in
$\Aug(\alg^{(\bullet)}, \coeffs)$ and $\AugBC(\alg^{(\bullet)}, \coeffs)$
that correspond to a common diagonal augmentation $\e = (\e_1, \ldots, \e_{k+1})$ of $(\alg^{(k+1)}, \partial)$.

\item[(ii)]  The subscripts of the $h$ maps are {\em the same} as in \eqref{eq:composition}.  However, in these two settings, they are applied to different generators from the $\sS_m$.
\end{itemize}
\end{remark}

\begin{proposition}\label{prop:bifunctor}
The map $(\epsilon_1, \epsilon_2) \mapsto \hom(\epsilon_1, \epsilon_2)$ underlies the structure of an $A_\infty$ bifunctor from $\AugBC$ to
chain complexes and likewise
the map $(\epsilon_1, \epsilon_2) \mapsto \homBC(\epsilon_1, \epsilon_2)$ underlies the structure of an $A_\infty$ bifunctor from $\Aug$ to
chain complexes.
\end{proposition}
\begin{proof}
Consider the diagonal augmentation $\epsilon$ on $\alg^{(3)}$ induced by the tuple $(\epsilon_1, \epsilon_2, \epsilon_3)$.
Then the composition on $C^\epsilon$ gives us in particular:
\begin{align*}
m_2&: \hom(\epsilon_1,\epsilon_3) \otimes \homBC(\epsilon_2,\epsilon_1)  =
C_{13}^\vee \otimes C_{21}^\vee  \to C_{23}^\vee = \hom(\epsilon_2, \epsilon_3) \\
m_2&: \homBC(\epsilon_3,\epsilon_2) \otimes \hom(\epsilon_1,\epsilon_3)   =
C_{32}^\vee \otimes C_{13}^\vee  \to C_{12}^\vee = \hom(\epsilon_1,
\epsilon_2) \\
m_2&: \hom(\epsilon_1,\epsilon_2) \otimes \homBC(\epsilon_3,\epsilon_1)  =
C_{12}^\vee \otimes C_{31}^\vee  \to C_{32}^\vee = \homBC(\epsilon_3, \epsilon_2) \\
m_2&: \homBC(\epsilon_3,\epsilon_1)  \otimes \hom(\epsilon_2,\epsilon_3)  =
C_{31}^\vee \otimes C_{23}^\vee  \to C_{21}^\vee = \homBC(\epsilon_2, \epsilon_1).
\end{align*}
The first two and the analogous higher compositions give $\hom$ the structure of a bifunctor on $\AugBC$, since the compositions
are taking place in an $A_\infty$ algebra as described in Proposition \ref{prop:linkcat}.
Similarly, the second two and their higher variants give $\homBC$ the structure of a bifunctor on $\Aug$.
\end{proof}

\begin{remark}
Note from the proof of Proposition~\ref{prop:bifunctor} that we have
maps
\[
m_2 : \hompm(\epsilon_2,\epsilon_3) \otimes
\hompm(\epsilon_1,\epsilon_2) \to \hompm(\epsilon_1,\epsilon_3)
\]
for all choices of $(\pm,\pm,\pm)$ except $(+,+,-)$ and
$(-,-,+)$. These six choices correspond to the six different ways to
augment the components of the $3$-copy with
$\epsilon_1,\epsilon_2,\epsilon_3$ in some order. For $(+,+,+)$ and
$(-,-,-)$, we recover the usual $m_2$ multiplication in the $A_\infty$
categories $\Aug$ and $\AugBC$.
\end{remark}

\subsection{A construction of unital categories.} \label{ssec:unital-construction}

Let $(\alg,\partial)$ be a semi-free \dga{} with generating set $\sS = \rR \sqcup \tT$ where $\rR = \{a_1, \ldots, a_r\}$ and $\cT= \{t_1, t_1^{-1}, \ldots, t_M, t_M^{-1}\}$.  Suppose further that $(\alg,\partial)$ is equipped with a weak link grading $(r \times c): \sS \to \{1, \ldots, l\} \times \{1, \ldots, l\}$.
(As in Definition \ref{def:linkgrading}, this means $r \times c$ satisfies all the conditions of a link grading {\em except} that the elements of $\tT$ are not required to be diagonal.)

We will construct a consistent sequence from the above data.\footnote{The following construction comes from the geometry
of the $m$-copies of a Lagrangian projection (cf.\ Proposition \ref{prop:mcopyMultiple}), but we require it in some non-geometric
settings in order to prove invariance.  Thus it is convenient to carry out the algebra first.
In the geometric case, the identity
$(\partial^m)^2 = 0$ is automatic because $\partial^m$ is the differential of a C--E \dga{}.}

\begin{proposition} \label{prop:xysequence}
Let $(\alg,\dd)$ be a semi-free DGA with a weak link grading as above.  We define a sequence of algebras $\alg^{(\bullet)}$ with $\alg^{(1)} = \alg$,
where
$\alg^{(m)}$ has the following generators:
\begin{itemize}
\item
$a_k^{ij}$, where $1 \leq k \leq r$ and $1\leq i,j\leq m$, with degree $|a_k^{ij}| = |a_k|$;
\item
$x_k^{ij}$, where $1 \leq k \leq M$ and $1\leq i<j\leq m$, with degree $|x_k^{ij}| = 0$;
\item
$y_k^{ij}$, where $1 \leq k \leq M$ and $1\leq i<j\leq m$, with degree $|y_k^{ij}| = -1$;
\item invertible generators $(t_k^{i})^{\pm 1}$
where $1 \leq k \leq M$ and $1 \le i \le m$.
\end{itemize}

We organize the generators with matrices.
Consider the following elements of $\operatorname{Mat}(m, \alg^{(m)})$ :
$A_k = (a^{ij}_k)$,  $\Delta_k = \operatorname{Diag}(t_k^1, \ldots, t_k^m)$,
\begin{align*}
X_k &= \left[
\begin{matrix}
1 & x_k^{12} & \cdots & x_k^{1m} \\
0 & 1 & \cdots & x_k^{2m} \\
\vdots & \vdots & \ddots & \vdots \\
0 & 0 & \cdots & 1\\
\end{matrix}
\right],
& 
Y_k &= \left[
\begin{matrix}
0 & y_k^{12} & \cdots & y_k^{1m} \\
0 & 0 & \cdots & y_k^{2m} \\
\vdots & \vdots & \ddots & \vdots \\
0 & 0 & \cdots & 0\\
\end{matrix}
\right].
\end{align*}
We introduce a ring homomorphism
\begin{align*}
\Phi: \alg &\rightarrow \operatorname{Mat}(m, \alg^{(m)}) \\
a_k &\mapsto A_k \\
t_k &\mapsto \Delta_k X_k \\
t_k^{-1} &\mapsto X_k^{-1} \Delta_k^{-1}
\end{align*}
and a $(\Phi, \Phi)$-derivation
\begin{align*}
\alpha_Y: \alg & \rightarrow \operatorname{Mat}(m, \alg^{(m)}) \\
s & \mapsto Y_{r(s)} \Phi(s) - (-1)^{|s|} \Phi(s) Y_{c(s)}, \hspace{6ex} s \in \sS.
\end{align*}
Then there is a unique derivation $\partial^m$ on $\alg^{(m)}$ such that (applying $\partial^m$ to matrices entry by entry):
\begin{align*}
\dd^m \Delta & = 0 \\
\dd^m Y_k & = Y_k^2 \\
\dd^m \circ \Phi & = \Phi \circ \dd + \alpha_Y. 
\end{align*}
Furthermore, this derivation is a differential: $(\partial^m)^2 = 0$.
\end{proposition}

\begin{proof}
The uniqueness of such a derivation follows  because taking $s = t_k$ determines
$\Delta^{-1} \dd^m \Phi(t_k) = \Delta^{-1} \dd^m (\Delta_k X_k) = \dd^m X_k$, and taking $s = a$ determines
$\dd^m \Phi(a_k) = \dd^m A_k$.  Existence follows because the above specifies its behavior on the generators,
and the equation $\partial^m \circ \Phi = \Phi \circ \partial + \alpha_Y$ need only be checked on generators
since both sides are $(\Phi, \Phi)$-derivations.  (Recall that $f$ is a $(\Phi, \Phi)$-derivation when
$f(ab) = f(a) \Phi(b) + (-1)^{|a|} \Phi(a) f(b)$.)

We turn to checking $(\partial^m)^2 =0$.  Evidently
\[
(\dd^m)^2(\Delta_k)=0, \quad (\dd^m)^2 Y_k = (\dd Y_k)Y_k + (-1)^{-1}Y_k(\dd Y_k) = Y_k^3 - Y_k^3 = 0,
\]
and we compute
\[
(\partial^m)^2 \circ \Phi =  \partial^m\circ( \Phi \circ \dd + \alpha_Y )  =
 \Phi \circ \dd^2 + \alpha_Y\circ \dd + \partial^m \circ\alpha_Y  = \alpha_Y\circ \dd + \partial^m \circ\alpha_Y
 \]
so it remains only to show, for any $s \in \sS$, that
\begin{equation} \label{eq:pmaYsaYps}
\partial^m \alpha_Y(s) = -\alpha_Y(\partial s).
\end{equation}

In order to verify this identity, recall from Definition \ref{def:linkgrading} the DGA homomorphism $\ell: \alg' \rightarrow \alg \otimes \mbox{End}(\bZ^l)$ arising from the weak link grading on $\alg$, where $(\alg', \partial')$ denotes the composable algebra and $\alg \otimes \operatorname{End}(\bZ^l)$ has differential $\partial \otimes 1$.
We compose the maps $\Phi \otimes 1$ and $\alpha_Y \otimes 1$ with $\ell$ to define maps
\begin{align*}
\tilde{\Phi} &: \alg' \xrightarrow{\ell} \alg \otimes  \operatorname{End}(\bZ^l) \xrightarrow{\Phi \otimes 1} \operatorname{Mat}(m, \alg^{(m)}) \otimes \operatorname{End}(\bZ^l) \\
\tilde{\alpha}_Y &: \alg' \xrightarrow{\ell} \alg \otimes \operatorname{End}(\bZ^l) \xrightarrow{\alpha_Y \otimes 1} \operatorname{Mat}(m, \alg^{(m)}) \otimes \operatorname{End}(\bZ^l).
\end{align*}
The identity $\dd^m \circ \Phi  = \Phi \circ \dd + \alpha_Y$ immediately implies $(\dd^m \otimes 1) \circ \tilde{\Phi} = \tilde{\Phi} \circ \dd' + \tilde{\alpha}_Y$.
Moreover, if we can show for any $s \in \sS$ that
\begin{equation} \label{eq:pmomidcirc}
(\partial^m \otimes 1) \circ \tilde{\alpha}_Y(s) = -\tilde{\alpha}_Y \circ \partial' (s),
\end{equation}
then \eqref{eq:pmaYsaYps} will follow.
This is because we can then compute
\begin{align*}
(\partial^m \circ \alpha_Y(s) ) \otimes |r(s)\rangle\langle c(s)| &=
(\partial^m \otimes 1) \circ (\alpha_Y\otimes 1) \circ \ell(s) \\
&= (\partial^m \otimes 1) \circ \tilde{\alpha}_Y(s) \\
(-\alpha_Y \circ \partial(s)) \otimes |r(s)\rangle\langle c(s)| &=
(-\alpha_Y \otimes 1) \circ (\partial \otimes 1) \circ \ell (s) \\
&= (-\alpha_Y \otimes 1) \circ \ell \circ \partial'(s) \\
&= -\tilde{\alpha}_Y \circ \partial'(s),
\end{align*}
and these last two quantities are equal.

To establish (\ref{eq:pmomidcirc}), we define an element of $\operatorname{Mat}(m, \alg^{(m)}) \otimes \operatorname{End}(\bZ^l)$ by the formula
\[ \mathbb{Y} = \sum_{i=1}^l Y_i \otimes |i\rangle\langle i| \]
and verify the identities
\begin{align*}
\dd^m \mathbb{Y} &= \mathbb{Y}^2, & \tilde{\alpha}_Y(s) &= [\mathbb{Y}, \tilde{\Phi}(s)]
\end{align*}
where $s\in \sS$ and $[x,y] = xy - (-1)^{|x||y|}yx$ denotes the graded commutator.
Note that $\tilde{\alpha}_Y$ and $[\mathbb{Y}, \tilde{\Phi}(\cdot)]$ are both $(\tilde{\Phi}, \tilde{\Phi})$-derivations from $\alg'$ to $\operatorname{Mat}(m, \alg^{(m)}) \otimes \operatorname{End}(\bZ^l)$.  Therefore, since they agree on a generating set for $\mathcal{A}'$, it follows that
$\tilde{\alpha}_Y(x) = [\mathbb{Y}, \tilde{\Phi}(x)]$ holds for any $x \in \alg'$.

Now the Leibniz rule $\dd[x,y] = [\dd x, y] + (-1)^{|x|}[x,\dd y]$, together with $|\mathbb{Y}|=-1$, gives
\begin{align*}
(\dd^m \otimes 1) \circ \tilde{\alpha}_Y(s) &= [(\dd^m \otimes 1)  \mathbb{Y}, \tilde{\Phi}(s)] - [\mathbb{Y}, (\dd^m \otimes 1) \tilde{\Phi}(s)] \\
&= [\mathbb{Y}^2, \tilde{\Phi}(s)] - [\mathbb{Y}, (\dd^m \otimes 1) \tilde{\Phi}(s)].
\end{align*}
Similarly, we compute that
\begin{align*}
\tilde{\alpha}_Y(\dd' s) &= [\mathbb{Y}, \tilde{\Phi}(\dd' s)] \\
&= [\mathbb{Y}, (\dd^m \otimes 1)  \tilde{\Phi}(s)] - [\mathbb{Y}, \tilde{\alpha}_Y(s)] \\
&= [\mathbb{Y}, (\dd^m \otimes 1)  \tilde{\Phi}(s)] - [\mathbb{Y}, [\mathbb{Y}, \tilde{\Phi}(s)]]
\end{align*}
and we can verify either directly or using the graded Jacobi identity that the last term on the right is equal to $[\mathbb{Y}^2, \tilde{\Phi}(s)]$.  Thus,  $(\dd^m \otimes 1) \circ \tilde{\alpha}_Y(s) = -\tilde{\alpha}_Y(\dd s)$ holds as desired.
\end{proof}

\begin{proposition}  \label{prop:Thealgmabove}
The $\alg^{(m)}$ above comes with a $m$-component link grading given by $(r \times c)(a_k^{ij}) = (r \times c)(x_k^{ij}) = (r \times c)(y_k^{ij})=  (i, j)$  and $(r \times c)(t_k^{i})=  (i, i)$.
Given $I: [m] \hookrightarrow [n]$, we define $h_I(s^{ij}) = s^{I(i), I(j)}$.
This gives $\alg^{(\bullet)}$ the structure of a consistent sequence of \dgas{}.
\end{proposition}
\begin{proof}
By inspection.  The fact that the above formula gives a link grading follows because the differential was defined by a matrix formula in the
first place.  Also, the matrix formulas are identical for all $m \geq 1$, so the identification of generators
extends to a \dga{} isomorphism $(\mathcal{A}^{(m)}, \partial^m) \rightarrow (\mathcal{A}^{(n)}_I, \partial^n)$.
\end{proof}

\begin{remark}
The link grading defined in Proposition \ref{prop:Thealgmabove}
is unrelated to the initial weak link grading on $\alg$ that was used in Proposition \ref{prop:xysequence} in defining differentials on the $\alg^{(m)}$.  In particular, for $\alg^{(1)} = \alg$ the two gradings are distinct if the initial weak link grading has $l > 1$.
\end{remark}

\begin{proposition}  \label{thm:unital}
Let $\alg$ be a DGA with weak link grading, and $\alg^{(\bullet)}$ the consistent sequence from
Proposition \ref{prop:xysequence}.
Then the $A_\infty$ category $\Aug(\alg^{(\bullet)})$ 
is strictly unital, with the unit being given by
\[
e_\epsilon = -\sum_{j =1}^M (y^{12}_j)^{\vee} \in \hom(\epsilon,\epsilon)
\]
for any $\epsilon \in \Aug(\alg^{(\bullet)})$.
\end{proposition}

\begin{proof}
We recall the properties of a strict unit element:
we must show that $m_1(e_\epsilon)=0$, that $m_2(e_{\epsilon_1}, a) = m_2(a, e_{\epsilon_2}) = a$ for any $a \in \hom(\epsilon_1, \epsilon_2)$, and that all higher compositions involving $e_\epsilon$ vanish.

Inspection of the formula for $\dd^2: \alg^{(2)}  \rightarrow \alg^{(2)}$ yields
\begin{align*}
\dd^2(a_k^{12}) &= y_{r(a_k)}^{12} a_k^{22} - (-1)^{|a_k|} a_k^{11} y_{c(a_k)}^{12} + \cdots \\
\dd^2(x_k^{12}) &= (t_{k}^{11})^{-1} y_{r(t_k)}^{12} t_{k}^{22} - y_{c(t_k)}^{12} \\
\dd^2(y_k^{12}) &= 0,
\end{align*}
and so if we write $\dd_{(\epsilon,\epsilon)}$ for the differential $\phi_{(\epsilon,\epsilon)} \circ \dd \circ \phi_{(\epsilon,\epsilon)}^{-1}$ on $\alg^{(2)}$, then for $1 \leq k \leq r$
the coefficient of $(a_k^{12})^{\vee}$ in $-m_1 e_\epsilon$ is
\[
\langle m_1 \sum_{j=1}^M (y^{12}_j)^{\vee}, (a^{12}_k)^{\vee} \rangle = \langle \sum_{j=1}^M y^{12 }_j , \partial_{(\e,\e)} a^{12}_k \rangle = \langle y_{r(k)}^{12} + y_{c(k)}^{12}, \partial_{(\e,\e)} a^{12}_k \rangle =  \epsilon(a_k) - (-1)^{|a_k|} \epsilon(a_k) = 0.
\]
(In the final equality, we used the fact that $\epsilon(a_k)=0$ unless
$|a_k|=0$.)
A similar computation shows that $\langle m_1 \sum_{j=1}^M (y^{12}_j)^{\vee}, (x^{12}_k)^{\vee} \rangle = 0$, and $\langle m_1 \sum_{j=1}^M (y^{12}_j)^{\vee}, (y^{12}_k)^{\vee} \rangle = 0$ holds since $\partial y^{12}_k = 0$.  Thus $m_1(e_\e) = 0$.

The formula for $\dd^3: \alg^{(3)}  \rightarrow \alg^{(3)}$ yields
\begin{align*}
\dd^3(a_k^{13}) &= y_{r(k)}^{12} a_k^{23} - (-1)^{|a_k|} a_k^{12} y_{c(k)}^{23} + \cdots \\
\dd^3(x_k^{13}) &= (t_{k}^{11})^{-1} y_{r(k)}^{12} t_{k}^{22} x_k^{23} - x_k^{12} y_{c(k)}^{23} + \cdots \\
\dd^3(y_k^{13}) &= y_k^{12} y_k^{23}.
\end{align*}
Using \eqref{eq:ms}, we calculate that 
\[
m_2(e_\e,(a_k^{12})^{\vee}) = 
(-1)^{|a_k^\vee|+1}(-1)^{|a_k|} (a_k^{12})^{\vee} = (a_k^{12})^{\vee}
\]
and similarly $m_2((a_k^{12})^{\vee},e_\e) = (a_k^{12})^{\vee}$.
In the same manner, we find that
$m_2(e_\e,(x_k^{12})^{\vee}) = m_2((x_k^{12})^{\vee},e_\e) =
(x_k^{12})^{\vee}$
and $m_2(e_\e, (y_k^{12})^{\vee}) = (y_k^{12})^{\vee}$; note that for
$m_2((x_k^{12})^{\vee},e_\e) =
(x_k^{12})^{\vee}$, we have $e_\e \in \hom(\e,\e)$ and $x_k^\vee \in
\hom(\e,\e')$ for some $\e,\e'$, and the corresponding diagonal augmentation $(\e,\e,\e')$ of $\alg^{(3)}$ sends both $t_k^{11}$ and $t_k^{22}$ to $\e(t_k)$.

Finally, all higher order compositions involving $e_\e$ vanish for the following reason: In any differential of a generator in any of the $\alg^{(m)}$, the $y_k^{ij}$ appear only in words that have at most $2$ non-$t$ generators.  Therefore, when $\e$ is a pure augmentation of $\alg^{(m)}$, occurences of $y_k^{ij}$ in the differential of the other generators must be in words of length $2$ or less.
\end{proof}

\begin{proposition}  \label{prop:functorialityprop}
Let $f: (\alg, \dd) \rightarrow (\mathcal{B}, \dd)$ be a DGA morphism between
algebras with weak link gradings (with the same number of components),
 which respects the weak link gradings in the sense of (2) from Definition \ref{def:consistentsequencemorphisms}.
 Then $f$ extends, in a canonical way, to a consistent sequence of morphisms
\[
f^{(m)} :(\alg^{(m)}, \dd^m) \rightarrow (\mathcal{B}^{(m)}, \dd^m)
\]
inducing a unital $A_\infty$ morphism of categories
$\Aug(\mathcal{B}^{(\bullet)}) \to \Aug(\mathcal{A}^{(\bullet)})$.
This construction defines a functor, i.e. it preserves identity morphisms and compositions.
\end{proposition}
\begin{proof}

Given $f: (\alg, \partial) \rightarrow (\mathcal{B}, \partial)$ we produce morphisms $f^{(m)}$, $m\geq1$, by requiring that the following matrix formulas hold
(again applying $f^{(m)}$ entry-by-entry):
\begin{align*}
f^{(m)}(\Delta_k) &= \Delta_k, & f^{(m)}(Y_k) &= Y_k,
\end{align*}
and when $x \in \alg$ is a generator,
\begin{equation} \label{eq:fmPhialg}
f^{(m)} \circ \Phi_\alg(x) = \Phi_{\mathcal{B}} \circ f(x).
\end{equation}
(Note that taking $x=t_k$ uniquely specifies $f^{(m)}(X_k)= \Delta_k^{-1} \cdot \Phi_\mathcal{B} \circ f(t_k) $.)  This characterizes the value of $f^{(m)}$ on generators, and we extend $f^{(m)}$ as an algebra homomorphism.  Equation (\ref{eq:fmPhialg}) then holds for all $x \in \alg$, as the morphisms on both sides are algebra homomorphisms.

Next, note that the $(\Phi,\Phi)$-derivation $\alpha_Y:\alg \rightarrow \operatorname{Mat}(m,\alg^{(m)})$ satisfies
\[
\alpha_Y(w) = Y_i \, \Phi(w) - (-1)^{|w|} \Phi(w)\, Y_j
\]
for any composable word in $\alg$ from $i$ to $j$.  This is verified by inducting on the length of $w$:  if $w = a \cdot b$ with $a$ composable from $i$ to $k$ and $b$ composable from $k$ to $j$, then
\begin{align*}
\alpha_Y(a b) &= \alpha_Y(a) \, \Phi(b) + (-1)^{|a|} \Phi(a) \, \alpha_Y(b) \\
&= (Y_i \, \Phi(a) - (-1)^{|a|} \Phi(a) \, Y_k) \Phi(b) + (-1)^{|a|}\Phi(a) (Y_k \, \Phi(a) - (-1)^{|b|} \Phi(a)\, Y_j) \\
&= Y_i \,\Phi(a b) -(-1)^{|a\cdot b|} \Phi(a b)\, Y_j.
\end{align*}
Because $f$ respects the link gradings, if $x \in \sS^{ij}$ is a generator of $\alg$ then $f(x)$ is a $\bZ$-linear combination of composable words from $i$ to $j$ in $\mathcal{B}$, so we  have
\begin{align}\label{eq:alphaYf}
\begin{split}
f^{(m)}\circ \alpha_Y(x) & =  f^{(m)}(Y_i \,\Phi_\alg(x) - (-1)^{|x|} \Phi_\alg(x) \, Y_j) \\
&= Y_i\, (f^{(m)}\circ \Phi_\alg)(x) - (-1)^{|x|} (f^{(m)} \circ \Phi_\alg)(x) \, Y_j \\
&= Y_i \, (\Phi_{\mathcal{B}}\circ f)(x) - (-1)^{|f(x)|}
(\Phi_{\mathcal{B}}\circ f)(x) \, Y_j \\
& =
\alpha_Y \circ f(x).
\end{split}
\end{align}

To verify that $f^{(m)}$ is a DGA map, we need to verify that $f^{(m)} \partial^m = \partial^m f^{(m)}$ holds when applied to any generator of $\alg^{(m)}$.  For the entries of $\Delta$ or $Y$, this is immediate.  For the remaining generators, it suffices to compute using (\ref{eq:fmPhialg}) and (\ref{eq:alphaYf})
that for $x \in \sS$,
\begin{align*}
f^{(m)} \circ \partial^m \circ\Phi (x) &= f^{(m)} \circ \Phi \circ\partial(x) + f^{(m)} \circ \alpha_Y(x) \\
&= \Phi \circ f \circ \partial(x) + \alpha_Y \circ f(x) \\
&= \Phi \circ \partial \circ f(x) + \alpha_Y \circ f(x) \\
&= \partial ^m \circ \Phi  \circ f (x) - \alpha_Y \circ f(x)+ \alpha_Y \circ f(x) \\
&= \partial^m \circ f^{(m)} \circ \Phi(x).
\end{align*}
The consistency of the $f^{(m)}$ follows since the matrix formulas used for different $m$ all appear identical; we get a morphism of $A_\infty$ categories
 by Proposition \ref{prop:consistentF}.  The construction preserves identities by inspection.

That the construction of this Proposition defines a functor is clear from the definitions combined with the functoriality of the construction in Proposition \ref{prop:consistentF}.
\end{proof}

\newpage

\section{The augmentation category of a Legendrian link}

\label{sec:augcat}

In this section, we apply the machinery from Section~\ref{sec:augcatalg} to
define a new category $\Aug(\Lambda)$ whose objects are
augmentations of a Legendrian knot or link $\Lambda$ in $\bR^3$. As mentioned in the Introduction, this category
is similar to, but in some respects crucially different from, the
augmentation category constructed by Bourgeois and Chantraine in
\cite{BC}, which we write as $\AugBC(\Lambda)$. Our approach in fact allows us to treat the two categories
as two versions of a single construction, and to investigate the
relationship between them.

We begin in Section~\ref{ssec:augcatdef} by considering the link consisting of $m$ parallel copies of $\Lambda$
for $m \geq 1$, differing from each other by translation in the
Reeb direction, and numbered sequentially. In the language of
Section~\ref{sec:augcatalg}, the \dgas{} for these $m$-copy links form a
consistent sequence of \dgas{}, and we can dualize, using Proposition~\ref{prop:consistentC}, to obtain an
$A_\infty$ category: $\Aug$ if the components are ordered from top to
bottom, and $\AugBC$ if from bottom to top.

Associating a \dga{} to the $m$-copy of $\Lambda$ requires a choice of
perturbation; the construction of $\AugBC$ is independent of this
perturbation, but $\Aug$ is not. For the purposes of defining $\Aug$,
we consider two explicit perturbations, the Lagrangian and the
front projection $m$-copies. In Section~\ref{ssec:dgas-unitality}, we show that the $A_\infty$ category
associated to the Lagrangian perturbation is constructed algebraically
from the \dga{} of $\Lambda$
using Proposition~\ref{prop:xysequence}, and conclude that $\Aug$ is
unital.

In Section~\ref{ssec:invariance}, we then proceed to prove invariance
of $\Aug$ under choice of perturbation and Legendrian isotopy of
$\Lambda$. In Section~\ref{ssec:exs}, we present computations of
$\Aug$ and $\AugBC$ for some examples.

\subsection{Definition of the augmentation category} \label{ssec:augcatdef}

We recall our
contact conventions.  For a manifold $M$, we denote the first jet space
by $J^1(M) = T^*M \times
\bR_z$, the subscript indicating that we use $z$ as the coordinate in the $\bR$ direction.
We choose the contact form $dz-\lambda$ on $J^1(M)$,
where $\lambda$ is the Liouville $1$-form on $T^*M$
(e.g. $\lambda=y\,dx$ on $T^*\bR = \bR^2$).  With these conventions, the Reeb
vector field is $\partial/\partial z$.

\begin{definition}
Let $\Lambda \subset J^1(M)$ be a Legendrian.
For $m\geq 1$, the \textit{$m$-copy} of $\Lambda$, denoted
$\Lambda^m$, is the disjoint union of $m$ parallel copies of
$\Lambda$, separated by small translations in the Reeb ($z$)
direction. We label the $m$ parallel copies $\Lambda_1,\ldots,\Lambda_m$ {\bf from top} (highest $z$ coordinate)
{\bf to bottom} (lowest $z$ coordinate).
\end{definition}

The $m$-copy defined above is not immediately suitable for Legendrian contact homology, as the space of
Reeb chords is not discrete; we need to perturb the $m$-copy so that there are finitely many Reeb chords.
A standard method for
perturbing a Legendrian is to work within a Weinstein
neighborhood of $\Lambda$, contactomorphic to a neighborhood of the
$0$-section in $J^1(\Lambda)$.  One then chooses a $C^1$-small
function $f: \Lambda \rightarrow \bR$, and replaces $\Lambda$ with the
$1$-jet of several small multiples of $f$ (along with another small perturbation to make the picture generic, see ``Lagrangian projection $m$-copy'' below).  In order to apply the algebraic constructions of the previous section, it will be important to perturb the $m$-copies of $\Lambda$ in a consistent manner, i.e, in a way that produces a consistent sequence of DGAs. We will do this only in the $1$-dimensional case; see Remark \ref{rem:higher-dim} for a discussion of issues involved with extending to higher dimensions.


We now specialize to the case of a $1$-dimensional Legendrian
$\Lambda \subset J^1(\bR)$ where we use coordinates $(x,y,z) \in J^1(\bR) = T^*\R \times \R$.  
In this case, we introduce two perturbation\footnote{Strictly speaking the resolution construction does not produce a $C^0$-small  perturbation of the original Legendrian, although we occasionally make this abuse in our terminology.} methods for the $m$-copies of $\Lambda$ that result in consistent sequences of DGAs, one described in terms of the Lagrangian ($xy$) projection of $\Lambda$ and other via the front projection ($xz$) projection of $\Lambda$. 
Recall from Section \ref{ssec:dga-background} that the resolution procedure \cite{NgCLI} gives a
 Legendrian isotopic link  whose Reeb chords (crossings in the $xy$ diagram)
are in one-to-one correspondence with the crossings and right cusps of
the front.


\begin{figure}[H]
\centering
\labellist
\small\hair 2pt
\pinlabel $c_1$ [l] at 288 145
\pinlabel $c_2$ [l] at 288 73
\pinlabel $(c_1)$ [r] at 0 145
\pinlabel $(c_2)$ [r] at 0 73
\pinlabel $a_1$ at 49 126
\pinlabel $a_2$ at 145 126
\pinlabel $a_3$ at 238 126
\endlabellist
\includegraphics[scale=0.6]{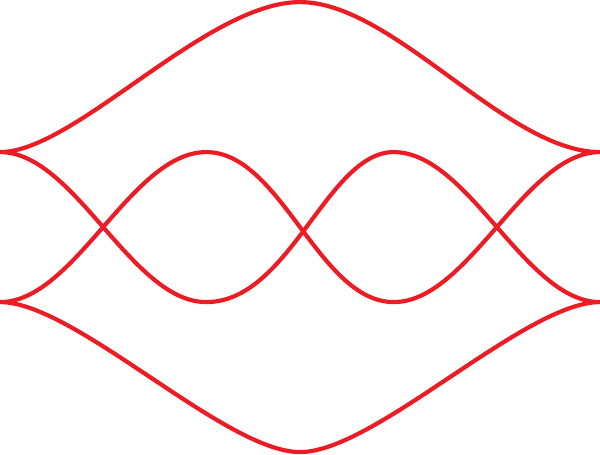}
\hspace{5ex}
\labellist
\small\hair 2pt
\pinlabel $a_1$ at 22 130
\pinlabel $a_2$ at 128 130
\pinlabel $a_3$ at 234 128
\pinlabel $a_4$ at 289 182
\pinlabel $a_5$ at 290 74
\endlabellist
\includegraphics[scale=0.5]{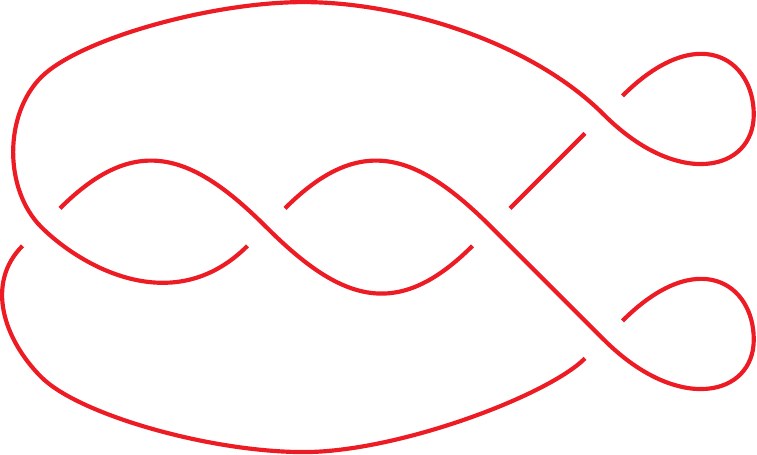}
\caption{
The Legendrian trefoil, in the front (left) and $xy$ (right)
projections, with Reeb chords labeled (and a correspondence chosen
between left and right cusps in the front projection).
}
\label{fig:trefoil-1}
\end{figure}
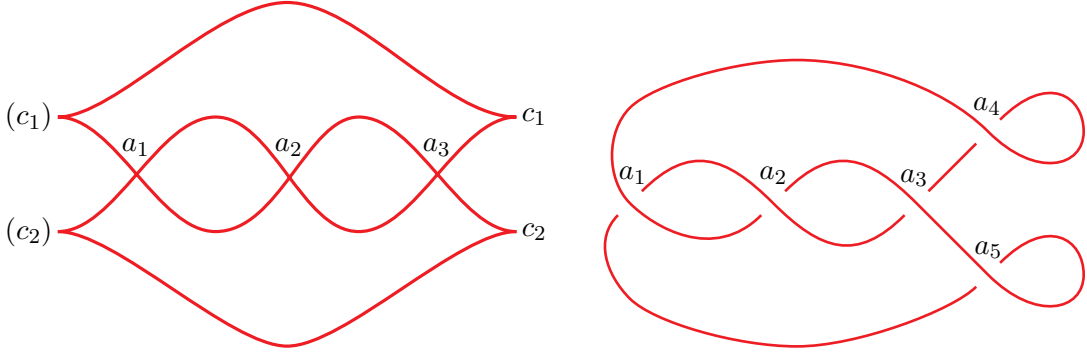

Here are our two perturbation schemes in more detail:
\begin{itemize}
\item ``Front projection $m$-copy.'' Beginning with a front projection
  for $\Lambda$, take $m$ copies of this front, separated by
  small translations in the Reeb direction, and labeled $1,\ldots,m$
  from top to bottom; then resolve to get an $xy$ projection, or
  equivalently use the formulation for the \dga{} for fronts from
  \cite{NgCLI}.  Typically, we denote this version of the $m$-copy as $\Lambda^m_{xz}$.

\begin{figure}[H]
\labellist
\small\hair 2pt
\pinlabel $\Lambda_1$ [r] at -4 136
\pinlabel $\Lambda_2$ [r] at -4 104
\pinlabel $\Lambda_3$ [r] at -4 72
\pinlabel $t_{k,1}$ [l] at 42 148
\pinlabel $t_{k,2}$ [l] at 42 116
\pinlabel $t_{k,3}$ [l] at 42 84
\pinlabel $x_k^{23}$ [tr] at 82 64
\pinlabel $x_k^{13}$ [bl] at 134 78
\pinlabel $x_k^{12}$ [tr] at 116 32
\pinlabel $\Lambda_1$ [r] at 332 8
\pinlabel $\Lambda_2$ [r] at 332 40
\pinlabel $\Lambda_3$ [r] at 332 72
\pinlabel $y_k^{23}$ [tl] at 460 64
\pinlabel $y_k^{13}$ [br] at 406 78
\pinlabel $y_k^{12}$ [tl] at 428 32
\endlabellist
\centering
\includegraphics[scale=0.6]{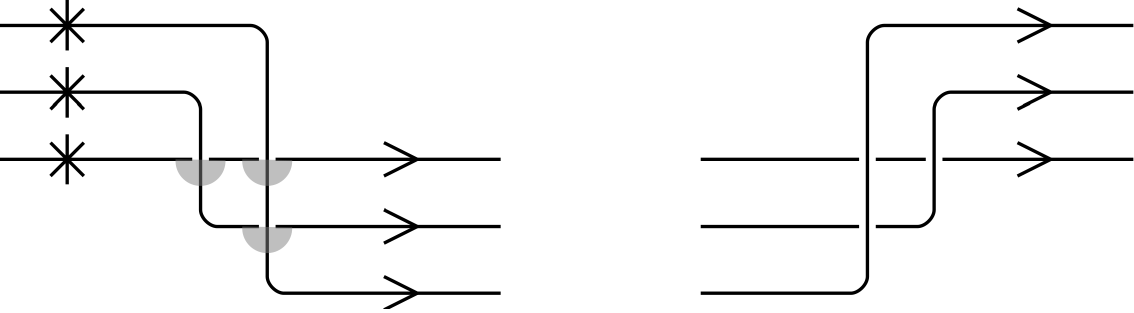}
\caption{
The $xy$
projection of $\Lambda^m_f$ pictured near local maxima (left) and local minima (right) of $f$. The shaded quadrants of the $x^{ij}_k$ indicate negative orientation signs.  Note that in intervals bordered on the left by a local minimum of $f$ and on the right by a local maximum of $f$ the components appear from top to bottom (with respect to the oriented normal to $\Lambda$) in order $\Lambda_1, \ldots, \Lambda_m$, and in remaining intervals the top to bottom ordering is $\Lambda_m, \ldots, \Lambda_1$.
}
\label{fig:HalfTwists}
\end{figure}
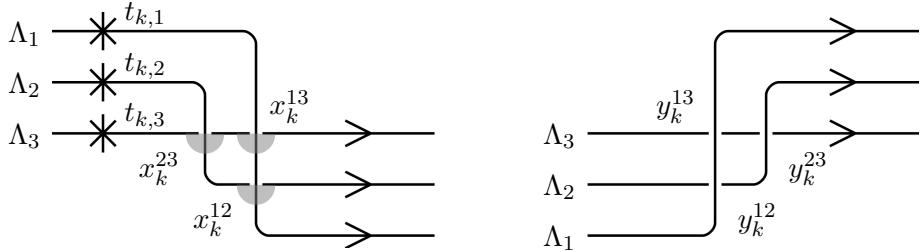

\item ``Lagrangian projection $m$-copy.'' Beginning with an $xy$
  projection for $\Lambda$ (which e.g.\ can be obtained by resolving a
  front projection), take $m$ copies separated by small translations
  in the Reeb ($z$) direction.
	Let $f: \Lambda \rightarrow \R$ be a Morse function whose critical points are distinct from the crossing points of the $xy$-projection. Use this
  function to perturb the copies in the
normal direction to the knot in the $xy$ plane. Away from critical points of $f$, the result appears as $m$
parallel copies of the $xy$ projection of $\Lambda$, 
while the $xy$ projection remains $m$-to-$1$ at critical points of $f$.  Finally, perturb the projection near critical points of $f$ so that a left-handed (resp. right-handed) half twist appears as in Figure \ref{fig:HalfTwists} when passing local minima (resp. local maxima) of $f$ according to the orientation of $\Lambda$.  We denote this perturbed $m$-copy as $\Lambda^m_f$.

An example of this construction where $f$ has only two critical points, with the local minimum placed just to the right of the local maximum appears, in Figure~\ref{fig:trefoil-3-xy}.  Here, the two half twists fit together to form what is commonly called a  ``dip'' in the $xy$-projection, cf. \cite{Sab}.
\end{itemize}

Associating a Legendrian contact homology \dga{} to the perturbed
$m$-copy $\Lambda^m_{xz}$ or $\Lambda^m_f$ requires a further choice
of Maslov potentials to determine the grading, as well as a choice of
orientation signs and base points.  Suppose that a choice of Maslov
potential, orientation signs, and base points has been made for $\Lambda$ itself.  As usual, we require that each component of $\Lambda$ contains at least one base point, and we further assume that the locations of base points are distinct from local maxima and minima of $f$.
Then,  we equip each of the parallel components of $\Lambda^m_{xz}$
and $\Lambda^m_f$ with the identical Maslov potential, and place base points on each of the copies of $\Lambda$ in $\Lambda^m_{xz}$ or $\Lambda^m_{f}$ in the same locations as the base points of $\Lambda$.  Finally, we
assign orientation signs as follows.  Any even-degree crossing of
$\Lambda^m_{xz}$ corresponds to an even-degree crossing of $\Lambda$
(the crossings that appear near cusps all have odd degree): we assign
orientation signs to agree with the orientation signs of $\Lambda$.  A
similar assignment of orientation signs to $\Lambda^m_f$ is made, with
the following addition for the crossings of $\pi_{xy}(\Lambda^m_f)$
that are created near critical points of $f$ during the perturbation
process, which do not correspond to any crossing of
$\pi_{xy}(\Lambda)$: only the crossings near local maxima of $f$ have even degree, and they are assigned  orientation signs as pictured in Figure \ref{fig:HalfTwists}.

\begin{proposition} \label{prop:mcopyseq}
Given a Legendrian $\Lambda \subset J^1(\bR)$, the following collections of \dgas{} underlie consistent sequences:
\begin{itemize}
\item The ``front projection $m$-copy'' algebras $(\alg(\Lambda^m_{xz}), \partial)$.
\item The ``Lagrangian projection $m$-copy'' algebras $(\alg(\Lambda^m_f), \partial)$, for a fixed Morse function $f$.
\end{itemize}
\end{proposition}
\begin{proof}
The data of a consistent sequence is an $m$-component link grading on the $m$-th algebra, plus the structure of a
co-$\Delta_+$ set on the generators.  Writing $\sS_m$ for the generators, i.e. Reeb chords and base points, of $\Lambda^m$,
the data of the link grading is associated to the decomposition
$\Lambda^m = \Lambda_1 \sqcup \cdots \sqcup \Lambda_m$ as discussed in
Section \ref{sssec:linkG}.  That is, the  map $r \times c:
\mathcal{S}_m \to \{1,\ldots,m\} \times \{1,\ldots, m\}$  sends a base
point on the $i$-th copy
to $(i,i)$, and a Reeb chord that {\bf ends on} the $i$-th copy and {\bf begins on} the $j$-th copy to $(i, j)$.
In both of the $m$-copy constructions above, the Lagrangian projection of the link resulting from removing any $n-m$ pieces of $\Lambda^n$ looks identical to $\Lambda^m$;
this gives the co-$\Delta_+$ set structure, and makes the desired isomorphisms obviously hold.
\end{proof}

\begin{definition}
We write $\Aug(\Lambda_f,\coeffs)$ for the $A_\infty$-category that is
associated by Definition~\ref{def:aug-from-consistent-sequence} to the
sequence of $m$-copy \dgas{} $(\alg(\Lambda^\bullet_f), \partial)$.
\label{def:augplus-xy-xz}
Likewise we write $\Aug(\Lambda_{xz},\coeffs)$ for the category associated to $(\alg(\Lambda^\bullet_{xz}), \partial)$.
\end{definition}

\begin{remark}[grading]
If $r(\Lambda)$ denotes the $\gcd$ of the rotation numbers of the
components of $\Lambda$, then recall from
Section~\ref{ssec:dga-background} that the \dga{} for $\Lambda$ is
graded over $\bZ/2r$.
\label{rmk:grading}
Later in this paper, when we prove the equivalence of augmentation and sheaf categories, we will assume that $r(\Lambda)=0$ and thus that the \dga{} is $\bZ$-graded. For the purposes of constructing the augmentation category, however, $r(\Lambda)$ can be arbitrary; note then that augmentations $\e$ must satisfy the condition $\e(a_i) = 0$ for $a_i \not\equiv 0 \bmod{2r}$. Indeed, we can further relax the grading on the \dga{} and on augmentations to a $\bZ/m$ grading where $m\,|\,2r$, as long as either $m$ is even or we work over a ring with $-1=1$, cf.\ the first paragraph of Section~\ref{ssec:dga-augs}.
\end{remark}

In Proposition \ref{prop:mcopyMultiple}, we will show the sequence
$(\alg(\Lambda^m_f), \partial)$ arises by applying the construction of Proposition \ref{prop:xysequence} to $\alg(\Lambda_f)$,
and deduce that $\Aug(\Lambda_f,\coeffs)$ is unital.  In Theorem
\ref{prop:Invariance}, we will show that, up to $A_\infty$ equivalence,
the category $\Aug(\Lambda_f,\coeffs)$ does not depend on the choice of $f$, and moreover is invariant under Legendrian isotopy.  In addition, if $\Lambda$ is assumed to be in plat position, $\Aug(\Lambda_f,\coeffs)$ and $\Aug(\Lambda_{xz},\coeffs)$ are shown to be equivalent.
Thus we will usually suppress the perturbation method from notation
and denote any of these categories simply by $\Aug(\Lambda,\coeffs)$, which we call the {\bf positive
  augmentation category} of $\Lambda$ (with coefficients in
$\coeffs$).

The category
$\Aug(\Lambda, \coeffs)$ is summarized in the following:
\begin{itemize}
\item
The objects are augmentations $\epsilon: \alg(\Lambda) \to \coeffs$.
\item
The morphisms are
\[
\hom(\epsilon_1,\epsilon_2) := C_{12}^\vee,
\]
the
$\coeffs$-module generated by Reeb chords that end on $\Lambda_1$ and begin on $\Lambda_2$
in the $2$-copy $\Lambda^2$.
\item
For $k\geq 1$, the
composition map
\[
m_k :\thinspace \hom(\epsilon_{k} ,\epsilon_{k+1})  \otimes \cdots \otimes \hom(\epsilon_{1},\epsilon_{2})
\to \hom(\epsilon_{1},\epsilon_{k+1})
\]
is defined to be
the map $m_k :\thinspace  C_{k,k+1}^\vee \otimes \cdots \otimes C_{12}^{\vee}
\to C_{1,k+1}^\vee$
given by the
diagonal
  augmentation $\epsilon = (\epsilon_1,\ldots,\epsilon_{k+1})$ on
the $(k+1)$-copy $\Lambda^{k+1}$. (Note that in the Legendrian
literature, diagonal augmentations are often called ``pure.'')
\end{itemize}
Here, one of the allowed perturbation methods, as in Sections \ref{sssec:frontcopy} and \ref{sssec:lagrcopy}, must be used when producing the \dgas{} of the $m$-copies $\Lambda^m$.

\begin{remark}  \label{rem:higher-dim}
It should be possible to define the augmentation category in an 
analogous manner for any Legendrian submanifold $\Lambda$ of a $1$-jet
space $J^1(M)$.  Some key technical points that would need to be addressed to rigorously establish the augmentation category in higher dimensions include: producing
a  consistent sequence of \dgas{} via appropriate perturbations of the
$m$-copies of $\Lambda$ (or showing how to work around this point); proving independence of choices made to produce such perturbations; and establishing Legendrian isotopy invariance.  The construction of augmentation categories for Legendrians in $J^1(\mathbb{R})$ given in this article is also valid for Legendrians in $J^1(S^1)$.  When $\dim(M) \geq 2$, we leave the rigorous construction of the positive augmentation category as an open problem.
\end{remark}

Before turning to a more concrete description of the $m$-copy algebras
$(\alg(\Lambda^m_{xz}), \partial)$ and $(\alg(\Lambda^m_f),\partial)$
underlying the definition of $\Aug(\Lambda_{xz},\coeffs)$ and
$\Aug(\Lambda_f,\coeffs)$, we consider the corresponding negative augmentation
category.

\begin{definition} \label{def:AugBC}
Given a Legendrian submanifold $\Lambda \subset J^1(M)$ and a
coefficient ring $\coeffs$, we define the
\textit{negative augmentation category} to be the $A_\infty$ category
$\AugBC(\Lambda,\coeffs)$ obtained by applying Definition~\ref{def:augbc} to any of the consistent
sequences of \dgas{} introduced in Proposition~\ref{prop:mcopyseq}.
\end{definition}

The category $\AugBC(\Lambda,\coeffs)$ is summarized as follows:
\begin{itemize}
\item
The objects are augmentations $\epsilon: \alg(\Lambda) \to \coeffs$.
\item
The morphisms are
\[
\homBC(\epsilon_2,\epsilon_1) := C_{21}^\vee,
\]
the vector space generated by Reeb chords that end on $\Lambda_2$ and begin on $\Lambda_1$ in the $2$-copy $\Lambda^2$.
\item
For $k\geq 1$, the
composition map
\[
m_k :\thinspace \homBC(\epsilon_{2},\epsilon_{1}) \otimes \homBC(\epsilon_{3},\epsilon_2) \otimes \cdots \otimes
\homBC(\epsilon_{k+1} ,\epsilon_{k})  \to \homBC(\epsilon_{k+1},\epsilon_{1})
\]
is defined to be
the map $m_k :\thinspace  C_{2 1}^{\vee} \otimes \cdots \otimes \thinspace
C_{k+1,k}^\vee
\to C_{k+1,1}^\vee$
given by the pure augmentation $\epsilon = (\epsilon_1,\ldots,\epsilon_{k+1})$ on
the $(k+1)$-copy $\Lambda^{k+1}$.
\end{itemize}

The key distinction between augmentation categories $\Aug$ and
$\AugBC$ is that $\AugBC$ does not depend on the choice of
perturbation. This is because the short Reeb chords introduced in the
perturbation belong to $C_{ij}^\vee$ for $i<j$ but not for $i>j$.
Note that $C_{ij}^\vee$ is always a space of homs from $\epsilon_i$ to $\epsilon_j$, but is $\hom$ if $i < j$ and $\homBC$ if $i > j$.
One might ask about $C_{ii}^\vee$; one can show this to be the same as $\homBC(\epsilon_i, \epsilon_i)$.

The negative augmentation category $\AugBC(\Lambda,\coeffs)$ is
not new: it was defined by Bourgeois and Chantraine \cite{BC},
and was the principal
inspiration and motivation for our definition of $\Aug(\Lambda,\coeffs)$.

\begin{proposition} The category $\AugBC(\Lambda,\coeffs)$ is the
  augmentation category of
\label{prop:neg-is-BC}
Bourgeois and Chantraine
\cite{BC}.
\end{proposition}

\begin{proof}
This is
proven in Theorem~3.2 of \cite{BC} and the
discussion surrounding it. There it is shown that the \dga{} for the
$n$-copy of $\Lambda$, quotiented out by short Reeb chords
corresponding to critical points of the perturbing Morse function,
produces the $A_\infty$ operation $m_{n-1}$ on their augmentation
category. In our formulation for the Lagrangian projection $m$-copy in
Section~\ref{sssec:lagrcopy}, the critical points of the
perturbing Morse function are of the form $x^{ij}_k,y^{ij}_k$ with
$i<j$. It follows the short Reeb chords do not contribute in our
definition of $\AugBC(\Lambda,\coeffs)$, and thence that our
definition agrees with Bourgeois--Chantraine's.
\end{proof}

\begin{remark}
Our sign conventions for $\AugBC$ differ from the conventions of
Bourgeois and Chantraine, because of differing sign conventions
for $A_\infty$ operations. See the discussion at the beginning of
Section~\ref{sec:a-infinity}.
\end{remark}

\begin{remark}
To follow up on the previous discussion of short chords, the absence
of short chords in $\rR^{ij}$ when $i>j$ allows one to describe $\AugBC(\Lambda)$
algebraically from $(\alg(\Lambda),\partial)$  in a manner that is more direct than for $\Aug(\Lambda)$,
as the extra data of a perturbing function $f$ is unnecessary.  In fact, Bourgeois--Chantraine's original definition of
$\AugBC$ is purely algebraic.
\end{remark}

\begin{remark}Our choice of symbols $+$ and $-$ has to do with the interpretation that,
for augmentations which come from fillings, the first corresponds
to computing positively infinitesimally wrapped Floer homology, and
the second to computing negatively infinitesimally wrapped Floer
homology.
\label{rmk:infwrapped}
See Section \ref{sec:exseq}.
\end{remark}

\begin{remark}
Bourgeois and Chantraine prove invariance of $\AugBC$ in
\cite{BC}.
\label{rmk:BCinvariance}
One can give an alternate proof using the techniques of
the present paper, using the invariance of $\Aug$
(Theorem~\ref{prop:Invariance}), the existence of a morphism from
$\AugBC$ to $\Aug$ (Proposition~\ref{prop:nu-morphism}) and the exact
sequence relating the two (Proposition~\ref{prop:pmexactseq}), and the
fact that isomorphism in $\Aug$ implies isomorphism in $\AugBC$
(Proposition~\ref{prop:isom-pm}). We omit
the details here.
\end{remark}

\subsection{\dgas{} for the perturbations and unitality of $\Aug$}
\label{ssec:dgas-unitality}

We now turn to an explicit description of the \dgas{} for the $m$-copy
of $\Lambda$, in terms of the two perturbations introduced in
Section~\ref{ssec:augcatdef}. The front projection $m$-copy
$\Lambda^m_{xz}$ is useful
for computations (cf.\ Section~\ref{sssec:TrefoilEx}), while the
Lagrangian projection $m$-copy $\Lambda_f^m$ leads immediately to a proof that
$\Aug(\Lambda_f,\coeffs)$ is unital.

\subsubsection{Front projection $m$-copy}  \label{sssec:frontcopy}

For the front projection $m$-copy, we adopt matching notation for the Reeb chords of $\Lambda$ and $\Lambda^m_{xz}$.
Label the crossings of $\Lambda$ by $a_1,\ldots,a_p$ and the right cusps of $\Lambda$ by $c_1,\ldots,c_q$, and choose a pairing of right cusps of $\Lambda$ with left cusps of $\Lambda$. See the left side of Figure~\ref{fig:trefoil-1} for an illustration.
Then each crossing $a_k$ in the front for $\Lambda$ gives rise to $m^2$ crossings $a_k^{ij}$ in $\Lambda^m_{xz}$, where $a_k^{ij} \in \rR^{ij}$; note that the overstrand (more negatively sloped strand) at $a_k^{ij}$ belongs to component $i$, while the understrand (more positively sloped strand) belongs to component $j$. Each right cusp $c_k$ for $\Lambda$ similarly gives rise to $m^2$ crossings and right cusps $c_k^{ij}$ in $\Lambda^m_{xz}$, where $c_k^{ij} \in \rR^{ij}$:
\begin{itemize}
\item
$c_k^{ii}$ is the cusp $c_k$ in copy $\Lambda_i$;
\item
for $i>j$, $c_k^{ij}$ is the crossing between components $\Lambda_i$ and $\Lambda_j$ by the right cusp $c_k$;
\item
for $i<j$, $c_k^{ij}$ is the crossing between components $\Lambda_i$ and $\Lambda_j$ by the \textit{left} cusp paired with the right cusp $c_k$.
\end{itemize}
See Figure~\ref{fig:trefoil-2-front}.

\begin{figure}
\labellist
\small\hair 2pt
\pinlabel $\color{red}{\Lambda_1}$ at 145 29
\pinlabel $\color{blue}{\Lambda_2}$ at 145 10
\pinlabel $c_{1}^{12}$ at 25 160
\pinlabel $c_1^{21}$ at 262 160
\pinlabel $c_2^{12}$ at 25 93
\pinlabel $c_2^{21}$ at 262 93
\pinlabel $c_1^{11}$ [l] at 288 163
\pinlabel $c_1^{22}$ [l] at 288 145
\pinlabel $c_2^{11}$ [l] at 288 91
\pinlabel $c_2^{22}$ [l] at 288 72
\pinlabel $a_1^{11}$ at 49 139
\pinlabel $a_1^{12}$ at 67 117
\pinlabel $a_1^{21}$ at 31 117
\pinlabel $a_1^{22}$ at 49 97
\pinlabel $a_2^{11}$ at 145 139
\pinlabel $a_2^{12}$ at 163 117
\pinlabel $a_2^{21}$ at 127 117
\pinlabel $a_2^{22}$ at 145 97
\pinlabel $a_3^{11}$ at 239 139
\pinlabel $a_3^{12}$ at 257 117
\pinlabel $a_3^{21}$ at 221 117
\pinlabel $a_3^{22}$ at 239 97
\endlabellist
\centering
\includegraphics[scale=1.25]{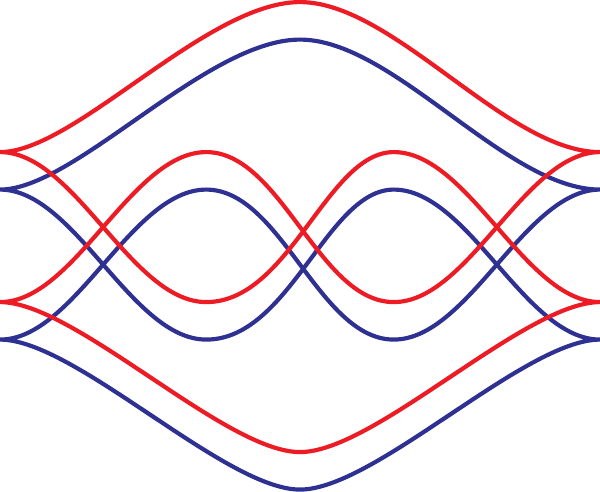}
\caption{
Reeb chords for the double of the Legendrian trefoil, in the front projection.
}
\label{fig:trefoil-2-front}
\end{figure}
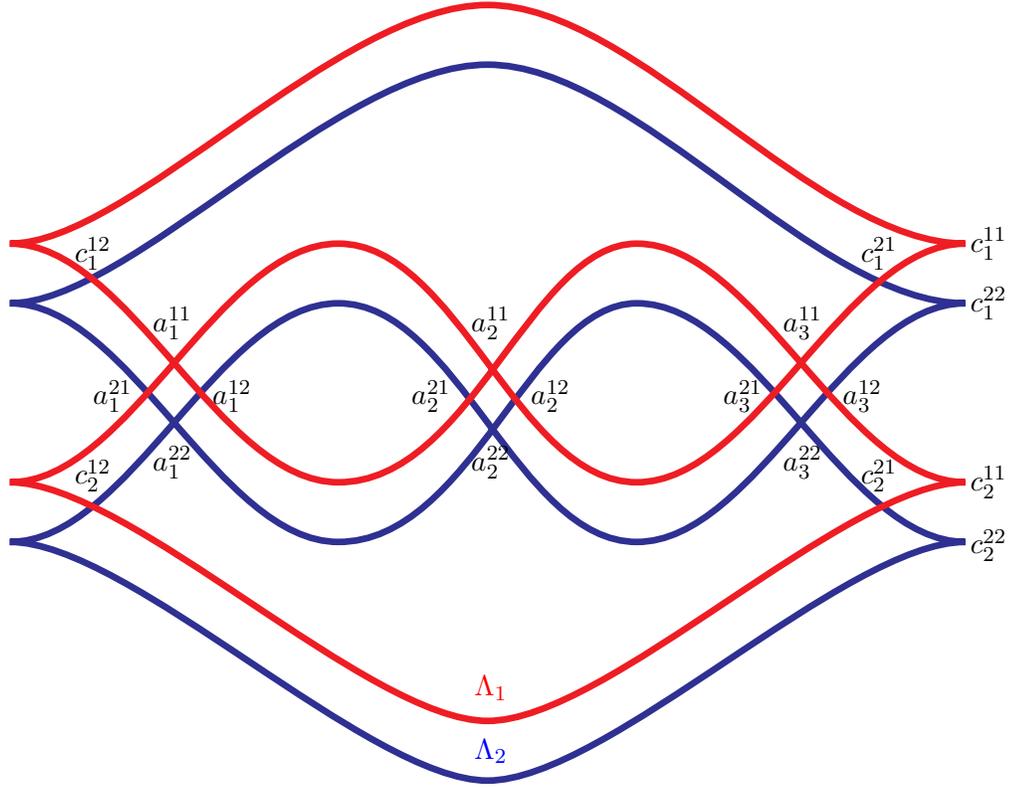

\subsubsection{Lagrangian projection $m$-copy}
\label{sssec:lagrcopy}

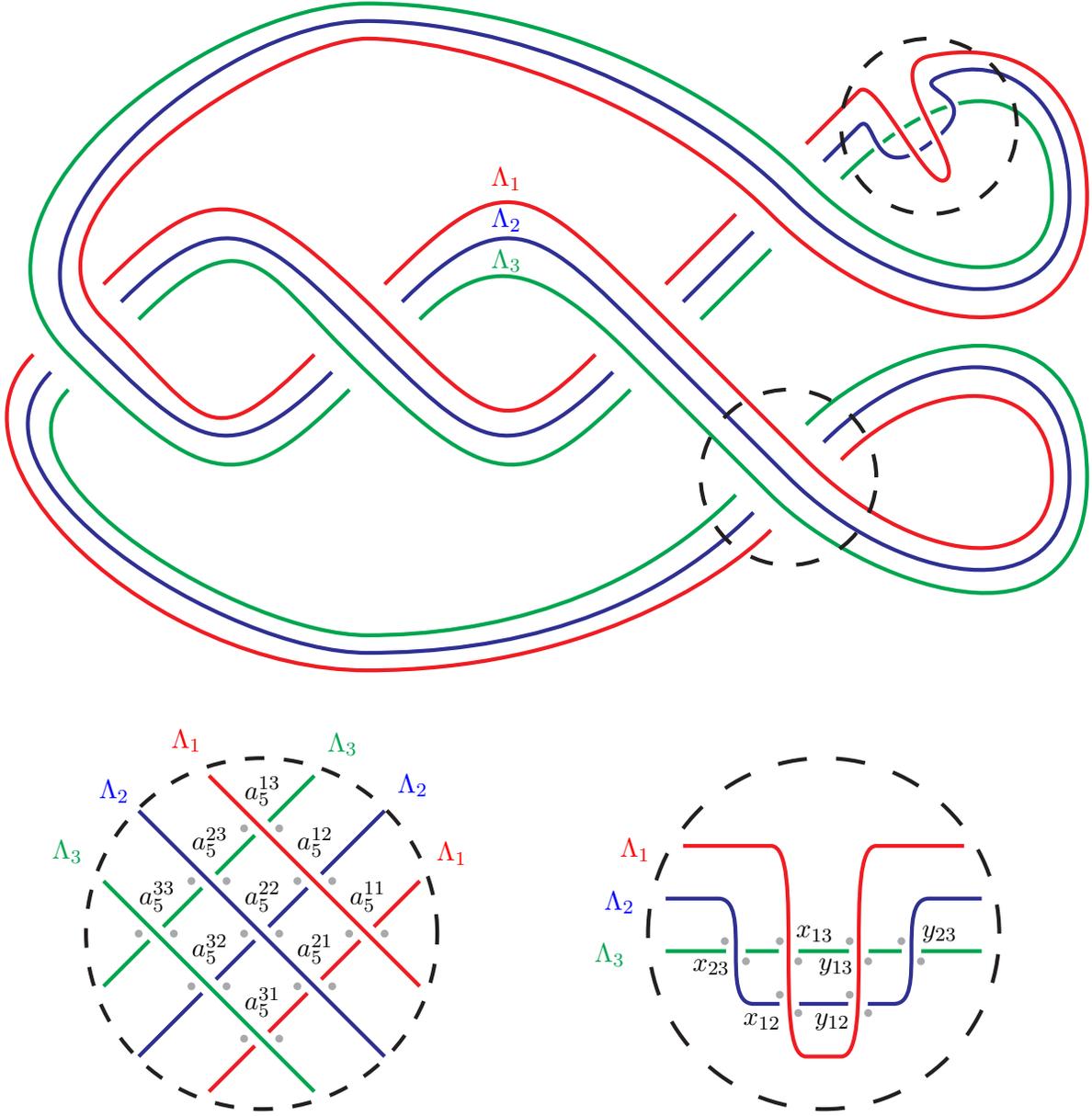
\begin{figure}
\labellist
\small\hair 2pt
\pinlabel $\color{red}{\Lambda_1}$ at 257 478
\pinlabel $\color{blue}{\Lambda_2}$ at 257 457
\pinlabel $\color{figuregreen}{\Lambda_3}$ at 257 436
\pinlabel $\color{red}{\Lambda_1}$ at 93 190
\pinlabel $\color{blue}{\Lambda_2}$ at 56 164
\pinlabel $\color{figuregreen}{\Lambda_3}$ at 32 133
\pinlabel $\color{red}{\Lambda_1}$ at 229 132
\pinlabel $\color{blue}{\Lambda_2}$ at 209 167
\pinlabel $\color{figuregreen}{\Lambda_3}$ at 173 188
\pinlabel $\color{red}{\Lambda_1}$ at 323 134
\pinlabel $\color{blue}{\Lambda_2}$ at 315 107
\pinlabel $\color{figuregreen}{\Lambda_3}$ at 310 80
\pinlabel $a_5^{33}$ at 78 110
\pinlabel $a_5^{23}$ at 105 138
\pinlabel $a_5^{32}$ at 105 83
\pinlabel $a_5^{13}$ at 132 164
\pinlabel $a_5^{22}$ at 132 110
\pinlabel $a_5^{31}$ at 132 56
\pinlabel $a_5^{12}$ at 159 138
\pinlabel $a_5^{21}$ at 159 83
\pinlabel $a_5^{11}$ at 186 110
\pinlabel $x_{23}$ at 362 74
\pinlabel $x_{13}$ at 415 91
\pinlabel $x_{12}$ at 388 46
\pinlabel $y_{13}$ at 426 74
\pinlabel $y_{12}$ at 424 46
\pinlabel $y_{23}$ at 479 91
\endlabellist
\centering
\includegraphics[scale=0.8]{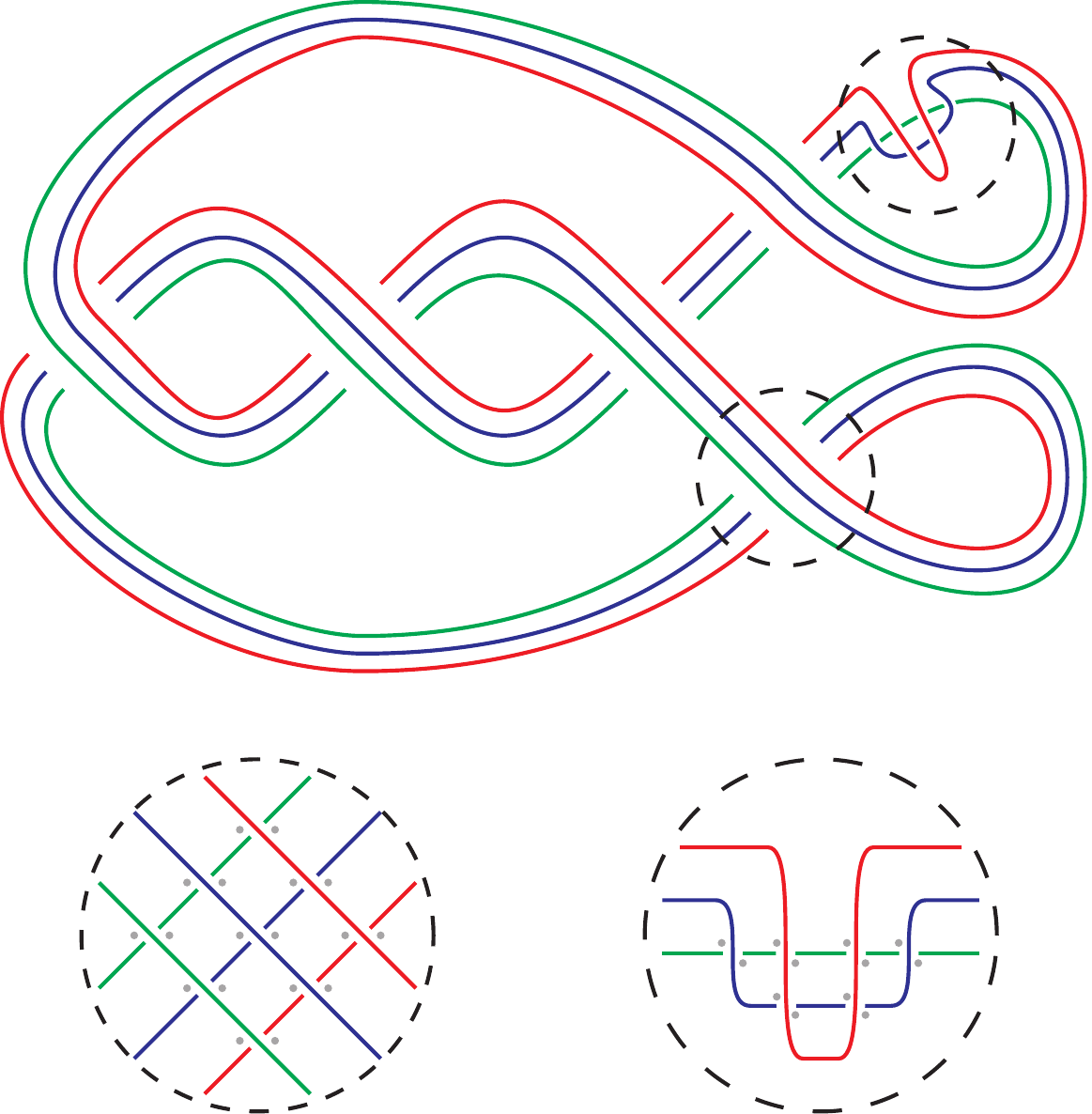}
\caption{
The $3$-copy of the Legendrian trefoil, in the $xy$
projection. Insets, with crossings labeled and positive quadrants
marked with dots: a neighborhood of the crossing labeled $a_5$ in
Figure~\ref{fig:trefoil-1}, and the dip. The $x$ crossings in the dip
correspond to the maximum of the Morse function on $S^1$, and the $y$
crossings to the minimum.
}
\label{fig:trefoil-3-xy}
\end{figure}

Label the crossings in the $xy$ projection of $\Lambda$ by
$a_1,\ldots,a_r$, and suppose that $f: \Lambda \rightarrow \R$ is a
Morse function with $M$ local maxima and $M$ local minima, enumerated so that the $k$-th local minimum follows the $k$-th maximum of $f$ with respect to the orientation of $\Lambda$.  Then the
$xy$ projection of $\Lambda^m_f$ has $m^2r+M m(m-1)$ crossings, which we can label as follows:
\begin{itemize}
\item
$a_k^{ij}$, $1\leq i,j\leq m$, between components $\Lambda_i$ and $\Lambda_j$ by crossing $a_k$;
\item
$x_k^{ij}$, $1\leq i<j\leq m$, between components $\Lambda_i$ and $\Lambda_j$ by the $k$th maximum of $f$;
\item
$y_k^{ij}$, $1\leq i<j\leq m$, between components $\Lambda_i$ and $\Lambda_j$ by the $k$th minimum of $f$.
\end{itemize}
Here the superscripts are chosen so that $a_k^{ij},x_k^{ij},y_k^{ij} \in \rR^{ij}$, i.e. upper strand belongs to $\Lambda_i$ and the lower strand belongs to $\Lambda_j$.
Since the $m$-copies are separated by a very small distance in the $z$ direction, the length of the Reeb chords $x_k^{ij},y_k^{ij}$ is much smaller than the length of the Reeb chords $a_k^{ij}$, and as a consequence we call the former chords ``short chords'' and the latter chords ``long chords.''

Both maxima and minima of $f$ give rise to Reeb chords of  $\Lambda_f^m$, but it will turn out that in fact
moving the local minima while leaving the locations of local maxima fixed does not change the differential of $\alg(\Lambda_f^m)$.
For this reason, we place base points $*_1, \ldots, *_M$ on $\Lambda$ at the local maxima of $f$, and denote the resulting base-pointed Legendrian as $\Lambda_f$ so that the \dga{} $\alg(\Lambda_f)$ has invertible generators $t_1^{\pm1}, \ldots, t^{\pm1}_{M}$.  For each one of these base points,
we place base points on all of the $m$-copies of $\Lambda^m_f$ preceding the corresponding half twist as pictured in Figure \ref{fig:HalfTwists}.
We label the corresponding invertible generators of $\alg(\Lambda^m_f)$ as $(t^i_k)^{\pm 1}$ for $1 \leq k \leq M$ and $1 \leq i \leq m$,
with $k$ specifying by the corresponding base point of $\Lambda_f$ and $i$ specifying the copy of $\Lambda$ where the base point appears.

Note that the generators of $\alg(\Lambda^m_f)$ are related to the generators of $\alg(\Lambda_f)$ as in the construction of Proposition \ref{prop:xysequence}.  In fact, with respect to a suitable weak link grading, the differentials will coincide as well.

\begin{definition}
Removing all base points of $\Lambda_f$ leaves a union of open intervals $\Lambda\setminus\{*_1, \ldots, *_M\} = \sqcup_{i = 1}^m U_i$ where we index the $U_i$ so that the initial endpoint of $U_i$ (with respect to the orientation of $\Lambda$) is at $*_i$.  Define $(r\times c):  \sS \rightarrow \{1, \ldots, m\}$ so that $(r\times c)(a_l) = (i,j)$  for a Reeb chord whose upper endpoint is on $U_i$ and whose lower endpoint is on $U_j$, and $(r\times c)(t_l) = (i,j)$ if the component of $\Lambda \setminus \{*_1, \ldots, *_M\}$ preceding (resp. following) $*_l$ is $U_i$ (resp. $U_j$).  We call $r \times c$ the {\it internal grading} of $\Lambda_f$.
\end{definition}

\begin{proposition}  \label{prop:mcopyWeakLinkGrading}
The internal grading is a weak link grading for $\alg(\Lambda_f)$.
\end{proposition}
\begin{proof}
We need to check that if $(r\times c)(a_l) = (i,j)$, then $\partial a_l$ is a $\bZ$-linear combination of composable words in $\alg(\Lambda_f)$ from $i$ to $j$.  This is verified by following along the boundaries of the disks used to define $\partial a_l$.
\end{proof}

We now give a purely algebraic description of the \dga{} of the
$m$-copy $\Lambda^m_f$ in terms of the \dga{} $\alg(\Lambda_f)$ of a
single copy of $f$, as presaged by Proposition~\ref{prop:xysequence}.
We note that the algebraic content given here is probably well-known
to experts, and is in particular strongly reminiscent of constructions
in \cite{BEE} (see e.g. \cite[Section~7.2]{BEE}).

\begin{proposition}  \label{prop:mcopyMultiple}
The \dga{} $\alg(\Lambda^m_f)$ arises by applying the construction of Proposition \ref{prop:xysequence} to $\alg(\Lambda_f)$ equipped with its initial grading.
More explicitly, $\alg(\Lambda^m_f)$
is generated by:
\begin{itemize}
\item invertible generators $(t^i_k)^{\pm 1}$,  $1 \leq i \leq m$, $1 \leq k \leq M$;
\item
$a_k^{ij}$, $1\leq i,j\leq m$, $1\leq k\leq r$, with $|a_k^{ij}| = |a_k|$;
\item
$x_k^{ij}$, $1\leq i<j\leq m$, $1 \leq k \leq M$, with $|x_k^{ij}| = 0$;
\item
$y_k^{ij}$, $1\leq i<j\leq m$, $1 \leq k \leq M$, with $|y_k^{ij}| = -1$.
\end{itemize}
The differential $\dd^m$ of $\alg(\Lambda^m_f)$ can be described as follows.
Assemble the generators of $\alg(\Lambda^m_f)$
into $m\times m$ matrices $A_1,\ldots,A_r,X_1, \ldots, X_M,Y_1, \ldots Y_M, \Delta_1, \ldots, \Delta_M$, with $A_k = (a^{ij}_k)$,
\[
X_k = \left[
\begin{matrix}
1 & x_k^{12} & \cdots & x_k^{1m} \\
0 & 1 & \cdots & x_k^{2m} \\
\vdots & \vdots & \ddots & \vdots \\
0 & 0 & \cdots & 1\\
\end{matrix}
\right],
\hspace{6ex}
Y_k = \left[
\begin{matrix}
0 & y_k^{12} & \cdots & y_k^{1m} \\
0 & 0 & \cdots & y_k^{2m} \\
\vdots & \vdots & \ddots & \vdots \\
0 & 0 & \cdots & 0\\
\end{matrix}
\right],
\]
and $\Delta_k = \operatorname{Diag}(t_k^{1}, \ldots, t_k^{m})$.

Then, applying $\dd^m$ to matrices entry-by-entry, we have
\begin{align*}
\dd^m(A_k) &= \Phi(\dd(a_k))+Y_{r(a_k)} A_k-(-1)^{|a_k|} A_k Y_{c(a_k)} \\
\dd^m(X_k) &= \Delta_k^{-1} Y_{r(t_k)} \Delta_k X_k -X_k Y_{c(t_k)}\\
\dd^m(Y_k) &= Y_k^2,
\end{align*}
where $\Phi: \alg(\Lambda) \rightarrow \operatorname{Mat}(m,\alg(\Lambda^m_f))$ is the ring homomorphism determined by $\Phi(a_k) = A_k$,
$\Phi(t_k) = \Delta_k X_k$, and $\Phi(t_k^{-1}) = X_k^{-1}\Delta_k^{-1}$.

\end{proposition}

\begin{remark}
Note that by this result (or by geometric considerations), short Reeb
chords form a sub-\dga{} of $\alg(\Lambda^m)$.
\end{remark}

\begin{proof}

The Reeb chords of $\Lambda^m_f$ are $a_k^{ij},x_k^{ij},y_k^{ij}$ as
described previously.  (See
Figure~\ref{fig:trefoil-3-xy} for an illustration for the trefoil from
Figure~\ref{fig:trefoil-1}, where there is a single base point in the loop to the
right of $a_4$ and the knot is oriented clockwise around this loop.) It is straightforward to calculate their
gradings as explained in Section \ref{ssec:dga-background}.

To associate signs to disks that determine the differential of $\alg(\Lambda^m_f)$, we use the choice of orientation signs given above Proposition \ref{prop:mcopyseq}.
The sign of a disk is then determined by the number of its
corners that occupy quadrants with negative orientation signs.  At each even-degree generator, two  quadrants, as in Figure \ref{fig:ReebSigns}, are assumed to
have been chosen for $\Lambda$ to calculate the differential on
$\alg(\Lambda)$. For even-degree generators of $\Lambda^m_f$, we have assigned the location of quadrants with negative orientation signs  as follows:
for $a_k^{ij}$, we take the
quadrants that correspond to the quadrants chosen for $a_k$; for
$x^{ij}$, we take the quadrants to the right of $\Lambda^j$ as we follow
the orientation of $\Lambda^j$ (in Figure~\ref{fig:HalfTwists},
these are the bottom two quadrants at each $x^{ij}$).

We next identify disks that contribute to the differential on $\alg(\Lambda^m_f)$. These disks
consist of two types, ``thick'' and ``thin'':
 viewed in the Lagrangian projection of $\Lambda^m_f$, thin disks are those disks whose images are entirely in the neighborhood of $\Lambda$ that contains the $m$-copies of $\Lambda^m_f$, and all other disks are thick, cf. \cite{Mishachev, NgR}.  It is not hard to see from the combinatorics of $\Lambda^m_f$ that  thick disks limit to disks for $\alg(\Lambda)$ in the limit
that the $m$ copies of $\Lambda^m$ approach each other, while thin
disks limit to curves along $\Lambda$ following the negative gradient flow
for the Morse function $f$.  

Since
the height of Reeb chords induces a
filtration on $\alg(\Lambda^m)$, the $x_k^{ij},y_k^{ij}$ form a differential
subalgebra and the differentials of these generators only involve thin
disks. An inspection of
Figure~\ref{fig:HalfTwists} shows that the only disks contributing
to $\partial(y_k^{ij})$ (i.e., with positive quadrant at $y_k^{ij}$ and
negative quadrants at all other corners) are triangles
that remain within the half twist, with two
negative corners at $y_k^{i\ell},y_k^{\ell j}$ for some $i<\ell<j$.  (See also the right-hand
inset in Figure~\ref{fig:trefoil-3-xy} where positive quadrants at crossings are decorated with dots.)
The disks contributing to
$\partial(x_k^{ij})$, which have a positive corner at $x_k^{ij}$, are of
four types, as follows.  There are bigons with
negative corner at $y_k^{ij}$, and triangles
with negative corners
at $x_k^{i\ell},y_k^{\ell j}$; both of these types of disks follow
$\Lambda$ from $*_k$ to the local minimum of $f$ that follows $*_k$.
In addition, there are bigons and triangles that follow the $\Lambda$
from $*_k$ to the preceding local minimum (which has the same
enumeration as $*_{c(t_k)}$); the bigons have
negative corner at $y_{c(t_k)}^{ij}$, and the
triangles have negative corners at
$y_{c(t_k)}^{i\ell},x_k^{\ell j}$. It follows that the differentials for $X_k,Y_k$ are as in
the statement of the proposition.

The disks for $\partial^m(a^{ij}_k)$ can be either thick or
thin. The thick disks are in many-to-one correspondence to the disks
for $\partial(a_k)$. The negative corners of a disk for
$\partial^m(a^{ij}_k)$ correspond to the negative corners of a disk
for $\partial(a_k)$, with one exception: where the boundary of the
disk passes through a maximum of the Morse function, there can be
one negative corner at an $x$ (if the boundary of the disk agrees with the
orientation of $\Lambda$ there) or some number of negative corners at
$x$'s (if it disagrees).  More precisely, if the boundary of a disk for $\Lambda^m_{f}$ lies on $\Lambda_i$ before passing the location of $*_k$ and lies on $\Lambda_j$ afterwards, then the possible products arising from negative corners and base points encountered when passing through the half twist are $t_{k}^i x^{i,j}_k$ if the orientations agree and $(-x_k^{i,i_1})(-x_k^{i_1,i_2}) \cdots (-x_k^{i_l, j}) (t_k^j)^{-1}$ for $i<i_1< \cdots < i_l<j$ when the orientations disagree.   The $(i,j)$ entries of $\Phi(t_k) = \Delta_k X_k$ and $\Phi(t_k^{-1}) = X_k^{-1} \Delta_k^{-1}$ are respectively
\[
t_{k}^i x^{i,j}_k \quad \mbox{and} \quad \sum_{i< i_1< \cdots < i_l <j} (-x^{i,i_1})(-x^{i_1,i_2}) \cdots (-x^{i_l, j}) (t_{k}^j)^{-1},
\]
so we see that the contribution of thick disks to $\partial^m(A_k)$ is precisely the term $\Phi(\partial(a_k))$.  (An alternate discussion of thick disks in a related setting may be found in Theorem 4.16 of \cite{NgR}, where the presence of the matrices $\Phi(t_k)$ and $\Phi(t_k^{-1})$ is established in a slightly more systematic manner using properties of the ``path matrix'' proved in \cite{Kal}.)

The thin disks contributing to $\partial a^{ij}_k$ have a
positive corner at $a^{ij}_k$ and two negative corners, one at a $y$
and the other in the
same $a_k$ region; in the limit as the copies approach each other,
these disks limit to paths
from the $a_k$ to a local minimum of $f$ that avoid local maxima.  When following $\Lambda$ along the upper strand (resp. lower strand) of $a_k$ in this manner, we reach the local minimum that follows  $*_{r(a_k)}$ (resp. $*_{c(a_k)}$).
The two corresponding disks have their negative
corners at  $y_{r(a_k)}^{i\ell},a^{\ell j}_k$ and $a^{i\ell}_k,y_{c(a_k)}^{\ell j}$. This
leads to the remaining $Y_{r(a_k)}A_k$ and $A_kY_{c(a_k)}$ terms in $\partial A_k$.  It is straightforward to verify that  the signs are as given in the
statement of the proposition.
\end{proof}

\begin{corollary} \label{cor:unital}
The augmentation category $\Aug(\Lambda_f,\coeffs)$ is strictly unital.
\end{corollary}
\begin{proof}
Follows from Propositions \ref{prop:mcopyMultiple} and \ref{thm:unital}.
\end{proof}

\begin{corollary}  \label{cor:unitalcohomology}
The (usual) category $H^*\Aug(\Lambda,\coeffs)$ is unital.
In particular, $\LCH{*}(\epsilon,\e)$ is a unital ring for any
augmentation $\epsilon$.
\end{corollary}

\begin{remark} \label{rem:front-unit}
We will show in Proposition~\ref{prop:xyandxz} that
$\Aug(\Lambda_{xz},\coeffs) \simeq \Aug(\Lambda_f,\coeffs)$, whence
$\Aug(\Lambda_{xz},\coeffs)$ has unital cohomology category.
In fact, $\Aug(\Lambda_{xz},\coeffs)$ is also strictly unital: it is straightforward to calculate directly that 
there is a unit in the category $\Aug(\Lambda_{xz},\coeffs)$, given by $-\sum_{k} (c_k^{12})^{\vee}$ where the sum is over all Reeb chords in $\rR^{12}(\Lambda^2_{xz})$ located near left cusps of $\Lambda$.  See also the proof of Proposition \ref{prop:xyandxz} and the example in \ref{sssec:TrefoilEx}.
\end{remark}

\begin{remark}
We expect that Corollary \ref{cor:unital} holds in arbitrary $1$-jet spaces
$J^1(M)$ as well, provided one has a suitable construction of $\Aug$, with the unit given by $\pm y^\vee$ where $y$ is the local minimum of a Morse function used to perturb the $2$-copy of $\Lambda$.  
We note by contrast that Proposition~\ref{prop:mcopyMultiple} does not hold in higher dimensions. In general, holomorphic disks $\overline{\Delta}$ for $\Lambda^m$ are in correspondence with holomorphic disks $\Delta$ for $\Lambda$ together with gradient flow trees attached along the boundary of $\Delta$; see \cite[Theorem~3.8]{ekholm-lekili} for a general statement, and \cite{EESab,EENS} for special cases worked out in more detail in the settings of Sabloff duality and knot contact homology, respectively. Some of the rigid holomorphic disks $\overline{\Delta}$ contributing to the differential in $\alg(\Lambda^m)$ come from rigid disks $\Delta$, but others come from disks $\Delta$ in some positive-dimensional moduli space and are rigidified by the flow trees. The disks in this latter case (which do not appear nontrivially when $\dim \Lambda=1$) are not counted by the differential in $\alg(\Lambda)$.
\end{remark}

\subsection{Invariance}  \label{ssec:invariance}

We now show that up to $A_\infty$ equivalence, our various
constructions of $\Aug$, $\Aug(\Lambda_f,\coeffs)$ and
$\Aug(\Lambda_{xz},\coeffs)$, are independent of choices and
Legendrian isotopy. We will
suppress the coefficient ring $\coeffs$ from the notation.

\begin{theorem} \label{prop:Invariance}
Up to $A_\infty$ equivalence, $\Aug(\Lambda_f)$ does not depend on the choice of $f$.
Moreover, if $\Lambda$ and $\Lambda'$ are Legendrian isotopic, then $\Aug(\Lambda_f)$ and $\Aug(\Lambda'_{f'})$ are $A_\infty$ equivalent.  In addition, if $\Lambda$ is in plat position, then $\Aug(\Lambda_f)$ and $\Aug(\Lambda_{xz})$ are $A_\infty$ equivalent.
\end{theorem}

The proof of Theorem \ref{prop:Invariance} is carried out in the following steps.
First, we show in Proposition \ref{prop:xyinvariance} that the categories defined using $f: \Lambda \rightarrow \R$
with a single local maximum on each component are invariant (up to $A_\infty$ equivalence) under Legendrian isotopy.
In Propositions \ref{prop:independentoff} and \ref{prop:xyandxz} we show that for fixed $\Lambda$ the
categories $\Aug(\Lambda_f)$ are independent of $f$ and, assuming $\Lambda$ is in plat position, are $A_\infty$
equivalent to $\Aug(\Lambda_{xz})$.  In proving Propositions \ref{prop:xyinvariance}-\ref{prop:xyandxz}, we continue to assume that base points are placed on the $\Lambda^m_f$ near local maxima of $f$ as indicated in Section \ref{sssec:lagrcopy}.  This assumption is removed in Proposition \ref{prop:invarianceBasePoints} where we show that both of the categories $\Aug(\Lambda_f)$ and $\Aug(\Lambda_{xz})$ are independent of the choice of base points on $\Lambda$.

\begin{proposition}
Let $\Lambda_0,\Lambda_1 \subset J^1(\bR)$ be Legendrian isotopic, and for $i=1,2$ let $f_i: \Lambda_i \rightarrow \R$ be a Morse function with a single local maximum on each component.  Then the augmentation categories $\Aug((\Lambda_0)_{f_0})$ and $\Aug((\Lambda_1)_{f_1})$  are $A_\infty$ equivalent.
\label{prop:xyinvariance}
\end{proposition}
\begin{proof}
Suppose that the links $\Lambda_i$ have components $\Lambda_i = \sqcup_{j = 1}^{c} \Lambda_{i,j}$ and that there is a Legendrian isotopy from $\Lambda_0$ to $\Lambda_1$ that takes $\Lambda_{0,j}$ to $\Lambda_{1,j}$ for all $1 \leq j \leq c$.  Then each \dga{} $(\alg(\Lambda_i), \partial)$ fits into the setting of Proposition \ref{prop:xysequence} with weak link grading given by the internal grading on $\Lambda_i$.
(The generator  $t_j$ corresponds to the unique base point on the $j$th component.)  Moreover, by Proposition \ref{prop:mcopyMultiple} the augmentation category $\Aug((\Lambda_i)_{f_i})$ agrees with the category $\Aug( \alg(\Lambda_i))$ that is constructed as a consequence of Proposition \ref{prop:xysequence}.

According to Proposition \ref{prop:InvarianceOfLinkGrading}, after stabilizing both $\alg(\Lambda_0)$ and $\alg(\Lambda_1)$ some (possibly different) number of times, they become isomorphic by a
\dga{} map that takes $t_j$ to $t_j$ and generators to linear
combinations of composable words, i.e. it satisfies the hypothesis of
the map  $f$ from Proposition \ref{prop:functorialityprop}.  The
construction from Proposition \ref{prop:functorialityprop} then shows
that the $A_\infty$ categories associated to these stabilized \dgas{}
are isomorphic.  Thus it suffices to show that if
$(S(\alg), \partial')$ is an algebraic stabilization of $(\alg, \partial)$, then $\Aug( S(\alg))$ and $\Aug( \alg)$ are $A_\infty$ equivalent.

Recall that $S(\alg)$ has the same generators as
$\alg$ but with two additional generators $a_{r+1},a_{r+2}$, and
$\partial'$ is defined so that $(\alg,\partial)$ is a sub-\dga{} and $\partial'(a_{r+1}) =
a_{r+2}$, $\partial'(a_{r+2}) = 0$.
The $A_\infty$-functor $\Aug(S(\alg)) \rightarrow \Aug(\alg)$ induced by the inclusion $i:\alg \hookrightarrow S(\alg)$ is surjective on objects.  (Any augmentation of $\alg$ extends to an augmentation of $S(\alg)$ by sending the two new generators to $0$.)  Moreover, for any $\epsilon_1, \e_2 \in \Aug(S(\alg))$ the map
\[
\hom(i^*\epsilon_1, i^*\e_2) \rightarrow \hom(\e_1,\e_2)
\]
is simply the projection with kernel spanned by $\{(a^{12}_{r+1})^{\vee}, (a^{12}_{r+2})^{\vee}\}$.  This is a quasi-isomorphism since, independent of $\e_1$ and $\e_2$, $m_1(a^{12}_{r+2})^{\vee} =  (a^{12}_{r+1})^{\vee}$.  Thus the  corresponding cohomology functor is indeed an equivalence.
\end{proof}

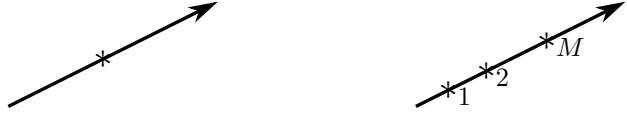
\begin{figure}
\labellist
\large\hair 2pt
\pinlabel $*$ at 45 23
\pinlabel $*_1$ at 211 7
\pinlabel $*_2$ at 229 16
\pinlabel $*_M$ at 261 30
\endlabellist
\centering
\includegraphics[scale=.8]{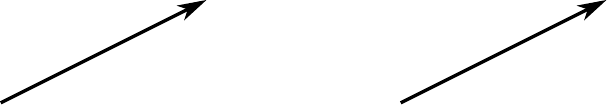}
\caption{
Locations of the local maxima of $f_0$ (left) and $f_1$ (right).
}
\label{fig:Basepoints}
\end{figure}

\begin{proposition}  \label{prop:independentoff}
For fixed $\Lambda \in J^1(\mathbb{R})$, the $A_\infty$ category $\Aug(\Lambda_f)$ is independent of the choice of $f$.
\end{proposition}

\begin{proof}
In Proposition \ref{prop:mcopyMultiple} the \dgas{} $\alg(\Lambda^m_f)$ are computed based on the location of base points placed at local maxima of $f$.  To simplify notation, we suppose that $\Lambda$ has a single component; a similar argument applies in the multi-component case.

Fix a Morse function $f_0$ with a single local maximum at $*$, and begin by considering the case of a second Morse function $f_1$ that has local maxima located at base points $*_1 \ldots, *_M$ that appear, in this order, on a small arc that contains $*$ and is disjoint from all crossings of $\pi_{xy}(\Lambda)$.  See Figure \ref{fig:Basepoints}.  Then there is a consistent sequence of \dga{} morphisms
\[
f^{(m)}: (\alg(\Lambda^m_{f_0}), \partial) \rightarrow (\alg(\Lambda^m_{f_0}), \partial)
\]
determined uniquely on generators by the matrix formulas
\begin{align*}
f^{(m)}(Y) &= Y_M \\
f^{(m)}(\Delta X) &= (\Delta_1 X_1)(\Delta_2 X_2) \cdots (\Delta_M
X_M) \\
f^{(m)}(A_k) &= A_k,  \quad 1 \leq k \leq r.
\end{align*}
Note that considering the diagonal entries of $f^{(m)}$ shows that $f^{(m)}(\Delta) = \Delta_1 \Delta_2 \cdots \Delta_M$.  In particular, $f^{(1)}(t) = t_1 \cdots t_M$.  The consistency of the sequence follows as usual from the uniformity of the matrix formulas.

We check that the extension of $f^{(m)}$ as an algebra homomorphism is a \dga{} map.  Note that for a Reeb chord $a_i$ of $\Lambda$, the only difference between the differential $\partial a_i$ in $\alg(\Lambda_{f_0})$ and $\alg(\Lambda_{f_1})$ is that words associated to holomorphic disks have all occurrences of $t$ replaced with $t_1 \cdots t_m$.  When comparing $\partial A_k$ in  $\operatorname{Mat}(m,\alg(\Lambda^m_{f_0}))$ and $\operatorname{Mat}(m,\alg(\Lambda^m_{f_1}))$, this results in all occurrences of $\Delta X$ being replaced with $(\Delta_1 X_1)(\Delta_2 X_2) \cdots (\Delta_M X_M)$.  Moreover, the $Y A_k - (-1)^{|a_k|} A_k Y$ term becomes $Y_M A_k - (-1)^{|a_k|} A_k Y_M$ since when approaching the arc containing the base points $*_1, \ldots, *_M$ in a manner that is opposite to the orientation of $\Lambda$, it is always $*_M$ that is reached first.  Together, these observations show that
\[
\partial f^{(m)} (A_k) = f^{(m)} \partial ( A_k).
\]
That $\partial f^{(m)} (Y) = f^{(m)} \partial (Y)$ is a immediate direct calculation.  Finally, note that using the Leibniz rule
\begin{align*}
\partial f^{(m)} (\Delta X) &= \partial [(\Delta_1 X_1)(\Delta_2 X_2) \cdots (\Delta_M X_M)] \\
&=  [\partial (\Delta_1 X_1)](\Delta_2 X_2) \cdots (\Delta_M X_M) + (\Delta_1 X_1)[\partial(\Delta_2 X_2)] \cdots (\Delta_M X_M)  \\
&\phantom{= }+ \ldots + (\Delta_1 X_1)(\Delta_2 X_2)\cdots [\partial(\Delta_M X_M)]
\end{align*}
and the sum telescopes to leave
\begin{align*}
Y_M(\Delta_1 X_1)(\Delta_2 X_2) \cdots (\Delta_M X_M) - (\Delta_1 X_1)(\Delta_2 X_2) \cdots (\Delta_M X_M) Y_M \\
=f^{(m)}(Y \Delta X - \Delta X Y) = f^{(m)}\partial (\Delta X).
\end{align*}

We check that the induced $A_\infty$ functor $F$, as in Proposition \ref{prop:consistentF}, is an $A_\infty$ equivalence.  The correspondence $\epsilon \rightarrow (f^{(1)})^* \epsilon$ is surjective on objects: Given $\epsilon': (\alg(\Lambda_{f_0}), \partial) \rightarrow (\coeffs, 0)$, we can define $\epsilon(t_1) = \epsilon'(t)$ and $\epsilon(t_k) = 1$ for $2 \leq k \leq M$ and $\epsilon(a_k) = \epsilon'(a_k)$ for $1\leq k \leq r$. The resulting augmentation of $\alg(\Lambda_{f_1})$ satisfies $f^*\epsilon = \epsilon'$.

Next, we verify that for $\e_1, \e_2 \in \Aug(\Lambda_{f_1})$, $F$
gives a quasi-isomorphism $F_1: \hom(\e_1, \e_2) \rightarrow \hom(f^*
\e_1, f^* \e_2)$.  We compute from the definitions:
\begin{align*}
F_1((x_1^{12})^{\vee}) & = \e_1(t_2 \cdots t_M)^{-1} \e_2(t_2 \ldots t_M) (x^{12})^{\vee} \\
F_1((x_2^{12})^{\vee}) &=  \e_1(t_3 \cdots t_M)^{-1} \e_2(t_3 \ldots t_M) (x^{12})^{\vee} \\
 &  \quad \vdots \\
F_1((x_M^{12})^{\vee}) &= (x^{12})^{\vee} \\
F_1((y_k^{12})^{\vee}) &= 0,  \quad 1 \leq k \leq M-1  \\
F_1((y_M^{12})^{\vee}) &= (y^{12})^{\vee} \\
F_1((a_k^{ij})^{\vee}) &= (a_k^{ij})^{\vee} \quad 1 \leq i,j \leq 2,~ 1\leq k \leq r,
\end{align*}
so $F_1$ is clearly surjective.
In addition, the differential $m_1 : \hom(\e_1, \e_2)\rightarrow \hom(\e_1, \e_2)$ satisfies
\[
m_1((y_k^{12})^{\vee}) = \epsilon_1(t_{k+1})^{-1} \epsilon_2(t_{k+1}) (x^{12}_{k+1})^\vee - (x^{12}_k)^\vee
\]
for $1 \leq k \leq M-1$, and it follows that $\ker(F_1)$ is free with basis
\[
\{ (y_1^{12})^{\vee}, \ldots, (y_{M-1}^{12})^{\vee}, m_1((y_1^{12})^{\vee}), \ldots, m_1((y_{M-1}^{12})^{\vee}) \}.
\]
Thus $\ker(F_1)$ is clearly acyclic, and the induced map on cohomology $F_1:H^*\hom(\e_1, \e_2) \rightarrow H^*\hom(f^* \e_1, f^* \e_2)$ is an isomorphism since it fits into an exact triangle with third term $H^*\ker(F_1) \cong 0$.

To complete the proof, we now show that the $A_\infty$-category is
unchanged up to isomorphism when the location of the base points is
changed.  Let $\Lambda_0$ and $\Lambda_1$ denote the same Legendrian
but with two different collections of base points $(*_1, \ldots, *_M)$
and $(*_1', \ldots, *_{M}')$ which appear cyclically ordered.  It
suffices to consider the case where the locations of the base points
agree except that $*_i'$ is obtained by pushing $*_i$ in the direction of the orientation of $\Lambda$ so that it passes through a
crossing $a_l$.

In the case that $*_i$ and $*_i'$ lie on the overstrand of $a_l$, we have a  \dga{} isomorphism
\[
f: (\alg(\Lambda_0), \partial) \rightarrow (\alg(\Lambda_1), \partial)
\]
given by
\[
f(a_l) = (t_i)^{-1} a_l
\]
and $f^{(m)}(x)  = x$ for any generator other than $a_l$, as in
\cite{NgR}.  To see that the $f^{(m)}$ are chain maps, note that the
holomorphic disks for $\Lambda_0$ and $\Lambda_1$ are identical, and the words associated to disks change only for disks with corners at $a_l$.
Note also that this isomorphism is compatible with the internal gradings on $\Lambda_0$ and $\Lambda_1$ which differ only on $a_l$.  Therefore,
 Proposition \ref{prop:functorialityprop} shows that $\Aug(\Lambda_0)$ and $\Aug(\Lambda_1)$ are isomorphic.

When $*_i$ and $*_i'$ sit on the understrand of $a_l$, similar considerations show that a DGA isomorphism with
\[
f(a_l) = a_l t_i
\]
 leads to an isomorphism of $A_\infty$-categories.
\end{proof}

\begin{proposition} \label{prop:xyandxz}
Suppose that  $\Lambda \subset J^1(\bR)$ has its front projection in preferred plat position.  Then the category $\Aug(\Lambda_{xz})$
is $A_\infty$ equivalent to $\Aug(\Lambda_f)$ for any Morse function $f$. 
\end{proposition}

\begin{proof}
Again, we suppose that $\Lambda$ has a single component, as a similar argument applies in the multi-component case.

We compare  $\Aug(\Lambda_{xz})$ with
$\Aug(\Lambda_f)$
for the function $f(x,y,z) = x$ whose local minima are at left cusps and local maxima are at right cusps.  Label crossings of $\pi_{xz}(\Lambda)$ as $a_1, \ldots, a_r$.  Label left and right cusps of $\Lambda$ as $b_1, \ldots, b_q$ and $c_1, \ldots, c_q$ so that, when the front projection is traced according to its orientation, the cusps appear in order, with $b_r$ immediately following $c_r$ for all $1 \leq r \leq q$.  Assuming the resolution procedure has been applied, we label the crossings of the $xy$-projection, $\pi_{xy}(\Lambda)$, as $a_1, \ldots, a_r, a_{r+1}, \ldots, a_{r+q}$ so that the crossings $a_{r+1}, \ldots, a_{r+q}$ correspond to the right cusps $c_1, \ldots, c_q$.  We assume that the base points $*_1, \ldots, *_q$, which are located at the far right of the loops that appear on $\pi_{xy}(L)$ in place of right cusps,  are labeled in the same manner as the $c_1, \ldots, c_q$.

Collect generators of $\alg(\Lambda_f^m)$ as usual into matrices $A_k, X_k, Y_k, \Delta_k$, and form matrices $A_k, B_k, C_k, \Delta_k$ out of the generators of $\alg(\Lambda_{xz}^m)$.  Note that $B_k$ is strictly upper triangular, while $C_k$ is lower triangular with diagonal entries given by the generators $c_k^{ii}$ that correspond to the right cusps of $\Lambda^m$.

There is a consistent sequence of \dga{} inclusions
\[
f^{(m)}: (\alg(\Lambda^m_{xz}), \partial ) \rightarrow (\alg(\Lambda^m_f), \partial)
\]
obtained by identifying generators so that we have
\begin{align*}
f^{(m)}(A_k) &= A_k &  f^{(m)}(B_k) &= Y_k \\
f^{(m)}(C_k) &= \pi_{\mathrm{low}}(A_{r+k}) &  f^{(m)}(\Delta_k) &= \Delta_k,
\end{align*}
where $\pi_{\mathrm{low}}(A_{r+k})$ is $A_{r+k}$ with all entries above the main diagonal replaced by $0$.
To verify that these identifications provide a chain map, note that the $xy$-projections of $\Lambda^m_{xz}$ and $\Lambda^m_f$ are identical to the left of the location of the crossings associated with right cusps.  Moreover, for crossings that appear in this portion of the diagram, all disks involved in the computation of differentials are entirely to the left of the crossings from right cusps as well.   Thus $\partial f^{(m)} = f^{(m)} \partial$ follows when applied to any of the matrices $A_k$ or $B_k$.  As in the proof of Proposition \ref{prop:mcopyMultiple}, examining thin and thick disks that begin at generators $c_k^{ij}$ leads to the matrix formula
\[
\partial C_k = \pi_{\mathrm{low}}\left(\tilde{\Phi}(\partial(c_k))  +  B_{k-1} C_k + C_k B_k\right)
\]
where $\tilde{\Phi}: \alg(\Lambda) \rightarrow \operatorname{Mat}(m, \alg(\Lambda^m_{xz}))$ denotes the ring homomorphism with $\tilde{\Phi}(a_k) = A_k$ for $1\leq k \leq r$ and $\tilde{\Phi}(t_k^{\pm1}) = \Delta_k^{\pm1}$.  (None of the $c_k$ appear in differentials of generators of $\alg(\Lambda)$ due to the plat position assumption.)  Notice that, for $1 \leq k \leq q$, $\pi_{\mathrm{low}}\left(\tilde{\Phi}(\partial(c_k))\right)$ agrees with $\pi_{\mathrm{low}} \left( \Phi (\partial(a_{r+k}))\right)$ (here $\Phi$ is from Proposition \ref{prop:mcopyMultiple})  because the only appearance of any of the $t_i$ in $\partial c_k = \partial a_{r+k}$ is as a single $t_k^{\pm1}$ term coming from the disk without negative punctures whose boundary maps to the loop to the right of $c_{r+k}$.  Moreover,
\[
\pi_{\mathrm{low}}( \tilde{\Phi}(t_k^{\pm1})) = \pi_{\mathrm{low}}(\Delta_k^{\pm1}) = \pi_{\mathrm{low}}((\Delta_k X_k)^{\pm1}) = \pi_{\mathrm{low}}( \Phi(t_k^{\pm1})),
\]
and $\tilde{\Phi}$ and $\Phi$ agree on all other generators that appear in $\partial c_k$.  Finally, we note that
\[ f^{(m)}\left(\pi_{\mathrm{low}}(B_{k-1} C_k + C_k B_k) \right) = \pi_{\mathrm{low}}(Y_{k-1} A_{r+k} + A_{r+k} Y_k)
\]
because none of the entries $a_{r+k}^{ij}$ with $i<j$ can appear below
the diagonal in $Y_{k-1} A_{r+k} + A_{r+k} Y_k$.  Combined with the
previous observation, this implies that $\partial f^{(m)}(C_k) = f^{(m)} \partial(C_k)$.

We claim that the $A_\infty$ functor $F:\Aug(\Lambda_f) \rightarrow \Aug(\Lambda_{xz})$ arising from Proposition \ref{prop:consistentF} is an $A_\infty$ equivalence.  Indeed, since $f^{(1)}$ is an isomorphism, $F$ is bijective on objects.  The maps $F_1: \hom(\epsilon_1, \epsilon_2) \rightarrow \hom(f^*\epsilon_1, f^*\epsilon)$ are surjections with $\ker(F_1) = \mbox{Span}_{\coeffs} \{(x_k^{12})^{\vee}, (a_{r+k}^{12})^{\vee} \,|\, 1\leq k \leq q  \}$.  Moreover, we have
$m_1 (x_k^{12})^{\vee} = \epsilon_1(t_k) (a_{r+k}^{12})^{\vee}$  (resp. $-\epsilon_2(t_k)^{-1} (a_{r+k}^{12})^{\vee}$) when the upper strand at $c_k$ points into (resp. away from) the cusp point.  It is therefore clear that $\ker(F_1)$ is acyclic, so that $F_1$ is a quasi-isomorphism.
\end{proof}

\begin{proposition}  \label{prop:invarianceBasePoints}
The categories $\Aug(\Lambda_f)$ and $\Aug(\Lambda_{xz})$ are independent of the number of base points chosen on $\Lambda$ as well as their location, provided each component of $\Lambda$ has at least one base point.
\end{proposition}

\begin{proof}
For simplicity, we assume $\Lambda$ is connected.  Let $\Lambda_0$ and $\Lambda_1$ denote $\Lambda$ equipped with two different collections $(*_1, \ldots, *_M)$ and $(*_1', \ldots, *_{M'}')$ of base points.  First, we suppose that these base points have the same number and appear in the same cyclic order along $\Lambda$.   We claim that the categories of $\Lambda_0$ and $\Lambda_1$ are isomorphic.  To show this, it suffices to consider the case where $*_k = *_k'$ for $k \geq 2$ and $*'_1$ is obtained by pushing $*_1$ in the direction of the orientation of $\Lambda$ either through a crossing, past a local maximum or local minimum of $f$ (in the case of the $xy$-perturbed category), or past a cusp of $\Lambda$ (in the case of the $xz$ category).  The proof is uniform for all of these cases.

For each $m \geq 1$, we always have some (possibly upper triangular) matrix $(w^{ij})$ of Reeb chords on $\Lambda^m$ from the $j$-th copy of $\Lambda$ to the $i$-th copy, and the movement of $*_1$ to $*_1'$ results in sliding $m$ base points $t^1_1, \ldots, t^m_1$ through this collection of Reeb chords.  As discussed in the proof of Proposition \ref{prop:independentoff}, we then have isomorphisms $f^{(m)}: \alg(\Lambda^m) \rightarrow \alg(\Lambda^m)$ satisfying
\begin{align*}
f^{(m)}(w^{ij}) &= (t^{i}_1)^{-1} w^{ij} && \mbox{for all $i, j$ and $m$, or}\\
f^{(m)}(w^{ij}) &= w^{ij} t^{j}_1 && \mbox{for all $i, j$ and $m$,}
\end{align*}
and fixing all other generators.  Clearly, the $f^{(m)}$ form a consistent sequence of \dga{} isomorphisms, and the isomorphism of the augmentation categories follows from Proposition \ref{prop:consistentF}.

Finally, to make the number of base points the same, it suffices to consider the case where $\Lambda_0$ has a single base point, $*_1$, and $\Lambda_1$ has base points $*_1, \ldots, *_M$ located in a small interval around $*_1$ as in Figure \ref{fig:Basepoints}.  Then, for $m \geq 1$, we have \dga{} morphisms $f^{(m)}: \alg(\Lambda^m) \rightarrow \alg(\Lambda^m)$ fixing all Reeb chords and with
\[
f^{(m)}(t^i_1) = t^i_1t^i_2 \ldots t^i_M,  \quad \mbox{for all $1 \leq i \leq m$}.
\]
The $f^{(m)}$ clearly form a consistent sequence, so there is an $A_\infty$ functor $F: \Aug(\Lambda_1) \rightarrow \Aug(\Lambda_0)$ induced by Proposition \ref{prop:consistentF}.  As in the proof of Proposition \ref{prop:independentoff}, $F$ is surjective on objects.  Moreover, $F$ induces an isomorphism on all hom spaces (before taking cohomology), and is thus an equivalence.
\end{proof}

\subsection{Examples}
\label{ssec:exs}

Here we present computations of the augmentation category\footnote{For some computations of the sheaf category of a similar spirit, see \cite[section 7.2]{STZ}.}
 for the Legendrian unknot and the Legendrian trefoil, as well as an application of the augmentation category to the Legendrian mirror problem.

These calculations require computing the \dga{} for the $m$-copy of the knot.
For this purpose, each of the $m$-copy perturbations described in Section~\ref{ssec:augcatdef}, front projection $m$-copy and Lagrangian projection $m$-copy, has its advantages and disadvantages. The advantage of the Lagrangian $m$-copy is that its \dga{} can be computed directly from the \dga{} of the original knot by Proposition~\ref{prop:mcopyMultiple}; for reference, we summarize this computation and the resulting definition of $\Aug(\Lambda)$ in Section \ref{sec:aug-knot}, assuming $\Lambda$ is a knot with a single base point. The advantage of the front $m$-copy is that it has fewer Reeb chords and thus simplifies computations somewhat: that is, if we begin with the front projection of the knot, resolving and then taking the Lagrangian $m$-copy results in more crossings (because of the $x,y$ crossings) than taking the front $m$-copy and the resolving. We compute for the unknot using the Lagrangian $m$-copy and for the trefoil using the front $m$-copy, to illustrate both.

\begin{convention}
We recall
$\hom(\epsilon_1, \epsilon_2) = C_{12}^\vee$ and $\homBC(\epsilon_2, \epsilon_1) = C_{21}^\vee$.
Often our notational convention would require elements of
$C_{12}^\vee$ to be written in the form $(a^{12})^\vee$,
but when viewing them as elements of $\hom(\epsilon_1, \epsilon_2)$, we denote
them simply as $a^+$.
\label{conv:pm}
Likewise, an element of $C_{21}^\vee$, which would otherwise be denoted as $(a^{21})^\vee$,
we will instead write as $a^- \in \homBC(\epsilon_2, \epsilon_1)$.
\end{convention}
This convention is made both to decrease indices, and to decrease cognitive dissonance associated
with the relabeling of strands required by the definition of composition, as in \eqref{eq:composition}.

\subsubsection{The augmentation category in terms of Lagrangian $m$-copies}
\label{sec:aug-knot}

Since the construction and proof of invariance of the augmentation category involved a large amount of technical details, we record here a complete description of it in the simplest case, namely a Legendrian knot with a single base point, in terms of the \dga{} associated to its Lagrangian projection.  This is an application of Definition \ref{def:aug-from-consistent-sequence} to the corresponding consistent sequence of \dgas{} from Proposition \ref{prop:mcopyMultiple}.

\begin{proposition} \label{prop:complete-aug-description}
Let $\Lambda$ be a Legendrian knot with a single base point, and let $(\alg(\Lambda),\dd)$ be its C--E \dga{}, constructed from a Lagrangian projection of $\Lambda$, which is generated by $\cS = \cR \sqcup \cT$ where $\cR = \{a_1,\dots,a_r\}$ and $\cT = \{t,t^{-1}\}$, with only the relation $t\cdot t^{-1} = t^{-1}\cdot t = 1$.  Then the objects of $\Aug(\Lambda,\coeffs)$ are exactly the augmentations of $\alg(\Lambda)$, i.e.\ the \dga{} morphisms $\epsilon: \alg(\Lambda) \to \coeffs$.  Each $\hom(\epsilon_1,\epsilon_2)$ is freely generated over $\coeffs$ by elements $a_k^+$ ($1 \leq k \leq r$), $x^+$, and $y^+$, with $|a_k^+| = |a_k|+1$, $|x^+| = 1$, and $|y^+| = 0$.

We describe the composition maps in terms of the corresponding \dgas{} $(\alg^{m},\dd^{m})$ of the $m$-copies of $\Lambda$, which are defined as follows.  The generators of $\alg^m$ are
\begin{enumerate}
\item $(t^i)^{\pm 1}$ for $1 \leq i \leq m$, with $|t^i| = 0$;
\item $a^{ij}_k$ for $1 \leq i,j \leq m$ and $1 \leq k \leq r$, with $|a^{ij}_k| = |a_k|$;
\item $x^{ij}$ for $1 \leq i < j \leq m$, with $|x^{ij}|=0$;
\item $y^{ij}$ for $1 \leq i < j \leq m$, with $|y^{ij}|=-1$,
\end{enumerate}
and the only relations among them are $t^i \cdot(t^i)^{-1} =
(t^i)^{-1}\cdot t^i = 1$ for each $i$.  If we assemble these into
$m\times m$ matrices $A_k$, $X$, $Y$, and
$\Delta=\operatorname{Diag}(t^1,\dots,t^m)$ as before, where $X$ is
upper triangular with all diagonal entries equal to 1 and $Y$ is
strictly upper triangular, then the differential $\dd^m$ satisfies
\begin{align*}
\dd^m(A_k) &= \Phi(\dd a_k) + YA_k - (-1)^{|a_k|}A_k Y \\
\dd^m(X) &= \Delta^{-1} Y \Delta X - XY \\
\dd^m(Y) &= Y^2
\end{align*}
where $\Phi$ is the graded algebra homomorphism defined by $\Phi(a_k)=A_k$ and $\Phi(t) = \Delta X$.

To determine the composition maps
\[ m_k: \hom(\epsilon_k,\epsilon_{k+1}) \otimes \dots \otimes \hom(\epsilon_2,\epsilon_3) \otimes \hom(\epsilon_1,\epsilon_2) \to \hom(\epsilon_1,\epsilon_{k+1}), \]
recall that a tuple of augmentations
$(\epsilon_1,\dots,\epsilon_{k+1})$ of $\alg(\Lambda)$ produces an
augmentation $\epsilon: \alg^{k+1} \to \coeffs$ by setting
$\epsilon(a_j^{ii}) = \epsilon_i(a_j)$, $\epsilon((t^i)^{\pm 1}) =
\epsilon_i(t^{\pm 1})$, and $\epsilon=0$ for all other generators.  We define a twisted \dga{} $((\alg^{k+1})^\epsilon, \dd^{k+1}_\epsilon)$ by noting that $\dd^{k+1}$ descends to $(\alg^{k+1})^\epsilon := (\alg^{k+1} \otimes \coeffs) / (t^i = \epsilon(t^i))$ and letting $\dd^{k+1}_\epsilon = \phi_\epsilon \circ \dd^{k+1} \circ \phi_\epsilon^{-1}$, where $\phi_\epsilon(a)=a+\epsilon(a)$.  Then
\[ m_k(\alpha_{k}^+, \dots, \alpha_{2}^+, \alpha_{1}^+) = (-1)^{\sigma} \sum_{a\in\cR\cup \{x,y\}} a^+ \cdot \operatorname{Coeff}_{\alpha_{1}^{12}\alpha_{2}^{23}\dots \alpha_{k}^{k,k+1}}(\dd^{k+1}_\epsilon a^{1,k+1}), \]
where $\alpha_i \in \{a_1,a_2,\dots,a_r,x,y\}$ for each $i$, and
$\sigma = k(k-1)/2 + \sum_{p<q}
|\alpha_p^+||\alpha_q^+|+|\alpha_{k-1}^+|+|\alpha_{k-3}^+|+ \cdots$.
\end{proposition}

\begin{remark}
The construction of each $(\alg^m,\dd^m)$ can be expressed more concisely as follows.  Having defined the graded algebra homomorphism $\Phi: \alg(\Lambda) \to \alg^m \otimes \operatorname{End}(\bZ^m)$ and the elements $A_k, X, Y, \Delta$, the differential $\dd^m$ is equivalent to a differential on $\alg^m \otimes \operatorname{End}(\bZ^m)$ once we know that $\operatorname{End}(\bZ^m)$ has the trivial differential.  It is characterized by the facts that $\dd^m \Delta = 0$; that $-Y$ is a Maurer-Cartan element, i.e.\ that
\[ \dd^m(-Y) + \frac{1}{2}[-Y,-Y] = 0; \]
and that if we define a map $D\Phi: \alg(\Lambda) \to \alg^m\otimes \operatorname{End}(\bZ^m)$ by $D\Phi = \dd^m \Phi - \Phi \dd$, then
\[ D\Phi + \operatorname{ad}(-Y) \circ \Phi = 0. \]
Here $D\Phi$ is a $(\Phi,\Phi)$-derivation, meaning that $D\Phi(ab) = D\Phi(a)\cdot \Phi(b) + (-1)^{|a|}\Phi(a)\cdot D\Phi(b)$, and $[\cdot,\cdot]$ denotes the graded commutator $[A,B] = AB - (-1)^{|A||B|}BA$.
\end{remark}

\subsubsection{Unknot}
\label{sec:unknot}
We first compute the augmentation categories $\Augpm(\Lambda,\coeffs)$
for the standard
Legendrian unknot $\Lambda$ shown in Figure~\ref{fig:unknot-xy}, and any coefficients $\coeffs$, using the
Lagrangian projection $m$-copy and via Proposition~\ref{prop:complete-aug-description}.
Then the
\dga{} for $\Lambda$ is generated by
$t^{\pm 1}$ and a single Reeb chord $a$, with $|t|=0$, $|a|=1$, and
\[
\partial(a) = 1+t^{-1}.
\]
This has a unique augmentation $\epsilon$ to $\coeffs$, with
$\epsilon(a)=0$ and $\epsilon(t) = -1$.

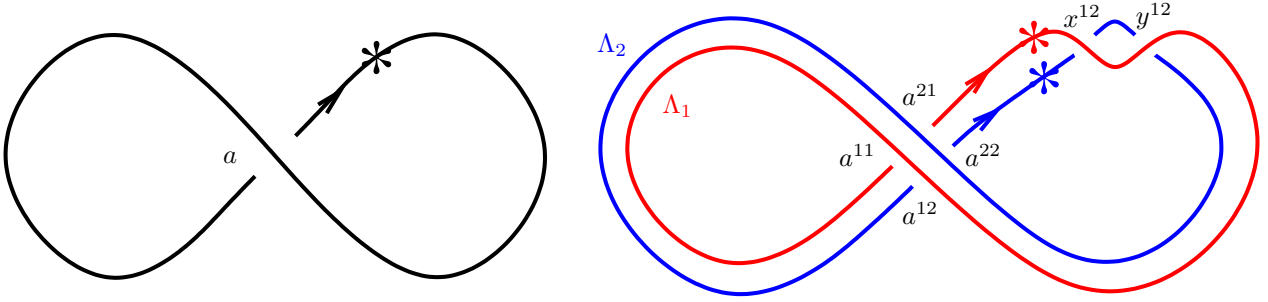
\begin{figure}
\labellist
\footnotesize\hair 2pt
\pinlabel $a$ at 102 63
\pinlabel $a^{21}$ at 409 91
\pinlabel $a^{12}$ at 409 38
\pinlabel $a^{11}$ at 381 64
\pinlabel $a^{22}$ at 437 64
\pinlabel $x^{12}$ at 481 126
\pinlabel $y^{12}$ at 513 126
\pinlabel $\color{red}{\Lambda_1}$ at 301 86
\pinlabel $\color{blue}{\Lambda_2}$ at 272 113
\endlabellist
\centering
\includegraphics[width=\textwidth]{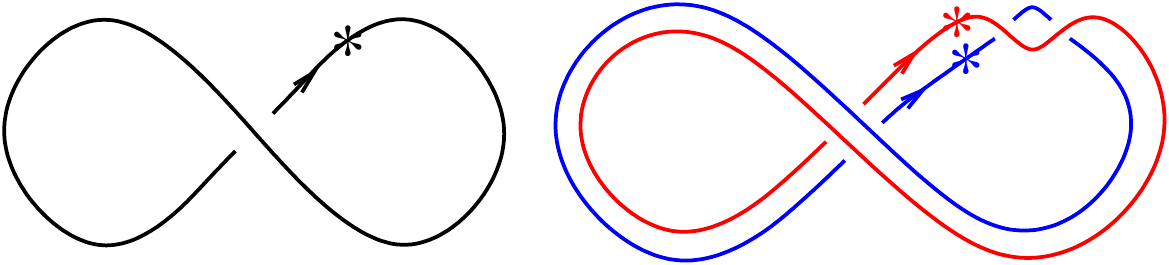}
\caption{
The Legendrian unknot $\Lambda$ (left) and its $2$-copy $\Lambda^2$
(right), in the $xy$ projection, with base points and Reeb chords labeled.
}
\label{fig:unknot-xy}
\end{figure}

We can read the \dga{} for the $m$-copy of $\Lambda$ from
Proposition~\ref{prop:complete-aug-description}. For $m=2$, there are $6$ Reeb chords
$a^{11},a^{12},a^{21},a^{22},x^{12},y^{12}$ with $|a^{ij}|=1$,
$|x^{12}|=0$, $|y^{12}|=-1$, and the differential is
\begin{alignat*}{2}
\partial(a^{11}) &= 1+(t^1)^{-1}+y^{12}a^{21} & \qquad \partial(x^{12}) &=
(t^1)^{-1}y^{12}t^2-y^{12} \\
\partial(a^{12}) &= -x^{12}(t^2)^{-1}+y^{12}a^{22}+a^{11}y^{12} &
\partial(y^{12}) &= 0 \\
\partial(a^{21}) &= 0 &&\\
\partial(a^{22}) &= 1+(t^2)^{-1}+a^{21}y^{12}. &&
\end{alignat*}
Note that the differential on $\Lambda^2$ can also be read by
inspection from Figure~\ref{fig:unknot-xy}.

We have
\begin{alignat*}{3}
\LCC{0}(\epsilon,\epsilon) &= \langle y^+ \rangle & \qquad
\LCC{1}(\epsilon,\epsilon) &= \langle x^+ \rangle & \qquad
\LCC{2}(\epsilon,\epsilon) &= \langle a^+ \rangle \\
&&
\LCCBC{2}(\epsilon,\epsilon) &= \langle a^- \rangle &
\end{alignat*}
and all other $\LCCpm{*}(\epsilon,\epsilon)$ are $0$.
The linear part $\partial_{(\epsilon,\epsilon)}^{\lin}$ of the differential $\partial_{(\epsilon,\epsilon)}$ on
$C_{12} = \langle a^{12},x^{12},y^{12}\rangle$ is given by
$\partial_{(\epsilon,\epsilon)}^{\lin}(a^{12}) =
x^{12}$,
$\partial_{(\epsilon,\epsilon)}^{\lin}(x^{12})
= \partial_{(\epsilon,\epsilon)}^{\lin}(y^{12}) = 0$, while on
$C_{21} = \langle a^{21} \rangle$ it is identically
zero. Dualizing gives differentials $m_1$ on $\hompm^*$ with $m_1(x^+)=a^+$ in $\LCC{}$ and $m_1=0$ otherwise, and $m_1=0$ on $\LCCBC{}$; thus
\[
\LCH{0}(\epsilon,\epsilon) \cong \langle y^+ \rangle
\qquad
\LCHBC{2}(\epsilon,\epsilon) \cong \langle a^- \rangle
\]
and $\LCHpm{*}(\epsilon,\epsilon)=0$ otherwise. (Recall from
Convention~\ref{conv:pm} that $a^+,a^-$ represent
$(a^{12})^\vee,(a^{21})^\vee$ in $\hom,\homBC$, respectively.)

It is evident that the augmentation category $\AugBC(\Lambda,\coeffs)$
is non-unital -- there are no degree zero morphisms at all.  Indeed, all
higher compositions $m_k$, $k \geq 2$, on $\homBC(\epsilon,\epsilon)$
must vanish
for degree reasons. To calculate the composition maps on
$\Aug(\Lambda,\coeffs)$, we need the differential for the $3$-copy
$\Lambda^3$. Again from Proposition~\ref{prop:complete-aug-description}, the relevant
part of the differential for $\Lambda^3$ is
\begin{align*}
\partial(a^{13}) &=
-x^{13}(t^3)^{-1}+x^{12}x^{23}(t^3)^{-1}+y^{12}a^{23}+y^{13}a^{33}+a^{11}y^{13}+a^{12}y^{23} \\
\partial(x^{13}) &= (t^1)^{-1}y^{13}t^3+(t^1)^{-1}y^{12}t^2 x^{23}
-y^{13}-x^{12}y^{23} \\
\partial(y^{13}) &= y^{12}y^{23}.
\end{align*}
Augmenting each copy by $\epsilon$ sends each $t^i$ to $-1$, which
by \eqref{eq:ms} leads to
\begin{align*}
m_2(x^+,x^+) &= a^+ \\
m_2(y^+,a^+) = m_2(a^+,y^+) &= -a^+ \\
m_2(y^+,x^+) = m_2(x^+,y^+) &= -x^+ \\
m_2(y^+,y^+) &= -y^+.
\end{align*}
Note in particular that in $\Aug(\Lambda,\coeffs)$, $-y^+$ is the unit, in agreement with
Theorem~\ref{thm:unital}.

One can check from Proposition~\ref{prop:complete-aug-description} that
\[
m_k(x^+,x^+,\ldots,x^+) = (-1)^{\lfloor (k-1)/2\rfloor} a^+
\]
and all other higher products $m_k$ vanish for $k \geq 3$: the only
contributions to $m_k$ come from entries of $(\Delta X)^{-1}$.

\begin{remark}
If we instead choose the opposite orientation for $\Lambda$ (which does not change $\Lambda$ up to Legendrian isotopy), then the differential for $\Lambda$ contains no negative powers of $t$, and no words of length $\geq 2$; it follows that in the resulting $A_\infty$ category $\Aug$, $m_k$ vanishes identically for $k \geq 3$.
\end{remark}

\subsubsection{Trefoil} \label{sssec:TrefoilEx}

Here we compute the augmentation categories to $\bZ/2$ for the
right-handed trefoil $\Lambda$ shown in Figure~\ref{fig:trefoil-1},
using the front projection $m$-copy, cf.\ Section~\ref{sssec:frontcopy}.
Place a single base point at the right cusp $c_1$
(i.e., along the loop at $c_1$ in the $xy$ resolution of the front),
and set $t=-1$ to reduce to coefficient ring $\bZ$ (we will keep the
signs for reference, although for our calculation it suffices to
reduce mod $2$ everywhere). Then the \dga{} for $\Lambda$ is generated by $c_1,c_2,a_1,a_2,a_3$, with $|c_1|=|c_2|=1$ and $|a_1|=|a_2|=|a_3|=0$, with differential
\begin{align*}
\partial (c_1) &= -1+a_1+a_3+a_1a_2a_3 \\
\partial(c_2) &= 1-a_1-a_3-a_3a_2a_1 \\
\partial(a_1) &= \partial(a_2) = \partial(a_3) = 0.
\end{align*}
There are five augmentations $\epsilon_1,\epsilon_2,\epsilon_3,\epsilon_4,\epsilon_5$ from this \dga{} to $\bZ/2$: $\epsilon_i(c_j) = 0$ for all $i,j$, and the augmentations are determined by where they send $(a_1,a_2,a_3)$: $\epsilon_1 = (1,0,0)$,
$\epsilon_2 = (1,1,0)$, $\epsilon_3 = (0,0,1)$, $\epsilon_4 = (0,1,1)$, $\epsilon_5 = (1,1,1)$.

Next consider the double $\Lambda^2$ of the trefoil as shown in Figure~\ref{fig:trefoil-2-front}. For completeness, we give here the full differential on mixed Reeb chords of $\Lambda^2$ (over $\bZ$, with base points at $c_1^{11}$ and $c_1^{22}$):
\begin{alignat*}{2}
\partial(c_1^{12}) &= 0 &
\partial(c_1^{21}) &= a_1^{21}(1+a_2^{11}a_3^{11}+a_2^{12}a_3^{21})
+ a_1^{22}(a_2^{21}a_3^{11}+a_2^{22}a_3^{21})+a_3^{21} \\
\partial(c_2^{12}) &= 0 &
\partial(c_2^{21}) &= -a_3^{21}(1+a_2^{12}a_1^{21}+a_2^{11}a_1^{11})
-a_3^{22}(a_2^{21}a_1^{11}+a_2^{22}a_1^{21})-a_1^{21} \\
\partial(a_1^{12}) &= c_1^{12}a_1^{22}-a_1^{11}c_2^{12} &
\qquad \partial(a_1^{21}) &= 0\\
\partial(a_2^{12}) &= c_2^{12}a_2^{22}-a_2^{11}c_1^{12} &
\partial(a_2^{21}) &= 0 \\
\partial(a_3^{12}) &= c_1^{12}a_3^{22}-a_3^{11}c_2^{12} &
\partial(a_3^{21}) &= 0.
\end{alignat*}

For any augmentations $\epsilon_i,\epsilon_j$, we have
\begin{alignat*}{2}
\LCC{0}(\epsilon_i,\epsilon_j) &\cong (\bZ/2) \langle c_1^+,c_2^+ \rangle & \qquad \LCCBC{1}(\epsilon_i,\epsilon_j) &\cong (\bZ/2) \langle a_1^{-},a_2^{-},a_3^{-} \rangle \\
\LCC{1}(\epsilon_i,\epsilon_j) &\cong (\bZ/2) \langle a_1^+,a_2^+,a_3^+ \rangle & \qquad
\LCCBC{2}(\epsilon_i,\epsilon_j) &\cong (\bZ/2) \langle c_1^-,c_2^- \rangle
\end{alignat*}
and $\LCCpm{*}(\epsilon_i,\epsilon_j)=0$ otherwise.
The
linear part $\partial_{(\epsilon_1,\epsilon_1)}^{\lin}$ of the differential $\partial_{(\epsilon_1,\epsilon_1)}$ on $C_{12}$ sends $a_1^{12}$ to $c_1^{12}+c_2^{12}$ and the other four generators $c_1^{12},c_2^{12},a_2^{12},a_3^{12}$ to $0$, while $\partial_{(\epsilon_1,\epsilon_1)}$ on $C_{21}$ sends $c_1^{21}$ to $a_1^{21}+a_3^{21}$, $c_2^{21}$ to $a_1^{21}+a_3^{21}$, and $a_1^{21},a_2^{21},a_3^{21}$ to $0$.
Dualizing gives
\begin{alignat*}{2}
\LCH{0}(\epsilon_1,\epsilon_1) &\cong (\bZ/2)\langle [c_1^+ +c_2^+]\rangle & \qquad
\LCHBC{1}(\epsilon_1,\epsilon_1) &\cong (\bZ/2)\langle [a_1^- +a_3^-],[a_2^-]\rangle \\
\LCH{1}(\epsilon_1,\epsilon_1) &\cong (\bZ/2)\langle [a_2^+],[a_3^+]\rangle & \qquad
\LCHBC{2}(\epsilon_1,\epsilon_1) &\cong (\bZ/2)\langle [c_1^-]\rangle
\end{alignat*}
and $\LCHpm{*}(\epsilon_1,\epsilon_1)=0$ otherwise.
As in the previous example, note that $\LCH{*}(\epsilon_1,\epsilon_1)$
has support in degree $0$, while $\LCHBC{*}(\epsilon_1,\epsilon_1)$ does not.

A similar computation with the pair of augmentations $(\epsilon_1,\epsilon_2)$ gives, on $C_{12}$,
$\partial_{(\epsilon_1,\epsilon_2)}^{\lin}(a_1^{12}) =
c_1^{12}+c_2^{12}$,
$\partial_{(\epsilon_1,\epsilon_2)}^{\lin}(a_2^{12}) = c_2^{12}$, and
$\partial_{\epsilon_1,\epsilon_2}^{\lin} = 0$ on other generators.
On $C_{21}$, we have $\partial_{(\epsilon_2,\epsilon_1)}^{\lin}(c_1^{21}) = a_1^{21}+a_3^{21}$, $\partial_{(\epsilon_2,\epsilon_1)}^{\lin}(c_2^{21}) = a_1^{21}$, and $\partial_{(\epsilon_2,\epsilon_1)}^{\lin} = 0$ on other generators. Thus we have:
\begin{alignat*}{2}
\LCH{1}(\epsilon_1,\epsilon_2) &\cong (\bZ/2)\langle [a_3^+]\rangle & \qquad
\LCHBC{1}(\epsilon_1,\epsilon_2) &\cong (\bZ/2)\langle [a_2^-]\rangle
\end{alignat*}
and $\LCHpm{*}(\epsilon_1,\epsilon_2)=0$ otherwise.

\begin{remark}
Note that either of $\LCH{*}(\epsilon_1,\epsilon_1) \not\cong \LCH{*}(\epsilon_1,\epsilon_2)$ or
$\LCHBC{*}(\epsilon_1,\epsilon_1) \not\cong \LCHBC{*}(\epsilon_1,\epsilon_2)$ implies that $\epsilon_1 \not\cong \epsilon_2$ in $\Aug$: see Section~\ref{ssec:augisom} below for a discussion of isomorphism in $\Aug$. Indeed, an analogous computation shows that all five augmentations $\epsilon_1,\epsilon_2,\epsilon_3,\epsilon_4,\epsilon_5$ are nonisomorphic. (The analogous statement in $\AugBC$ was established in \cite[\S 5]{BC}.) As shown in \cite{EHK}, these five augmentations correspond to five Lagrangian fillings of the trefoil, and these fillings are all distinct; compare the discussion in \cite[\S 5]{BC} as well as Corollary~\ref{cor:filling-equiv} below.
\end{remark}

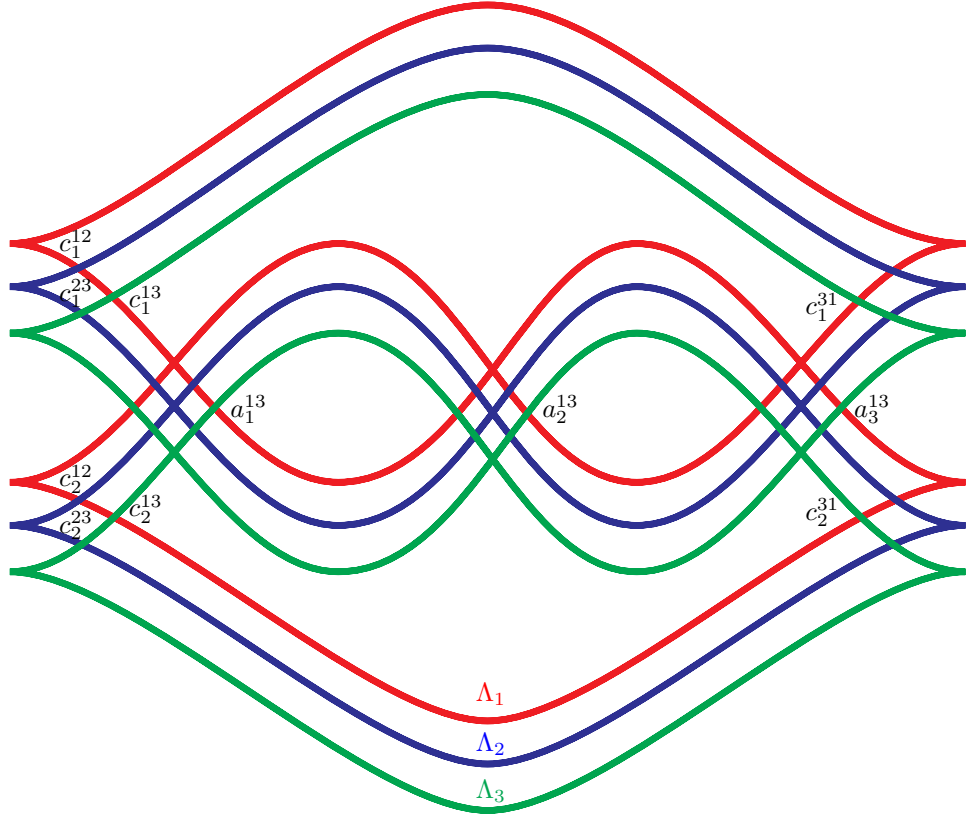
\begin{figure}
\labellist
\footnotesize\hair 2pt
\pinlabel $\color{red}{\Lambda_1}$ at 145 36
\pinlabel $\color{blue}{\Lambda_2}$ at 145 21
\pinlabel $\color{figuregreen}{\Lambda_3}$ at 145 7
\pinlabel $c_1^{13}$ at 41 155
\pinlabel $c_1^{12}$ at 20 172
\pinlabel $c_1^{23}$ at 20 157
\pinlabel $c_2^{13}$ at 41 92
\pinlabel $c_2^{12}$ at 20 101
\pinlabel $c_2^{23}$ at 20 87
\pinlabel $a_1^{13}$ at 72 122
\pinlabel $a_2^{13}$ at 166 122
\pinlabel $a_3^{13}$ at 260 122
\pinlabel $c_1^{31}$ at 245 153
\pinlabel $c_2^{31}$ at 245 91
\endlabellist
\centering
\includegraphics[scale=1.25]{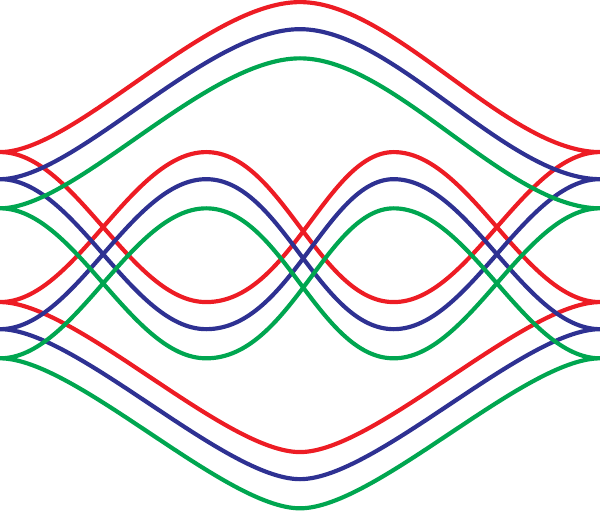}
\caption{
The $3$-copy of the Legendrian trefoil, in the front projection, with some Reeb chords labeled.
}
\label{fig:trefoil-3-front}
\end{figure}

We now compute $m_2$ as a product on
$\hompm(\epsilon_1,\epsilon_1)$. For this we use the front projection
$3$-copy $\Lambda^3$ of $\Lambda$, as shown in
Figure~\ref{fig:trefoil-3-front}. The relevant portion of the
differential for $\Lambda^3$ (with irrelevant signs) is:
\begin{alignat*}{2}
\partial(c_1^{13}) &= c_1^{12}c_1^{23}
& \quad
\partial(c_1^{31}) &= a_1^{33}a_2^{32}a_3^{21} +
a_1^{32}a_2^{22}a_3^{21} + a_1^{32}a_2^{21}a_3^{11} \\
\partial(c_2^{13}) &= c_2^{12}c_2^{23} &
\partial(c_2^{31}) &= -a_3^{33}a_2^{32}a_1^{21} -
a_3^{32}a_2^{22}a_1^{21} - a_3^{32}a_2^{21}a_1^{11} \\
\partial(a_1^{13}) &= c_1^{12}a_1^{23}-a_1^{12}c_2^{23} && \\
\partial(a_2^{13}) &= c_2^{12}a_2^{23}-a_2^{12}c_1^{23} && \\
\partial(a_3^{13}) &= c_1^{12}a_3^{23}-a_3^{12}c_2^{23}. &&
\end{alignat*}
Linearizing with respect to the augmentation $(\epsilon_1,\epsilon_1,\epsilon_1)$ on $\Lambda^3$, we find that the nonzero parts of $m_2 : \hom(\epsilon_1,\epsilon_1) \otimes \hom(\epsilon_1,\epsilon_1) \to \hom(\epsilon_1,\epsilon_1)$ are $m_2(c_1^+,c_1^+) = c_1^+$, $m_2(c_2^+,c_2^+) = c_2^+$, $m_2(c_2^+,a_1^+) = m_2(a_1^+,c_1^+) = a_1^+$, $m_2(c_1^+,a_2^+) = m_2(a_2^+,c_2^+) = a_2^+$, and $m_2(c_2^+,a_3^+) = m_2(a_3^+,c_1^+) = a_3^+$. This gives the following multiplication $m_2$ on $\LCH{*}(\epsilon_1,\epsilon_1)$:
\[
\begin{array}{|c|ccc|} \hline
m_2 & \lbrack c_1^+ +c_2^+ \rbrack & \lbrack a_2^+ \rbrack & \lbrack a_3^+ \rbrack \\ \hline
\lbrack c_1^+ +c_2^+ \rbrack & \lbrack c_1^+ +c_2^+ \rbrack & \lbrack a_2^+ \rbrack & \lbrack a_3^+ \rbrack \\
\lbrack a_2^+ \rbrack & \lbrack a_2^+ \rbrack & 0 & 0 \\
\lbrack a_3^+ \rbrack &\lbrack a_3^+ \rbrack & 0 & 0 \\ \hline
\end{array}
\]
Thus $[c_1^+ +c_2^+]$ acts as the identity in $\LCH{*}(\epsilon_1,\epsilon_1)$, exactly as predicted in Remark~\ref{rem:front-unit}.

For composition in $\AugBC$,
the nonzero parts of $m_2 : \homBC(\epsilon_1,\epsilon_1) \otimes \homBC(\epsilon_1,\epsilon_1) \to \homBC(\epsilon_1,\epsilon_1)$ are $m_2(a_3^-,a_2^-) = c_1^-$ and $m_2(a_2^-,a_3^-) = c_2^-$. This gives the following multiplication on $\LCHBC{*}(\epsilon_1,\epsilon_1)$:
\[
\begin{array}{|c|ccc|} \hline
m_2 & \lbrack a_1^- + a_3^- \rbrack & \lbrack a_2^- \rbrack & \lbrack c_1^- \rbrack \\ \hline
\lbrack a_1^- + a_3^- \rbrack & 0 & \lbrack c_1^- \rbrack & 0 \\
\lbrack a_2^- \rbrack & \lbrack c_1^- \rbrack & 0 & 0 \\
\lbrack c_1^- \rbrack & 0 & 0 & 0 \\ \hline
\end{array}
\]
This last multiplication table illustrates Sabloff duality \cite{Sabloff}: cohomology classes pair together, off of the fundamental class $[c_1^-]$.

\subsubsection{$m(9_{45})$} \label{sssec:m945}

Let $\Lambda$ be the Legendrian knot in Figure~\ref{fig:m945}. This is
of topological type $m(9_{45})$, and has previously appeared in work
of Melvin and Shrestha \cite{MelvinShrestha}, as the mirror diagram
for $9_{45}$, as well as in the Legendrian knot atlas \cite{atlas},
where it appears as the second diagram for $m(9_{45})$. In particular,
Melvin and Shrestha note that $\Lambda$ has two different linearized
contact homologies (see the discussion following
\cite[Theorem~4.2]{MelvinShrestha}).

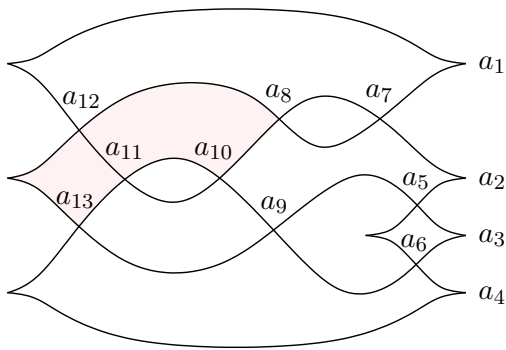
\begin{figure}
\centering
\labellist
\small\hair 2pt
\pinlabel $a_1$ [l] at 294 179
\pinlabel $a_2$ [l] at 294 106
\pinlabel $a_3$ [l] at 294 70
\pinlabel $a_4$ [l] at 294 34
\pinlabel $a_5$ at 258 106
\pinlabel $a_6$ at 257 67
\pinlabel $a_7$ at 235 161
\pinlabel $a_8$ at 172 161
\pinlabel $a_9$ at 169 91
\pinlabel $a_{10}$ at 131 125
\pinlabel $a_{11}$ at 75 125
\pinlabel $a_{12}$ at 48 157
\pinlabel $a_{13}$ at 44 94
\endlabellist
\includegraphics[scale=0.6]{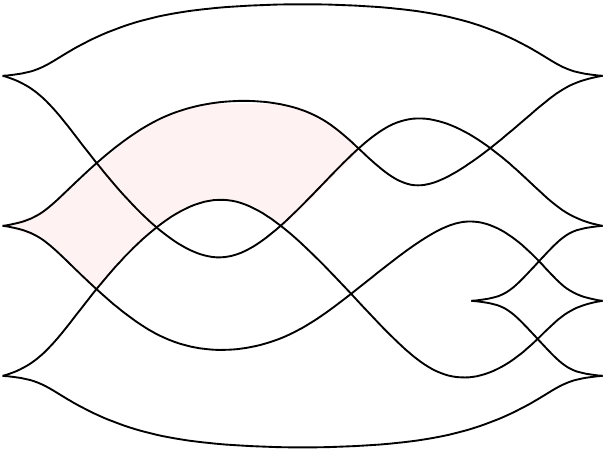}
\hspace{5ex}
\caption{
The knot $m(9_{45})$, with one particular disk shaded.}
\label{fig:m945}
\end{figure}

We can use the multiplicative structure on $\Aug$ to prove the
following, which was unknown until now according to the tabulation in
\cite{atlas}:

\begin{proposition}
$\Lambda$ is not isotopic to its Legendrian mirror.
\label{prop:m945}
\end{proposition}

\noindent
Here the \textit{Legendrian mirror} of a Legendrian knot in $\R^3$ is
its image under the diffeomorphism $(x,y,z) \mapsto (x,-y,-z)$. The
problem of distinguishing Legendrian knots from their mirrors is known
to be quite subtle; see, e.g., \cite{NgCLI} and \cite{CEKSW}. It was already noted in \cite{CEKSW} that the ring structure on $\AugBC$ for $\Lambda$ is noncommutative, and we will use this noncommutativity here.

\begin{proof}[Proof of Proposition~\ref{prop:m945}]
We use the calculation of the augmentation category from
Proposition~\ref{prop:complete-aug-description}, where we resolve $\Lambda$ to produce a Lagrangian projection, and choose any
orientation and base point. We claim that: (1) there is an augmentation $\epsilon$ from $\mathcal{A}(\Lambda)$ to $\bZ/2$ for which
\[
m_2 :\thinspace \LCH{-1}(\epsilon,\epsilon)
\otimes \LCH{2}(\epsilon,\epsilon) \to \LCH{1}(\epsilon,\epsilon)
\]
is nonzero; and (2) there is no such augmentation for the Legendrian mirror of $\Lambda$. Since $\Lambda$ and its Legendrian mirror have the same \dga{} over $\bZ/2$ but with the order of multiplication reversed, (2) is equivalent to
$m_2 :\thinspace \LCH{2}(\epsilon,\epsilon)
\otimes \LCH{-1}(\epsilon,\epsilon) \to \LCH{1}(\epsilon,\epsilon)$ being zero for all augmentations $\epsilon$ for $\Lambda$.

To establish (2), note that the only Reeb chord of $\Lambda$ of degree $-2$ is $a_{10}$, while the Reeb chords of degree $1$ are $a_1,a_2,a_3,a_4,a_{13}$. By inspection, for $i \in \{1,2,3,4,13\}$, there is no disk whose negative corners include $a_{10}$ and $a_i$, with $a_{10}$ appearing first, and so $m_2(a_i^+,a_{10}^+) = 0$. Since $[a_{10}^+]$ generates $\LCC{-1}$ and $[a_i^+]$ generate $\LCC{2}$ for $i \in \{1,2,3,4,13\}$, (2) follows.

It remains to prove (1).
There are five augmentations from $\mathcal{A}(\Lambda)$ to $\bZ/2$. Two of these are given (on the degree $0$ generators) by
$\epsilon(a_5)=\epsilon(a_6)=\epsilon(a_7)=1$, $\epsilon(a_{12})=0$,
and $\epsilon(a_8) = 0 \text{ or } 1$. (The other three have $\LCH{2}(\epsilon,\epsilon) = 0$.) Let $\epsilon$ be either of
these two augmentations; then
\begin{align*}
\LCH{2}(\epsilon,\epsilon) &\cong (\bZ/2)\langle [a_{13}^+]\rangle \\
\LCH{1}(\epsilon,\epsilon) &\cong (\bZ/2)\langle
[a_8^+],[a_{12}^++(1+\epsilon(a_8))a_7^+]\rangle \\
\LCH{-1}(\epsilon,\epsilon) &\cong (\bZ/2)\langle [a_{10}^+]\rangle.
\end{align*}
Now the fact that $\partial(a_8) = a_{13}a_{10}$ (the relevant disk is
shaded in Figure~\ref{fig:m945}) leads to
$m_2(a_{10}^+,a_{13}^+) = a_8^+$, and thus $m_2 :\thinspace \LCH{-1}
\otimes \LCH{2} \to \LCH{1}$ is nonzero.
%
\end{proof}

\begin{remark}
It turns out that for either of the two augmentations specified in the
proof of Proposition~\ref{prop:m945},
$m_2 :\thinspace \LCH{i} \otimes \LCH{j} \to
\LCH{i+j}$ is nonzero for $(i,j) = (-1,2)$, $(1,-1)$, and $(2,1)$, but
zero for $(i,j) = (2,-1)$, $(-1,1)$, and $(1,2)$; any of these can be
used to prove the result. In addition, we note
that the same proof also works if we use $\AugBC$ instead of $\Aug$.
\end{remark} 
\newpage

\section{Properties of the augmentation category}
\label{sec:augprops}

This section explores certain properties of the augmentation category
$\Aug(\Lambda,\coeffs)$ defined in Section~\ref{sec:augcat}. In
Section~\ref{ssec:augstructure}, we give categorical formulations
of Sabloff duality and the duality exact sequence, and also explain
the relation of the cohomology and compactly supported cohomology
of a Lagrangian filling to the $+/-$ endomorphism spaces of the corresponding
augmentation. Some of the results from Section~\ref{ssec:augstructure}
are very similar if not essentially identical to previously known
results in the literature, and in Section~\ref{ssec:dictionary}, we
provide a dictionary that allows comparison.
Finally, in Section~\ref{ssec:augisom}, we
discuss relations between different notions of equivalence of
augmentations, showing in particular that being isomorphic in $\Aug$ is
the same as being \dga{} homotopic.

\subsection{Duality and long exact sequences}

\label{ssec:augstructure}

Let $\Lambda \subset J^1(\bR)$ be a Legendrian link.
Here we examine the relationship between the positive and negative
augmentation categories $\Augpm(\Lambda,\coeffs)$; recall from
Proposition~\ref{prop:neg-is-BC} that $\AugBC(\Lambda,\coeffs)$ is the
Bourgeois--Chantraine augmentation category.
We note that many of the results in this subsection are inspired by,
and sometimes essentially identical to, previously known results, and
we will attempt to include citations wherever appropriate.

\subsubsection{Exact sequence relating the hom spaces}

\begin{proposition}
There is a morphism of non-unital $A_\infty$ categories
$\AugBC(\Lambda) \to \Aug(\Lambda)$ carrying
\label{prop:nu-morphism}
the augmentations to themselves.
\end{proposition}
\begin{proof}
In Proposition \ref{prop:bifunctor}, we observed that from the 3-copy, we obtained a map
$$m_2: \hom(\epsilon_1,\epsilon_3) \otimes \homBC(\epsilon_2,\epsilon_1) 
=
C_{13}^\vee \otimes C_{21}^\vee  \to C_{23}^\vee = \hom(\epsilon_2, \epsilon_3).$$
Taking $\epsilon_1 = \epsilon_3$ and specializing to $\mathrm{id} \in  \hom(\epsilon_1,\epsilon_3 = \epsilon_1)$,
we get a map $$\homBC(\epsilon_2,\epsilon_1) \to \hom(\epsilon_2, \epsilon_3 = \epsilon_1).$$
The higher data characterizing an $A_\infty$ functor and related compatibilities
comes from similar compositions obtained from higher numbers of copies.
\end{proof}

\begin{proposition} \label{prop:pmexactseq}
Let $\Lambda \subset J^1(\bR)$ be a
Legendrian link, and let
$\Aug(\Lambda_f)$ and $\AugBC(\Lambda)$ be the
positive augmentation
category as constructed in Definition~\ref{def:augplus-xy-xz} (with some
Morse function $f$ chosen on $\Lambda$), and the
negative augmentation category as constructed in Definition~\ref{def:AugBC}.
Let $\epsilon_1,\epsilon_2$ be augmentations of $\Lambda$, and suppose
that $\epsilon_1,\epsilon_2$ agree on $\mathcal{T}$ (that is, on the $t_k$'s).
Then the map
determined by the functor from Proposition~\ref{prop:nu-morphism}
fits into a short exact sequence of chain complexes 
\[
0 \to \homBC(\epsilon_1, \epsilon_2) \to \hom(\epsilon_1, \epsilon_2)
\to C^*(\Lambda) \to 0,
\]
where $C^*(\Lambda)$ is a chain complex computing the ordinary cohomology of $\Lambda$.
It follows there is a long exact sequence
\[
\cdots \to H^{i-1}(\Lambda) \to H^i \homBC(\epsilon_1, \epsilon_2)
\to H^i \hom(\epsilon_1, \epsilon_2) \to H^{i} (\Lambda) \to \cdots.
\]
\end{proposition}

\begin{proof}
The proof consists of explicitly writing the complex $\LCC{*}(\epsilon_1,\epsilon_2)$ as a mapping cone:
\[
\LCC{*}(\epsilon_1,\epsilon_2) = \operatorname{Cone} \left(C^{*+1}(\Lambda) \to
\LCCBC{*}(\epsilon_1,\epsilon_2)\right).
\]
For simplicity, we assume that $\Lambda$ is a single-component knot;
the multi-component case is a straightforward generalization.
Let $\Lambda^2$ be the Lagrangian projection $2$-copy of $\Lambda$. In
the notation of Proposition~\ref{prop:complete-aug-description},
$\hom(\epsilon_1,\epsilon_2) = C_{12}^{(\epsilon_1,\epsilon_2)}$ is
generated by the $a_k^+$'s
as well as $x^+$ and $y^+$, while $\homBC(\epsilon_1,\epsilon_2)
= C_{21}^{(\epsilon_2,\epsilon_1)}$ is generated by just the
$a_k^-$'s: that is, if we identify $a_k^+=a_k^-=a_k^\vee$, then
\[
\hom(\epsilon_1,\epsilon_2) = \homBC(\epsilon_1,\epsilon_2) \oplus
\langle x^+,y^+ \rangle.
\]

The differential $m_1^+$ on
$\hom(\epsilon_1,\epsilon_2)$ is given by dualizing the linear part
of the twisted differential $\partial^2_{(\epsilon_1,\epsilon_2)}$ on
$C^{12}_{(\epsilon_1,\epsilon_2)}$, while the differential $m_1^-$ on
$\homBC(\epsilon_1,\epsilon_2)$ is given by dualizing the linear part
of $\partial^2_{(\epsilon_2,\epsilon_1)}$ on
$C^{21}_{(\epsilon_2,\epsilon_1)}$. Inspecting
Proposition~\ref{prop:mcopyMultiple} gives that $m_1^+$ and $m_1^-$ coincide
on the $a_k^\vee$'s, while
$m_1^+(y^+) = 0$ and $m_1^+(x^+) \in \langle
a_1^+,\ldots,a_r^+ \rangle$ as in the proof of
Proposition~\ref{thm:unital}. (Note that for $m_1^+(y^+)=0$, we need
the fact that $\epsilon_1(t)=\epsilon_2(t)$, which is true by
assumption.)
The quotient
complex $\langle
x^+,y^+ \rangle$ of $\hom(\epsilon_1,\epsilon_2)$ is then the
usual Morse complex $C^*(S^1)=C^*(\Lambda)$, and the statement about
the mapping cone follows.
\end{proof}

\begin{remark}
The condition in Proposition~\ref{prop:pmexactseq}
that $\epsilon_1$ and $\epsilon_2$ agree on
$\mathcal{T}$ is automatically satisfied for any single-component knot
with a Morse function with a unique minimum and maximum: in this case,
there is only one $t$, and $\epsilon_1(t) = \epsilon_2(t) = -1$ by a
result of Leverson \cite{Leverson}. Here we implicitly assume that the
augmentation categories are of $\bZ$-graded augmentations, although
the same is true for $(\bZ/m)$-graded augmentations if $m$ is even.
However, Proposition~\ref{prop:pmexactseq} fails to hold for
multi-component links if we remove the assumption that
$\epsilon_1,\epsilon_2$ must agree on $\mathcal{T}$.
\end{remark}

\subsubsection{Sabloff duality}
Here we present a repackaging of Sabloff duality \cite{Sabloff,EESab}
in our language. Roughly speaking, Sabloff duality states that
linearized contact homology and linearized contact cohomology fit into
a long exact sequence with the homology of the Legendrian.
In our notation, linearized contact
cohomology is the cohomology of $\homBC$, while linearized contact
homology is the homology of the dual to $\homBC$; see
Section~\ref{ssec:dictionary}.

For a  cochain complex $K$, we write $K^\dagger$ to denote the cochain complex obtained by dualizing the underlying
vector space and differential and negating all the gradings.  By comparison, if $K$ were a chain complex, we would write
$K^*$ to denote the cochain complex obtained by dualizing the underlying vector space and differential, but
leaving all the gradings alone.
We now have the following result, which can roughly be summarized as
``homology in $\AugBC$ is cohomology in $\Aug$.''

\begin{proposition}
There is a quasi-isomorphism $\homBC(\epsilon_2, \epsilon_1)^\dagger[-2]  \xrightarrow{\sim} \hom(\epsilon_1, \epsilon_2)$.
\label{prop:duality}
\end{proposition}

\begin{proof} 
This proof is given in \cite{EESab} (though not in the language stated here); we include the proof in our language for the convenience of the reader.
Let $\Lambda^{(2)}$ be an (appropriately perturbed) $2$-copy of $\Lambda$.  Let
$\overline{\Lambda}^{(2)}$ the link with the same $xy$ projection as $\Lambda^{(2)}$, but
with $\Lambda_1$ lying very far above $\Lambda_2$ in the $z$ direction.

We write $\overline{C}$ for the space spanned by the Reeb chords of $\overline{\Lambda}^{(2)}$.
Note that, since these are in correspondence with self intersections in the $xy$ projection, which
is the same for $\Lambda^{(2)}$ and $\overline{\Lambda}^{(2)}$, the Reeb chords of these links are
in bijection.  However, in $\Lambda^{(2)}$, all Reeb chords go from $\Lambda_2$ to $\Lambda_1$,
so $\overline{C}^{21} = 0$.  Note that if $r \in \mathcal{R}^{21}$ corresponds to a chord $\overline{r} \in \overline{\mathcal{R}}^{12}$, then
$\mu(\overline{r}) = - \mu(r) - 1$ because the Reeb chord is now oppositely oriented between Maslov potentials, and moreover
is a minimizer of front projection distance if it was previously a maximizer, and vice versa.  We will write
$C^{21}_{-*-1}$ to indicate the graded module with this corrected grading.   We have explained that,
as a graded module,
\[
\overline{C}^{12} = C^{12} \oplus C^{21}_{-*-1}.
\]

Let $\epsilon_1, \epsilon_2$ be augmentations of $\alg(\Lambda)$.  We write $\epsilon = (\epsilon_1, \epsilon_2)$ for
the corresponding  augmentation of $\Lambda^{(2)}$ and $\overline{\epsilon} = (\epsilon_1, \epsilon_2)$ for the
corresponding augmentation of $\overline{\Lambda}^{(2)}$.
If we pass from $\Lambda^{(2)}$ to $\overline{\Lambda}^{(2)}$ by moving the components further apart by some large
distance $Z$ in the $z$ direction, then every Reeb chord of $\overline{\Lambda}^{(2)}$ corresponding to a generator of
$C_{12}$ has length larger than $Z$, and every Reeb chord corresponding to a generator of $C_{21}^*$ has length smaller
than $Z$.  Because the differential is filtered by chord length, it follows that we have an exact sequence of dg modules
\[
0 \to (C^{21}_{-*-1}, \partial_{\overline{\epsilon}}|_{C^{21}_{-*-1}}) \to (\overline{C}^{12}, \partial_{\overline{\epsilon}}) \to
(C^{12},  \partial_{\overline{\epsilon}} ) \to 0.
\]
Here $(C^{12},  \partial_{\overline{\epsilon}})$ means the dg module which is $C^{12}$ equipped with the quotient differential
coming from the exact sequence.

Now geometric considerations imply that
\[
(C^{12},  \partial_{\overline{\epsilon}})  =
(C^{12}, \partial_{\epsilon}|_{C^{12}}):
\]
if $r \in \mathcal{R}^{12} \subset \overline{\mathcal{R}}^{12}$, then
$\partial_{\overline{\epsilon}}(r)$ in $\overline{C}^{12}$ counts
disks with a negative corner at some $r' \in
\overline{\mathcal{R}}^{12}$, which is either in $\mathcal{R}^{12}$
(in which case it contributes equally to $\partial_{\epsilon}$ and the quotient differential
$\partial_{\overline{\epsilon}}$) or in $\mathcal{R}^{21}$ (in which
it does not contribute to either, since in $\Lambda^{(2)}$ it
corresponds to a disk with two positive punctures).
Additionally, we have
\[
(C^{21}_{-*-1}, \partial_{\overline{\epsilon}}|_{C^{21}_{-*-1}})^* =
((C_{21}^*, \partial^*_\epsilon|_{C_{21}^*})[-1])^\dag:
\]
this is a manifestation of the fact that when we push the two copies
past each other, a disk with a positive and a negative corner at
chords in $\mathcal{R}^{21}$ becomes a disk with a positive and a
negative corner at chords in $\overline{\mathcal{R}}^{12}$, but with
the positive and negative corners switched.

Dualizing and shifting, we have
\[
0 \to \hom(\epsilon_1, \epsilon_2) \to
(\overline{C}^{12}, \partial_{\overline{\epsilon}})^*[-1] \to
\homBC(\epsilon_2, \epsilon_1)^\dag[-1] \to 0.
\]
View the central term as a mapping cone to obtain a morphism
$\homBC(\epsilon_2, \epsilon_1)^\dag[-2] \to \hom(\epsilon_1, \epsilon_2)$.
Since $\overline{\Lambda}^2$
can be isotoped
so that there are no Reeb chords between the two components, the central term is acyclic, so this morphism
is a quasi-isomorphism.
\end{proof}

\begin{remark}
Proposition~\ref{prop:duality} holds for $n$-dimensional Legendrians
as well, with $n+1$ replacing $2$.
\end{remark}

\begin{corollary}
We have $H^*\hom(\epsilon_1,\epsilon_2) \cong
H^{*-2}\homBC(\epsilon_2,\epsilon_1)^\dagger$: the cohomology of the
hom spaces in $\Aug$ is isomorphic to (bi)linearized Legendrian
contact homology.
\label{cor:hom-is-LCH}
\end{corollary}

Here bilinearized Legendrian contact homology, as constructed in
\cite{BC}, is the cohomology of $\homBC^\dagger$;
see Section~\ref{ssec:dictionary} below for the precise equality, and
for further discussion of the
relation of Proposition~\ref{prop:duality} to Sabloff duality.

\subsubsection{Fillings}

As described in the Introduction, an important source of augmentations
is exact Lagrangian fillings, whose definition we recall here. For a
contact manifold $(V,\alpha)$, the cylinder $\bR \times V$ is a
symplectic manifold with symplectic form $\omega = d(e^t\alpha)$,
where $t$ is the $\bR$ coordinate. Let $\Lambda \subset V$ be
Legendrian. A \textit{Lagrangian filling} of $\Lambda$ is a compact
$L \subset (-\infty,0] \times V$ such that $\omega|_L = 0$, $L \cap
\{t=0\} = \{0\} \times \Lambda$, and $L \cup ([0,\infty) \times
\Lambda)$ is smooth. The filling is \textit{exact} if $e^t\alpha|_L$
is an exact $1$-form and its primitive is constant for $t \gg 0$ and for $t \ll 0$, cf.\ \cite{Cha-exact}. As part of the functoriality of Symplectic Field
Theory, any exact Lagrangian filling of a Legendrian $\Lambda$ induces
an augmentation of the \dga{} for $\Lambda$; see e.g. \cite{Ekh09}, and
\cite{EHK} for the special case $V = J^1(\bR)$.

We now restrict as usual to this special case.
For an augmentation obtained from an exact Lagrangian filling,
$\Hom_\pm$ is determined by the topology of the filling.  This is essentially a result of \cite{Ekh09} (see also \cite[\S4.1]{BC});
translated into our language, it becomes the following:

\begin{proposition}
\label{prop:plusminusexseq}
Suppose that $L$ is an exact Lagrangian filling of $\Lambda$ in
$(-\infty,0]\times J^1(\bR)$, with Maslov number $0$, and let
$\epsilon_L$ be the augmentation of $\Lambda$ corresponding to the filling. Then
\[
H^k \hom(\epsilon_L, \epsilon_L) \cong H^k(L),
\hspace{5ex}
H^k \homBC(\epsilon_L, \epsilon_L) \cong H^k(L,\Lambda),
\]
and the long exact sequence
\[
\cdots \to H^{k-1}(\Lambda) \to H^k \homBC(\epsilon_L, \epsilon_L) \to H^k \hom(\epsilon_L, \epsilon_L) \to H^k(\Lambda) \to \cdots
\]
is the standard long exact sequence in relative cohomology.
\end{proposition}

\begin{proof}
This result has appeared in various guises and degrees of completeness in
\cite{BC,DR,EHK,Ekh09} (in \cite{Ekh09} as a conjecture); the basic
result that linearized contact
homology for $\epsilon_L$ is the homology of $L$ is often attributed
to Seidel. For completeness, we indicate how to obtain the precise
statement of Proposition~\ref{prop:plusminusexseq} via wrapped Floer
homology, using the terminology and results from
\cite{DR}.

Theorem~6.2 in \cite{DR} expresses a wrapped Floer complex
$(CF_\bullet(L,L_+^{\eta,\epsilon}),\partial)$ as a direct sum
\[
CF_\bullet(L,L_+^{\eta,\epsilon}) \cong C^\bullet_{\mathit{Morse}}(F_+) \oplus
C^{\bullet-1}_{\mathit{Morse}}(f) \oplus CL^{\bullet-2}(\Lambda),
\]
where the differential $\partial$ is block upper triangular with
respect to this decomposition, so that $\partial$ maps each summand to
itself and to the summands to the right. In this decomposition,
$C^\bullet_{\mathit{Morse}}(F_+)$,
$C^\bullet_{\mathit{Morse}}(f)$, and
$\operatorname{Cone}(C^\bullet_{\mathit{Morse}}(F_+) \to
C^\bullet_{\mathit{Morse}}(f))$ are Morse complexes for
$C^\bullet(L)$, $C^\bullet(\Lambda)$, and $C^\bullet(L,\Lambda)$,
respectively. Furthermore, inspecting the definitions of \cite[\S 6.1.2]{DR} (and
recalling  we shift degree by $1$) gives
$CL^{\bullet-2}(\Lambda) = \homBC^{\bullet-1}(\epsilon_L, \epsilon_L)$ and
$\operatorname{Cone}(C^{\bullet}_{\mathit{Morse}}(f) \to
CL^{\bullet}(\Lambda)) = \hom^{\bullet+1}(\epsilon_L, \epsilon_L)$.

Now the wrapped Floer homology for the exact Lagrangian fillings
$L,L_+^{\eta,\epsilon}$ vanishes (see
e.g. \cite[Proposition~5.12]{DR}), and so the complex
$CF_\bullet(L,L_+^{\eta,\epsilon})$ is acyclic. It follows
from $CF_\bullet(L,L_+^{\eta,\epsilon}) \cong C^\bullet(L) \oplus
\hom^{\bullet-1}(\epsilon_L, \epsilon_L)$ that
$H^k \hom(\epsilon_L, \epsilon_L) \cong H^k(L)$, and from
$CF_\bullet(L,L_+^{\eta,\epsilon}) \cong C^{\bullet}(L,\Lambda) \oplus
\homBC^{\bullet-1}(\epsilon_L)$ that
$H^k \homBC(\epsilon_L, \epsilon_L) \cong H^k(L,\Lambda)$.
The statement about the long exact sequence similarly follows.
\end{proof}

\begin{remark}
Proposition~\ref{prop:plusminusexseq} relies on the Lagrangian filling
$L$ having Maslov number $0$, where the Maslov number of $L$ is the
$\gcd$ of the Maslov numbers of all closed loops in $L$; see
\cite{Ekh09,EHK}. However, a version of
Proposition~\ref{prop:plusminusexseq} holds for exact Lagrangian
fillings of arbitrary Maslov number $m$. In this case, $\epsilon_L$
is not graded but $m$-graded: that is, $\epsilon(a) = 0$ if $m
\nmid |a|$, but $\epsilon(a)$ can be nonzero if $|a|$ is a multiple of
$m$. The isomorphisms and long exact sequence in
Proposition~\ref{prop:plusminusexseq} continue to hold when all
gradings are taken mod $m$.
\end{remark}

\begin{remark}
\label{rmk:dgasigns}
Here we make an extended comment on signs as they relate to augmentations coming from fillings. For simplicity we restrict our discussion to a Legendrian knot $\Lambda \subset \R^3$ with DGA $(\cA,\dd)$, which we recall for emphasis is generated by Reeb chords of $\Lambda$ along with $t^{\pm 1}$. Given an exact Lagrangian filling $L$ of $\Lambda$, the augmentation $\epsilon_L$ as constructed in \cite{EHK} (cf.\ \cite{Ekh09}) is a map to $\field = \bZ/2$. This is lifted to an augmentation $\cA \to \bZ$ by Karlsson \cite{Karlsson-ori} by a choice of coherent orientations of various moduli spaces.

More precisely, what is constructed in \cite{Karlsson-ori} is an augmentation of the Chekanov--Eliashberg DGA of $\Lambda$, but taken with $\bZ$ coefficients. A natural way to define such a DGA is to `forget' the homology coefficients $t^{\pm 1}$ in $(\cA,\dd)$, which is to say, set $t=1$ in $(\cA,\dd)$ to yield a DGA $(\cA_1,\dd_1)$ where $\cA_1$ is the tensor algebra over $\bZ$ generated by Reeb chords. However, one could also set $t=-1$ in $(\cA,\dd)$ to yield another DGA $(\cA_{-1},\dd_{-1})$ with the same underlying algebra $\cA_{-1}=\cA_1$ but distinct differentials. For example, for the standard unknot, $\cA_1=\cA_{-1}$ is generated by a single Reeb chord $a$ with differential $\dd_1(a) = 2$ and $\dd_{-1}(a) = 0$.

To expand on this a bit further, signs in the differential in $(\cA,\dd)$ are determined geometrically by a choice of spin structure on the Legendrian $\Lambda$ \cite{EES-ori}. When $\Lambda$ is topologically $S^1$, there are two spin structures, one (called the ``Lie group spin structure'' in \cite{EES-ori}) coming from the canonical trivialization of $TS^1$, and the other (the ``null-cobordant spin structure'') from the unique spin structure on $D^2$ by viewing $S^1$ as its boundary. As shown in \cite[Theorem~4.29]{EES-ori}, the differentials resulting from the two spin structures are not independent, but are related by the automorphism of $\cA$ which sends $t$ to $-t$ and sends each Reeb chord to itself.

The standard combinatorial sign conventions for the Chekanov--Eliashberg DGA, as originally defined in \cite{ENS} and presented here in Section~\ref{ssec:dga-background}, correspond to the Lie group spin structure; see \cite[Theorem~4.32]{EES-ori} and \cite[Appendix~A]{NgSFT}. In the context of fillings, however, it is more natural to choose the null-cobordant spin structure. What Karlsson shows in \cite{Karlsson-ori} is that a filling induces an augmentation of the DGA over $\bZ$ obtained from forgetting the homology coefficients (setting $t=1$) in the DGA $(\cA,\dd)$ for the null-cobordant spin structure. In light of the preceding discussion about how changing spin structure negates $t$, this is the DGA $(\cA_{-1},\dd_{-1})$. In other words:

\begin{proposition}[\cite{Karlsson-ori}]
An exact Lagrangian filling $L$ of a Legendrian knot $\Lambda \subset \R^3$ induces an augmentation of the DGA $(\cA,\dd)$ of $\Lambda$, $\epsilon_L :\thinspace \cA \to \bZ$, satisfying $\epsilon_L(t) = -1$.
\label{prop:filling-signs}
\end{proposition}
\noindent
Note that this augmentation induces an augmentation to any field $\field$, also sending $t$ to $-1$; this is in line with the result of Leverson \cite{Leverson} that any augmentation to $\field$ must send $t$ to $-1$.

We conclude this remark by comparing with the sheaf picture.  As defined in \cite{STZ}, microlocal monodromy does not explicitly depend on the choice of a spin structure on $\Lambda$.  
However, from a more abstract point of view, microlocal monodromy is naturally valued in a category of (in general twisted) local systems on $\Lambda$, but the isomorphism with the category of local systems
is not entirely canonical.  (The autoequivalences of the identity functor of the category of chain complexes over a ring $k$ is naturally identified with $k^*$.  This leads to $H^1(X, k^*)$ acting by autoequivalences
on the category of local systems on $X$.  When $k = \mathbb{Z}$, this means that isomorphisms with the category of local systems of $X$ are a torsor over $H^1(X, \pm 1)$, just like spin structures.)  The work 
 \cite{STZ}  made a choice at the cusps which in effect fixes this isomorphism.  
A more abstract discussion of how such `brane structure' choices enter into microlocal sheaf theory can be found in \cite{JinTreumann}, which in turn was partially inspired by an account \cite{lurie-rotation} explaining 
among other things a homotopical setup well suited to understanding certain orientation choices in Floer theory. See also the discussion of obstruction classes in \cite[Part~10]{Guillermou}.

In any case, under the correspondence we will set up, the filling $L$ of $\Lambda$ yields a sheaf in $\cC_1(\Lambda; \mathbbm{k})$ microsupported along $\Lambda$. 
Moreover, the microlocal monodromy of the corresponding sheaf will be the restriction of the rank one local system on $L$ to the boundary $\Lambda$, hence trivial because 
this boundary circle is a commutator in $\pi_1(L)$. The correspondence between sheaves and augmentations sends this sheaf to $\epsilon_L$, and the triviality of the monodromy to the condition $\epsilon_L(t)=1$. 
This indicates that the choices made in \cite{STZ} to define microlocal monodromy correspond to the choice of 
 the null-cobordant spin structure on $\Lambda$.

%
%
%
%

\end{remark}

\subsection{Dictionary and comparison to previously known results}
\label{ssec:dictionary}

Here we compare our notions and notations with pre-existing ones,
especially from \cite{BC}.
We have considered a number of constructions derived from the
Bourgeois--Chantraine category
$\AugBC(\Lambda,\coeffs)$ that previously appeared
in \cite{BC} or elsewhere in the literature. For convenience, we
present here a table translating between our notation and notation
from other sources, primarily \cite{BC}.

\vspace{2ex}

\centerline{
\renewcommand{\arraystretch}{1.2}
\begin{tabular}{|l|l|} \hline
Notation here & Notation in other sources \\ \hline\hline
$\AugBC(\Lambda,\coeffs)$ & Bourgeois--Chantraine augmentation
category \cite{BC} \\ \hline
$\homBC^*(\epsilon_1,\epsilon_2)$ &
$\operatorname{Hom}^{*-1}(\epsilon_2,\epsilon_1) = C^{*-1}_{\epsilon_1,\epsilon_2}$
\cite{BC} \\ \hline
$H^*\homBC(\epsilon_1,\epsilon_2)$ &
bilinearized Legendrian contact cohomology
$LCH^{*-1}_{\epsilon_1,\epsilon_2}(\Lambda)$ \cite{BC} \\ \hline
$H^*\homBC(\epsilon,\epsilon)$ & linearized Legendrian contact
cohomology $LCH^{*-1}_{\epsilon}(\Lambda)$ \cite{Sabloff,EESab} \\
\hline
$\homBC^*(\epsilon_1,\epsilon_2)^\dagger$ &
$C^{\epsilon_1,\epsilon_2}_{-*-1}$ \cite{BC} \\ \hline
$H^*\homBC(\epsilon_1,\epsilon_2)^\dagger$ &
bilinearized Legendrian contact homology
$LCH_{-*-1}^{\epsilon_1,\epsilon_2}(\Lambda)$ \cite{BC} \\ \hline
$H^*\homBC(\epsilon,\epsilon)^\dagger$ & linearized Legendrian contact
homology $LCH_{-*-1}^\epsilon(\Lambda)$ \cite{C} \\ \hline
\end{tabular}
}

\vspace{2ex}

Using this dictionary, we can interpret various results from the
literature in our language.
For instance, Sabloff duality, or more precisely the
Ekholm--Etnyre--Sabloff duality exact sequence
\cite[Theorem~1.1]{EESab} relating linearized Legendrian contact
homology and cohomology, is:
\[
\cdots \to H_{k+1}(\Lambda) \to LCH_\epsilon^{-k}(\Lambda) \to
LCH^k_\epsilon(\Lambda) \to H_k(\Lambda) \to \cdots.
\]
This was generalized in \cite[Theorem~1.5]{BC} to bilinearized contact
homology and
cohomology:
\[
\cdots \to H_{k+1}(\Lambda) \to
LCH_{\epsilon_2,\epsilon_1}^{-k}(\Lambda) \to
LCH^{\epsilon_1,\epsilon_2}_k(\Lambda) \to H_k(\Lambda) \to \cdots.
\]
Since this long exact sequence is derived from a
chain-level argument using mapping cones, we can dualize to give:
\[
\cdots \to H^k(\Lambda) \to LCH^k_{\epsilon_1,\epsilon_2}(\Lambda) \to
LCH_{-k}^{\epsilon_2,\epsilon_1} \to H^{k+1}(\Lambda) \to \cdots.
\]
But we have $LCH^k_{\epsilon_1,\epsilon_2}(\Lambda) =
H^{k+1}\homBC(\epsilon_1,\epsilon_2)$, while by Corollary~\ref{cor:hom-is-LCH},
\[
LCH_{-k}^{\epsilon_2,\epsilon_1}(\Lambda) = H^{k-1}
\homBC(\epsilon_2,\epsilon_1)^\dagger \cong H^{k+1}
\hom(\epsilon_1,\epsilon_2),
\]
and so
the last exact sequence now becomes the exact sequence in
Proposition~\ref{prop:pmexactseq}.

In the case when $\Lambda$ has an exact Lagrangian filling $L$ with
corresponding augmentation $\epsilon_L$, the
fundamental result that the linearized contact
cohomology is the homology of the filling is written in the literature
as:
\begin{equation}
LCH_{\epsilon_L}^{1-k}(\Lambda) \cong H_k(L).\label{eq:filling}
\end{equation}
As discussed in the proof of Proposition~\ref{prop:plusminusexseq},
this was first stated in \cite{Ekh09} and also appears in
\cite{BC,DR,EHK}. Now we have
$LCH_{\epsilon_L}^{1-k}(\Lambda) =
H^{2-k}\homBC(\epsilon_L,\epsilon_L)$, while
$H_k(L) \cong H^{2-k}(L,\Lambda)$ by Poincar\'e duality; thus
\eqref{eq:filling} agrees with our
Proposition~\ref{prop:plusminusexseq} (which, after all, was
essentially proven using \eqref{eq:filling}).

To summarize the relations between the various constructions in the
presence of a filling:
\begin{align*}
LCH^{\epsilon_L}_{1-k}(\Lambda) \cong H^k\hom(\epsilon_L,\epsilon_L)
&\cong H^k(L) \cong H_{2-k}(L,\Lambda) \\
LCH^{k-1}_{\epsilon_L}(\Lambda) \cong H^k\homBC(\epsilon_L,\epsilon_L)
&\cong H^k(L,\Lambda) \cong H_{2-k}(L).
\end{align*}

\begin{remark}
With the benefit of hindsight, the terminology ``linearized contact
cohomology'' applied to $H^*\homBC(\epsilon, \epsilon)$ is perhaps
less than optimal on general philosophical grounds: cohomology
should contain a unit, and $H^*\homBC(\epsilon, \epsilon)$ does not.
Moreover, in the case when $\epsilon=\epsilon_L$ is given by a filling
and so $H^*\homBC(\epsilon,\epsilon)$ has a geometric meaning, it is
compactly supported cohomology (or, by Poincar\'e duality, regular homology):
\[
H^*\homBC(\epsilon_L,\epsilon_L) \cong H^*(L,\Lambda) \cong H_{2-*}(L).
\]
By contrast, we have
\[
H^*\hom(\epsilon_L,\epsilon_L) \cong H^*(L),
\]
and so it may be more suggestive to refer to
$H^*\hom(\epsilon_L,\epsilon_L)$ rather than
$H^*\homBC(\epsilon_L,\epsilon_L)$ as linearized contact cohomology.

To push this slightly further, ``linearized contact homology''
$LCH_*^\epsilon(\Lambda)$ is
$H^{-*-1}\homBC(\epsilon,\epsilon)^\dagger$, which by
Proposition~\ref{prop:duality} is isomorphic to
$H^{-*+1}\hom(\epsilon,\epsilon)$. Thus linearized contact homology,
confusingly enough, is a unital ring, and indeed
for $\epsilon=\epsilon_L$ it is the \textit{cohomology} ring of $L$!
\end{remark}

\subsection{Equivalence of augmentations}
\label{ssec:augisom}

Having formed a unital category $\Aug$ from the set of augmentations, we have a
natural notion of when two augmentations are isomorphic.   Note that the following
are  equivalent by definition: isomorphism in $\Aug$, isomorphism in the cohomology
category $H^* \Aug$, and isomorphism in the degree zero part $H^0 \Aug$.

This notion implies in particular
that the corresponding linearized contact homologies are isomorphic:

\begin{proposition}[cf.\ {\cite[Theorem 1.4]{BC}}]
If
$\epsilon_1,\epsilon_2$ are isomorphic in $\Aug$, then
\[
\LCH{*}(\epsilon_1,\epsilon_3) \cong \LCH{*}(\epsilon_2,\epsilon_3)
\text{ and }
\LCH{*}(\epsilon_3,\epsilon_1) \cong \LCH{*}(\epsilon_3,\epsilon_2)
\]
for any augmentation $\epsilon_3$. In particular,
$$
\LCH{*}(\epsilon_1,\epsilon_1) \cong \LCH{*}(\epsilon_1,\epsilon_2) \cong
\LCH{*}(\epsilon_2,\epsilon_2).$$
\end{proposition}

\begin{proof}
Obvious.
\end{proof}

We now investigate the relation of this notion to other notions of equivalence of augmentations
which have been introduced in the literature.
We will consider three notions of equivalence, of which (\ref{it:augBC})
and (\ref{it:dga}) will be defined below:
\begin{enumerate}
\item
isomorphism in $\Aug$; \label{it:aug}
\item
isomorphism in $\mathcal{Y} \AugBC$; \label{it:augBC}
\item
\dga{} homotopy. \label{it:dga}
\end{enumerate}
We will see that (\ref{it:aug}) implies (\ref{it:augBC}), and that
(\ref{it:aug}) and (\ref{it:dga}) are equivalent if $\Lambda$ is connected with a single base point; we do not know if
(\ref{it:augBC}) implies (\ref{it:aug}). A fourth notion of
equivalence, involving exponentials and necessitating that we work over a field of characteristic $0$, usually $\bR$
(see \cite{Bourgeois,BC}), is not addressed here. 

Note that (\ref{it:dga}) has been shown to be closely related to
isotopy of Lagrangians in the case where the augmentations come from
exact Lagrangian fillings: see \cite{EHK} and
Corollary~\ref{cor:filling-equiv} below.

\subsubsection{Isomorphism in $\mathcal{Y} \AugBC$}
In \cite{BC}, equivalence was defined using $\AugBC$ as follows.
While $\AugBC$ is not unital, the category of $\AugBC$-modules (functors to chain complexes)
is, and the Yoneda
construction $\epsilon \mapsto \homBC(\, \cdot \, , \epsilon)$ gives a morphism $\mathcal{Y}: \AugBC \to
{\AugBC}$-modules.
This morphism is cohomologically faithful but not cohomologically full since $\AugBC$ is non-unital.
We write $\mathcal{Y} \AugBC$ for the
full subcategory on the image objects.
In any case, \cite{BC} defined two augmentations to be equivalent if their images in $\mathcal{Y} \AugBC$ are isomorphic.
As noted in \cite[Theorem~1.4]{BC}, essentially by definition, if
$\mathcal{Y} \epsilon_1 \cong \mathcal{Y} \epsilon_2$ in $\mathcal{Y}
\AugBC$, then
\[
\LCHBC{*}(\epsilon_1,\epsilon_3) \cong \LCHBC{*}(\epsilon_2,\epsilon_3)
\text{ and }
\LCHBC{*}(\epsilon_3,\epsilon_1) \cong \LCHBC{*}(\epsilon_3,\epsilon_2)
\]
for any augmentation $\epsilon_3$.

\begin{proposition}
If $\epsilon_1 \cong \epsilon_2$ in $\Aug$, then
\label{prop:isom-pm}
$\mathcal{Y} \epsilon_1 \cong \mathcal{Y} \epsilon_2$ in $\mathcal{Y} \AugBC$.
\end{proposition}
\begin{proof}
According to Proposition \ref{prop:bifunctor}, we have a map 
\begin{eqnarray*} \mathcal{Y}_-: \Aug & \to & \mathcal{Y} \AugBC \\
\epsilon & \mapsto & \homBC(\, \cdot \, , \epsilon).
\end{eqnarray*}
The fact that the identity in $\hom$ acts trivially on the space $\homBC$ under the morphisms in
Proposition \ref{prop:bifunctor} implies that this is a unital morphism of categories.  It follows that
the image of an isomorphism in $\Aug$ is an isomorphism in $\mathcal{Y} \AugBC$.
\end{proof}

\begin{corollary}
If $\epsilon_1 \cong \epsilon_2$ in $\Aug$, then
\[
\LCHBC{*}(\epsilon_1,\epsilon_3) \cong \LCHBC{*}(\epsilon_2,\epsilon_3)
\text{ and }
\LCHBC{*}(\epsilon_3,\epsilon_1) \cong \LCHBC{*}(\epsilon_3,\epsilon_2)
\]
for any augmentation $\epsilon_3$.
\end{corollary}

\subsubsection{\dga{} homotopy}
Another notion of equivalence that has appeared in the literature is
\dga{} homotopy \cite{Kal05, Henry2011, EHK,HR14}.
This arises from viewing augmentations as \dga{} maps from
$(\alg, \partial)$ to $(\coeffs, 0)$ and considering an appropriate
version of chain homotopy for \dga{} maps.

\begin{definition}
Two \dga{} maps $f_1,f_2: (\alg_1, \partial_1) \rightarrow (\alg_2, \partial_2)$ are \emph{\dga{} homotopic} if they are chain homotopic via a chain homotopy operator $K: \alg_1 \rightarrow \alg_2$ which is an $(f_1,f_2)$-derivation.  This means that
\begin{itemize}
\item $K$ has degree $+1$,
\item $f_1-f_2 = \partial_2 K +K \partial_1$, and
\item $K(x\cdot y) = K(x) \cdot f_2(y) + (-1)^{|x|} f_1(x)\cdot K(y)$ for all $x,y \in \alg_1$.
\end{itemize}
\end{definition}
Note that if $K$ is an $(f_1,f_2)$-derivation and $f_1$ and $f_2$ are \dga{} maps, then it suffices to check the second condition on a generating set for $\alg_1$.  In addition, if $\alg_1$ is freely generated by $a_1, \ldots, a_k$, then, once $f_1$ and $f_2$ are fixed, any choice of values $K(a_i) \in \alg_2$ extends uniquely to an $(f_1,f_2)$-derivation.  Although it is not immediate, \dga{} homotopy is an equivalence relation on the set of \dga{} morphisms from $(\alg_1, \partial_1)$ to $(\alg_2, \partial_2)$ (see e.g. \cite[Chapter 26]{FHT}).

We will show that if $\Lambda$ is a Legendrian knot with a single base point, then
two augmentations of $\alg(\Lambda)$ are isomorphic in $\Aug(\Lambda,\coeffs)$ if and
only if they are \dga{} homotopic as \dga{} maps to $(\coeffs,0)$. To do
this, we compute with $m$-copies constructed from
the Lagrangian projection as described in Section \ref{sec:aug-knot}.
For any $\e_1,\e_2 \in \Aug$, $\hom(\e_1,\e_2)$ is spanned as an $\coeffs$-module by elements $a_i^+, x^+,y^+$.  The $a_i^+$ are dual to the crossings $a^{12}_i$ of the $2$-copy, which are in bijection with the generators $a_1, \ldots, a_r$ of $\alg(\Lambda)$, while $x^+$ and $y^+$  are dual to the crossings $x^{12}$ and $y^{12}$ that arise from the perturbation process.

The definition of $\Aug$ together with the description of the differential in $\alg^m = \alg(\Lambda^m_f)$ from Proposition \ref{prop:complete-aug-description} lead to the following formulas.
\begin{lemma}  \label{lem:m1Homotopy}
In $\hom(\e_1,\e_2)$, we have
\begin{align*}
m_1 (a_i^+) &= \sum_{a_j, b_1, \ldots, b_n} \sum_{u \in \Delta(a_j; b_1, \ldots, b_n)} \sum_{1 \leq l \leq n} \delta_{b_l, a_i} \sigma_u \e_1(b_1\cdots b_{l-1}) \e_2(b_{l+1}\cdots b_n) a_j^+; \\
m_1 (y^+) &= (\epsilon_1(t)^{-1} \epsilon_2(t) - 1) x^+ + \sum_{i} \big(\e_2(a_i) - (-1)^{|a_i|}\e_1(a_i)\big)a_i^+; \\
m_1 (x^+) &\in \operatorname{Span}_\coeffs \{a_1^+, \ldots, a_r^+\}.
\end{align*}
Here we abuse notation slightly to allow the 
$b_i$ to include the base point on $\Lambda$ as well as the corresponding generators $t^{\pm 1}$.  The factor $\sigma_u \in \{\pm 1\}$ denotes the product of all orientation signs at the corners of the disk $u$, i.e. the coefficient of the monomial $w(u)$ (see Section \ref{ssec:dga-background}).
\end{lemma}
We remark that if $\e_1$ and $\e_2$ are homotopic via the operator $K$, then we have $\e_1(t) - \e_2(t) = \dd_{\coeffs} K(t) + K(\dd t) = 0$ and so $\e_1(t) = \e_2(t)$.

We will also need the following properties of composition in $\Aug$.

\begin{lemma} \label{lem:m2Homotopy} 
Assume that the crossings $a_1, \ldots, a_r$ of the $xy$-projection of $\Lambda$ are labeled with increasing height, $h(a_1) \leq h(a_{2}) \leq \ldots \leq h(a_r)$.

For any $\e_1,\e_2,\e_3$, the composition $m_2: \hom(\e_2,\e_3) \otimes \hom(\e_1,\e_2) \rightarrow \hom(\e_1, \e_3)$ satisfies the following properties.
\begin{itemize}
\item $m_2(a_i^+, a_j^+)  \in \operatorname{Span}_{\coeffs} \left\{ a_l^+ \mid l \geq \max(i,j)\right\}$ for all $i$ and $j$, $1 \leq i,j \leq r$.
\item Each of $m_2(x^+, a_i^+)$, $m_2(a_i^+,x^+)$, and $m_2(x^+,x^+)$ belongs to $\operatorname{Span}_{\coeffs}\{a_l^+ \mid 1 \leq l \leq r\}$ for $1 \leq i \leq r$.
\item For any $\alpha \in \operatorname{Span}_{\coeffs}\{a_1^+, \ldots, a_r^+, x^+,y^+\}$, we have
\[
m_2(y^+, \alpha) = m_2(\alpha, y^+) = -\alpha.
\]
\end{itemize}
\end{lemma}

\begin{proposition}  \label{prop:HomotopyCo}
Consider an element $\alpha \in \hom^0(\e_1,\e_2)$ of the form
\[
\alpha = -y^+ - \sum_i K(a_i) a_i^+.
\]
Then $m_1(\alpha)=0$ if and only if the extension of $K$ to an $(\e_1,\e_2)$-derivation, $\widetilde{K}: \alg \rightarrow \coeffs$, is a \dga{} homotopy from $\e_1$ to $\e_2$.
\end{proposition}

\begin{proof}
We note that $\e_1(a_i) = (-1)^{|a_i|}\e_1(a_i)$ for all $i$, since $\e_1$ is supported in grading $0$.  Using Lemma \ref{lem:m1Homotopy}, we compute
\begin{align*}
-m_1(\alpha) &= m_1 (y^+) + \sum_{i} K(a_i) m_1(a_i^+) \\
&= \sum_j\left[\e_2(a_j) - \e_1(a_j)\right]a_j^+ \\
 &\qquad + \sum_{i} K(a_i) \left(\sum_{a_j,b_1,\dots,b_n}\sum_{u \in \Delta(a_j; b_1, \ldots, b_n)} \sum_{\substack{1\leq l\leq n \\ b_l=a_i}} \sigma_u \e_1(b_1\cdots b_{l-1}) \e_2(b_{l+1}\cdots b_n) a_j^+ \right) \\
&= \sum_j\left[\e_2(a_j) - \e_1(a_j)\right]a_j^+ \\
 &\qquad + \sum_j \left(\sum_{\substack{b_1,\dots,b_n \\ u\in\Delta(a_j;b_1,\dots,b_n)}} \sum_{\substack{1 \leq l \leq n \\ b_l \neq t^{\pm 1}}} (-1)^{|b_1 \cdots b_{l-1}|} \sigma_u \e_1(b_1\cdots b_{l-1}) K(b_l) \e_2(b_{l+1}\cdots b_n)\right) a_j^+ \\
&= \sum_j \left[\e_2(a_j) - \e_1(a_j) + \widetilde{K} \circ \partial (a_j)\right] a_j^+
\end{align*}
where $\widetilde{K}$ denotes the unique $(\e_1,\e_2)$-derivation with $\widetilde{K}(a_j) = K(a_j)$.  (The innermost sum above is equal to $\widetilde{K}(\dd a_j)$ only once we also include the terms where $b_l = t^{\pm 1}$, but $K(t^{\pm 1}) = 0$ since it must be an element of $\coeffs$ with grading 1, so this does not change anything.)  Therefore, $m_1(\alpha)=0$ if and only if the equation
\[
\e_1 - \e_2 = \widetilde{K} \circ \partial
\]
holds when applied to generators, and the proposition follows.
\end{proof}

We can now state our result relating notions of equivalence.

\begin{proposition}  \label{prop:Homotopy} If $\Lambda$ is a knot with
  a single base point, then two augmentations $\e_1,\e_2:\alg(\Lambda)
  \rightarrow \coeffs$ are homotopic as \dga{} maps if and only if they
  are isomorphic in 
$\Aug(\Lambda)$.
\end{proposition}

\begin{proof} 
First, suppose that $\e_1$ and $\e_2$ are isomorphic in $H^*\Aug$.  In particular, there exist cocycles $\alpha \in \hom^0(\e_1,\e_2)$ and $\beta \in \hom^0(\e_2,\e_1)$ with $[m_2(\alpha,\beta)]= -[y^+]$ in $H^0\hom(\e_2,\e_2)$.
That is, $m_2(\alpha,\beta) + y^+ = m_1(\gamma)$ for some $\gamma \in
\hom(\e_2,\e_2)$.  Using Lemma \ref{lem:m1Homotopy}, we see that
$\langle m_1(\gamma), y^+\rangle= \langle m_1(\gamma), x^+\rangle =
0$, 
where $\langle m_1(\gamma), y^+\rangle$ denotes the coefficient of
$y^+$ in $m_1(\gamma)$ and so forth. Thus we can write
\[ m_2(\alpha, \beta) = -y^+ + \sum_{i} K(a_i)a_i^+ \]
for some $K(a_i) \in \coeffs$. 
(To see that $\langle m_1(\gamma), x^+\rangle =0$, we used the fact that we are working in $\hom(\e_2,\e_2)$, hence $\langle m_1(y^+), x^+\rangle = 0$.)  Moreover, Lemma \ref{lem:m2Homotopy}  shows that both $\alpha$ and $\beta$ must also have this same form, except that the $y^+$ coefficients need not be $-1$: we have $\langle \alpha, y^+\rangle = c_\alpha$ and $\langle \alpha, y^+\rangle = c_\beta$ for some $c_\alpha,c_\beta \in \coeffs^\times$ with $c_\alpha c_\beta = 1$.  (Note that $\langle\alpha,x^+\rangle = \langle\beta,x^+\rangle = 0$ because $\alpha$ and $\beta$ are both elements of $\hom^0$, whereas $|x^+|=1$.)  Replacing $\alpha$ and $\beta$ with $-c_\alpha^{-1}\alpha$ and $-c_\beta^{-1}\beta$ respectively preserves $m_2(\alpha,\beta)$ and $m_1(\alpha)=m_1(\beta)=0$, so we can assume that both $\alpha$ and $\beta$ have $y^+$-coefficient equal to $-1$ after all.  Now, since $m_1(\alpha)=0$,  Proposition \ref{prop:HomotopyCo} applies to show that $\e_1$ and $\e_2$ are homotopic.

Conversely, suppose that $\e_1$ and $\e_2$ are homotopic, with $K: \alg \rightarrow \coeffs$ an $(\e_1,\e_2)$-derivation with $\e_1-\e_2 = K \circ \partial$.  Note that since $\coeffs$  sits in grading $0$ when viewing $(\coeffs,0)$ as a \dga{}, we have $K(a_i) = 0$ unless $|a_i|=-1$ in $\alg$.  As $|a_i^+| = |a_i|+1$, it follows that
\[
\alpha = -y^+ - \sum_{i} K(a_i) a_i^+
\]
defines a cocycle in $\hom^0(\e_1,\e_2)$ by Proposition \ref{prop:HomotopyCo}.  We show that $[\alpha] \in H^0\hom(\e_1,\e_2)$ has a multiplicative inverse in $H^0\hom(\e_2,\e_1)$.  In fact, we prove a stronger statement by showing that there are elements $\beta, \gamma \in \hom^0(\e_2,\e_1)$ satisfying
\begin{equation} \label{eq:m1betagamma}
m_1(\beta) = m_1(\gamma) = 0
\end{equation}
and
\begin{equation} \label{eq:m2betagamma}
m_2(\beta, \alpha) = m_2(\alpha, \gamma) = -y^+.
\end{equation}
It will then follow that $[\beta] = [\gamma] \in H^0\hom(\e_2,\e_1)$ is the desired multiplicative inverse.  (It is not clear whether $\beta = \gamma$ as cochains, since the $m_2$ operations may not be associative if $m_3$ is nontrivial.)  We will construct $\beta$ of the form
\[
\beta = -y^+ + \sum_{i} B_i a_i^+,
\]
and omit the construction of $\gamma$ which is similar.

Writing $\alpha = -y^+ - A$ and $\beta = -y^+ + B$ with $A, B \in \operatorname{Span}_\coeffs\{a_1^+, \ldots, a_r^+\}$, we note, using Lemma \ref{lem:m2Homotopy}, that $m_2(\beta, \alpha) = -y^+$ is equivalent to
\[
B = A + m_2(B,A).
\]
The coefficients $B_i$ can then be defined inductively to satisfy this property.  Indeed, assuming $a_1, \ldots, a_r$ are labeled according to height, Lemma \ref{lem:m2Homotopy} shows that the coefficient of $a_i^+$ in $m_2(B,A)$ is determined by $A$ and those $B_j$ with $j<i$.

Now that we have found $\beta = -y^+ + B$ satisfying (\ref{eq:m2betagamma}), we verify (\ref{eq:m1betagamma}).  The $A_\infty$ relations on $\Aug(\Lambda)$ imply that
\[ m_1(-y^+) = m_1(m_2(\beta, \alpha)) = m_2(m_1(\beta), \alpha) + m_2(\beta, m_1(\alpha)), \]
and the left side is zero since we evaluate $m_1(-y^+)$ in $\hom(\e_1,\e_1)$, while the term $m_2(\beta,m_1(\alpha))$ on the right side is zero since $m_1(\alpha)=0$; hence
\[ m_2(m_1(\beta),\alpha) = 0. \]
We claim that $m_2(X, \alpha)=0$ implies that $X = 0$ for any $X \in \operatorname{Span}_{\coeffs}\{y^+,x^+, a_1^+, \ldots, a_r^+\}$; in the case $X=m_1(\beta)$, it will immediately follow that $m_1(\beta)=0$ as desired.
Using Lemma \ref{lem:m2Homotopy}, we have that $m_2(X,A) \in \operatorname{Span}_\coeffs\{a_1^+,\ldots,a_r^+\}$, so
\[
0 = m_2(X,\alpha) = m_2(X, -y^+ - A) = X - m_2(X,A),
\]
which implies that $X  = m_2(X,A) \in \operatorname{Span}_\coeffs\{a_1^+,\dots,a_n^+\}$ as well.  That $\langle X, a^+_i\rangle  =  0$ is then verified from the same equation using Lemma \ref{lem:m2Homotopy} and induction on height.
\end{proof}

\begin{corollary}
Let $L_1,L_2$ be exact Lagrangian fillings of a Legendrian knot
$\Lambda$ with trivial Maslov number, and let
$\epsilon_{L_1},\epsilon_{L_2}$ be the corresponding augmentations of
the \dga{} of $\Lambda$.
\label{cor:filling-equiv}
If $L_1,L_2$ are isotopic through exact
Lagrangian fillings, then $\epsilon_{L_1} \cong \epsilon_{L_2}$  in $\Aug(\Lambda)$.
\end{corollary}

\begin{proof}
From \cite{EHK}, given these hypotheses,
$\epsilon_{L_1},\epsilon_{L_2}$ are \dga{} homotopic.
\end{proof}

\begin{remark}
As before, we can generalize Corollary~\ref{cor:filling-equiv} to
exact fillings of Maslov number $m$ as long as we consider $\Aug$ to
be $(\bZ/m)$-graded rather than $\bZ$-graded.
\end{remark}

\newpage

\section{Localization of the augmentation category}
\label{sec:bordered}

A preferred plat diagram of a Legendrian knot in $\R^3$ can be split along vertical lines which avoid the crossings, cusps,
and base points into a sequence of ``bordered'' plats.
Each of these bordered plats
was assigned a \dga{} in \cite{sivek-bordered}, generalizing the Chekanov--Eliashberg construction.\footnote{
In \cite{sivek-bordered}, mod 2 coefficients were used, and the language of co-sheaves was avoided.  There the vertical lines
bounding a bordered plat diagram $T$ on the left and right are assigned ``line algebras'' $I^L_n$ and $I^R_n$, and
 a ``type DA'' algebra $\alg(T)$ was associated to the oriented tangle $T$ along with natural \dga{} morphisms $I^L_n \to \alg(T)$
and $I^R_n \to \alg(T)$.  If $T$ decomposes into two smaller bordered plats as $T = T_1 \cup T_2$, with the two diagrams
intersecting along a single vertical line with $n$ points whose algebra is denoted $I^M_n$, then these morphisms fit into a pushout square
\begin{equation}
\label{eq:bordered-pushout}
\xymatrix{
I^M_n \ar[r] \ar[d] & \alg(T_2) \ar[d] \\
\alg(T_1) \ar[r] & \alg(T)
}
\end{equation}
and the corresponding morphisms $I^L_n \to \alg(T)$ and $I^R_n \to \alg(T)$ corresponding to the left and right boundary
lines of $T$ factor through the morphisms $I^L_n \to \alg(T_1)$ and $I^R_n \to \alg(T_2)$ respectively.

In the present treatment, the line algebras are the co-sections over a neighborhood of a boundary of the interval in question,
the \dga{} morphisms above are co-restrictions, and the pushout square reflects the cosheaf axiom.
}
Here we generalize the ideas of \cite{sivek-bordered} to yield the following result.

\begin{theorem} \label{thm:cosheaf}
Let $\Lambda \subset J^1(\R)$ be in preferred plat position.
Then there is a constructible co-sheaf of dg algebras
$\underline{\alg}(\Lambda)$
over $\R$ with global sections $\alg(\Lambda)$.

The sections
$\underline{\alg}(\Lambda)(U)$ are all semi-free, and we have a co-sheaf in the strict
sense that the underlying
graded algebras already form a co-sheaf.
\end{theorem}

We will prove this result over the course of this section, but first we interpret it and draw several corollaries.

The statement means the following.  For each open set $U \subset \R$, there is a \dga{} $\underline{\alg}(\Lambda)(U)$.
When $V \subset U$, there is a map (defined by counting disks) $\iota_{VU}: \underline{\alg}(\Lambda)(V) \to \underline{\alg}(\Lambda)(U)$.
For $W \subset V \subset U$, one has $ \iota_{VU}\iota_{WV} = \iota_{WU}$.
Finally, when
$U = L \cup_{V} R$, one has a pushout in the category of \dga{}
\[ \underline{\alg}(\Lambda)(U) = \underline{\alg}(\Lambda)(L) \star_{\underline{\alg}(\Lambda)(V)} \underline{\alg}(\Lambda)(R). \]

Cosheaves are determined by their behavior on any base of the topology; to prove the theorem it suffices to
give the sections and corestrictions
for open intervals to open intervals
and prove the cosheaf axiom for overlaps of intervals.  We give a new construction of these sections, which
is equivalent to that of \cite{sivek-bordered} if we restrict to mod 2 coefficients.

\begin{corollary} \label{cor:aug-set-sheaf}
Augmentations form a sheaf of sets over $\bR_x$.
That is, $U \mapsto \mathrm{Hom}_{\mathrm{\dga{}}}(\underline{\alg}(\Lambda)(U), \coeffs)$ determines a sheaf.
\end{corollary}
\begin{proof}
Given $U = L \cup_V R$, suppose we have augmentations $\epsilon_L: \underline{\alg}(\Lambda)(L) \to \coeffs$ and $\epsilon_R: \underline{\alg}(\Lambda)(R) \to \coeffs$ such that $\epsilon_L|_V = \epsilon_L \circ \iota_{VL}$ equals $\epsilon_R|_V = \epsilon_R \circ \iota_{VR}$ as augmentations of $\underline{\alg}(\Lambda)(V)$.  By the pushout axiom above, there is a unique $\epsilon: \underline{\alg}(\Lambda)(U) \to \coeffs$ such that $\epsilon_L = \epsilon \circ \iota_{LU} = \epsilon|_L$ and $\epsilon_R = \epsilon \circ \iota_{RU} = \epsilon|_R$, verifying the gluing axiom.
\end{proof}

\begin{corollary} \label{cor:a-infty-alg-sheaf}
Fix a global augmentation $\epsilon: \alg(\Lambda) \to \coeffs$.  This induces local augmentations
$\epsilon|_U: \underline{\alg}(\Lambda)(U) \to \coeffs$, which determine $A_\infty$ algebras
$C^{\epsilon}(U)$.  The co-restriction maps of the \dga{} determine restriction maps of the $C^{\epsilon}(U)$,
and the association $U \to C^{\epsilon}(U)$ is a sheaf of $A_\infty$ algebras.
\end{corollary}
\begin{proof}
This follows formally from Theorem \ref{thm:cosheaf} and Proposition \ref{prop:bar}.  Note the statement is asserting
the existence of $A_\infty$ restriction morphisms and $A_\infty$ pushouts.
\end{proof}

One of our definitions of the augmentation category of $\Lambda$ proceeded by first forming a front projection $m$-copy $\Lambda^m$
and then using the corresponding $C^{\epsilon}$ to define and compose homs.  Exactly the same construction can be made for a restriction $\Lambda|_{J^1(U)}$ for $U \subset \R$.  We note that since we work with the front projection rather than the Lagrangian projection, we do not have to choose perturbations of $\Lambda|_{J^1(U)}$.  We also do not require $\Lambda|_{J^1(U)}$ to contain any base points.

\begin{corollary} \label{cor:aug-sheaf}
There exists a presheaf of $A_\infty$ categories $\underline{\Aug}(\Lambda, \coeffs)$ with global sections
$\Aug(\Lambda, \coeffs)$, given by sending $U$ to the augmentation category of $\Lambda|_{J^1(U)}$.  Denoting
by $\underline{\Aug}(\Lambda, \coeffs)^\sim$ its sheafification, the map
$\underline{\Aug}(\Lambda, \coeffs)(U) \to \underline{\Aug}(\Lambda, \coeffs)^\sim(U)$ is fully faithful for all $U$.
\end{corollary}
\begin{proof}
This follows formally from Corollary~\ref{cor:aug-set-sheaf} and from applying Corollary~\ref{cor:a-infty-alg-sheaf} to the front projection $m$-copy for each $m$.
\end{proof}

\begin{remark} \label{rem:notsheaf} The presheaf of categories $\underline{\Aug}(\Lambda, \coeffs)$ need {\em not} be a sheaf of categories.  That is,
the map $\underline{\Aug}(\Lambda, \coeffs)(U) \to \underline{\Aug}(\Lambda, \coeffs)^\sim(U)$ need not be essentially surjective.
In fact, it never will be unless $\Lambda$ carries enough base points.
This may seem strange given Corollary \ref{cor:aug-set-sheaf}, but the point is that objects of a fibre product of categories $\mathcal{B} \times_\mathcal{C} \mathcal{D}$
are not the fibre product of the sets of objects, i.e. not pairs $(b, d)$ such that $b|_\mathcal{C} = d|_\mathcal{C}$, but instead triples $(b, d, \phi)$ where
$\phi: b|_\mathcal{C} \cong d|_\mathcal{C}$ is an isomorphism in $\mathcal{C}$.  The objects of the more naive product, where $\phi$ is required to be the identity,
will suffice under the condition that the map $\mathcal{B} \to \mathcal{C}$ has the ``isomorphism lifting property,'' i.e., that any isomorphism $\phi(b) \sim c'$ in $\mathcal{C}$
lifts to an isomorphism $b \sim b'$ in $\mathcal{B}$.  We will ultimately show that this holds for restriction to the left when $\Lambda$ has base points at all the right cusps, and
conclude in this case that $\underline{\Aug}(\Lambda, \coeffs)$ is a sheaf.
\end{remark}

We now turn to proving Theorem \ref{thm:cosheaf}.
Let $U \subset \R$ be an open interval, and $T \subset J^1(U)$ be a Legendrian tangle transverse to $\partial J^1(U)$.
We will assume that $T$ is oriented, that its front projection is generic and
equipped with a Maslov potential $\mu$ such that two strands are oriented in the same
direction as they cross $\partial J^1(U)$ if and only if their Maslov potentials agree mod 2.
Suppose that $T$ also has $k\geq 0$ base points, labeled $*_{\alpha_1},*_{\alpha_2},\dots,*_{\alpha_k}$ for distinct positive integers $\alpha_j$.

We require that any right cusps in $T$ abut the unbounded region of $T$ containing all points with $z \ll 0$, which can be arranged by Reidemeister 2 moves, but which will certainly be the case if $T$ comes from a preferred plat.  We will let $n_L$ and $n_R$ denote the number of endpoints on the left and right sides of $T$, respectively.

\begin{definition}
\label{def:tangle-algebra}
The graded algebra $\alg(T)$ is freely generated over $\zz$ by the following elements:
\begin{itemize}
\item Pairs of left endpoints, denoted $a_{ij}$ for $1 \leq i < j \leq n_L$.
\item Crossings and right cusps of $T$.
\item A pair of elements $t_{\alpha_j}^{\pm 1}$ for each $j$, with $t_{\alpha_j}\cdot t_{\alpha_j}^{-1} = t_{\alpha_j}^{-1} \cdot t_{\alpha_j} = 1$.
\end{itemize}
These have gradings $|c| = \mu(s_{\mathrm{over}})-\mu(s_{\mathrm{under}})$ for crossings, $1$ for right cusps, and $0$ for $t_{\alpha_{j}}^{\pm 1}$, 
and $|a_{ij}| = \mu(i)-\mu(j)-1$.
We take the Maslov potential $\mu$ to be $\bZ/2r$-valued for some integer $r$, which may be zero; if $T$ comes from a Legendrian link $\Lambda$, as in Theorem \ref{thm:cosheaf}, then we will generally take $r$ to be the gcd of the rotation numbers of the components of $\Lambda$.

The differential $\dd$ is given on the $t_{\alpha_j}$ by  $\dd(t_{\alpha_j}^{\pm 1}) = 0$  and on the $a_{ij}$ by
\[ \dd a_{ij} = \sum_{i<k<j} (-1)^{|a_{ik}|+1}a_{ik}a_{kj}. \]
For crossings and right cusps, we define $\dd c$ in terms of the set $\Delta(c;b_1,\dots,b_l)$, which consists of embeddings
\[ u: (D_l^2,\partial D_l^2) \to (\R^2, \pi_{xz}(T)) \]
of a boundary-punctured disk $D^2_l = D^2 \smallsetminus \{p,q_1,\dots,q_l\}$ up to reparametrization.  These maps must satisfy $u(p) = c$; $u(q_i)$ is a crossing for each $i$, except that we can also allow the image $u(D_l^2)$ to limit to the segment $[i,j]$ of the left boundary of $T$ between points $i<j$ at a single puncture $q_k$; and the $x$-coordinate on $\overline{u(D_l^2)}$ has a unique local maximum at $c$ and local minima precisely along $[i,j]$ if it occurs in the image, or at a single left cusp otherwise.  For each such embedding we define $w(u)$ to be the product, in counterclockwise order from $c$ along the boundary of $\overline{u(D_l^2)}$, of the following terms:
\begin{itemize}
\item $c_j$ or $(-1)^{|c_j|+1}c_j$ at a corner $c_j$, depending on whether the disk occupies the top or bottom quadrant near $c_j$;
\item $t_j$ or $t_j^{-1}$ at a base point $\ast_j$, depending on whether the orientation of $u(\partial D_l^2)$ agrees or disagrees with that of $T$ near $\ast_j$;
\item $a_{ij}$ at the segment $[i,j]$ of the left boundary of $T$.
\end{itemize}
We then define $\dd c = \sum_{u} w(u)$, and note that if $c$ is a right cusp then this also includes an ``invisible'' disk $u$ with $w(u) = 1$ or $t_j^{\pm 1}$ depending on whether there is a base point $\ast_j$ at the cusp.
\end{definition}

We remark that the differential $\dd$ on $\alg(T)$ is defined exactly as in the usual link \dga{} from Section~\ref{ssec:dga-background}, except that we enlarge the collection $\Delta(c;b_1,\dots,b_l)$ of disks by also allowing the $x$-coordinate of a disk to have local minima along some segment $[i,j]$ of the left boundary of $T$, in which case it contributes a factor of $a_{ij}$, rather than at a left cusp.

\begin{example}
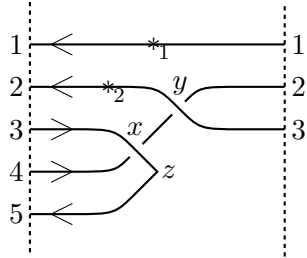
\begin{figure}[ht]
\labellist
\small \hair 2pt
\pinlabel $1$ [r] at 0 80
\pinlabel $2$ [r] at 0 64
\pinlabel $3$ [r] at 0 48
\pinlabel $4$ [r] at 0 32
\pinlabel $5$ [r] at 0 16
\pinlabel $1$ [l] at 97 80
\pinlabel $2$ [l] at 97 64
\pinlabel $3$ [l] at 97 48
\pinlabel $x$ [b] at 40 44
\pinlabel $y$ [b] at 57 59
\pinlabel $z$ [l] at 48 32
\pinlabel $*$ at 47 80
\pinlabel $*$ at 30 64
\tiny
\pinlabel $1$ at 51 77
\pinlabel $2$ at 34 61
\endlabellist
\begin{centering}
\includegraphics{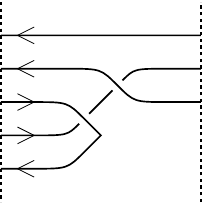}
\end{centering}
\caption{An example of a bordered front $T$.}
\label{fig:bordered-tangle}
\end{figure}
The oriented front $T$ in Figure~\ref{fig:bordered-tangle} has differential
\begin{align*}
\dd x &= a_{34} &
\dd a_{13} &= a_{12}a_{24} &
\dd a_{14} &= a_{12}a_{24} - a_{13}a_{34} \\
\dd y &= t_2a_{24} + t_2a_{23}x &
\dd a_{24} &= -a_{23}a_{34} &
\dd a_{25} &= -a_{23}a_{35} - a_{24}a_{45} \\
\dd z &= 1 + a_{35} - xa_{45} &
\dd a_{35} &= a_{34}a_{45} &
\dd a_{15} &= a_{12}a_{25} - a_{13}a_{35} - a_{14}a_{45}
\end{align*}
and $\dd a_{i,i+1} = 0$ for $1 \leq i \leq 4$.  Note that the orientation suffices to determine the signs, since $(-1)^{|c|+1} = -1$ (resp.\ $(-1)^{|a_{ij}|+1} = 1$) if and only if both strands through $c$ (resp.\ through points $i$ and $j$ on the left boundary of $T$) have the same orientation from left to right or vice versa.
\end{example}

\begin{proposition}
\label{prop:bordered-differential}
The differential $\dd$ on $\alg(T)$ has degree $-1$ and satisfies $\dd^2 = 0$.
\end{proposition}

\begin{proof}
The claim that $\deg(\dd)=-1$ is proved exactly as in \cite{sivek-bordered}.  In order to prove $\dd^2 = 0$, we will embed $T$ in a simple (in the sense of \cite{NgCLI}) front diagram for some closed, oriented Legendrian link $L$ so that $(\alg(T), \dd)$ is a sub-\dga{} of $(\alg(L), \dd)$, and then observe that we already know that $\dd^2 = 0$ in $\alg(L)$.

\begin{figure}[ht]
\labellist
\small \hair 2pt
\pinlabel $1$ [r] at 133 236
\pinlabel $2$ [r] at 133 220
\pinlabel $3$ [r] at 133 204
\pinlabel $4$ [r] at 133 188
\pinlabel $5$ [r] at 133 172
\pinlabel $1$ [r] at 229 236
\pinlabel $2$ [r] at 229 220
\pinlabel $3$ [r] at 229 204
\pinlabel $x$ [b] at 173 194
\pinlabel $y$ [b] at 189 209
\pinlabel $z$ [l] at 180 182
\pinlabel $*$ at 181 230.5
\pinlabel $*$ at 164 214.5
\pinlabel $\alpha_{12}$ [r] at 75 168
\pinlabel $\alpha_{23}$ [r] at 75 152
\pinlabel $\alpha_{34}$ [r] at 75 136
\pinlabel $\alpha_{45}$ [r] at 75 120
\pinlabel $\alpha_{15}$ [l] at 103 143
\pinlabel $\beta_1$ [r] at 249 246
\pinlabel $\beta_2$ [r] at 258 237
\pinlabel $\beta_3$ [r] at 267 229
\tiny
\pinlabel $1$ at 185 227
\pinlabel $2$ at 168 211
\endlabellist
\begin{centering}
\includegraphics{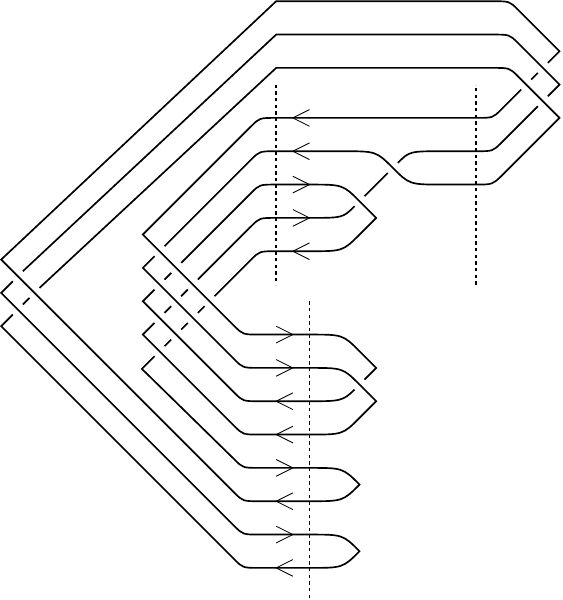}
\end{centering}
\caption{Embedding the bordered front $T$ in a simple front diagram of a closed link.}
\label{fig:bordered-embedding}
\end{figure}
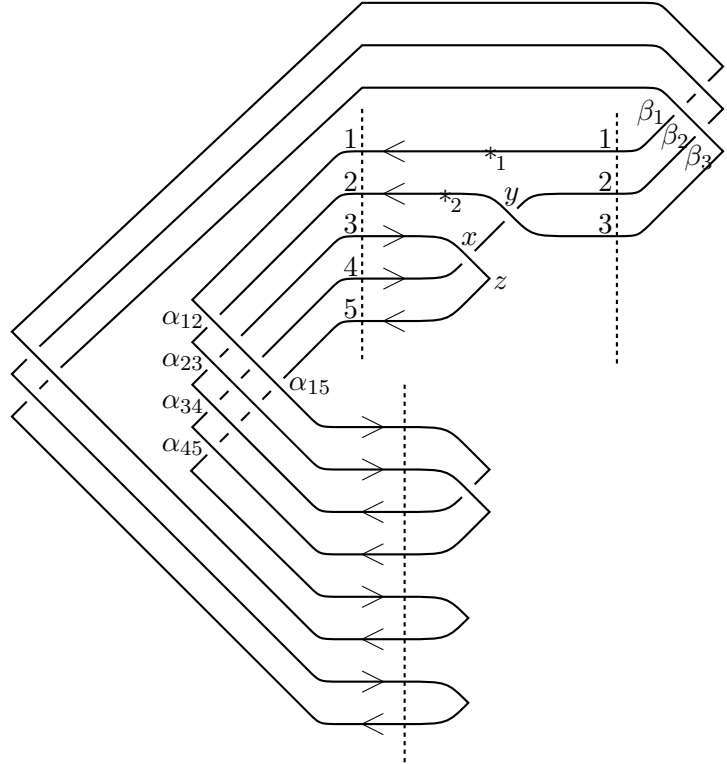

Figure~\ref{fig:bordered-embedding} illustrates the construction of $L$. We glue the $n_L$-copy of a left cusp to the left edge of $T$, attaching the top $n_L$ endpoints to $T$, and similarly we glue the $n_R$-copy of a right cusp to the right edge of $T$ along the bottom $n_R$ endpoints.
We then attach the $n_R$-copy of a left cusp, placed to the left of this diagram, by gluing its top $n_R$ endpoints to those of the $n_R$-copy of the right cusp, as shown in the figure; the resulting tangle diagram has $n_L+n_R$ points on its boundary, which is represented as the dotted line at the bottom, and it is an easy exercise to check that the tangle is oriented to the left at as many endpoints as to the right.  Thus we can add some crossings and right cusps to the tangle in any way at all, as long as they intersect the tangle diagram exactly at its endpoints and the resulting link diagram is simple,
to produce the desired front for the link $L$.  Since $T$ embeds in $L$ as an oriented tangle, its Maslov potential $\mu$ mod 2 extends to a potential $\tilde{\mu}$ on the front diagram for $L$.

The $n_L$-copy of the left cusp which was glued to the left end of $T$ has ${n_L \choose 2}$ crossings; we will let $\alpha_{ij}$ denote the crossing between the strands connected to points $i$ and $j$ on the left boundary of $T$.  Then $|\alpha_{ij}| = \tilde{\mu}(i)-\tilde{\mu}(j)-1$ since the overcrossing strand has potential $\tilde{\mu}(i)-1$, so $|\alpha_{ij}| \equiv |a_{ij}|\pmod{2}$ and thus we verify that
\[ \dd \alpha_{ij} = \sum_{i<k<j} (-1)^{|a_{ik}|+1} \alpha_{ik} \alpha_{kj}. \]
Moreover, given a right cusp or crossing $c$ of $T$, any $u\in \Delta(c;b_1,\dots,b_l)$ which intersected this left boundary between points $i$ and $j$ now extends in $L$ to a unique disk with the same corners as before, except that the puncture along the dividing line is replaced by a corner filling the top quadrant at $\alpha_{ij}$.  Thus the differentials $\dd_{\alg(L)} \alpha_{ij}$ and $\dd_{\alg(L)} c$ are identical to $\dd_{\alg(T)}(a_{ij})$ and $\dd_{\alg(T)}(c)$, except that we have replaced each $a_{ij}$ with $\alpha_{ij}$, and this identifies $\alg(T)$ as a sub-\dga{} of $\alg(L)$ (after potentially reducing the gradings mod 2) as desired.
\end{proof}

\begin{remark} \label{def:line-algebra}
A particularly important special case occurs when $T$ contains no crossings, cusps, or base points at all; i.e., $T$ consists merely of $n$ horizontal strands.  The resulting algebra is termed the {\em line algebra}, and denoted $I_n$ or $I_n(\mu)$ to emphasize
the dependence of the grading on the Maslov potential.
It is generated freely over $\zz$ by ${n\choose 2}$ elements $a_{ij}$, $1 \leq i < j \leq n$, with grading $|a_{ij}| = \mu(i)-\mu(j)-1$ and differential
\[ \dd a_{ij} = \sum_{i < k < j} (-1)^{|a_{ik}|+1}a_{ik}a_{kj} = \sum_{i < k < j} (-1)^{\mu(i)-\mu(k)} a_{ik}a_{kj}. \]
\end{remark}

If $V \subset U$ is an open interval, $T|_V:= T|_{J^{1}(V)}$ retains the properties assumed above of $T$, and moreover inherits
a Maslov potential.  Thus there is an algebra $\alg(T|_V)$.  It admits maps to $\alg(T)$, as we explain:

\begin{lemma}\label{lem:left-sub-dga}
Let $V \subset U$ be an open interval extending to the left boundary of $U$.  Then $\alg(T|_V)$ is naturally a sub-\dga{} of $\alg(T)$.
\end{lemma}
\begin{proof}
The generators of $\alg(T|_V)$ are a subset of the generators of $\alg(T)$, and the differential only counts disks extending to the left, so
the differential of any generator of $\alg(T|_V)$ will be the same whether computed in $V$ or in $U$.
\end{proof}

In fact, there is a similar map for any subinterval.  It is defined as follows.  Let $V \subset U$ be a subinterval.  Then the map
$\iota_{VU}: \alg(T|_V) \to \alg(T)$ takes the generators in $\alg(T|_V)$ naming crossings, cusps, and base points in $T|_V$  to the corresponding generators
of $\alg(T)$.  The action on the pair-of-left-endpoint generators of $\alg(T|_V)$ -- denoted $b_{ij}$ to avoid confusion -- is however nontrivial:
$\iota_{VU}(b_{ij})$ counts disks extending from the left boundary of
$V$ to the left boundary of $U$, meeting the boundary of $V$ exactly along the interval named by $b_{ij}$.

More precisely, we define a set of disks $\Delta(b_{ij}; c_1,\dots,c_l)$ to consist of embeddings
\[ u: (D_l^2,\partial D_l^2) \to (\R^2, \pi_{xz}(T)) \]
which limit at the puncture $p\in\partial D^2$ to the segment of the left boundary of $V$ between points $i$ and $j$, and which otherwise satisfy the same conditions as the embeddings of disks used to define the differential on $\alg(T)$ for crossings.  We then define $\iota_{VU}$ for the generator $b_{ij}$ by
\[ \iota_{VU}(b_{ij}) = \sum_{u \in \Delta(b_{ij};c_1,\dots,c_l)} w(u). \]

\begin{example}
For the front in Figure~\ref{fig:bordered-tangle}, let $V \subset U$ denote a small open interval of the right endpoint of $U$, so that $T|_V$ has no crossings, cusps, or base points and $\alg(T|_V) = I_3$.  Then the morphism $\iota_{VU}: I_3 \to \alg(T)$ satisfies
\begin{align*}
\iota_{VU}(b_{12}) &= t_1a_{14} + t_1a_{13}x + t_1a_{12}t_2^{-1}y \\
\iota_{VU}(b_{13}) &= t_1a_{12}t_2^{-1} \\
\iota_{VU}(b_{23}) &= 0.
\end{align*}
\end{example}

\begin{proposition}
The above map $\iota_{VU}: \alg(T|_V) \to \alg(T)$ is a morphism of \dgas{}.
\end{proposition}

\begin{proof}
It is straightforward to check that $\iota_{VU}(b_{ij})$ has grading $|b_{ij}|$, exactly as in \cite{sivek-bordered}.  In order to prove that $\dd \circ \iota_{VU} = \iota_{VU} \circ \dd$ for each of the generators $b_{ij}$, we embed the leftmost region $T_L$ of $T|_{U\smallsetminus V}$ (i.e.\ everything to the left of $V$) in the closed link $L$ shown in Figure~\ref{fig:bordered-embedding}, realizing $\alg(T_L)$ as a sub-\dga{} of $\alg(L)$ just as in the proof of Proposition~\ref{prop:bordered-differential}.  Let $n_R$ denote the number of endpoints on the right side of $T_L$, or equivalently the number of left endpoints of $T|_V$.  We identify the generator $\beta_i$ of $\alg(L)$, $1 \leq i \leq n_R$, as the crossing or right cusp of $L$ immediately to the right of $T_L$ on the strand through point $i$.  We note for the sake of determining signs that
\[ |\beta_{i}| = (\tilde{\mu}(n_R)+1) - \tilde{\mu}(i) \equiv \mu(i)-\mu(n_R)-1 \pmod{2}, \]
hence $(-1)^{|\beta_i| - |\beta_k|} = (-1)^{\mu(i)-\mu(k)} = (-1)^{|b_{ik}|+1}$.

We will now show that $\dd (\iota_{VU}(b_{ij})) = \iota_{VU}(\dd b_{ij})$ follows from $\dd^2 \beta_j=0$ for each $i<j$.  We first compute
\[ \dd \beta_j = \delta_{j,n_R} + \sum_{k=1}^{j-1} (-1)^{|\beta_k|+1} \beta_k \iota_{VU}(b_{kj}), \]
and then applying $\dd$ again yields
\begin{align*}
\dd^2 \beta_j &= \sum_{k=1}^{j-1} (-1)^{|\beta_k|+1}\left[
\left(\sum_{i=1}^{k-1} (-1)^{|\beta_i|+1} \beta_i \iota_{VU}(b_{ik})\right)\iota_{VU}(b_{kj}) + (-1)^{|\beta_k|}\beta_k \dd (\iota_{VU}(b_{kj}))\right] \\
&= \sum_{i=1}^{j-2} \beta_i\left(\sum_{i<k<j} (-1)^{|\beta_i|-|\beta_k|} \iota_{VU}(b_{ik})\iota_{VU}(b_{kj})\right) - \sum_{i=1}^{j-1} \beta_i \dd (\iota_{VU}(b_{ij})).
\end{align*}
Since this sum vanishes, so does the coefficient of $\beta_i$, which is equal to $\iota_{VU}(\dd b_{ij}) - \dd (\iota_{VU}(b_{ij}))$.  We conclude that $\iota_{VU}\circ \dd = \dd \circ \iota_{VU}$ on the subalgebra generated by the $b_{ij}$.

It remains to be seen that $\iota_{VU}(\partial c) = \partial(\iota_{VU}(c))$, where $c \in \alg(T|_V)$ is a generator corresponding to a crossing or right cusp of $T|_V$.  But we compute $\iota_{VU}(\partial c)$ by taking all of the appropriate embedded disks in $T|_V$, some of which may limit at punctures to the segment of the left boundary of $T|_V$ between strands $i$ and $j$, and replacing the corresponding $b_{ij}$ with the expression $\iota_{VU}(b_{ij})$.  The resulting expressions coming from all terms of $\partial c$ with a $b_{ij}$ factor count all of the embedded disks $u \in \Delta(c;b_1,\dots,b_l)$ in $T$ which cross the left end of $T|_V$ along the interval between strands $i$ and $j$.  Summing over all $i$ and $j$, as well as the terms with no $b_{ij}$ factor corresponding to disks in $T|_V$ which never reach the left end of $T|_V$, we see that $\iota_{VU}(\partial c)$ counts exactly the same embedded disks in $T$ as the expression $\partial (\iota_{VU}(c))$, hence the two are equal.
\end{proof}

Finally, we check that the co-restriction maps $\iota_{VU}$ satisfy the co-sheaf axiom.

\begin{theorem}
Let $U = L \cup_V R$, where $L, R$ are connected open subsets of $\bR$ with nonempty intersection $V$.  Then the diagram
\[ \xymatrix{
\alg(T|_V) \ar[r]^{\iota_{VR}} \ar[d]_{\iota_{VL}} & \alg(T|_R) \ar[d]^{\iota_{RU}} \\
\alg(T|_L) \ar[r]^{\iota_{LU}} & \alg(T)
} \]
commutes and is a pushout square in the category of \dgas{}.
\end{theorem}

\begin{proof}
The proof is exactly as in \cite{sivek-bordered}.  If $c \in \alg(T|_V)$ is the generator corresponding to a crossing, cusp, or base point of $T|_V$, then both $\iota_{LU}(\iota_{VL}(c))$ and $\iota_{RU}(\iota_{VR}(c))$ equal the analogous generator of $\alg(T)$.  Otherwise, if $b_{ij} \in \alg(T|_V)$ is a pair-of-left-endpoints generator, then $\iota_{VR}(b_{ij})$ is the corresponding generator of $\alg(T|_R)$, and if we view $\alg(T|_L)$ as a subalgebra of $\alg(T)$ as in Lemma~\ref{lem:left-sub-dga} then $\iota_{VL}(b_{ij})$ and $\iota_{RU}(\iota_{VR}(b_{ij}))$ are defined identically.  Thus the diagram commutes.

Now suppose we have a commutative diagram of \dgas{} of the form
\[ \xymatrix{
\alg(T|_V) \ar[r]^{\iota_{VR}} \ar[d]_{\iota_{VL}} & \alg(T|_R) \ar[d]^{f_R} \\
\alg(T|_L) \ar[r]^{f_L} & \alg
} \]
where $\alg$ is some \dga{}.  If $f: \alg(T) \to \alg$ is a \dga{} morphism such that $f_L = f \circ \iota_{LU}$ and $f_R = f \circ \iota_{RU}$, then $f$ is uniquely determined by $f_L$ on the subalgebra $\alg(T|_L) \subset \alg(T)$, and since $f_L$ is a \dga{} morphism, so is $f|_{\alg(T|_L)}$.  Any generator $c \in \alg(T)$ which does not belong to $\alg(T|_L)$ corresponds to a crossing, cusp, or base point of $T|_R$, meaning that $c = \iota_{RU}(c')$ for some generator $c' \in \alg(T|_R)$, so we must have $f(c) = f_R(c')$ and
\[ \partial(f(c)) = \partial (f_R(c')) = f_R(\partial c')  = f(\iota_{RU}(\partial c')) = f(\partial(\iota_{RU}(c'))) = f(\partial c) \]
since $f_R$ and $\iota_{RU}$ are both chain maps.  It is easy to check that this specifies $f$ as a well-defined \dga{} morphism, and since it is unique we conclude that the diagram in the statement of this theorem is a pushout square.
\end{proof}

This completes the proof of Theorem \ref{thm:cosheaf}.

\newpage

\section{Augmentations are sheaves}
\label{sec:aug=sh}

It is known that some augmentations arise from geometry: given an exact symplectic filling $(W, L)$ of
$(V, \Lambda)$, we get an augmentation $\phi_{(W, L)}$ by sending each Reeb chord to the count of disks in $(W, L)$ ending on
that Reeb chord.  But not all augmentations can arise in this way; see
the Introduction.
It is natural to hope that more augmentations can be constructed by ``filling $\Lambda$ with an element of the derived category
of the Fukaya category,'' but making direct sense of this notion is difficult.  Instead we will pass
from the Fukaya category to an equivalent category $Sh(\Lambda,\mu;\coeffs)$ introduced
in \cite{STZ}: constructible sheaves on $\mathbb{R}^2$ whose singular support meets $T^\infty \bR^2$ in some subset of
$\Lambda \subset \bR^3 = T^{\infty, -}\bR^2 \subset T^\infty \bR^2$, with coefficients in $\coeffs.$

In this section we realize this hope, by identifying
the subcategory $\cC_1(\Lambda,\mu;\coeffs)\subset Sh(\Lambda,\mu;\coeffs)$ of
``microlocal rank-one sheaves'' having acyclic stalk when $z \ll 0$ --- i.e., those corresponding to
Lagrangian branes with rank-one bundles --- with the category of augmentations.

\begin{theorem} \label{thm:main} Let $\Lambda$ be a Legendrian knot or link whose front diagram is equipped with a $\bZ$-graded Maslov potential $\mu$.  Let $\field$ be a field.
Then there is an $A_\infty$ equivalence of categories $$\Aug(\Lambda, \mu;  \field) \to \cC_1(\Lambda,\mu;\field).$$
\end{theorem}

\begin{remark}
As defined in Section~\ref{sec:augcat}, the augmentation category
$\Aug(\Lambda;\field)$ depends on a choice of Maslov potential $\mu$
on $\Lambda$, but we have suppressed $\mu$ in the notation up to
now. If $\Lambda$ is a single-component knot, then both categories in
Theorem~\ref{thm:main} are independent of the choice of $\mu$.
\end{remark}

\begin{remark}
More generally, one can consider $\Lambda$ equipped with a
$(\bZ/m)$-graded Maslov potential where $m\,|\,2r(\Lambda)$, and
define the category of $(\bZ/m)$-graded augmentations; see
Remark~\ref{rmk:grading}. There is a corresponding
 category of sheaves for $m$-periodic
complexes, and we expect that the equivalence in
Theorem~\ref{thm:main} would continue to hold in this more general
setting, with proof along similar lines.  However, in this paper, we
restrict ourselves to the
case of $\bZ$-graded Maslov potential; in particular, $\Lambda$ must
have rotation number $0$.
\end{remark}

\begin{proof}[Sketch of proof of Theorem~\ref{thm:main}]
(A detailed version comprises this entire section.)
As both the augmentation category and the sheaf category are known to transform by equivalences when the knot is altered by Hamiltonian
isotopy (from Theorem \ref{prop:Invariance} and \cite{STZ}, respectively), we may put the knot in any desired form.
Thus we take $\Lambda$ to be given by a front diagram in plat position, say
with $n$ left cusps and $n$ right cusps.  We stratify the $x$ line so that 
above each open interval, the front diagram above them contains only one interesting feature of the knot.  That is, the picture
above this interval must be one of the four possibilities shown in
Figure \ref{fig:4figs}.

\begin{figure}
\includegraphics[scale = 1.2]{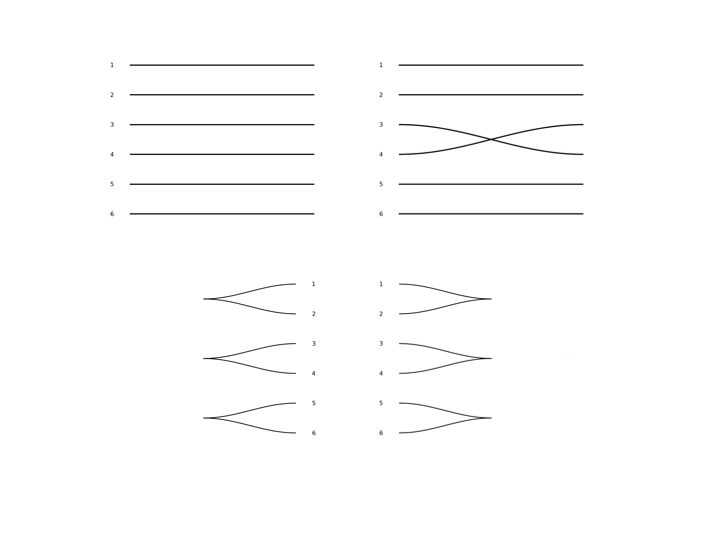}
\caption{Front diagrams for an $n$-strand knot, $n = 6.$  Clockwise from the upper left:
$n$-lines $\equiv_n$, crossing $\crossk$ (with $k=3$), right cusps $\succ$, left cusps $\prec$.
The numbers represent
strand labels, not Maslov potential values.}
\label{fig:4figs}
\end{figure}

To facilitate the proof we introduce in Section \ref{sec:MC} yet a third category, $MC$.  It is a categorical formulation of the Morse complex sequences
of Henry \cite{Henry2011}.  We define it locally, so it is by construction a sheaf on the $x$-line.  It is a dg category,
and significantly simpler than either the augmentation category or the sheaf category.

In Section \ref{sec:locAug} we calculate the local augmentation categories, and then define locally equivalences of $A_\infty$  categories
$\hh: \Aug \to MC$.  Then in Section \ref{sec:locSh} we calculate the sheaf categories, and produce equivalences of sheaves of dg categories
$\mathfrak{r}: MC \to \cC_1$.

Composing these functors and taking global sections, we learn that there is an equivalence
$\Aug^\sim \to \cC_1$, where $\Aug^\sim$ is the global sections {\em of the sheafification} of the
augmentation category.  This has, a priori, more objects than $\Aug$.  In fact this is already true for the unknot
without base points---the sheaf category, hence $\Aug^\sim$, has an object, whereas $\Aug$ does not unless $\coeffs$ has characteristic $2$ (see Remark~\ref{rmk:dgasigns}, where removing base points corresponds to setting $t=1$).
However,
by checking the criterion of Lemma \ref{lem:strictify} (by computing all local categories and restriction functors), we learn that when $\Lambda$ has a base point at each right cusp, the
augmentation category indeed forms a sheaf, giving the desired result.

%

A more explicit way to describe what is going on is the following.  We split the front diagram of the knot
into a union of pieces $T_L \cup T_{k_1} \cup \dots \cup T_{k_m} \cup T_R$, where
\begin{itemize}
\item $T_L$ consists of all $n$ left cusps,
\item each $T_{k_i}$ contains $n$ strands, with a single crossing between the $k_i$-th and $(k_i+1)$-st strands, and
\item $T_R$ consists of all $n$ right cusps.
\end{itemize}
In each case we will determine the augmentation category of that piece, together with the relevant functors from it to the
augmentation categories of the associated line algebras.  These augmentation categories will form pullback squares dual to the
diagram \eqref{eq:bordered-pushout}, so that the augmentation category of $K$
can be recovered up to equivalence from this information.  We do the same for the sheaf category, and match local pieces.

More precisely, we must prove a compatibility among the equivalences and restrictions.  It suffices
to establish for each triple of
Maslov-graded bordered plats $(T,\mu_L) \rightarrow \;(\equiv,\mu)\; \leftarrow (T',\mu_R)$
the following diagram, commuting up to homotopies indicated by dotted lines:
\begin{equation}
\label{eq:bigdiagram}
 \xymatrix{\Aug(T,\mu_L)\ar[r]\ar[d]^{\hh}\ar@{.>}[dr]&\Aug(\equiv,\mu)\ar[d]^{\hh}&
\Aug(T',\mu_R)\ar[l]\ar[d]^{\hh}\ar@{.>}[dl]\\
MC(T,\mu_L)\ar[r]&MC(\equiv,\mu)&MC(T',\mu_R)\ar[l]\\
\cC_1(T,\mu_L)\ar[r]\ar[u]^{\mathfrak r}\ar@{.>}[ur]&\cC_1(\equiv,\mu)\ar[u]^{\mathfrak r}&
\cC_1(T',\mu_R)\ar[l]\ar[u]^{\mathfrak r}\ar@{.>}[ul] \\
} 
\end{equation}
(Note the homotopy may be the zero map.)  Here $T$ will be either $T_L$ or one of the $T_{k_i}$, while $T'$ is either a $T_{k_i}$ or $T_R$. 
Remember that each vertical line is an isomorphism.
\end{proof}

\subsection{The Morse complex category} $ $
\label{sec:MC}

We define a constructible sheaf of dg categories on the $x$-line, denoted $\underline{MC}$, by sheafifying the following local descriptions.
In this section ``$\mu$'' should be viewed as providing fictional Morse indices.  Throughout we work with a fixed ring $\coeffs$.

\subsubsection{Lines}

\begin{definition}For $\mu: \{1\ldots n\} \to \bZ$,
we write $\Bmu$ for the free graded $\coeffs$-module with basis $|1\rangle, \ldots, |n \rangle $ where $\deg |i \rangle  = - \mu(i)$, and decreasing filtration
${}^k \Bmu := \mathrm{Span}( |n \rangle, \ldots, |k+1 \rangle  )$.  That is,
$${}^0 \Bmu = V, \,\,\, {}^1 \Bmu = \mathrm{Span}( |n \rangle, \ldots, |2 \rangle  ),\,\,\,\,\,\,\, \cdots\,\,\,\,\,\,\,\,  {}^{n-1} \Bmu = \mathrm{Span} |n \rangle, \,\,\,
{}^n \Bmu = 0.$$
\end{definition}

\begin{remark} (To be read only when referring back to this section from the sheaf category section.)
The correspondence to sheaves takes ${}^k \Bmu$ to the stalk on the $k$-th line. \end{remark}

\begin{definition}
Fix an integer $n$ (the number of lines) and a function
$\mu:\{1,\ldots, n\} \to \bZ$.
\label{def:MClines}
We define $MC(\equiv; \mu)$ to be the dg category with:
\begin{itemize}
\item Objects: square-zero operators $d$ on $\Bmu$, which preserve the filtration on $\Bmu$ and are degree 1 with
respect to the grading on $\Bmu$.
\item Morphisms: $\Hom_{MC(\equiv; \mu)}(d, d')$ is $\Hom_{filt}(\Bmu, \Bmu)$ as a graded vector space; i.e., it
consists of the linear, filtration preserving maps
$\Bmu \to \Bmu$ and carries the usual grading of a Hom of graded vector spaces.
Only its differential depends on $d, d'$, and is $$D\phi = d' \circ \phi - (-1)^{|\phi|} \phi \circ d.$$
\item Composition: usual composition of maps.
\end{itemize}
That is, we allow maps $|j \rangle \langle i|$ for $i \le j$, i.e. lower triangular matrices, and
$$\deg |j \rangle \langle i| = \deg |j\rangle - \deg |i \rangle = \mu(i) - \mu(j)$$
and the differential is $D(|i \rangle \langle j|) = d' |i\rangle \langle j | - (-1)^{\mu(i) - \mu(j)} | i \rangle \langle j | d$.
\end{definition}

\begin{lemma} \label{lem:dansremark}
Assume $d \cong d' \in MC(\equiv; \mu)$.  Then, for any $k$,
\begin{itemize}
\item $\langle k + 1 | d | k \rangle = 0$ if and only if $\langle k + 1 | d' | k \rangle = 0$
\item $\langle k + 1 | d | k \rangle \in \coeffs^*$ if and only if $\langle k + 1 | d' | k \rangle \in \coeffs^*$
\end{itemize}
\end{lemma}
\begin{proof}
By assumption, we have $d = s^{-1}  d' s$ for some lower triangular matrix $s$.  As $d, d'$ are strictly
lower triangular, we have
$$\langle k + 1 | d | k \rangle = \langle k + 1 | s^{-1} d' s | k \rangle =
\langle k + 1 | s^{-1} | k+1 \rangle \cdot \langle k+1 | d' |k \rangle \cdot \langle k | s | k \rangle$$
and $\langle k + 1 | s^{-1} | k+1 \rangle,  \langle k | s | k \rangle \in \coeffs^*$ since $s$ is invertible.
\end{proof}

\begin{remark} Over a field, Barannikov has classified the isomorphism classes of Morse complexes:
each has a unique representative whose matrix in the basis $|i\rangle$ at most one nonzero entry in
each row and column, and moreover these are all $1$s.
\end{remark}

\subsubsection{Crossings}

We now describe the Morse complex category $MC(\crossk; \mu)$ associated to a crossing between the $k$-th and $(k+1)$-st strands.
It will be built from $MC(\equiv; \mu)$.  To define it we first note some equivalences between conditions.

\begin{lemma} \label{lem:crossobj}
Let $d \in MC(\equiv, \mu)$, and $z \in \coeffs$.  We write $\mu_R := \mu \circ (k, k+1) : \{1, \ldots, n \} \to \bZ$.  We use
$|i\rangle$ for the basis of $\Bmu$, and $|i_R \rangle$ for the basis of $\Bmu_R$.  We identify these vector spaces by
\begin{eqnarray*}
|i_R \rangle & = & |i\rangle \qquad \qquad \qquad i \ne k, k+1 \\
|k_R \rangle & = & |k+1 \rangle \\
|k+1_R \rangle & = & |k \rangle + z |k + 1 \rangle.
\end{eqnarray*}
  Then the following are equivalent.
\begin{itemize}
\item Under this identification, $d \in MC(\equiv, \mu_R)$.
\item We have $z = 0$ unless $\mu(k) = \mu(k+1)$, and we have $\langle k + 1 | d | k  \rangle = 0$.
\end{itemize}
\end{lemma}
\begin{proof}
The condition $d \in MC(\equiv, \mu_R)$ means that $\deg |j_R \rangle = - \mu_R(j)$ 
and
that $d$ preserves the decreasing filtration ${}^i \Bmu_R  =   \mathrm{Span}( |n_R \rangle, \ldots, |i+1_R \rangle  )$.
The first condition amounts to $z = 0$ unless $\mu(k) = \mu(k+1)$.  As
${}^i \Bmu_R  =   \mathrm{Span}( |n \rangle, \ldots, |i+1 \rangle  ) = {}^i \Bmu$ for $i \ne k$ but ${}^k \Bmu_R \ne {}^k \Bmu$,
the second condition is equivalent to $\langle k_R | d |k + 1_R\rangle  = 0$.
Changing basis and recalling that $d$ has square zero, hence $\langle k +1 | d | k+1 \rangle = 0$, this is equivalent
to $\langle k + 1 | d | k  \rangle = 0$.
\end{proof}

\begin{lemma} \label{lem:crosshom}
Let $d, z$ and $d', z'$ satisfy the  conditions of Lemma \ref{lem:crossobj}, and
let $\xi \in \Hom_{MC(\equiv; \mu)} (d, d')$, i.e., it is a filtration preserving linear map
$\xi: \Bmu \to \Bmu$.  Then the following are equivalent:
\begin{itemize}
\item The map $\xi$ preserves the filtration on $\Bmu_R$.
\item $ \xi |k + 1_R \rangle \in \mathrm{Span}(|n_R'\rangle, \ldots, |k+1_R' \rangle )$.
\item $ \langle k + 1 | \xi | k \rangle = z' \langle k | \xi | k \rangle - z \langle k+1 | \xi | k+1 \rangle $.
\end{itemize}
\end{lemma}
\begin{proof}
The first and second are equivalent since by assumption $\xi$ already preserves the filtration on $\Bmu$, hence
all but possibly one of the steps of the filtration of $\Bmu_R$. To check whether this step is preserved, we need to check
$\left( z' \langle k| - \langle k+1 | \right) \xi \left(   |k \rangle + z | k + 1 \rangle \right)  = 0$; the fact that
$\langle k | \xi | k + 1 \rangle$ vanishes shows the equivalence of the second and third conditions.
\end{proof}

\begin{definition}
Fix an integer $n$ (the number of lines), and a
function $\mu: \{1,\ldots, n\} \to \bZ$ as before.  We write $MC(\crossk; \mu)$ for the dg
category whose objects are pairs $(d, z)$ for $d \in MC(\equiv, \mu)$ and $z \in \coeffs$, satisfying the equivalent
conditions of Lemma \ref{lem:crossobj}, and whose morphisms are those morphisms in $MC(\equiv, \mu)$ which satisfy the equivalent
conditions of Lemma \ref{lem:crosshom}.  The composition and differential are the restrictions of those of $MC(\equiv, \mu)$.
\end{definition}

\begin{definition}
There is an evident forgetful dg functor
\begin{eqnarray*} \rho_L : MC(\crossk; \mu) & \to & MC(\equiv; \mu) \\ (d,z) & \mapsto & d. \end{eqnarray*}
We define this to be
the restriction map to the left.

\label{def:MCcrossres}
Recall $\mu_R := \mu \circ (k, k+1)$, and the element $z$ gives an identification $\theta_z: \Bmu \to \Bmu_R$.
Essentially by definition, we also have a dg functor which on objects is
\begin{eqnarray*} \rho_R : MC(\crossk; \mu) & \to & MC(\equiv; \mu_R) \\ (d,z) & \mapsto & \theta_z \circ d \circ \theta_z^{-1}  \end{eqnarray*}
and on morphisms is given by
\begin{eqnarray*} \rho_R : \Hom_{MC(\small\crossk; \mu)} ((d,z), (d', z')) & \to & \Hom_{MC(\equiv; \mu)} (\theta_z \circ d \circ \theta_z^{-1},
\theta_{z'} \circ d' \circ \theta_{z'}^{-1})
\\ \xi & \mapsto & \theta_{z'} \circ \xi \circ \theta_z^{-1}.  \end{eqnarray*}
We define this to be the restriction map to the right.
\end{definition}

\begin{remark}
Both restrictions are injective on homs at the chain level, but of course need not be injective on homs after passing to cohomology. \end{remark}

\begin{proposition} \label{prop:mcrosssat}
Every object in $MC(\equiv; \mu)$ isomorphic to an object in the image of $\rho_L$ is already
in the image of $\rho_L$.  Similarly, every object in $MC(\equiv; \mu_R)$ isomorphic to an object in the image of
$\rho_R$ is already in the image of $\rho_R$.
\end{proposition}
\begin{proof}
Objects in the image are characterized by $\langle k + 1 | d | k \rangle = 0$; by Lemma \ref{lem:dansremark} this is a union of isomorphism
classes.
\end{proof}

\subsubsection{Cusps}

Let ``$\succ$'' denote a front diagram with $n$ right cusps.
Near the left, it is $2n$ horizontal lines, which we number $1,2,\dots,2n$ from top to bottom,
and each pair $2k-1,2k$ is connected by a right cusp.  We fix a function $\mu: \{1, \ldots, 2n\} \to \bZ$,
such that $\mu(2k) + 1= \mu(2k-1)$.

\begin{definition}
The category $MC(\succ, \mu)$ is the full subcategory of $MC(\equiv, \mu)$ on objects $d$ such that
$\langle 2k-1 | d | 2k \rangle \in \coeffs^*$ for all $1 \le k \le n$.  The left restriction map $\rho_L: MC(\succ, \mu) \to MC(\equiv, \mu)$ is
just the inclusion.
\end{definition}

\begin{proposition} \label{prop:onlyonecusp}
All objects in $MC(\succ; \mu)$ are isomorphic.
\end{proposition}
\begin{proof}
Let $d \in MC(\succ; \mu)$, we will show it is
isomorphic to $d_0 = \sum |2k \rangle \langle 2k-1|$.
Note that to do so means to given an invertible degree zero lower triangular matrix $u$, such that
$d_0 u = u d$.  We take $u := d_0 d_0^T + d_0^T d$, so that since $d^2 = d_0^2 = 0$, we have
$$d_0 u = d_0 (d_0 d_0^T + d_0^T d) =  d_0 d_0^T d = (d_0 d_0^T + d_0^T d ) d = u d. $$
Moreover, $u$ has degree zero since $\deg x^T = - \deg x$, and also $u$ is lower triangular
since $d, d_0$ are strictly lower triangular and $d_0^T$ has entries only on the first diagonal above
the main diagonal.  Finally, $u$ is invertible since its diagonal entries are either $1$s or the
$\langle 2k-1 | d | 2k \rangle$, which are invertible by definition.
\end{proof}

Similarly, for a diagram of left cusps, we define

\begin{definition}
The category $MC(\prec, \mu)$ is the full subcategory of $MC(\equiv, \mu)$ on objects $d$ such that
$\langle 2k-1 | d | 2k \rangle \in \coeffs^*$ for all $1 \le k \le n$.  The right restriction map $\rho_R: MC(\prec, \mu) \to MC(\equiv, \mu)$ is
just the inclusion.
\end{definition}

\begin{proposition}
All objects in $MC(\prec,\mu)$ are isomorphic.
\end{proposition}

\subsubsection{Sheafifying the Morse complex category} \label{MCissheaf}

We observe that, comparing Lemma \ref{lem:dansremark} to the characterizations of the image maps
on the crossing and cusp categories, the condition of Lemma \ref{lem:strictify} is satisfied.  Thus,
we can discuss sections of the sheaf of Morse complex categories naively.

\subsection{Local calculations in the augmentation category}
\label{sec:locAug}

In this section, we determine the local augmentation categories for the line, crossing,
left cusp, and right cusp diagrams. We define the isomorphisms $\mathfrak{h}$ to the corresponding
local categories of $MC$ and study the compatibility with left and right restriction functors to $\equiv$, as in the diagram \eqref{eq:bigdiagram}.
We conclude by proving that the presheaf of
$\underline{\Aug}(\Lambda)$ is a sheaf when $\Lambda$ is a front diagram with
base points at all right cusps.

\begin{notation}
Recall that if $\alg(T)$ is the Chekanov--Eliashberg algebra of a tangle $T$, and $\epsilon_1, \epsilon_2: \alg(T) \to \coeffs$ are two
augmentations, then $\Hom_{\Aug}(\epsilon_1, \epsilon_2)$ is generated by symbols dual to the names of certain Reeb chords
in the 2-copy; specifically those
chords from $T$ (viewed as carrying $\epsilon_2$) to its pushoff in the positive Reeb direction (viewed as carrying $\epsilon_1$).

Thus if $x$ is a Reeb chord of $T$ itself, it gives rise to a ``long'' chord $x^{12}$ in the $2$-copy, and a corresponding generator
$x^{12} \in \alg(T^2)$.  There will however be additional ``short'' chords $y^{12}$ in the $2$-copy, and corresponding
generators $y^{12} \in \alg(T^2)$.
Recall from Convention~\ref{conv:pm} that we write their duals in $\Hom_{\Aug}(\epsilon_1,\epsilon_2)$
as $x^+ := (x^{12})^\vee$ and $y^+:= (y^{12})^\vee$ with $|x^+| = |x|+1$ and $|y^+|=|y|+1$.
\end{notation}

\begin{remark}
We will find that applying the differential to any generator of
any of the local \dgas{} gives a sum of monomials of word length at most 2.  It follows that
all higher compositions $m_k$ in the respective augmentation categories will vanish
for $k \geq 3$ --- that is, all the categories will in fact be dg categories.  The $A_\infty$ behavior,
from this point of view, comes entirely from the right restriction map on the crossing category, $\rho_R: \Aug(\crossk, \mu) \to \Aug(\equiv, \mu)$,
which is an $A_\infty$ but not a dg morphism, i.e. it does not respect composition on the nose, but only up to
homotopy --- see Theorem \ref{thm:crossmc}.
\end{remark}

\subsubsection{Lines} \label{sec:linealg}

We write $\equiv_n$ or just $\equiv$ for the front diagram consisting of $n$ horizontal lines, numbered $1 \ldots n$ from top to bottom.  (See Figure \ref{fig:4figs}, upper left.)
Fix a Maslov potential $\mu: \{1, \ldots, n\} \to \bZ$.  The algebra $\alg(\equiv, \mu)$ of this tangle is
freely generated by ${n \choose 2}$ elements $a_{ij}$, $1 \leq i < j \leq n$, with $|a_{ij}| = \mu(i) - \mu(j) - 1$, and
\[ \dd a_{ij} = \sum_{i < k < j} (-1)^{\mu(i)-\mu(k)} a_{ik} a_{kj} = \sum_{i < k < j} (-1)^{\mu(i)}a_{ik} \cdot (-1)^{\mu(k)}a_{kj}. \]
Throughout this section we will let $(-1)^\mu$ denote the matrix $\operatorname{diag}((-1)^{\mu(1)},(-1)^{\mu(2)},\dots,(-1)^{\mu(n)})$.

Package the generators into a strictly upper triangular matrix
\[ A := \left[ \begin{array}{ccccc}
0 & a_{12} & a_{13} & \dots & a_{1n} \\
0 & 0 & a_{23} & \dots & a_{2n} \\
\vdots & \vdots & \vdots & \ddots & \vdots \\
0 & 0 & 0 & \dots & a_{n-1,n} \\
0 & 0 & 0 & \dots & 0
\end{array} \right] = \sum_{i < j} a_{ij} |i  \rangle \langle j |  . \]
Then
$$\dd A = (-1)^{\mu} A (-1)^\mu A.$$

\begin{theorem} \label{thm:lineaugcategory}
There is a (strict) isomorphism of dg categories
\[ \hh: \Aug(\equiv, \mu) \to MC(\equiv, \mu). \]
It is given on objects by:
\[ [ \epsilon: \alg(\equiv, \mu) \to \coeffs] \mapsto [ d =  (-1)^\mu \epsilon(A)^T: \Bmu \to \Bmu] \]
and on morphisms $\Hom_{\Aug}(\e_1,\e_2) \to \Hom_{MC}(d_1,d_2)$ by
\[ a_{ij}^+ \mapsto (-1)^{(\mu(i)+1)\mu(j)+1} |j \rangle \langle i|. \]
In other words:
\begin{itemize}
\item
$\epsilon$ is an augmentation if and only if $(-1)^\mu \epsilon(A)^T$ is a degree one, square zero, filtration preserving operator on $\Bmu$.
\item For the $i, j$ for which there's an element $a_{ij}^+ \in \hom(\epsilon_1, \epsilon_2)$, the operator $|j \rangle \langle i|$ preserves
the filtration on $\Bmu$, and this induces an isomorphism on underlying spaces of morphisms.
\item  Degrees are preserved: $$\deg a_{ij}^+ = \mu(i) - \mu(j) = \deg (|j \rangle \langle i|).$$
\item The differential is preserved: $\hh \circ \mu_1 = d \circ \hh$, where
\[
d (|j \rangle \langle i|) = ((-1)^\mu\epsilon_2(A)^T)|j\rangle\langle
i| - (-1)^{\mu(i)-\mu(j)}|j\rangle\langle i|((-1)^\mu\epsilon_1(A)^T).
\]
\item The composition is preserved: i.e., the only nonvanishing compositions are
\[ m_2(a_{kj}^+, a_{ik}^+)  = (-1)^{|a_{kj}^+||a_{ik}^+|+1}a_{ij}^+ \]
compatibly with
\begin{align*}
| j \rangle \langle k | \circ |k \rangle \langle i| = |j \rangle \langle i |.
\end{align*}
\end{itemize}
\end{theorem}
\begin{proof}
First we show that the map makes sense on objects, ie.
a map $\epsilon: \alg(\equiv, \mu)\to \coeffs$ is an augmentation if and only if $(-1)^\mu\epsilon(A)^T$ is a filtered degree 1 derivation on $\Bmu$.
As $\epsilon(A)$ is upper triangular, its transpose is lower triangular, hence preserves the filtration on $\Bmu$.
The term $(-1)^{\mu(j)} a_{ij} |j \rangle \langle i | $ contributes to $(-1)^\mu\epsilon(A)^T$ only if $|a_{ij}| = \mu(i) - \mu(j) - 1 = 0$, i.e.
only if $\deg |j \rangle \langle i | = \mu(i) - \mu(j) = 1$, so $(-1)^\mu\epsilon(A)^T$ is degree 1.
Finally, the condition $\epsilon \circ \dd = 0$
translates directly into $\epsilon((-1)^\mu A (-1)^\mu A) = \left((-1)^\mu \epsilon(A)\right)^2 = 0$, hence $\left((-1)^\mu \epsilon(A)^T\right)^2 = 0$.

Given two augmentations $\epsilon_1,\epsilon_2: \alg(\equiv, \mu) \to \coeffs$, we compute $\hom(\epsilon_1,\epsilon_2)$
by first building the two-copy,
whose algebra we denote $\alg^2(\equiv, \mu)$.
Its generator $a_{ij}^{rs}$ represents a segment {\bf to} the $r$-th copy of point $i$ {\bf from}
the $s$-th copy of point $j$ ($1 \leq r,s \leq 2$);
here we must have either $i<j$ or $i=j$ and $r<s$.
There are  ${2n\choose 2}$ such generators.

The Hom space is free on the generators $a_{ij}^+$, dual to the $a_{ij}^{12}$ and of degree
$$|a_{ij}^+| = |a_{ij}^{12}| +1 = \mu(i) - \mu(j).$$
Since $i \le j$, the image  $| j \rangle \langle i|$ of $a_{ij}^+$  is lower triangular and hence preserves
the filtration on $\Bmu$, so $\Phi$ is well defined and an isomorphism of underlying spaces.
The grading of $|j \rangle \langle i|$ as an endomorphism of $\Bmu$ is $\deg |j \rangle - \deg |i\rangle = \mu(i) - \mu(j) =
\deg a_{ij}^+$, so $\Phi$ is a graded map.

The differential on the Hom space is given, according
to \eqref{eq:ms}, by the formula
\[ m_1 (a_{ij}^+) = \sum_{\alpha \in \rR} \alpha^+ \cdot \mathrm{Coeff}_{a_{ij}^{12}} (\dd_\epsilon \alpha). \]
Here, $\epsilon = (\epsilon_1, \epsilon_2)$ is the
pure augmentation of $\alg^2(\equiv, \mu)$ defined by $\epsilon(a_{ij}^{11}) = \epsilon_1(a_{ij})$ and $\epsilon(a_{ij}^{22}) = \epsilon_2(a_{ij})$ for $i<j$, and $\epsilon(a_{ij}^{rs}) = 0$ otherwise.

For any generator $a_{ij}^{12}$ of $I^2_n$, $i \leq j$, we have
\[ \dd a_{ij}^{12} =  \sum_{i < k \leq j} (-1)^{|a^{11}_{ik}|+1}a_{ik}^{11} a_{kj}^{12} + \sum_{i\leq k < j} (-1)^{|a_{ik}^{12}|+1}a_{ik}^{12} a_{kj}^{22} \]
and since $\epsilon(a_{kj}^{12}) = \epsilon(a_{ik}^{12}) = 0$, keeping only linear terms in the twisted differential, we have
\begin{align*}
[\mbox{linear part}]\dd_{\epsilon} (a_{ij}^{12}) &= \sum_{i < k \leq j} (-1)^{|a_{ik}^{11}|+1}\epsilon_1(a_{ik})a^{12}_{kj} + \sum_{i \leq k < j} (-1)^{|a_{ik}^{12}|+1}a^{12}_{ik} \epsilon_2(a_{kj}) \\
&= \sum_{i < k \leq j} (-1)^{\mu(i)} \epsilon_1(a_{ik}) (-1)^{\mu(k)}\cdot a^{12}_{kj} + \sum_{i \leq k < j} (-1)^{\mu(i)}a^{12}_{ik} (-1)^{\mu(k)} \cdot \epsilon_2(a_{kj}).
\end{align*}
Packaging these generators into
\[ A^{[12]} = \left[ \begin{array}{ccccc}
a_{11}^{12} & a_{12}^{12} & a_{13}^{12} & \dots & a_{1n}^{12} \\
0 & a_{22}^{12} & a_{23}^{12} & \dots & a_{2n}^{12} \\
\vdots & \vdots & \vdots & \ddots & \vdots \\
0 & 0 & 0 & \dots & a_{n-1,n}^{12} \\
0 & 0 & 0 & \dots & a_{nn}^{12}
\end{array} \right] \]
this equation reads simply
\[ [\mbox{linear part}] \dd_{\epsilon} A^{[12]} = (-1)^\mu \epsilon_1(A) (-1)^\mu A^{[12]} + (-1)^\mu A^{[12]} (-1)^\mu \epsilon_2(A). \]
We however want to compute $m_1$.  This is given by
\begin{align}
m_1(a_{rs}^+) & = \sum_{i < r}  (-1)^{\mu(i)} \epsilon_1(a_{ir} ) (-1)^{\mu(r)} a_{is}^+ + \sum_{s < j} \epsilon_2(a_{sj}) (-1)^{\mu(r) + \mu(s)}
a_{rj}^+ \nonumber \\
& = - \sum_{i < r}  \epsilon_1(a_{ir} )  a_{is}^+ + (-1)^{\mu(r) + \mu(s)} \sum_{s < j} \epsilon_2(a_{sj})  a_{rj}^+. \label{eq:line-category-m1}
\end{align}
By comparison, we have
\begin{align*}
d(|s \rangle \langle r| ) & =
  \left((-1)^\mu\epsilon_2(A)^T\right) |s \rangle \langle r |
 - (-1)^{\mu(r) - \mu(s)}  |s \rangle \langle r | \left((-1)^\mu\epsilon_1(A)^T\right)  \\
& =
\left(\sum_{i<  j} (-1)^{\mu(j)} \epsilon_2(a_{i j})|j \rangle \langle i|  \right)  |s \rangle \langle r |
 - (-1)^{\mu(r) - \mu(s)}
|s \rangle \langle r | \left(\sum_{i<  j} (-1)^{\mu(j)}  \epsilon_1(a_{ij})
|j \rangle \langle i|  \right)\\
& =
\sum_{s<j} (-1)^{\mu(j)}  \epsilon_2(a_{sj})  |j \rangle  \langle r |
 - (-1)^{\mu(s)}
\sum_{i< r}  \epsilon_1(a_{ir}) |s \rangle \langle i| \\
& =
(-1)^{\mu(s)+1} \left(\sum_{i<  r}  \epsilon_1(a_{ir}) |s \rangle
  \langle i|
+ \sum_{j>s}  \epsilon_2(a_{sj})  |j \rangle  \langle r | \right)
\end{align*}
since $(-1)^{\mu(j)}\epsilon_2(a_{sj}) = 0$ unless $\mu(s) - \mu(j)=1$.  Multiplying both sides by $(-1)^{(\mu(r)+1)\mu(s)+1}$ and recalling that $\epsilon_1(a_{ir}) = 0$ (resp.\ $\epsilon_2(a_{sj}) = 0$) unless $\mu(i)=\mu(r)+1$ (resp.\ $\mu(s)=\mu(j)+1$), we have
\begin{align*}
d\left((-1)^{(\mu(r)+1)\mu(s)+1}|s\rangle\langle r|\right) &=
(-1)^{\mu(r)\mu(s)} \left(\sum_{i<  r}  \epsilon_1(a_{ir}) |s \rangle
  \langle i|
+ \sum_{s<j}  \epsilon_2(a_{sj})  |j \rangle  \langle r | \right) \\
&= \sum_{i<r} \epsilon_1(a_{ir}) \left((-1)^{(\mu(i)+1)\mu(s)} |s\rangle\langle i|\right) \\
&\qquad + \sum_{s<j} \epsilon_2(a_{sj}) \left((-1)^{(\mu(j)+1)\mu(r)} |j\rangle\langle r|\right) \\
&= -\sum_{i<r} \epsilon_1(a_{ir}) \left((-1)^{(\mu(i)+1)\mu(s)+1} |s\rangle\langle i|\right) \\
&\qquad + (-1)^{\mu(r)+\mu(s)}\sum_{s<j} \epsilon_2(a_{sj}) \left((-1)^{(\mu(r)+1)\mu(j)+1} |j\rangle\langle r|\right).
\end{align*}
So $\hh$ commutes with the differential on Hom spaces. It remains to show that $\hh$ commutes with the composition.

We consider the algebra $\alg^3(\equiv, \mu)$ associated to the $3$-copy.
This is generated by elements $a_{ij}^{rs}$ as before, but now we have $1 \leq r,s \leq 3$; in particular, we compute
\[ \dd a_{ij}^{13} = \sum_{i<k\leq j} (-1)^{|a_{ik}^{11}|+1}a_{ik}^{11}a_{kj}^{13} + \sum_{i\leq k\leq j} (-1)^{|a_{ik}^{12}|+1}a_{ik}^{12}a_{kj}^{23} + \sum_{i\leq k<j} (-1)^{|a_{ik}^{13}|+1}a_{ik}^{13}a_{kj}^{33}. \]
Since the differential contains only quadratic terms, the quadratic term of its linearization is the same as the original quadratic term.
Only terms of the form $(-1)^{|a_{ik}^{12}|+1}a^{12}_{ik} a^{23}_{kj}$ contribute to $m_2$.
Each of these terms can only appear in the differential of a single generator of the form $a^{13}_{ij}$.

By \eqref{eq:ms},
\[ m_2( a_{kj}^+, a_{ik}^+) = (-1)^{|a_{kj}^+||a_{ik}^+|+|a_{ik}^+|+1} \cdot (-1)^{\mu(i)-\mu(k)} a_{ij}^+ = (-1)^{|a_{kj}^+||a_{ik}^+|+1}a_{ij}^+. \]
If $k\neq k'$, the term $a_{k'j}^{23} a_{ik}^{12}$
does not appear in the differential of any generator of $\alg^3(\equiv, \mu)$.  It follows that
then  $m_2( a_{k'j}^+, a_{ik}^+) = 0$.  That is,
\[ m_2:  \hom(\epsilon_2,\epsilon_3) \otimes \hom(\epsilon_1, \epsilon_2) \to \hom(\epsilon_1,\epsilon_3) \]
is given by the formula
\[ a_{kj}^+ \otimes a_{ik}^+  \mapsto (-1)^{|a_{kj}^+||a_{ik}^+|+1}a_{ij}^+. \]
This is compatible with composition in $MC(\equiv,\mu)$ once one checks that the signs are correct, which amounts to verifying the identity
\begin{align*}
\left[(\mu(k)+1)\mu(j)+1\right] &+ \left[(\mu(i)+1)\mu(k)+1\right] \\
&\qquad \equiv \left[(\mu(k)-\mu(j))(\mu(i)-\mu(k))+1\right] + [(\mu(i)+1)\mu(j)+1]
\end{align*}
modulo 2.  Finally, as the differentials of all $m$-copy algebras have no cubic or higher terms, all higher compositions vanish.
\end{proof}

\subsubsection{Crossings}

Let the symbol $\crossk$ denote a bordered plat consisting of $n$ strands, numbered from 1 at the top to $n$ at the bottom along the left,
with a single crossing between strands $k$ and $k+1$.  (See Figure \ref{fig:4figs}, upper right.)
Fix a Maslov potential $\mu$.  We write $\alg(\crossk, \mu)$ for the
Chekanov--Eliashberg \dga{} of this tangle with Maslov potential $\mu$.

We will write $\mu_L$ and $\mu_R$ for the induced Maslov potentials along the left and right of the diagram, respectively.
Note that if $s_k = (k, k+1) \in S_n$, then $\mu_R = \mu_L \circ s_k$.  We write the co-restriction maps from the left and right line algebras
as

\begin{eqnarray*}
\iota_L: \alg(\equiv, \mu_L) & \to & \alg(\crossk, \mu) \\
\iota_R: \alg(\equiv, \mu_R) & \to & \alg(\crossk, \mu)
\end{eqnarray*}

We view this as identifying $\alg(\equiv, \mu_L)$ and the subalgebra
of $\alg(\crossk, \mu)$ generated by elements $a_{ij}$ indexed by pairs of left endpoints of lines.
The algebra $\alg(\crossk, \mu)$ has one more generator,
$c$, naming the crossing, with $\dd c = a_{k,k+1}$.

\begin{lemma} \label{lem:crossobj-left} The map
$$[\epsilon: \alg(\crossk, \mu) \to \coeffs] \mapsto [(\epsilon_L: \alg(\equiv, \mu_L) \to \coeffs, \epsilon(c))]$$
is a bijection between augmentations of $\alg(\crossk, \mu)$ and pairs of an augmentation of $\alg(\equiv, \mu_L)$ carrying
$a_{k, k+1} \to 0$ and an element $\epsilon(c) \in \coeffs$, where $\epsilon(c)$ vanishes unless
$|c| = 0$, i.e., unless $\mu(k) = \mu(k+1)$.
\end{lemma}
\begin{proof}   An augmentation of $\alg(\crossk, \mu)$ is determined by its restriction $\epsilon_L: \alg(\equiv, \mu_L) \to \coeffs$ and its value on $c$.  The augmentation
must annihilate $c$ unless $|c| = 0$.  Finally, the only condition imposed on the restriction $\epsilon_L$ is
$\epsilon(a_{k, k+1})  = \epsilon(\dd c)  = \dd \epsilon(c) = 0$.
\end{proof}

\begin{lemma} \label{lem:crossdiff}
Consider augmentations $\epsilon_1, \epsilon_2: \alg(\crossk, \mu) \to \coeffs$.  The space $\hom(\epsilon_1, \epsilon_2)$
has as a basis $a_{ij}^+$ and $c^+$.  The differential is given explicitly by
\begin{align*}
m_1(a_{rs}^+) &=  - \sum_{i < r}  \epsilon_1(a_{ir} )  \cdot a_{is}^+ + (-1)^{\mu(r) + \mu(s)} \sum_{s < j} \epsilon_2(a_{sj})  \cdot a_{rj}^+ \qquad \qquad
\{r,s\} \not \subset \{k, k+1\} \\
m_1(a_{k, k}^+) &= \epsilon_2(c) \cdot c^+ - \sum_{i < k}  \epsilon_1(a_{ik} ) \cdot a_{ik}^+ + \sum_{k < j} \epsilon_2(a_{kj})  \cdot a_{kj}^+ \\
m_1( a_{k, k+1}^+) &= c^+
- \sum_{i < k}  \epsilon_1(a_{ik} ) \cdot  a_{i,k+1}^+ + (-1)^{\mu(k) + \mu(k+1)} \sum_{k+1 < j} \epsilon_2(a_{k+1,j}) \cdot a_{kj}^+ \\
m_1(a_{k+1, k+1}^+) &= - \epsilon_1(c) \cdot c^+   - \sum_{i < k+1}  \epsilon_1(a_{i,k+1} ) \cdot a_{i,k+1}^+ +
\sum_{k+1 < j} \epsilon_2(a_{k+1,j})  \cdot a_{k+1,j}^+ \\
m_1(c^+) &= 0.
\end{align*}
\end{lemma}
\begin{proof}
To compute the Hom spaces, we study the $2$-copy, whose algebra we denote $\alg^2(\crossk, \mu)$.  This has underlying algebra
$$\alg^2(\crossk, \mu) = \alg^2(\equiv, \mu_L) \langle c^{11}, c^{12}, c^{21}, c^{22} \rangle.$$
The differential
restricted to $\alg^2(\equiv, \mu_L)$  is just the differential there, and
$$\partial c^{12} = a^{12}_{k, k+1} + a^{12}_{kk} c^{22} - (-1)^{|c|} c^{11} a^{12}_{k+1, k+1}.$$

Taking $\epsilon = (\epsilon_1 , \epsilon_2) : \alg^2(\crossk, \mu) \to \coeffs$ we find the twisted differentials of the $a^{12}_{ij}$ are as in the line algebra, and:

$$\partial_{\epsilon} c^{12} = a^{12}_{k, k+1} + a^{12}_{kk} (c^{22} + \epsilon_2(c)) - (-1)^{|c|} (c^{11} + \epsilon_1(c)) a^{12}_{k+1, k+1}$$
of which the linear part is
$$\partial_{\epsilon, 1} c^{12} = a^{12}_{k, k+1} + a^{12}_{kk} \epsilon_2(c) -  \epsilon_1(c) a^{12}_{k+1, k+1}$$
where we have observed that $\epsilon(c) = 0$ unless $|c| = 0$.
Dualizing gives the stated formulas.
\end{proof}

\begin{proposition}
The only nonzero compositions in the category $\Aug(\crossk, \mu)$
are:
\begin{gather*}
m_2(a_{kj}^+, a_{ik}^+) = (-1)^{|a_{kj}^+||a_{ik}^+|+1}a_{ij}^+ \\
m_2(c^+, a_{kk}^+) = -c^+ =  m_2(a_{k+1,k+1}^+,  c^+).
\end{gather*}
\end{proposition}
\begin{proof}
In the algebra of the $3$-copy, the ``$a$'' generators have differentials as in the line algebra, and we have
\[ \dd c^{13} = a^{13}_{k,k+1} + a^{12}_{kk} c^{23} + a^{13}_{kk} c^{33} - (-1)^{|c|}c^{11}a^{13}_{k+1,k+1} - (-1)^{|c|}c^{12}a^{23}_{k+1,k+1}. \]
Since there are no terms higher than quadratic, the quadratic terms are not affected by twisting by the pure augmentation $\epsilon=(\epsilon_1,\epsilon_2,\epsilon_3)$.
Recalling that $|c^+| = |c|+1$ and that $|a_{kk}^+| = |a_{k+1,k+1}^+| = 0$ gives the desired formulas.
\end{proof}

We now study the restriction morphisms.  First, on objects:

\begin{proposition} \label{prop:crossresobs}
Let $\epsilon: \alg(\crossk, \mu) \to \coeffs$ be an augmentation.  Let $\epsilon_L, \epsilon_R$ be its restrictions to the line
algebras on the left and the right.  Take $A = \sum a_{ij} |i \rangle \langle j| $ and $B = \sum b_{ij} |i \rangle \langle j|$
to be strictly upper triangular
$n\times n$ matrices with entries $a_{ij}$ and $b_{ij}$ in position $(i,j)$,
collecting the respective generators of the left and right line algebras
as in Section \ref{sec:linealg}.  Let
\[ \phi := 1 + \epsilon(c) |k+1 \rangle \langle k| = \left[ \begin{array}{c|cc|c}
I_{k-1} & 0 & 0 & 0 \\
\hline
0 & 1 & 0 & 0 \\
0 & \epsilon(c) & 1 & 0 \\
\hline
0 & 0 & 0 & I_{n-(k+1)}
\end{array} \right] \]
and let $s_k = (k, k+1) \in S_n$.  Then
\[ \epsilon_R(B)  = s_k \cdot (\phi^T)^{-1} \cdot \epsilon_L(A) \cdot (\phi^T) \cdot s_k. \]
\end{proposition}
\begin{proof}
Denote the generators of the right line algebra $\alg(\equiv, \mu_R)$ by $b_{ij}$.  The right co-restriction morphism is given by:
\begin{eqnarray*}
\iota_R: \alg(\equiv, \mu_R) & \to & \alg(\crossk, \mu) \\
b_{ij} & \mapsto & a_{ij} \\
b_{ik} & \mapsto & a_{i, k+1} + a_{ik} c \\
b_{kj} & \mapsto & a_{k+1,j} \\
b_{i,k+1} & \mapsto & a_{ik} \\
b_{k+1,j} & \mapsto &  a_{kj} - (-1)^{|c|}ca_{k+1,j} \\
b_{k,k+1} & \mapsto & 0.
\end{eqnarray*}
The sign comes because each downward corner vertex with even grading contributes a factor of $-1$ to the sign of a disk,
so a downward corner at $c$ contributes $(-1)^{|c|+1}$.
We rewrite the above formula in matrix form as:
\[
B \mapsto
  \left[ \begin{array}{c|cc|c}
I_{k-1} & 0 & 0 & 0 \\
\hline
0 & 0 & 1 & 0 \\
0 & 1 & - (-1)^{|c|} c & 0 \\
\hline
0 & 0 & 0 & I_{n-(k+1)}
\end{array} \right] \cdot
A \cdot  \left[ \begin{array}{c|cc|c}
I_{k-1} & 0 & 0 & 0 \\
\hline
0 & c & 1 & 0 \\
0 & 1 & 0 & 0 \\
\hline
0 & 0 & 0 & I_{n-(k+1)}
\end{array} \right] - a_{k, k+1} |k +1 \rangle \langle k |.
\]
We now apply the augmentation and observe
$\epsilon(a_{k,k+1}) = \epsilon(\dd c) = 0$,
and  $\epsilon(c) = (-1)^{|c|} \epsilon(c)$ because $\epsilon(c)=0$ unless $|c|=0$.
\end{proof}

\begin{proposition} \label{prop:crossres}
Suppose we are given an element $\xi \in \Hom_{\Aug(\crossk, \mu)}(\epsilon, \epsilon')$.  We can restrict to the left or right, obtaining
$\xi_L \in \Hom_{\Aug(\equiv, \mu_L)}(\epsilon_L, \epsilon'_L) $ and $\xi_R \in \Hom_{\Aug(\equiv, \mu_R)}(\epsilon_R, \epsilon'_R) $.  We use
$a_{ij}^+$ to denote the generators of $\Hom_{\Aug(\crossk, \mu)}(\epsilon, \epsilon')$ or $\Hom_{\Aug(\equiv, \mu_L)}(\epsilon_L, \epsilon'_L)$,
and $b_{ij}^+$ to denote the generators of $\Hom_{\Aug(\equiv, \mu_R)}(\epsilon_R, \epsilon'_R)$.

Then the left restriction is just given by $a_{ij}^+ \mapsto a_{ij}^+$; it is a map of dg categories.

On the other hand, the right restriction, despite being between dg categories, has nontrival $A_\infty$ structure.
(See Section \ref{sec:a-infinity}.)
The first order term $\Hom_{\Aug(\crossk, \mu)}(\epsilon, \epsilon')  \to \Hom_{\Aug(\equiv, \mu_R)}(\epsilon_R, \epsilon'_R)$ is given by:
\begin{eqnarray*}
a_{ik}^+ & \mapsto & b_{i,k+1}^+ + \epsilon'(c) b_{ik}^+ \\
a_{kk}^+ & \mapsto &  b_{k+1,k+1}^+ \\
a_{kj}^+ & \mapsto &   b_{k+1,j}^+ \\
a_{i,k+1}^+ & \mapsto & b_{ik}^+ \\
a_{k+1, k+1}^+ & \mapsto &  b_{kk}^+\\
a_{k+1,j}^+ & \mapsto &  b_{kj}^+ - \epsilon(c) \cdot b_{k+1,j}^+ \\
a_{k,k+1}^+ & \mapsto & 0 \\
c^+ & \mapsto & \left(\sum_{i < k} \epsilon(a_{ik}) \cdot b_{ik}^+ \right) - (-1)^{|c|}
\left( \sum_{k+1 < j} \epsilon'(a_{k+1,j}) \cdot b_{k+1,j}^+ \right)
\end{eqnarray*}
for $i<k$ and $j>k+1$, and $a_{ij}^+ \mapsto b_{ij}^+$ for $i,j \notin \{k, k+1\}$.

The second order term
$\Hom_{\Aug(\crossk, \mu)}(\epsilon', \epsilon'')\otimes
 \Hom_{\Aug(\crossk, \mu)}(\epsilon, \epsilon')  \to \Hom_{\Aug(\equiv, \mu_R)}(\epsilon_R, \epsilon''_R)$
is defined by
\begin{align*}
c^+ \otimes a_{ik}^+ &\mapsto (-1)^{|c^+||a_{ik}^+| + |a_{ik}^+| + 1}\, b_{ik}^+, && i<k\\
a_{k+1,j}^+ \otimes c^+ &\mapsto (-1)^{|a_{k+1,j}^+||c^+| + 1}\, b_{k+1,j}^+, && j>k+1\\
\end{align*}
with all other tensor products of generators mapped to zero.
There are no higher order terms.
\end{proposition}
\begin{proof}
The statement about restriction to the left is obvious.

Examining the $2$-copy of $\crossk$, we can write the map $\rho_R^2 : \alg^2(\equiv, \mu_R) \to \alg^2(\crossk, \mu)$ as
\begin{eqnarray*}
b_{ik}^{12} & \mapsto & a^{11}_{ik}c^{12} + a^{12}_{ik}c^{22} + a^{12}_{i,k+1} \\
b_{kk}^{12} & \mapsto & a^{12}_{k+1,k+1} \\
b_{kj}^{12} & \mapsto & a^{12}_{k+1,j} \\
b_{i,k+1}^{12} & \mapsto & a^{12}_{ik} \\
b_{k+1, k+1}^{12} & \mapsto & a^{12}_{kk} \\
b_{k+1,j}^{12} & \mapsto &  a^{12}_{kj} - (-1)^{|c|}\left(c^{11}a^{12}_{k+1,j} + c^{12}a^{22}_{k+1,j}\right) \\
b_{k,k+1}^{12} & \mapsto & 0
\end{eqnarray*}
for all $i<k$ and $j>k+1$, and $b_{ij}^{12}  \mapsto  a_{ij}^{12}$ when $i,j \notin \{k, k+1\}$.

Twisting the differential by $\epsilon=(\epsilon,\epsilon')$ and taking the linear part gives
\begin{eqnarray*}
b_{ik}^{12} & \mapsto & \epsilon(a_{ik})c^{12} + a^{12}_{ik}\epsilon'(c) + a^{12}_{i,k+1} \\
b_{kk}^{12} & \mapsto & a^{12}_{k+1,k+1} \\
b_{kj}^{12} & \mapsto & a^{12}_{k+1,j}  \\
b_{i,k+1}^{12} & \mapsto & a^{12}_{ik} \\
b_{k+1, k+1}^{12} & \mapsto & a^{12}_{kk} \\
b_{k+1,j}^{12} & \mapsto &  a^{12}_{kj} - (-1)^{|c|}\left(\epsilon(c)a^{12}_{k+1,j} + c^{12}\epsilon'(a_{k+1,j})\right) \\
b_{k,k+1}^{12} & \mapsto & 0
\end{eqnarray*}
again with $i<k$ and $j>k+1$, and $b^{12}_{ij} \mapsto a^{12}_{ij}$ otherwise.

We now recall that $(-1)^{|c|} \epsilon(c) = \epsilon(c)$ and take duals to conclude.

The higher order term in the restriction functor comes from writing the inclusion
of the three-copy of the line algebra into the crossing algebra, then taking linear duals.
Explicitly, this inclusion is
\begin{align*}
b_{ij}^{13} &\mapsto a_{ij}^{13} && (i,j \not\in \{k,k+1\}) \\
b_{ik}^{13} &\mapsto a_{i,k+1}^{13} + a_{ik}^{11}c^{13} + a_{ik}^{12}c^{23} + a_{ik}^{13}c^{33} && (i < k) \\
b_{i,k+1}^{13} &\mapsto a_{ik}^{13} && (i < k) \\
b_{kj}^{13} &\mapsto a_{k+1,j}^{13} && (j > k+1) \\
b_{k+1,j}^{13} &\mapsto a_{kj}^{13} - (-1)^{|c|}\big( c^{11}a_{k+1,j}^{13} + c^{12}a_{k+1,j}^{23} + c^{13}a_{k+1,j}^{33} \big) && (j > k+1) \\
b_{kk}^{13} &\mapsto a_{k+1,k+1}^{13} && \\
b_{k+1,k+1}^{13} &\mapsto a_{kk}^{13} && \\
b_{k,k+1}^{13} &\mapsto 0. &&
\end{align*}
Selecting the terms of the form $\ast^{12}\ast^{23}$ and dualizing,
we conclude that the only higher parts of the restriction functor
are the terms stated.
\end{proof}

Consider the general element
$\xi =  \gamma \cdot c^+ + \sum_{i\le j} \alpha_{ji} \cdot a_{ij}^+ \in \Hom_{\Aug(\crossk,\mu)}(\epsilon, \epsilon')$.
We want to compare more explicitly $\xi_L$ and $\xi_R$.
To do this, we move to the Morse complex category, and consider
$\hh(\xi_L)$ and $\hh(\xi_R)$.  Note that these come to us as matrices.
Below we often adopt the convention for indices that $i < k < k+1 < j$, and for convenience we define $\sigma_{pq} = (-1)^{(\mu(p)+1)\mu(q) + 1}$, so that in $\Aug(\equiv,\mu)$ we have $\hh(a_{pq}) = \sigma_{pq}|q\rangle\langle p|$.

We have:
\[
\hh(\xi_L) =
\left[ \begin{array}{c|cc|c}
\sigma_{i_2i_1} \alpha_{i_1 i_2} & 0 & 0 & 0 \\
\hline
\sigma_{ik}\alpha_{ki} & \sigma_{kk}\alpha_{kk} & 0 & 0 \\
\sigma_{i,k+1}\alpha_{k+1,i} & \sigma_{k,k+1}\alpha_{k+1,k} & \sigma_{k+1,k+1}\alpha_{k+1, k+1} & 0 \\
\hline
\sigma_{ij}\alpha_{ji} & \sigma_{kj}\alpha_{jk} & \sigma_{k+1,j}\alpha_{j,k+1} & \sigma_{j_2j_1}\alpha_{j_1 j_2}
\end{array} \right],
\]
where the signs are defined using the Maslov potential $\mu = \mu_L$ on the left.
On the other hand, by the above proposition we have
\begin{align*}
\hh(\xi_R) & = \sum_{i \le j \notin \{k, k+1\}} \sigma_{ij}\alpha_{ji} |j \rangle \langle i| \\
& \qquad + \sum_{i < k}  \sigma_{i,k+1}\alpha_{k+1, i} |k \rangle \langle i|  + \alpha_{ki} \left( \sigma_{ik}|k + 1 \rangle \langle i|  + \sigma_{i,k+1}\epsilon'(c) |k \rangle \langle i|  \right)
\\
& \qquad + \sum_{k+1 < j} \sigma_{kj}\alpha_{jk} |j \rangle \langle k+1 |+  \alpha_{j, k+1} \left( \sigma_{k+1,j}|j \rangle \langle k| - \sigma_{kj}\epsilon(c) |j  \rangle \langle k+1| \right)
\\
& \qquad + \sigma_{kk}\alpha_{kk} |k+1 \rangle \langle k+1| + \sigma_{k+1,k+1}\alpha_{k+1, k+1} |k \rangle \langle k| \\
& \qquad + \gamma \left(  \sum_{i < k} \sigma_{i,k+1}\epsilon(a_{ik}) |k \rangle \langle i |  - (-1)^{|c|} \sum_{k+1 < j} \sigma_{kj}\epsilon'(a_{k+1,j}) | j \rangle \langle k+1| \right),
\end{align*}
where the signs $\sigma_{pq}$ are defined again in terms of $\mu_L$ for consistency; recall that $\mu_R = \mu_L \circ s_k$.  In matrix form, we have
\[
\hh(\xi_R) =  \left[ \begin{array}{c|cc|c}
\sigma_{i_2i_1}\alpha_{i_1 i_2} & 0 & 0 & 0 \\
\hline
\sigma_{i,k+1}x  & \sigma_{k+1,k+1}\alpha_{k+1,k+1} & 0 & 0 \\
\sigma_{ik}\alpha_{ki} & 0 & \sigma_{kk}\alpha_{kk} & 0 \\
\hline
\sigma_{ij}\alpha_{ji} & \sigma_{k+1,j}\alpha_{j,k+1} & \sigma_{kj}y  & \sigma_{j_2j_1}\alpha_{j_1 j_2}
\end{array} \right]
\]
where $x=\alpha_{k+1,i} + \epsilon'(c) \alpha_{ki} + \epsilon(a_{ik}) \gamma$ and $y=\alpha_{jk} - \epsilon(c) \alpha_{j,k+1} - (-1)^{|c|} \epsilon'(a_{k+1,j}) \gamma$.
So
\begin{equation} \label{eq:hcross}
\begin{split}
&(\phi')^{-1} s_k \hh(\xi_R) s_k \phi = \\
&\left[ \begin{array}{c|cc|c}
\sigma_{i_2i_1}\alpha_{i_1 i_2} & 0 & 0 & 0 \\
\hline
\sigma_{ik}\alpha_{ki}  & \sigma_{kk}\alpha_{kk} & 0 & 0 \\
\sigma_{i,k+1}(\alpha_{k+1,i} + \epsilon(a_{ik}) \gamma)  &
\sigma_{k+1,k+1}\epsilon(c) \alpha_{k+1, k+1}  - \sigma_{kk}
\epsilon'(c) \alpha_{kk} & \sigma_{k+1,k+1} \alpha_{k+1,k+1} & 0 \\
\hline
\sigma_{ij} \alpha_{ji} & \sigma_{kj}(\alpha_{jk}   - (-1)^{|c|}
\epsilon'(a_{k+1,j}) \gamma) & \sigma_{k+1,j}\alpha_{j,k+1}   &
\sigma_{j_2 j_1} \alpha_{j_1 j_2}
\end{array} \right],
\end{split}
\end{equation}
where again the Maslov potentials are for the left and not for the right.

\begin{theorem}
\label{thm:crossmc}
We define a morphism of $A_\infty$ categories
\[ \hh:  \Aug(\crossk, \mu) \to MC(\crossk, \mu) \]
on objects by
\[ \epsilon  \mapsto  (\hh(\epsilon_L), -\epsilon(c)) \]
and on morphisms $\xi \in \Hom(\epsilon,\epsilon')$ by
\[ \xi \mapsto (\phi')^{-1} s_k  \hh(\xi_R) s_k \phi, \]
where $\phi = 1 + \epsilon(c)|k+1\rangle\langle k|$ and $\phi' = 1 + \epsilon'(c)|k+1\rangle\langle k|$.  This morphism is a bijection on objects and an equivalence of categories.
It commutes with restriction in the following sense:
\begin{itemize}
\item
At the level of objects, $\hh$ commutes with restriction: $\hh(\epsilon)_L = \hh(\epsilon_L)$ and
$\hh(\epsilon)_R = \hh(\epsilon_R)$.
\item
At the level of morphisms, it commutes with restriction to the right:  $\hh(\xi_R) = \hh(\xi)_R$.
\item
At the level of morphisms, it commutes up to homotopy with restriction on the left: \[ \hh(\xi_L)  - \hh(\xi)_L  = (dH + Hm_1) \xi, \]
where $H$ is the homotopy given by sending $c^+ \mapsto \sigma_{k,k+1}|k\rangle \langle k+1|$ and all other generators to zero, i.e.,
\begin{align*}
H: \Hom_{\Aug(\crossk, \mu)}(\epsilon, \epsilon') & \to \Hom_{MC(\equiv, \mu)}(\hh(\epsilon_L), \hh(\epsilon'_L)) \\
\eta & \mapsto (-1)^{(\mu_L(k)+1)\mu_L(k+1)+1} (\mathrm{Coeff}_{c^+} \eta)  |k+1\rangle \langle k|.
\end{align*}
\end{itemize}
Higher order terms are determined by noting that the functor is just the right restriction map of the augmentation category --- which has higher terms (see Proposition \ref{prop:crossres}) --- followed by the isomorphism of augmentation and Morse complex line categories.
\end{theorem}

\begin{proof}
Lemma \ref{lem:crossobj-left} implies that, on objects, the map is well defined and a bijection.  Comparison
of \eqref{eq:hcross} and Lemma \ref{lem:crosshom} reveals that $(\phi')^{-1} s_k  \hh(\xi_R) s_k \phi$ is in fact
a morphism in $MC(\crossk, \mu)$.  The map was built from the A-infinity $\epsilon \mapsto \epsilon_R$ by composing with
isomorphisms, so is an A-infinity morphism.  Comparison of Lemma \ref{lem:crossdiff} with Proposition \ref{prop:crossres}
shows that the kernel of the map $\xi \mapsto \xi_R$ is exactly the two-dimensional space spanned by
$a_{k,k+1}^+$ and $m_1(a_{k,k+1}^+)$; the same is true for $\xi \mapsto (\phi')^{-1} s_k  \hh(\xi_R) s_k \phi$.
Counting dimensions, this is surjective to homs in $MC(\crossk, \mu)$.  Thus we have a map surjective on the chain level
which kills an acyclic piece; it is thus an equivalence.

We next check that $\hh$ commutes with restriction on the right. At the level of objects, by Proposition~\ref{prop:crossresobs}, we have
$\epsilon_R = s_k(\phi^T)^{-1}\epsilon_L(A)\phi^T s_k$, whence by Theorem~\ref{thm:lineaugcategory},
\[
\hh(\epsilon_R) = (-1)^{\mu_R} s_k \phi \epsilon_L(A)^T \phi^{-1} s_k.
\]
On the other hand, since $\hh(\epsilon) = ((-1)^{\mu_L}\epsilon(A)^T,-\epsilon(c))$, we compute from Definition~\ref{def:MCcrossres} that
\[
\hh(\epsilon)_R = \theta_{-\epsilon(c)} (-1)^{\mu_L}\epsilon_L(A)^T \theta_{-\epsilon(c)}^{-1},
\]
where $\theta_z$ is the identity matrix except for the $2\times 2$ block determined by rows $k,k+1$ and columns $k,k+1$, which is
$\left[\begin{smallmatrix} -z& 1\\1 & 0 \end{smallmatrix}\right]$. Matrix calculations show that $s_k\phi = \theta_{-\epsilon(c)}$
and $(-1)^{\mu_R}\theta_{-\epsilon(c)} = \theta_{-\epsilon(c)}(-1)^{\mu_L}$ (for the latter, note that 
$\mu_L(k)=\mu_R(k+1)$ and $\mu_L(k+1) = \mu_R(k)$ must be equal if $\epsilon(c) \neq 0$), and so $\hh(\epsilon_R) = \hh(\epsilon)_R$. At the level of morphisms,
$\hh$ commutes with right restriction essentially by definition:
\[
\hh(\xi)_R = ((\phi')^{-1} s_k  \hh(\xi_R) s_k \phi)_R= \theta_{-\epsilon'(c)}(\phi')^{-1} s_k  \hh(\xi_R) s_k \phi\theta_{-\epsilon(c)}^{-1} = \hh(\xi_R).
\]

For restriction on the left, note that $\hh(\epsilon_L) = \hh(\epsilon)_L$ by definition. It remains to show that $\hh$ commutes up to homotopy with left restriction on morphisms.
From \eqref{eq:hcross}, we find that
\[ \hh(\xi_L)  - \hh(\xi)_L =
\left[ \begin{array}{c|cc|c}
0 & 0 & 0 & 0 \\
\hline
0  & 0 & 0 & 0 \\
- \sigma_{i,k+1}\epsilon(a_{ik}) \gamma & \sigma_{k,k+1}\alpha_{k+1,k} + \sigma_{kk}\epsilon'(c) \alpha_{k,k} - \sigma_{k+1,k+1}\epsilon(c) \alpha_{k+1,k+1}  & 0 & 0 \\
\hline
0 &  \sigma_{kj}(-1)^{|c|} \epsilon'(a_{k+1,j}) \gamma  & 0 & 0
\end{array} \right].
\]
On the other hand we calculate
\begin{align*}
dH(\xi) &= \sigma_{k,k+1}\gamma \cdot d|k+1\rangle\langle k| \\
&= \sigma_{k,k+1} \gamma (-1)^{\mu(k+1)+1} \left(\sum_{i<k} \epsilon(a_{ik})|k+1\rangle\langle i|
+ \sum_{j>k+1} \epsilon'(a_{k+1,j})|j\rangle\langle k|\right) \\
&= \sum_{i<k} (-\sigma_{i,k+1} \epsilon(a_{ik})\gamma)|k+1\rangle\langle i|
+ \sum_{j>k+1} \sigma_{kj}(-1)^{|c|} \epsilon'(a_{k+1,j})\gamma|j\rangle\langle k| \\
&= \left[ \begin{array}{c|cc|c}
0 & 0 & 0 & 0 \\
\hline
0  & 0 & 0 & 0 \\
-  \sigma_{i,k+1}\epsilon(a_{ik})\gamma  & 0  & 0 & 0 \\
\hline
0 &   \sigma_{kj}(-1)^{|c|} \epsilon'(a_{k+1,j})\gamma  & 0 & 0
\end{array} \right],
\end{align*}
where we use the fact that $\epsilon(a_{ik}) = 0$ unless $\mu(i)-\mu(k)=1$ and $\epsilon'(a_{k+1,j})=0$ unless $\mu(k+1)-\mu(j) = 1$; and
\[ \mathrm{Coeff}_{c^+} m_1(\xi) = \alpha_{k+1, k} - \alpha_{k+1, k+1} \epsilon(c) + \alpha_{kk} \epsilon'(c). \]
We note that if $\epsilon(c) \neq 0$ or $\epsilon'(c) \neq 0$ then $\mu(k)=\mu(k+1)$, so $\sigma_{k,k+1}\epsilon(c) = \sigma_{k+1,k+1}\epsilon(c)$ and $\sigma_{k,k+1}\epsilon'(c)=\sigma_{kk}\epsilon'(c)$; thus multiplying this last equation by $\sigma_{k,k+1}|k+1\rangle\langle k|$ yields
\[ H(m_1(\xi)) = \left( \sigma_{k,k+1}\alpha_{k+1, k} - \sigma_{k+1,k+1} \alpha_{k+1, k+1} \epsilon(c) + \sigma_{kk}\alpha_{kk} \epsilon'(c) \right) |k+1\rangle\langle k|. \]
Thus we conclude that $\hh(\xi_L) - \hh(\xi)_L = (dH + Hm_1)\xi$.
\end{proof}

\subsubsection{Right cusps}

We now consider a bordered plat ``$\succ$'' which is the front projection of a set of $n$ right cusps.
Near the left, it is $2n$ horizontal lines, which we number $1,2,\dots,2n$ from top to bottom,
and each pair $2k-1,2k$ is connected by a right cusp; we place a base point $\ast_k$ at this cusp and let $\sigma_k = 1$ if the plat is oriented downward at this cusp or $\sigma_k = -1$ if it is oriented upward.  We fix a Maslov potential $\mu$, which is determined
by its restriction to the left $\mu_L: \{1, \ldots, 2n\} \to \bZ$.  The right cusps enforce that $\mu_L(2k) + 1= \mu_L(2k-1)$.

The left co-restriction $$\iota_L: \alg(\equiv, \mu_L)  \to  \alg(\succ, \mu)$$ identifies $\alg(\equiv, \mu_L)$ with a
subalgebra of $\alg(\succ, \mu)$ with ${2n \choose 2}$ generators $a_{ij}$.  The algebra $\alg(\succ, \mu)$ has $n$
additional generators $x_1,\dots,x_n$ naming the cusps, as well as generators $t_1,t_1^{-1},\dots,t_n,t_n^{-1}$ corresponding to the base points.  That is, the generator $x_k$ corresponds to the right cusp
connecting points $2k-1$ and $2k$, and has grading $|x_k|=1$ and satisfies $\dd x_k = t_k^{\sigma_k}+a_{2k-1,2k}$.
This ensures that if $\epsilon$ is an augmentation of $\alg(\succ, \mu)$, then $\epsilon(x_k)=0$ and $\epsilon(a_{2k-1,2k}) = -\epsilon(t_k)^{\sigma_k}$ for all $k$; since $t_k$ is invertible, so is $\epsilon(a_{2k-1,2k})$.

\begin{proposition}
The restriction $\rho_L: \Aug(\succ, \mu) \to \Aug(\equiv, \mu_L)$ is strictly fully faithful and an injection on objects.
Its image consists of all $\epsilon: \alg(\equiv, \mu_L) \to \coeffs$ such that $\epsilon(a_{2k-1,2k}) \in \coeffs^\times$ for $1\leq k \leq n$.
\end{proposition}

\begin{proof}
Injectivity on objects follows from the fact that $\epsilon(x_k) = 0$; the characterization of the image follows from
the discussion immediately above the proposition.

To see full faithfulness, note that the $2$-copy of the plat $\succ$ contains no crossings where the overcrossing is on copy $1$ (the upper copy) and the undercrossing is on copy $2$ (the lower copy). Thus if $\e_1,\e_2$ are augmentations of $\succ$, then $\Hom(\e_1,\e_2)$ in $\Aug(\succ,\mu)$ and $\Hom(\rho_L(\e_1),\rho_L(\e_2))$ in $\Aug(\equiv,\mu_L)$ are both generated by the same generators $a_{ij}^+$ where $1\leq i\leq j\leq 2n$, and
$\rho_L$ is an isomorphism on Hom spaces.
\end{proof}

\begin{corollary}
The isomorphism $\hh:\Aug(\equiv, \mu_L) \to MC(\equiv, \mu_L)$ identifies $\Aug(\succ, \mu)$ with $MC(\succ, \mu)$.
\end{corollary}
\begin{proof} Compare the definition of $MC(\succ, \mu)$ to the above proposition.
\end{proof}

We define $\hh: \Aug(\succ, \mu) \to MC(\succ, \mu)$ to be this restriction.

\begin{corollary}
All objects in $\Aug(\succ,\mu)$ are isomorphic.
\end{corollary}
\begin{proof}
We saw this for the Morse complex category in Proposition \ref{prop:onlyonecusp}.
\end{proof}

\subsubsection{Left cusps}

Let ``$\prec$'' denote the front projection of a set of $n$ left cusps.
Near the right, it is $2n$ horizontal lines, which we number $1,2,\dots,2n$ from top to bottom,
and each pair $2k-1,2k$ is connected by a left cusp.  We fix a Maslov potential $\mu$, which is determined
by its restriction to the right $\mu_R: \{1, \ldots, 2n\} \to \bZ$.  The left cusps enforce that $\mu_R(2k) + 1= \mu_R(2k-1)$.

The algebra $\alg(\prec, \mu)$ is simply the ground ring $\coeffs$; and hence there is a unique augmentation
$\epsilon: \alg(\prec, \mu) = \coeffs \xrightarrow{\operatorname{id}} \coeffs$.

The right co-restriction $\iota_R: \alg(\equiv, \mu_R) \to \alg(\prec, \mu)$ is given by the formula
\[ a_{ij}  \mapsto \begin{cases} 1 & (i,j)=(2k-1,2k) \\ 0 & \mathrm{otherwise}. \end{cases} \]
The restriction $\epsilon_R$ of the augmentation $\epsilon$ is given by the same formula.

To determine the $A_\infty$ structure
\[ m_p: \hom(\epsilon,\epsilon) \otimes \dots \otimes \hom(\epsilon,\epsilon) \to \hom(\epsilon,\epsilon) \]
on $\Aug(\prec, \mu)$, we consider the $(p+1)$-copy of $T_L$.  Here the $k$-th left cusp (i.e.\, the one connecting points $2k-1$ and $2k$ on the line $R$) gives rise to
${p+1 \choose 2}$ generators $y^{ij}_k$, each corresponding to a crossing of the $i$-th copy over the $j$-th copy for $i < j$.

\begin{proposition}
The chain complex $\Hom_{\prec}(\epsilon,\epsilon)$ is freely generated by the degree zero elements $y_1^+,\dots,y_n^+$,
and has vanishing differential.  The only nonvanishing composition is
$m_2(y_k^+, y_k^+) = -y_k^+$.
\end{proposition}

\begin{proof}
In the case $p=1$ above, corresponding to the $2$-copy of $\prec$, it is clear that $\dd y_k^{12} = 0$ for all $k$; it follows that
the dualized linearized differential also vanishes.
We have $|y_k^+| = |y_k^{12}| + 1 = 0$.

For the composition $m_p$, we study the differential on $\alg^{p+1}(\prec, \mu)$, which is
\[ \dd y^{ij}_k = \sum_{i<l<j} y^{il}_k y^{lj}_k, \]
the dualization of which gives the stated product (note the sign from \eqref{eq:ms}) and no more.
\end{proof}

\begin{proposition} The restriction map is
\begin{eqnarray*}
\rho_L: \Hom_{\prec}(\epsilon, \epsilon) & \to & \Hom_{\equiv}(\epsilon_L, \epsilon_L) \\
y_k^+ & \mapsto & a_{2k-1,2k-1}^+ + a_{2k,2k}^+.
\end{eqnarray*}
\end{proposition}
\begin{proof}
The co-restriction map on the two-copies is
\begin{eqnarray*}
\iota_L:  \alg^2(\equiv, \mu_L) & \to & \alg^2(\prec, \mu) \\
a^{12}_{2k-1,2k-1} & \mapsto & y^{12}_k \\
a^{12}_{2k,2k} & \mapsto & y^{12}_k \\
a^{12}_{i \ne j} & \mapsto & 0.
\end{eqnarray*}
Dualizing gives the stated restriction map.
\end{proof}

\subsubsection{The augmentation category is a sheaf}

\begin{theorem} \label{thm:augisasheaf} Let $\Lambda$ be a front diagram with base points at all right cusps. 
Then the  presheaf of categories $\underline{\Aug}(\Lambda)$ is a sheaf.
\end{theorem}
\begin{proof}
Given Corollary \ref{cor:aug-sheaf}, it remains to check that sections have sufficiently many objects, which
can be checked using the condition of Lemma \ref{lem:strictify}.  On objects, the local morphisms to the Morse
complex category were literally isomorphisms, so we may check in the Morse complex category.  In Section \ref{MCissheaf}, we noted that the hypothesis of  \ref{lem:strictify} holds for the Morse complex category.
\end{proof}

\subsection{Local calculations in the sheaf category}
\label{sec:locSh}

In this section, we determine the local sheaf categories for the line, crossing,
left cusp, and right cusp diagrams.  We define the isomorphisms $\mathfrak{r}$ to the corresponding
local categories of $MC$ and study the compatibility with left and right restriction functors to $\equiv$, as in the diagram \eqref{eq:bigdiagram}.

In fact, for convenience we use a slight variant $MC'$ on $MC$.  The difference is only at the cusps, 
and is that $MC'(\prec)$ is the full subcategory on the object called $d_0$ in 
Proposition \ref{prop:onlyonecusp}, and similarly for $MC'(\succ)$.  
By the same proposition, the inclusion of this subcategory is a quasi-equivalence.  Correspondingly
the global sections of the sheafifications of $MC, MC'$ agree.  We only distinguish between $MC$ and $MC'$ in the discussion of cusps. 

\begin{remark} The diagram of 
categories $MC'$ {\em does not} satisfy Lemma \ref{lem:strictify}; however this is irrelevant here 
because we will not be interested in directly computing global sections of the associated sheaf
of categories.  It is possible to avoid the use of $MC'$, but the construction of sheaves 
associated to the other objects of $MC(\succ)$ which strictly respect the restriction map is somewhat
more involved (one adds some auxiliary vertical strata to allow `handle slides', which however are 
invisible from the point of view of the microsupport).  
\end{remark}

In this section, it is essential for the arguments we give that the coefficients $\field$ form a field; this is because
we borrow from the theory of quiver representations. It is however conceivable that Theorem \ref{thm:main}
may hold for more general coefficient rings.

\subsubsection{Lines}

Let $I = (0,1)$ be the unit open interval and define $\strip:= I\times \bR$.
Let 
$\equiv_n$ be
the Legendrian associated to the front diagram consisting of $n$ horizontal lines --- see Figure \ref{fig:4figs} (upper left).

Recall that $Sh_{\equiv_n}(\strip)$ denotes the category of sheaves on $\strip = I\times \bR$ with singular support
meeting infinity in a subset of $\equiv_n$.
Objects of $Sh_{\equiv_n}$ can be constructed from representations (in
chain complexes) of the $A_{n+1}$ quiver, with nodes indexed and arrows oriented as follows:
$$ \begin{array}{ccccccc}
n & & n-1 & & & & 0 \\
\bullet & \to & \bullet & \to & \cdots & \to & \bullet \end{array} $$
To a representation $R$ of this quiver, i.e., a collection of chain complexes $R_i$ and morphisms
$$ R_n \to R_{n-1} \to \cdots \to R_0, $$
we write $Sh_{\equiv_n}(R)$ for the sheaf which has $R_0$ as its stalk in the upper region,
$R_i$ as its stalk along the $i$'th line and in
the region below it; downward generization maps identities, and upward generization maps given by
the quiver representation.  In fact this construction is an equivalence from the derived category of
representations of the $A_{n+1}$ quiver to $Sh_{\equiv_n}(\strip)$.  (See \cite[Sec. 3]{STZ}; essential
surjectivity is a special case of \cite[Prop. 3.22]{STZ}.)

Here we will prefer $A_{n+1}$ representations of a certain canonical form.  We recall that quiver
representations admit two-term projective resolutions. Explicitly, the irreducible projectives of the $A_{n+1}$ quiver are:

$$P_i := 0 \to \cdots \to 0 \to \field \to \field \to \cdots \to \field \to \field,$$

\noindent i.e.~a copy of $\field$ at all nodes $k \ge i$.  We have $\Hom(P_i, P_j) = 0$ for $i < j$ and $\field$ otherwise, and
$\Ext^{\ge 1}(P_i, P_j) = 0$.

On the other hand, the indecomposables of $Rep(A_{n})$ are \cite{G}: 
$$S_{ij} := P_i /P_{j+1} = \quad 0\to ... \to 0 \to \field \to \field \to ... \to \field \to \field \to 0 \to ... \to 0,$$
i.e., a copy of $\field$ at all nodes $k$ with $i\leq k\leq j$ and all maps identities --- and zero elsewhere.
These are of course quasi-isomorphic to
$$S_{ij}' := \left( \begin{array}{c} P_{j+1} \\ \downarrow \\ P_{i} \end{array} \right) =
\left( \begin{array}{ccccccccccccccccc}
 0 & \to & \cdots & \to & 0 & \to & 0     & \to & \cdots & \to & 0     & \to & \field &  \to & \cdots & \to & \field  \\
 \downarrow &       &           &       & \downarrow  &  &  \downarrow   &        &       &           & \downarrow           &      & \downarrow       &       &            &      & \downarrow \\
 0 & \to & \cdots & \to & 0 & \to & \field & \to & \cdots & \to & \field & \to & \field &  \to & \cdots & \to & \field
 \end{array} \right),
 $$
i.e. zero for nodes $k<i$, $\field$ for nodes $i\leq k\leq j$ and the acyclic complex $[\field\to\field]$ for $k>j$.

Since $Rep(A_{n+1})$ has cohomological dimension one, objects in its derived category split, hence any
representation in chain complexes is quasi-isomorphic to one of the form $\bigoplus S_{ij}[s]$, hence
quasi-isomorphic to one of the form $\bigoplus S'_{ij}[s]$.  (This latter object is just the minimal projective
resolution of the original object.)  We summarize properties of these as follows:

\begin{lemma} \label{lem:injective}
Over a field, every representation $R$ in chain complexes of the $A_{n+1}$ quiver is quasi-isomorphic to a representation
$$R'_{n} \to \cdots \to R'_{i+1} \to R'_{i} \to \cdots \to R'_0$$
such that:
\begin{itemize}
\item The (vector space) quiver representation $R'$ in each cohomological degree is projective.
\item The maps $R'_{i+1} \to R'_{i}$ are injections on the graded vector spaces underlying the complexes.
\item The differential on $R'_{i}/R'_{i+1}$ is zero.
\end{itemize}
Above we employ the convention $R'_{n+1} = 0$.  Note in particular that
there is an isomorphism of underlying graded vector spaces
$$R'_j \cong \bigoplus_{i \ge j} R'_i/R'_{i+1} \cong \bigoplus_{i \ge j} H^*(\mathrm{Cone}(R'_{i+1} \to R'_{i})).$$
\end{lemma}
\begin{proof}
The above construction shows every object is quasi-isomorphic to some $\bigoplus S'_{ij}[s]$ where $i > j$. The result
follows from its validity for each $S'_{ij}$, which holds by inspection.
\end{proof}

We now relate this to the category $MC(\equiv;\mu)$.

\begin{corollary} \label{cor:MCtoA}
There is a morphism $MC(\equiv_n; \mu) \to Rep_{ch}(A_{n+1})$, given on objects by sending the object $(\Bmu; d)$ to the
$A_n$ quiver representation which has the dg vector space ${}^i \Bmu$ at the node $i$.  The maps
are just inclusion of filtration steps.  Homs of the quiver representations are literally equal to homs of the
Morse complexes.

This map is fully faithful, and surjective onto the objects of $Rep_{ch}(A_{n+1}, \field)$ which (1) satisfy the conditions
of Lemma \ref{lem:injective} and (2) satisfy $R_{i-1} / R_i = \field[-\mu(i)]$.  It is essentially surjective
onto the portion of $Rep_{ch}(A_{n+1}, \field)$ in which
$\mathrm{Cone}(R_i \to R_{i-1}) = \field[-\mu(i)]$.
\end{corollary}
\begin{proof}
Essential surjectivity follows from Lemma \ref{lem:injective}.
\end{proof}

We write $\cC_1(\equiv_n; \mu) \subset Sh_{\equiv_n}$ for the full subcategory whose objects have microlocal monodromy
dictated by the Maslov potential $\mu$ --- see Section \ref{sec:mumon}, or recall briefly
in this case that microlocal rank one means that the cone of the upward generization
map from the $i$-th line has rank one in degree $-\mu(i).$

\begin{corollary} \label{cor:rlines}
The functor $\mathfrak{r}: MC(\mu) \to \cC_1(\equiv; \mu)$ given by composing the functor of
Corollary \ref{cor:MCtoA} with the equivalence of \cite[Prop. 3.22]{STZ} is an equivalence.
\end{corollary}

\subsubsection{Crossings}

Fix $n\geq 2$ and let $\crossk$ be a bordered plat consisting of $n$ strands with a single crossing between strands
$k$ and $k+1$ in the infinite vertical strip $\strip = I\times \bR.$  Fix a Maslov potential $\mu$.
We write $\cC_1(\crossk, \mu) \subset Sh_{\crossk}(\strip)$ for the category
of microlocal rank $1$ sheaves with vanishing stalks for $z \ll 0$.

By restriction to the first and second halves of the interval $I$, a sheaf $F\in \cC_1(\crossk, \mu)$ restricts to
a pair of objects $F_L$ and $F_R$ of the corresponding $n$-line sheaf categories, each
microlocal rank one with respect to the induced Maslov potentials $\mu_L$ and $\mu_R$.
These are related by $\mu_R = \mu_L \circ s_k$, where $s_k$ is the transposition of strands $k$ and $k+1$.

It is possible to build a sheaf in $Sh_{\crossk}(\strip)$ out of the following data:

\begin{definition}
A $\crossk$ triple on $n$ strands is a diagram $L \leftarrow M \rightarrow R$
of representations of $A_{n+1}$ in chain complexes as below:
$$\xymatrix{\vdots&\vdots&\vdots\\
L_{k-2}\ar[u]\ar@{=}[r]&M_{k-2}\ar[u]&\ar@{=}[l]R_{k-2}\ar[u]\\
L_{k-1}\ar[u]\ar@{=}[r]&M_{k-1}\ar[u]&\ar@{=}[l]R_{k-1}\ar[u]\\
L_k\ar[u]&\ar[l]M_k\ar[u]\ar[r]&R_k\ar[u]\\
L_{k+1}\ar[u]\ar@{=}[r]&M_{k+1}\ar@{=}[r]\ar@{=}[u]&R_{k+1}\ar[u]\\
L_{k+2}\ar[u]\ar@{=}[r]&M_{k+2}\ar@{=}[r]\ar[u]&R_{k+2}\ar[u]\\
\vdots\ar[u]&\vdots\ar[u]&\vdots\ar[u]}$$
such that $Tot = [M_{k+1}\rightarrow L_k \oplus R_k \rightarrow M_{k-1}]$ is
acyclic.
\end{definition}

A $\crossk$ triple determines an element of $Sh_{\crossk}(\strip)$.
To build the corresponding sheaf, the stalk along the $i$'th line and in the region below is $L_i$ on the left,
$M_i$ in the middle, and $R_i$ on the right; for $i \ne k$ these are all just equal.  The downward generization
map is the identity, and the upward generization map is the one pictured.  Finally, $M_k$ is the stalk at the crossing
and in the region below.   We will write $Sh_{\crossk}(L \leftarrow M \rightarrow R)$ for the corresponding sheaf.
As a special case of of \cite[Prop. 3.22]{STZ},
every object of $Sh_{\crossk}(\strip)$ is quasi-isomorphic to some $Sh_{\crossk}(L \leftarrow M \rightarrow R)$.
We sharpen this result as follows:

\begin{lemma} \label{lem:lmr}
Every object of $Sh_{\crossk}(\strip)$ is quasi-isomorphic to some  $Sh_{\crossk}(L \leftarrow M \rightarrow R)$,
in which  $L, M, R$ satisfy the conclusion of Lemma \ref{lem:injective}.
\end{lemma}
\begin{proof}

Begin with an object $\cF \in Sh_{\crossk}(\strip)$; pass to the quasi-isomorphic
$Sh(L \leftarrow M \rightarrow R)$ provided by \cite[Prop. 3.22]{STZ}.  We may replace
with quasi-isomorphic choices $L', M', R'$ by Lemma \ref{lem:injective}; then there exist
corresponding maps in the derived category $L' \leftarrow M' \rightarrow R'$.
Since
$L', M', R'$ are projective resolutions, the maps
$L' \leftarrow M' \rightarrow R'$ can be chosen to be maps of chain complexes so that we have a diagram
\[
\xymatrix{L' \ar[d]^\cong_{\alpha_L} & M' \ar[r]^{g_R} \ar[l]_{g_L} \ar[d]^\cong_{\beta} & R' \ar[d]^\cong_{\alpha_R} \\
L & M \ar[r]^{f_R} \ar[l]_{f_L}  & R 
}
\]
commutative up to homotopy.  Next, choose homotopy operators $K_L:M' \rightarrow L$ and $K_R:M' \rightarrow R$ with
\[
f_L \circ \beta - \alpha_L \circ g_L = \partial_L K_L + K_L \partial_{M'}    \quad \mbox{and} \quad  f_R \circ \beta - \alpha_R \circ g_R = \partial_R K_R + K_R\partial_{M'}, 
\]  and consider the mapping cylinder $\mathit{Map}(\beta) = M' \oplus M'[-1] \oplus M$ which has differential $D(a, b, c) = (\partial_{M'}a - b, -\partial_{M'}b, \partial_Mc +\beta(b))$ and inclusions $i_1:M' \stackrel{\cong}{\rightarrow} \mathit{Map}(\beta)$ and $i_2:M \stackrel{\cong}{\rightarrow} \mathit{Map}(\beta)$  which are quasi-isomorphisms (since $\beta$ is a quasi-isomorphism).  We then arrive at a fully commutative diagram: 
\[
\xymatrixcolsep{8pc}\xymatrix{L' \ar[d]^\cong_{\alpha_L} & M' \ar[l]_{g_L} \ar[r]^{g_R} \ar[d]^\cong_{i_1} & R' \ar[d]^\cong_{\alpha_R} \\ L & \mathit{Map}(\beta) \ar[l]_{(\alpha_L \circ g_L) \oplus K_L \oplus f_L} \ar[r]^{(\alpha_R \circ g_R) \oplus K_R \oplus f_R}  & R  \\
L \ar[u]_\cong^{\mathit{id}} & M \ar[l]_{f_L} \ar[r]^{f_R} \ar[u]_\cong^{i_2} & R  \ar[u]_\cong^{\mathit{id}}
}
\]

It remains to show that the maps $L'_i \leftarrow M'_i \rightarrow R'_i$ are (not just quasi-)isomorphisms
for $i \ne k$.  For $i \ne k, k-1$, we have the following maps of exact sequences of complexes:

$$\xymatrix{
0  & 0 & 0 \\
L'_{i}/L'_{i+1} \ar[u] & M'_{i} / M'_{i+1} \ar[u] \ar[l]_{\sim} \ar[r]^{\sim}  & R'_{i} / R'_{i+1}\ar[u]\\
L'_{i}\ar[u] & M'_{i}\ar[u]  \ar[r] \ar[l] & R'_{i}\ar[u]\\
L'_{i+1}\ar[u] &M'_{i+1}\ar[u] \ar[r] \ar[l] &   R'_{i+1}\ar[u]\\
0 \ar[u] & 0 \ar[u] & 0 \ar[u] }
$$

All horizontal maps are quasi-isomorphisms because this was true for the original $L, M, R$,
but now by construction the $L'_{i}/L'_{i+1}, M'_{i} / M'_{i+1}, R'_{i} / R'_{i+1}$
are  isomorphic to their cohomologies, hence the maps in the top row  are isomorphisms.  Thus if the arrows
$L'_{i+1} \leftarrow M'_{i+1} \rightarrow R'_{i+1}$ are isomorphisms, so are the
$L'_{i} \leftarrow M'_{i} \rightarrow R'_{i}$.

We also have

$$\xymatrix{
0  & 0 & 0 \\
L'_{k-1}/L'_{k+1} \ar[u] & M'_{k-1} / M'_{k+1} \ar[u] \ar[l]_{\sim} \ar[r]^{\sim}  & R'_{k-1} / R'_{k+1}\ar[u]\\
L'_{k-1}\ar[u] & M'_{k-1}\ar[u]  \ar[r] \ar[l] & R'_{k-1}\ar[u]\\
L'_{k+1}\ar[u] &M'_{k+1}\ar[u] \ar[r] \ar[l] &   R'_{k+1}\ar[u]\\
0 \ar[u] & 0 \ar[u] & 0 \ar[u] }
$$

All horizontal maps are quasi-isomorphisms because the same was true for $L, M, R$.
By construction, $M'_{k-1} / M'_{k+1} = M'_{k-1} / M'_k$ is isomorphic to its cohomology.
The only way that $L'_{k-1} / L'_{k+1}$ or $R'_{k-1}/R'_{k+1}$ could fail to have the same property is if they contained
a summand which were equal to a shift of the object $[P_{k+1} \to P_{k-1}]$.  However the sheaf corresponding
to this summand -- namely the constant sheaf stretching between the $k$-th and $(k+1)$-st strands --
violates the singular support condition at the crossing, so it cannot appear.
We conclude that
$L'_{k-1} / L'_{k+1}$ and $R'_{k-1}/R'_{k+1}$ are isomorphic to their cohomologies, hence that the maps in the top
row are isomorphisms.  Thus, if
the maps $L'_{k+1} \leftarrow M'_{k+1} \rightarrow R'_{k+1}$ are isomorphisms, then so too are
$L'_{k-1} \leftarrow M'_{k-1} \rightarrow R'_{k-1}$.

By induction, we conclude that $L'_i \leftarrow M'_i \rightarrow R'_i$ are isomorphisms for all $i \ne k$.
\end{proof}

We now relate this to the category $MC(\crossk, \mu)$.  An element of this category is a differential $d: \Bmu_L \to \Bmu_L$
and an element $z \in \field$, from
which we built an identification $\theta_z: \Bmu_L \to \Bmu_R$ such that $\theta_z \circ d \circ (\theta_z)^{-1} \in MC(\equiv, \mu_R)$.
We build a $L \leftarrow M \rightarrow R$ triple by setting $L_k = {}^k \Bmu_L$ and $R_k = {}^k \Bmu_R$; the Hom spaces in
$MC(\crossk, \mu)$ can be evidently interpreted as maps between these diagrams of quiver representations.  As in
Corollary \ref{cor:rlines}, we can define a functor $\mathfrak{r}: MC(\crossk, \mu) \to \cC_1(\crossk, \mu)$ by composing
this with the equivalence of \cite[Prop. 3.22]{STZ}.

\begin{proposition}
The functor $\mathfrak{r}: MC(\crossk, \mu) \to \cC_1(\crossk, \mu)$ is an equivalence which commutes with
the restriction maps.
\end{proposition}
\begin{proof}
Essential surjectivity follows from Lemma \ref{lem:lmr}.  The equivalence commutes with restrictions by construction.
\end{proof}

\subsubsection{Cusps}

Let ``$\succ$'' be the right-cusp diagram on $2n$ strands, carrying a Maslov potential $\mu$.
Let $sh_\succ(\bR^2, \coeffs)_0 \subset sh_\succ(\bR^2, \coeffs)$ be the full subcategory of
sheaves with acyclic stalk to the far right.

Let $V_i$ be the locally closed subsets comprised of the upper stratum and interior
region of the $i$'th cusp (numbered increasing from top to bottom, so that the $i$'th cusp connects
the $2i-1, 2i$ strands.  Let $v: \coprod V_i \to \bR^2$ be the inclusion.

\begin{lemma}
Every object of $sh_\succ(\bR^2, \coeffs)_0$
is quasi-isomorphic to the extension by zero of a locally constant sheaf on $V$.  
\end{lemma}
\begin{proof}
The microsupport condition translates directly into the constraint that the sheaf is locally constant
on $V$. 

We recall in general that 
for the inclusion of a locally closed subset $s: S \to T$, there is the extension by zero 
functor $s_!: Sh(S) \to Sh(T)$ with the property that 
$$s_! \cF(U) = \begin{cases} \cF(U) & U \subset S \\ 0 & otherwise \end{cases}$$
The properties of this functor can be found in any standard reference, e.g. \cite[Chap. 2]{KS}, 
and it is always true that sheaves which have zero stalks in the complement of a locally
closed subset are extensions by zero under the inclusion.  
\end{proof}

\begin{corollary}
Choose one point in each component of $V$, and consider the corresponding map 
$sh_\succ(\bR^2, \coeffs) \cong (\coeffs-mod)^n$ given by taking stalks.  It is a quasi-equivalence.
\end{corollary} 
\begin{proof}
The extension by zero is fully faithful, so it suffices to restrict attention to the 
locally constant sheaves on $V$ itself.  Since $V$ is a union of contractible components, taking
one stalk at each defines an quasi-equivalence of categories. 
\end{proof}

\begin{corollary}
Consider the map $\textit{left}: sh_\succ(\bR^2, \coeffs) \to sh_\equiv(\bR^2, \coeffs)$ given by restriction 
to a neighborhood of the left edge.  It is fully faithful, and has essential image the category $S$ 
of sheaves with acyclic stalks except between lines $2i-1, 2i$. 
\end{corollary}
\begin{proof}
By the same reason as the previous corollary, the map
$S \to (\coeffs-mod)^n$ given by taking one stalk in each component is an equivalence.  We can
factor the map of the previous corollary as 
$sh_\succ(\bR^2, \coeffs) \to S \subset sh_\equiv(\bR^2, \coeffs) \to (\coeffs-mod)^n$ by choosing points
for stalks near the left edge; it follows that the map $sh_\succ(\bR^2, \coeffs) \to S$ is an equivalence. 
\end{proof} 

\begin{corollary} 
The category $\cC_1(\succ, \mu; \coeffs)$ is empty
unless $\mu_{2k} = \mu_{2k-1} - 1$.
In this case $\cC_1(\succ, \mu; \coeffs)$ contains
up to isomorphism a unique object, whose stalks in the cusp regions are $\coeffs[-\mu_2], \coeffs[-\mu_4], \ldots, \coeffs[-\mu_{2n}]$.
\end{corollary}
\begin{proof}
We calculate the microstalks along the top strand of each cusp.  The sheaf on $V$ is locally constant; 
let the stalks in the $n$ cusp regions be $V_1, \ldots, V_n$.  Recalling the correspondence between
Maslov potential and degree of microstalk, we should have $V_i [1] = Cone(V_i \to 0) = \coeffs[\mu_{2i-1}]$ 
and $V_i = Cone(0 \to V_i) = \coeffs[\mu_{2i}]$.  
\end{proof}

Note that when nonempty, $\cC_1(\succ, \mu)$ contains a canonical object, namely the sheaf
which is {\em the} constant sheaf on each component on $V$, with an appropriate shift.  As follows
from the above, its endomorphisms are canonically the ring $k^n$.  

Recall that we write $MC'(\succ, \mu)$ for the full subcategory of $MC(\succ, \mu)$ containing
only the object $d_0$. 

\begin{proposition}
There is a commutative diagram with vertical diagrams equivalences
\begin{equation}
\label{eq:bigdiagram2}
 \xymatrix{MC(\equiv, \mu_L) \ar[d]  & \ar[l]^{\rho_L}  MC'(\succ, \mu) \ar[d]  \\
 \cC_1(\equiv, \mu_L) & \ar[l]^{\rho_L^\cC} \cC_1(\succ, \mu) \\
} 
\end{equation}
The left vertical arrow is that constructed in Corollary \ref{cor:MCtoA}. 
If the right categories are nonempty, the right vertical arrow sends the unique object $d_0$ of 
$MC'(\succ, \mu)$ to the canonical element of  $\cC_1(\succ, \mu)$. 
\end{proposition}
\begin{proof}
We have already seen that the horizontal arrows are fully faithful, and the left vertical arrow
is an equivalence.  We define the right vertical arrow through the corresponding fully
faithful embeddings.  (There is in any case no mystery about this arrow, both $d_0$ and the 
canonical element of $\cC_1(\succ, \mu)$ have endomorphisms $\coeffs^n$, where
the $i$'th component is canonically associated to the $i$'th cusp.)   
\end{proof}

\begin{remark}
We note that in the proposition we did not say whether we ask for homotopy commutativity or 
strict commutativity.  In fact it is irrelevant for our purposes: as we only ever consider maps
of linear diagrams of categories, and only check that these determine quasi-isomorphisms
by checking termwise (as opposed e.g. to trying to compose morphisms of diagrams), no higher
homotopical questions ever arise, so knowing the above square commutes up to some 
unspecified homotopy suffices.  On the other hand, by tracing through exactly how we 
associate a sheaf to an object of $\cC_1(\equiv, \mu_L)$, it is not difficult to describe the homotopy
explicitly. 
\end{remark}

The analogous statement holds for left cusps.

\subsection{Augmentations are sheaves}

Recall we assume that the right cusps are all equipped with base points.  We showed in 
in Theorem \ref{thm:augisasheaf} that under this hypothesis, the presheaf $\Aug$ is in fact a sheaf.  We did this by
verifying the hypothesis of Lemma \ref{lem:strictify} by explicitly computing the restriction maps, and then applying Proposition~\ref{prop:linesummary}.  (The necessity of doing
this was explained in Remark \ref{rem:notsheaf}.) 

A morphism of sheaves can be given by giving morphisms on
all sufficiently small open sets, compatibly with restriction.  The morphism may be checked to be an equivalence also on these sets.
We thusly defined morphisms
$\hh: \Aug \to MC$ and $\mathfrak{r}: MC \to \cC_1$, and showed
 each was an equivalence. 
In particular, we obtain an isomorphism of global sections $R\Gamma(\mathfrak{r} \hh) : \Aug(\Lambda) \cong \cC_1(\Lambda)$. 
 This completes the proof of Theorem \ref{thm:main}.

\newpage

\section{
Some exact sequences}
\label{sec:exseq}

This paper has established a host of relations among categories of sheaves, Lagrangians and augmentations.
Here we briefly discuss the Fukaya-theoretic viewpoint and gather the relationships in the
unifying Theorem
\ref{thm:bigexactsequence} below.

Let $X$ be a compact real analytic manifold.  Equip $T^*X$ with its canonical exact structure $\omega = -d\theta.$
Recall the infinitesimally wrapped Fukaya category $Fuk_\varepsilon(T^*X)$
from \cite{NZ}.  Its objects
are exact Lagrangian submanifolds equipped with local systems, brane structures, and perturbation data.
Morphisms of $Fuk_\varepsilon(T^*X)$, including higher morphisms, involving objects $L_1,\ldots,L_d$
are constructed by perturbing the Lagrangians, using a contractible fringed set $R_d\subset \bR_{>0}^d$ to organize
the perturbations.
A fringed set $R_d$ of dimension $d$ is a subset of $\bR_{>0}^d$ satisfying conditions defined inductively:
if $d=1,$ $R_1 = (0,r)$; if $d>1,$ then the projection of $R_d$ to the first $d-1$ coordinates is a fringed set,
and $(r_1,\ldots,r_d)\in R_d\Rightarrow (r_1,\ldots,r_d')\in R_d$ for all $0<r_d' < r_d.$
Loosely, to compute $\hfp(L_1,L_2)$ we must perturb $L_2$ more than $L_1;$
to compute compositions from
$\hfp(L_{d-1},L_d)\otimes \hfp(L_{d-2},L_{d-1})\otimes \cdots \otimes \hfp(L_1,L_2)$ and others, we perturb
by
\begin{equation}
\label{fringe}
\varepsilon_d > \varepsilon_{d-1} > \cdots > \varepsilon_1>0.
\end{equation}
The $d$-tuple of successive differences $\delta = (\varepsilon_1,\varepsilon_2-\varepsilon_1,\ldots,\varepsilon_d-\varepsilon_{d-1})$ lies in the
fringed set, $R_d$.\footnote{In fact the condition $\varepsilon_1>0$ is not necessary.
Only the relative positions of the perturbations is essential.}  The purpose of introducing this set $R_d$ is two-fold:
first, the perturbations bring intersections from infinity to finite space, so that the moduli spaces
defining compositions are compact; second, by perturbing in the Reeb direction
at infinity, morphisms compose as required for the isomorphism with the category of constructible sheaves.
So for the purposes of simply \emph{defining} a category, we can ignore the second of these purposes.
This leaves us with another choice of contractible set organizing the compositions.
We simply \emph{reverse} all the inequalities in \eqref{fringe} and \emph{negate} the definition of $\delta$ ---
it will then lie in a fringed set.
We call the category defined in this way the negatively wrapped Fukaya category, $Fuk_{-\varepsilon}(T^*X).$

Let us be a bit more specific and compare the two possibilities.
Recall from \cite{NZ} that we call a function $H:T^*X\to \bR$ a controlled Hamiltonian if $H(x,\xi) = |\xi|$ outside a compact set;
now let $\varphi_{H,t}$ denote Hamiltonian flow by $H$ for time $t.$
To compute the hom complex $\hfp(L,L')$, first choose controlled Hamiltonians
$H, H'$ and a fringed set
$R_2$ such that for all $\delta = (\varepsilon,\varepsilon'-\varepsilon)\in R_2$
we have $\varphi_{H,\varepsilon}(L)\cap
\varphi_{H',\varepsilon'}$ is transverse and contained in a compact subset of $T^*X.$
Now put $L_+ = \varphi_{H,\varepsilon}(L),$ $L'_+ = \varphi_{H',\varepsilon'}(L')$ (we suppress
the dependence on $\varepsilon,\varepsilon'$.
Then $\hfp(L,L')$ is defined by computing the Fukaya-Floer complex of the pair $(L_+,L'_+)$, counting
holomorphic strips in the usual way.
Alternatively, to study $\hfm(L,L')$ we choose controlled Hamiltonians $(H,H')$ and a fringed
set $R_2$ and require that for all $\delta = (\varepsilon,\varepsilon'-\varepsilon)$ in $R_2$,
$\varphi_{H,-\varepsilon}(L)\cap
\varphi_{H',-\varepsilon'}$ is transverse and contained in a compact subset of $T^*X.$
Then put $L_- = \varphi_{H,-\varepsilon}(L),$ $L'_- = \varphi_{H',-\varepsilon'}(L')$ and define
$\hfm(L,L')$ by the usual count of holomorphic strips.  Higher-order compositions in $UF_-$ are defined
exactly analogously to those in $Fuk_{\varepsilon}.$

\begin{remark}
$Fuk_{-\varepsilon}$ is not simply the opposite category of $Fuk_{\varepsilon}$, as no change has been made regarding the intersections
between Lagrangians which appear in compact space.  In particular,
reversing the order of the Lagrangians would have changed
the degrees of those intersections.
\end{remark}

When $X$ is not compact, we require that Lagrangian branes have compact image in $X$
or are the zero section outside a compact set.
With this set-up, the following lemma is then true by definition.  Let $L,L_+,L_-$ be as above.

\begin{lemma}
We have
$$\hfp(L,L_+)\cong \hfp(L,L)\cong \hfp(L_{-},L).$$
\end{lemma}
\noindent Note the symplectomorphism $\varphi_{H,\varepsilon}$ gives an identification
$\hfp(L_-,L)\cong \hfp(L,L_+).$
Further, each of these spaces contains an element isomorphic
to the identity of the middle term, and we denote them
respectively by $id_+,id,id_-.$

\medskip

\begin{lemma}
\label{lem:sslemma}
Let $M$ be a real analytic manifold and let $\coeffs$ be a field;
let $X = M\times \bR_z$, let
$F\in Sh_c(X;\coeffs)$
correspond to $L$ above under the microlocalization equivalence \cite{N,NZ},
and let $F_+, F_-$ correspond to $L_+,L_-$.  Then the following
quasi-isomorphisms also hold due to microlocalization:
$$\Hom_{Sh}(F,F_+) \cong \Hom_{Sh}(F,F) \cong \Hom_{Sh}(F_-,F).$$
\end{lemma}

Let $\Lambda\in J^1(\bR_x) \subset T^{\infty,-}(\bR_x\times \bR_z)$ be a Legendrian knot (or link)
with front diagram basepointed at all right cusps and with Maslov potential $\mu$.
First recall that from \cite{STZ} and Theorem \ref{thm:main} of
the present paper we have the following triangle of equivalences:
$$\xymatrix{Fuk_{\varepsilon}(T^*\bR^2,\Lambda,\mu;\coeffs)&&\Aug(\Lambda,\mu;\coeffs)\ar[ll]\ar[dl]_\cong^\psi\\&\cC_1(\Lambda,\mu;\coeffs)\ar[ul]_\cong^\mu}$$
The arrow across the top is defined to be the composition,
and as usual $\cC_1(\Lambda,\mu;\coeffs) \subset Sh(\bR^2,\Lambda,\mu;\coeffs)$
denotes the full subcategory of microlocal rank-one objects, as determined by $\mu.$

Now let $\Lambda \subset J^1(\bR)\subset T^{\infty,-}\bR^2$ be a Legendrian knot and
let $\mu$ be a Maslov potential.
Let $\epsilon\in \Aug(\Lambda,\mu;\coeffs)$ be an augmentation.  Let $F \in \cC_1(\Lambda,\mu;\coeffs)$ correspond to $\e$
under Theorem \ref{thm:main} and let $L \in Fuk_\varepsilon(T^*\bR^2,\Lambda;\coeffs)$
be a geometric Lagrangian object corresponding to $F$.  (Not all
such $L$ will be geometric.)
Write $\cL = \mu mon F$ for the microlocal monodromy local system, defined from
the Maslov potential $\mu$
(though note $End(\cL)_{Loc(\Lambda)}$ is canonical).
Let us denote for the moment $(A,B)_{\cC} := \Hom_{\cC}(A,B).$
Then we have the following.

\begin{theorem}
\label{thm:bigexactsequence}
$$\xymatrix{(L,L_-)_{Fuk_\varepsilon}\ar[r]^{\circ\,id_-}& (L,L)_{Fuk_\varepsilon} \ar[r]&{\sf Cone}(\circ\, id_-)\\
(F,F_-)_{Sh}\ar[r]^{\circ\,id_-}\ar[u]_{\cong}^\mu&(F,F)_{Sh}\ar[r]^{}\ar[u]_{\cong}^\mu
&{\sf Cone}(\circ\, id_-)\ar[u]_{\cong}^\mu&
\!\!\!\!\!\!\!\!\!\!\!\!\!\!\cong (\cL,\cL)_{Loc(\Lambda)}\ar@/^2pc/@{<->}[ddl]_\cong\\
(\epsilon,\epsilon)_{\mathcal Aug_-}\ar[r]^{can}\ar[d]^{\cong}_{\rho}\ar[u]_{\cong}^\psi&
(\epsilon,\epsilon)_{\Aug} \ar[u]_\cong^\psi \ar[r] \ar[d]^{\cong}_{\rho}&
{\sf Cone}(can)\ar[d]^{\cong}_\rho\ar[u]_\cong^\psi\\
C^*_c(L)\ar[r]^{\hookrightarrow}&C^*(L)\ar[r]&C^*(\Lambda)}$$
Here $\mu$ is short for the microlocalization theorem, which is a triangulated equivalence,
ensuring the isomorphism of cones.
Further, $\psi$ is the
isomorphism $Aug_+(\Lambda,\mu;\coeffs)\to \cC_1(\Lambda,\mu;\coeffs)$ proved in
Theorem \ref{thm:main},
and $\rho$ in the bottom row of vertical arrows indicates the isomorphism
proved in Proposition \ref{prop:plusminusexseq}.
The map ``$can$\!'' is the inclusion of \dgas{}
and the map $\hookrightarrow$ is inclusion of compactly supported forms.
Taking cohomology relates the rows to the long exact sequence
$H^*_c(L)\to H^*(L)\to H^*(\Lambda)\to$.\end{theorem}
\begin{proof}
The top line of vertical arrows is microlocalization \cite{NZ,N}.  The middle
line is Theorem \ref{thm:main}.  The bottom line is proven in Proposition \ref{prop:plusminusexseq}.
\end{proof}


\newpage

\bibliographystyle{alpha}
\bibliography{aug-references}

\newcommand{\etalchar}[1]{$^{#1}$}
\begin{thebibliography}{CKE{\etalchar{+}}11}

\bibitem[BC14]{BC}
Fr{\'e}d{\'e}ric Bourgeois and Baptiste Chantraine.
\newblock Bilinearized {L}egendrian contact homology and the augmentation
  category.
\newblock {\em J. Symplectic Geom.}, 12(3):553--583, 2014.

\bibitem[BEE12]{BEE}
Fr{\'e}d{\'e}ric Bourgeois, Tobias Ekholm, and Yasha Eliashberg.
\newblock Effect of {L}egendrian surgery.
\newblock {\em Geom. Topol.}, 16(1):301--389, 2012.
\newblock With an appendix by Sheel Ganatra and Maksim Maydanskiy.

\bibitem[Bou09]{Bourgeois}
Fr{\'e}d{\'e}ric Bourgeois.
\newblock A survey of contact homology.
\newblock In {\em New perspectives and challenges in symplectic field theory},
  volume~49 of {\em CRM Proc. Lecture Notes}, pages 45--71. Amer. Math. Soc.,
  Providence, RI, 2009.

\bibitem[Cha15]{Cha-exact}
Baptiste Chantraine.
\newblock A note on exact {L}agrangian cobordisms with disconnected
  {L}egendrian ends.
\newblock {\em Proc. Amer. Math. Soc.}, 143(3):1325--1331, 2015.

\bibitem[Che02]{C}
Yuri Chekanov.
\newblock Differential algebra of {L}egendrian links.
\newblock {\em Invent. Math.}, 150(3):441--483, 2002.

\bibitem[CKE{\etalchar{+}}11]{CEKSW}
Gokhan Civan, Paul Koprowski, John Etnyre, Joshua~M. Sabloff, and Alden Walker.
\newblock Product structures for {L}egendrian contact homology.
\newblock {\em Math. Proc. Cambridge Philos. Soc.}, 150(2):291--311, 2011.

\bibitem[CN13]{atlas}
Wutichai Chongchitmate and Lenhard Ng.
\newblock An atlas of {L}egendrian knots.
\newblock {\em Exp. Math.}, 22(1):26--37, 2013.

\bibitem[DR16]{DR}
Georgios Dimitroglou~Rizell.
\newblock Lifting pseudo-holomorphic polygons to the symplectisation of
  {$P\times\Bbb{R}$} and applications.
\newblock {\em Quantum Topol.}, 7(1):29--105, 2016.

\bibitem[Dri04]{D}
Vladimir Drinfeld.
\newblock D{G} quotients of {DG} categories.
\newblock {\em J. Algebra}, 272(2):643--691, 2004.

\bibitem[EENS13]{EENS}
Tobias Ekholm, John~B. Etnyre, Lenhard Ng, and Michael~G. Sullivan.
\newblock Knot contact homology.
\newblock {\em Geom. Topol.}, 17(2):975--1112, 2013.

\bibitem[EES05a]{EES-nonisotopic}
Tobias Ekholm, John Etnyre, and Michael Sullivan.
\newblock Non-isotopic {L}egendrian submanifolds in {$\mathbb R^{2n+1}$}.
\newblock {\em J. Differential Geom.}, 71(1):85--128, 2005.

\bibitem[EES05b]{EES-ori}
Tobias Ekholm, John Etnyre, and Michael Sullivan.
\newblock Orientations in {L}egendrian contact homology and exact {L}agrangian
  immersions.
\newblock {\em Internat. J. Math.}, 16(5):453--532, 2005.

\bibitem[EES09]{EESab}
Tobias Ekholm, John~B. Etnyre, and Joshua~M. Sabloff.
\newblock A duality exact sequence for {L}egendrian contact homology.
\newblock {\em Duke Math. J.}, 150(1):1--75, 2009.

\bibitem[EGH00]{EGH}
Y.~Eliashberg, A.~Givental, and H.~Hofer.
\newblock Introduction to symplectic field theory.
\newblock {\em Geom. Funct. Anal.}, (Special Volume, Part II):560--673, 2000.
\newblock GAFA 2000 (Tel Aviv, 1999).

\bibitem[EHK16]{EHK}
Tobias Ekholm, Ko~Honda, and Tam\'{a}s K\'{a}lm\'{a}n.
\newblock Legendrian knots and exact {L}agrangian cobordisms.
\newblock {\em J. Eur. Math. Soc. (JEMS)}, 18(11):2627--2689, 2016.

\bibitem[Ekh12]{Ekh09}
Tobias Ekholm.
\newblock Rational {SFT}, linearized {L}egendrian contact homology, and
  {L}agrangian {F}loer cohomology.
\newblock In {\em Perspectives in analysis, geometry, and topology}, volume 296
  of {\em Progr. Math.}, pages 109--145. Birkh\"auser/Springer, New York, 2012.

\bibitem[EL17]{ekholm-lekili}
Tobias Ekholm and Yank{\i} Lekili.
\newblock Duality between {L}agrangian and {L}egendrian invariants.
\newblock arXiv:1701.01284, 2017.

\bibitem[Eli98]{Eli}
Yakov Eliashberg.
\newblock Invariants in contact topology.
\newblock In {\em Proceedings of the {I}nternational {C}ongress of
  {M}athematicians, {V}ol. {II} ({B}erlin, 1998)}, number Extra Vol. II, pages
  327--338, 1998.

\bibitem[EN15]{ekholm-ng}
Tobias Ekholm and Lenhard Ng.
\newblock Legendrian contact homology in the boundary of a subcritical
  {W}einstein 4-manifold.
\newblock {\em J. Differential Geom.}, 101(1):67--157, 2015.

\bibitem[ENS02]{ENS}
John~B. Etnyre, Lenhard~L. Ng, and Joshua~M. Sabloff.
\newblock Invariants of {L}egendrian knots and coherent orientations.
\newblock {\em J. Symplectic Geom.}, 1(2):321--367, 2002.

\bibitem[FHT01]{FHT}
Yves F{\'e}lix, Stephen Halperin, and Jean-Claude Thomas.
\newblock {\em Rational homotopy theory}, volume 205 of {\em Graduate Texts in
  Mathematics}.
\newblock Springer-Verlag, New York, 2001.

\bibitem[Gab72]{G}
Peter Gabriel.
\newblock Unzerlegbare {D}arstellungen. {I}.
\newblock {\em Manuscripta Math.}, 6:71--103; correction, ibid. 6 (1972), 309,
  1972.

\bibitem[GJ90]{GJ}
Ezra Getzler and John D.~S. Jones.
\newblock {$A_\infty$}-algebras and the cyclic bar complex.
\newblock {\em Illinois J. Math.}, 34(2):256--283, 1990.

\bibitem[GKS12]{GKS}
St{\'e}phane Guillermou, Masaki Kashiwara, and Pierre Schapira.
\newblock Sheaf quantization of {H}amiltonian isotopies and applications to
  nondisplaceability problems.
\newblock {\em Duke Math. J.}, 161(2):201--245, 2012.

\bibitem[Gui]{Guillermou}
St\'ephane Guillermou.
\newblock Sheaves and symplectic geometry of cotangent bundles.
\newblock arXiv:1905.07341.

\bibitem[Hen11]{Henry2011}
Michael~B. Henry.
\newblock Connections between {F}loer-type invariants and {M}orse-type
  invariants of {L}egendrian knots.
\newblock {\em Pacific J. Math.}, 249(1):77--133, 2011.

\bibitem[HR14]{HR14}
Michael~B. Henry and Dan Rutherford.
\newblock Equivalence classes of augmentations and {M}orse complex sequences of
  {L}egendrian knots.
\newblock arXiv:1407.7003, 2014.

\bibitem[JT17]{JinTreumann}
Xin Jin and David Treumann.
\newblock Brane structures in microlocal sheaf theory.
\newblock {\em arXiv preprint arXiv:1704.04291}, 2017.

\bibitem[Kad85]{Kadeishvili}
T.~V. Kadeishvili.
\newblock The category of differential coalgebras and the category of
  {$A(\infty)$}-algebras.
\newblock {\em Trudy Tbiliss. Mat. Inst. Razmadze Akad. Nauk Gruzin. SSR},
  77:50--70, 1985.

\bibitem[K{\'a}l05]{Kal05}
Tam{\'a}s K{\'a}lm{\'a}n.
\newblock Contact homology and one parameter families of {L}egendrian knots.
\newblock {\em Geom. Topol.}, 9:2013--2078, 2005.

\bibitem[K{\'a}l06]{Kal}
Tam{\'a}s K{\'a}lm{\'a}n.
\newblock Braid-positive {L}egendrian links.
\newblock {\em Int. Math. Res. Not.}, pages Art ID 14874, 29, 2006.

\bibitem[Kar17]{Karlsson-ori}
Cecilia Karlsson.
\newblock A note on coherent orientations for exact {L}agrangian cobordisms.
\newblock arXiv:1707.04219, 2017.

\bibitem[Kas84]{kashiwara1984riemann}
Masaki Kashiwara.
\newblock The riemann-hilbert problem for holonomic systems.
\newblock {\em Publications of the Research Institute for Mathematical
  Sciences}, 20(2):319--365, 1984.

\bibitem[Kel01]{Keller}
Bernhard Keller.
\newblock Introduction to {$A$}-infinity algebras and modules.
\newblock {\em Homology Homotopy Appl.}, 3(1):1--35, 2001.

\bibitem[KS94]{KS}
Masaki Kashiwara and Pierre Schapira.
\newblock {\em Sheaves on manifolds}, volume 292 of {\em Grundlehren der
  Mathematischen Wissenschaften [Fundamental Principles of Mathematical
  Sciences]}.
\newblock Springer-Verlag, Berlin, 1994.
\newblock With a chapter in French by Christian Houzel, Corrected reprint of
  the 1990 original.

\bibitem[Lev16]{Leverson}
C.~Leverson.
\newblock Augmentations and rulings of {L}egendrian knots.
\newblock {\em J. Symplectic Geom.}, 14(4):1089--1143, 2016.

\bibitem[Lur09a]{Lurie-DAGV}
Jacob Lurie.
\newblock Derived algebraic geometry {V}: {S}tructured spaces.
\newblock {\em arXiv preprint arXiv:0905.0459}, 2009.

\bibitem[Lur09b]{Lurie-HTT}
Jacob Lurie.
\newblock {\em Higher Topos Theory (AM-170)}, volume 189.
\newblock Princeton University Press, 2009.

\bibitem[Lur15]{lurie-rotation}
Jacob Lurie.
\newblock Rotation invariance in algebraic k-theory.
\newblock {\em preprint}, 2015.

\bibitem[Lur17]{Lurie-HA}
Jacob Lurie.
\newblock Higher algebra.
\newblock {\em Preprint, available at http://www.math.ias.edu/\~{}lurie}, 2017.

\bibitem[Mis03]{Mishachev}
K.~Mishachev.
\newblock The {$N$}-copy of a topologically trivial {L}egendrian knot.
\newblock {\em J. Symplectic Geom.}, 1(4):659--682, 2003.

\bibitem[MS05]{MelvinShrestha}
Paul Melvin and Sumana Shrestha.
\newblock The nonuniqueness of {C}hekanov polynomials of {L}egendrian knots.
\newblock {\em Geom. Topol.}, 9:1221--1252 (electronic), 2005.

\bibitem[Nad09]{N}
David Nadler.
\newblock Microlocal branes are constructible sheaves.
\newblock {\em Selecta Math. (N.S.)}, 15(4):563--619, 2009.

\bibitem[Ng03]{NgCLI}
Lenhard~L. Ng.
\newblock Computable {L}egendrian invariants.
\newblock {\em Topology}, 42(1):55--82, 2003.

\bibitem[Ng10]{NgSFT}
Lenhard Ng.
\newblock Rational symplectic field theory for {L}egendrian knots.
\newblock {\em Invent. Math.}, 182(3):451--512, 2010.

\bibitem[NR13]{NgR}
Lenhard Ng and Daniel Rutherford.
\newblock Satellites of {L}egendrian knots and representations of the
  {C}hekanov-{E}liashberg algebra.
\newblock {\em Algebr. Geom. Topol.}, 13(5):3047--3097, 2013.

\bibitem[NRSS17]{NRSS}
Lenhard Ng, Dan Rutherford, Vivek Shende, and Steven Sivek.
\newblock The cardinality of the augmentation category of a {L}egendrian link.
\newblock {\em Math. Res. Lett.}, 24(6):1845--1874, 2017.

\bibitem[NZ09]{NZ}
David Nadler and Eric Zaslow.
\newblock Constructible sheaves and the {F}ukaya category.
\newblock {\em J. Amer. Math. Soc.}, 22(1):233--286, 2009.

\bibitem[Sab05]{Sab}
Joshua~M. Sabloff.
\newblock Augmentations and rulings of {L}egendrian knots.
\newblock {\em Int. Math. Res. Not.}, 19:1157--1180, 2005.

\bibitem[Sab06]{Sabloff}
Joshua~M. Sabloff.
\newblock Duality for {L}egendrian contact homology.
\newblock {\em Geom. Topol.}, 10:2351--2381 (electronic), 2006.

\bibitem[Sei08]{Seidel}
Paul Seidel.
\newblock {\em Fukaya categories and {P}icard-{L}efschetz theory}.
\newblock Zurich Lectures in Advanced Mathematics. European Mathematical
  Society (EMS), Z\"urich, 2008.

\bibitem[She85]{Shepard}
Allen~Dudley Shepard.
\newblock {\em A cellular description of the derived category of a stratified
  space}.
\newblock ProQuest LLC, Ann Arbor, MI, 1985.
\newblock Thesis (Ph.D.)--Brown University.

\bibitem[Siv11]{sivek-bordered}
Steven Sivek.
\newblock A bordered {C}hekanov-{E}liashberg algebra.
\newblock {\em J. Topol.}, 4(1):73--104, 2011.

\bibitem[Sta63]{Stasheff-II}
James~Dillon Stasheff.
\newblock Homotopy associativity of {$H$}-spaces. {II}.
\newblock {\em Trans. Amer. Math. Soc.}, 108:293--312, 1963.

\bibitem[STZ17]{STZ}
Vivek Shende, David Treumann, and Eric Zaslow.
\newblock Legendrian knots and constructible sheaves.
\newblock {\em Invent. Math.}, 207(3):1031--1133, 2017.

\bibitem[Tab05]{Tabuada}
Goncalo Tabuada.
\newblock Une structure de cat\'{e}gorie de mod\`eles de {Q}uillen sur la
  cat\'{e}gorie des dg-cat\'{e}gories.
\newblock {\em C. R. Math. Acad. Sci. Paris}, 340(1):15--19, 2005.

\bibitem[To{\"e}07]{Toen}
Bertrand To{\"e}n.
\newblock The homotopy theory of {$dg$}-categories and derived {M}orita theory.
\newblock {\em Invent. Math.}, 167(3):615--667, 2007.

\end{thebibliography}

\vspace{20mm}

\end{document}